\documentclass[10pt]{amsart}
\usepackage[margin=0.8in]{geometry} 
\usepackage{amsthm, amsrefs, amsmath,amsfonts,amssymb,euscript,hyperref,graphics,color,slashed,mathrsfs}
\usepackage{enumerate}
\usepackage{graphicx}
\usepackage[toc]{appendix}
\usepackage{comment}
\usepackage{import}
\usepackage{tikz}
\usepackage{latexsym}

\usepackage{pgfplots}  
\pgfplotsset{compat=1.18} 
\usepgfplotslibrary{patchplots} 
 
\usepackage[utf8]{inputenc}
\usepackage{enumitem}

\def\Xh{{\widehat{X}}}
\def\Th{{\widehat{T}}}

\def\Ntop{{N_{_{\rm top}}}}
\def\Ninf{N_{_{\infty}}}

\def\Xr{\mathring{X}}
\def\Tr{\mathring{T}}

\def\Trh{\widehat{\mathring{T}}}

\def\epsilonr{\mathring{\epsilon}}
\def\Mr{\mathring{M}}
\def\Tb{\underline{T}}
\def\Xb{\underline{X}}
\def\Omegab{\underline{\Omega}}

\def\Ur{\mathring{U}}

\def\Zr{{\mathring{Z}}}
\def\Lr{{\mathring{L}}}

\def\Lbr{{\mathring{\underline{L}}}}

\def\zr{{\mathring{z}}}
\def\yr{{\mathring{y}}}

\def\kappar{{\mathring{\kappa}}}
\def\mur{{\mathring{\mu}}}
\def\etar{{\mathring{\eta}}}
\def\deltar{{\mathring{\delta}}}
\def\deltasr{{\mathring{\slashed{\delta}}}}

\def\zetar{{\mathring{\zeta}}}
\def\chir{{\mathring{\chi}}}
\def\chibr{{\mathring{\underline{\chi}}}}

\def\ub{\underline{u}}
\def\wb{\underline{w}}

\def\chib{\underline{\chi}}

\def\Lb{\underline{L}}

\def\kappab{\underline{\kappa}}

\def\Cb{{\underline{C}}}

\def\cb{\underline{c}}
\def\vb{\underline{v}}
\def\psib{\underline{\psi}}
\def\thetab{\underline{\theta}}

\def\Hb{\underline{H}}

\newtheorem{theorem}{Theorem}[section]
\newtheorem{lemma}[theorem]{Lemma}
\newtheorem{proposition}[theorem]{Proposition}
\newtheorem{corollary}[theorem]{Corollary}
\newtheorem{definition}[theorem]{Definition}
\newtheorem{remark}[theorem]{Remark}

\numberwithin{equation}{section}

\newcommand{\missingfigurebox}[1]{%
\fbox{\parbox[c][1.8in][c]{0.78\linewidth}{\centering\small Figure not found\\\texttt{\detokenize{#1}}}}%
}
\let\Oldincludegraphics\includegraphics
\renewcommand{\includegraphics}[2][]{%
\IfFileExists{#2}{\Oldincludegraphics[#1]{#2}}{\missingfigurebox{#2}}%
}

\begin{document}
\title[Multi-Dimensional Structural Stability of Mixed Riemann Configurations]{Multi-Dimensional Structural Stability of Mixed Riemann Configurations Containing Centered Rarefaction Waves and Surfaces of Discontinuities of Gas Dynamics}

\author{Jin JIA and Tao LUO}

\address{School of Mathematics, Hunan University\\ Changsha, China}
\email{jiajin2023@hnu.edu.cn}

\address{Department of Mathematics, City University of Hong Kong\\ Hong Kong, China}
\email{taoluo@cityu.edu.hk}

\begin{abstract}
For 2D compressible Euler equations of isentropic gas, we prove the structural stability of mixed Riemann configurations containing centered rarefaction waves and surfaces of discontinuities
(such as  shock waves or vortex sheets), by deriving simultaneous energy estimates for acoustic and vorticity waves within the rarefaction wave region \emph{without loss of derivatives} and examinations of the nonlinear superpositions of rarefaction waves with other waves such as shock waves or vortex sheets. The nonlinear superpositions of \emph{shock wave-rarefaction wave} and \emph{rarefaction wave-vortex sheet-rarefaction wave} are achieved by reducing the problems in corner regions to the Cauchy problems with the data prescribed on the plane $\Sigma_0=\{(t, x_1, x_2): t=0, (x_1, x_2)\in \mathbb{R}\times\mathbb{R}/2\pi\mathbb{Z}\}$ with discontinuities at $\mathbf{S}_*:=\{(t,x_{1},x_{2})\mid t=0,\ x_{1}=0, x_2\in\mathbb{R}/2\pi\mathbb{Z} \}$.


\end{abstract}
\maketitle
\tableofcontents

\section{Introduction}

Rarefaction waves, shock fronts and vortex sheets are the elementary  wave patterns in gas dynamics. In one space dimension, the Riemann problem can be solved in terms of shocks and rarefaction waves, and the structure of solutions is by now classical. In several space dimensions, however, the structural stability of elementary waves becomes much challenging. For shock fronts, the local existence and stability theory was established by Majda \cite{MajdaShock2, MajdaShock3}. For vortex sheets, the corresponding local stability theory was established by Coulombel and Secchi \cite{Coulombel-Secchi1, Coulombel-Secchi2}. For rarefaction waves, the pioneering works of Alinhac \cite{AlinhacWaveRare1, AlinhacWaveRare2} gave the first multidimensional existence and uniqueness theory for general quasilinear hyperbolic systems.

Centered rarefaction waves in several space dimensions are analytically singular in a manner different from both shock fronts and vortex sheets. In the model one-dimensional fan, the solution itself stays continuous while the transversal derivative of a Riemann invariant behaves like $t^{-1}$ near the tip of the fan. In several dimensions, all outgoing rarefaction fronts emanate from the space time codimension-two set
\[
\mathbf{S}_*:=\{(t,x_{1},x_{2})\mid t=0,\ x_{1}=0, x_{2}\in \mathbb{R}/2\pi\mathbb{Z}\},
\]
the fronts themselves are unknown characteristic hypersurfaces, and the dependence on the transverse variable may usually generate vorticity. Therefore the multidimensional rarefaction problem is simultaneously singular and characteristic: it is neither a standard smooth Cauchy problem nor a standard free-boundary problem. Moreover, a naive differentiation of the eikonal equation for the rarefaction fronts leads to a loss of derivatives at top order.

The present paper is motivated by the structural stability of the mixed Riemann configurations \emph{shock wave-rarefaction wave} and \emph{rarefaction wave-vortex sheet-rarefaction wave}. In both configurations, a rarefaction region must first be constructed and then matched with either a shock front or a vortex sheet. The \(S\!-\!R\) configuration makes the main difficulty especially transparent: in several space dimensions, a shock front typically generates vorticity, so once one tries to carry out a nonlinear superposition of a rarefaction wave with a shock, one is naturally led beyond the irrotational setting. Therefore it is necessary to study rarefaction waves with vorticity, and any theorem intended for the \(S\!-\!R\) problem must allow the rarefaction wave to be coupled with rotational flows produced on the shock side. To be sure, Alinhac's theory \cite{AlinhacWaveRare1, AlinhacWaveRare2} already provides a multidimensional existence and uniqueness theory for single rarefaction waves. However, for the nonlinear superposition problems considered here, one needs a substantially more precise rarefaction theorem. Besides existence in the rarefaction region, one must recover the rarefaction region itself together with the associated inner characteristic boundaries and sufficiently accurate boundary jets along those boundaries, in a form compatible with the shock- and vortex-sheet theories of Majda \cite{MajdaShock2, MajdaShock3} and Coulombel--Secchi \cite{Coulombel-Secchi1, Coulombel-Secchi2}. Alinhac's construction is not directly suited to this purpose, because it loses derivatives, degenerates near the rarefaction fronts, and for Riemann-type data yields only an asymptotic description of the fronts. In this sense, one essential missing ingredient for the present program is a theory of a \emph{single family of rarefaction waves with vorticity} strong enough to be coupled with shock fronts and vortex sheets.

More broadly, the single-rarefaction wave theorem should be viewed as the universal front-end for perturbed Riemann problems containing rarefaction pieces. Once the rarefaction regions are constructed, the remaining interaction zone is naturally formulated as a smooth Goursat problem in the \(R\!-\!R\) regime, as a Goursat problem coupled to a single shock front in the \(S\!-\!R\) regime, and as a Goursat problem coupled to a single vortex sheet in the \(R\!-\!V\!-\!R\) regime. In this sense, the rarefaction theorem must do more than produce a solution inside the fan: it must also determine the nonlinear characteristic foliation, the inner boundaries inherited from the rarefaction waves, and the boundary jets required to solve the remaining interior problem. This viewpoint is also consistent with the compatibility theory for multidimensional Riemann problems developed by Chen--Li \cite{ChenLi}.

The structural stability of the \(S\!-\!R\) and \(R\!-\!V\!-\!R\) patterns is therefore not a routine application of the single-rarefaction theory. The natural problems left after opening the rarefaction fans have data prescribed partly on characteristic hypersurfaces \(H\) and \(\Hb\), and in the shock or vortex-sheet cases these data must also satisfy the nonlinear Rankine--Hugoniot relations and the corresponding high-order compatibility conditions. This gives nonstandard Goursat-type free-boundary problems for a shock front or a vortex sheet, outside the direct form of the classical Cauchy theories of Majda and Coulombel--Secchi. A main point of the present paper is to overcome this difficulty by a reduction from these Goursat problems to suitable Cauchy problems. More precisely, using the extremal structure of the rarefaction boundaries, the propagation of normal jets from the singular set \(\mathbf{S}_{*}\), and the characteristic well-posedness theory, we construct controlled auxiliary Cauchy data so that Majda's shock theory and Coulombel--Secchi's vortex-sheet theory can be applied. This reduction is one of the essential mechanisms by which the structural stability of the mixed \(S\!-\!R\) and \(R\!-\!V\!-\!R\) configurations is obtained.

Our work should also be compared with the recent irrotational theory of Luo and Yu \cite{Luo-YuRare1, Luo-YuRare2}. Their first paper established top-order energy estimates for multi-dimensional centered rarefaction waves without loss of derivatives, and their second paper constructed solutions and applied the theory to the Riemann problem with two families of rarefaction waves. The present paper removes the irrotational assumption and treats the  two-dimensional Euler system with non-zero vorticity. 

At the technical level, the analysis also draws on the acoustical-geometric approach to compressible Euler equations and on the wave-transport formulation for vorticity developed in the shock-formation works of Christodoulou--Miao and Luk--Speck \cite{ChristodoulouMiao, LukSpeck2D}. The present problem is, however, of a different geometric type: the singularity is not a codimension-one shock hypersurface, but the codimension-two tip of a centered rarefaction fan. This distinction is crucial for the analysis. It leads to a regime in which the strongest singularity is carried by the normal derivatives, while the tangential regularity remains comparatively better and can be exploited through a suitable choice of commutators and bootstrap weights. In the wave-transpot formulation, the wave equations are for Riemann invariants, and the transport equation is for the specific vorticity $\Omega$ as defined in \eqref {eq: specific vorticity}. The intrinsic reformulation \eqref{eq: reformulation of vorticity in the first null frame} shows that $\partial_{i}\Omega$ can be rewritten using only the rarefaction front-adapted derivatives. This propagates the good control of $\Omega$ from the right boundary of the rarefaction wave which is completely determined by the initial data into the rarefaction region.  The observation of the non-characteristic nature of the specific velocity  with respect to the rarefaction fronts is crucial. 

Compared with Alinhac's Nash-Moser framework, our theory is less general at the level of systems, since we work specifically with the compressible Euler equations. However, for Euler equations the theory obtained here is substantially sharper and more geometric. As emphasized already in \cite{Luo-YuRare1, Luo-YuRare2}, Alinhac's estimates lose derivatives and degenerate near the rarefaction fronts, while the description of the fronts for Riemann-type data is only asymptotic. In contrast, the present paper, like \cite{Luo-YuRare1, Luo-YuRare2}, works in standard Sobolev spaces and uses the acoustical geometry of Euler equations to obtain a nondegenerate description of the rarefaction region. The novelty here is that this stronger framework survives in the presence of vorticity and becomes usable for mixed-wave problems.

Throughout the paper, we work away from vacuum. This is a natural first step, since even in one space dimension rarefaction waves are the only elementary waves that may generate vacuum from non-vacuum initial states, and the up-to-vacuum regime introduces an additional degeneracy through the sound speed in the acoustical geometry. The multidimensional free expansion of gas into vacuum via rarefaction waves is therefore a natural next problem, but it lies beyond the scope of the present paper. This issue will be addressed in the forthcoming paper. 

\subsection{Review of Riemann problem for 2-D isentropic Euler system with 1-D symmetry} The $2$-D Euler equations are given by:
\begin{equation}\label{eq: Euler equations}
\begin{cases}
\partial_{t}\rho+\nabla_{x}\cdot(\rho v)=0,\\
\partial_{t}(\rho v)+\nabla_{x}\cdot (\rho v\otimes v+p)=0.
\end{cases}
\end{equation}
where $\rho, p, v=(v^{1}, v^{2})^{T}$ are the density, pressure, and velocity of the gas, respectively. The equation of the sate is given by $p(\rho)=k_{0}\rho^{\gamma}$ with constants $k_{0}>0$ and $\gamma\in (1, 3)$. The sound speed $c$ is defined as 
\[
c=\sqrt{\frac{dp}{d\rho}}=k_{0}^{\frac{1}{2}}\gamma^{\frac{1}{2}}\rho^{\frac{\gamma-1}{2}}.
\]

Under $1-D$ symmetry assumption, \eqref{eq: Euler equations} is expressed as:
\begin{equation}\label{eq: Euler equation in the form of conservation law}
\partial_{t}U+\partial_{x_1}F(U)=0,
\end{equation}
where
\[
U=\begin{pmatrix}
\rho\\
\rho v^{1}\\
\rho v^{2}
\end{pmatrix}, F(U)=\begin{pmatrix}
U^{2}\\ 
p(U^{1})+\frac{(U^{2})^2}{U^{1}}\\
\frac{U^{2}U^{3}}{U^1}
\end{pmatrix}.
\]
For a given solution $U$, the matrix $A(U):=DF(U)$ has three eigenvalues $\lambda_{1}(U)=v^{1}-c, \lambda_{2}(U)=v^{1}$, and $\lambda_{3}(U)=v^{1}+c$, which are distinct providing that $c>0$. The Riemann problem is the Cauchy problem with specific data $U(0, x_{1})$ posed on $t=0$:
\begin{equation}\label{eq: constant initial data}
U(0, x_1)=\begin{cases}
U_{l},\ \ x\leq 0,\\
U_{r},\ \ x\geq 0.
\end{cases}
\end{equation}
where $U_{l}=(\rho_{l}, v_{l}^{1}, v_{l}^{2})^{T}$ and $U_{r}=(\rho_{r}, v_{r}^{1}, v_{r}^{2})^{T}$ are two constant states. There are three fundamental solution patterns for the Riemann problem, which connect constant states through shocks, rarefaction waves, or vortex sheet:

\begin{itemize}
\item {\bf Shock Waves}

There are two families of shock waves  associated with the first eigenvalue $\lambda_{1}(U)$ and the third eigenvalue $\lambda_{3}(U)$, respectively. We will focus on the shock wave associated with the first eigenvalue $\lambda_{1}(U)$. A single shock wave associated with the first eigenvalue $\lambda_{1}(U)$ is a piecewise constant solution, consisting of two pieces separated by hyperplane $x_{1}=st$, where $s$ is the shock speed satisfying the Rankine-Hugoniot condition,
\[
s(U_{r}-U_{l})=F(U_{r})-F(U_{l})
\]
and the Lax-shock condition (entropy condition), yielding that
\begin{equation}\label{data: a single shock}
\begin{cases}
v_{r}^{1}-v_{l}^{1}=-\sqrt{\frac{\rho_{l}-\rho_{r}}{\rho_{l}\rho_{r}}\big(p(\rho_{l})-p(\rho_{r})\big)},\\
v_{r}^{2}-v_{l}^{2}=0,\\
 \rho_{r}>\rho_{l}.
\end{cases}
\end{equation}
It can be depicted as follows:
\begin{center}
\begin{tikzpicture}[scale=1]
\draw (-3,1) node {$(\rho_{r}, v^{1}_{r}, v^{2}_{r})$};   
\draw (2,1) node {$(\rho_{l}, v^{1}_{l}, v^{2}_{l})$};     
\draw[thick, color=blue] (0,0)--(-2,2)--(-1.5, 3.5)--(0.5, 1.5);
\draw[thick] (0,0)--(0.5, 1.5);
\draw[thin] (-6,0)--(6,0);
\draw[thin] (-6,2)--(6,2);
\draw[thin] (-5.5, 1.5)--(6, 1.5);
\draw[thin] (-5.5, 3.5)--(6, 3.5);
\draw[below] node {\text{Fig. a single shock wave associated with the first eigenvalue $\lambda_{1}(U)$.}};
 \end{tikzpicture}
 \end{center}

\item {\bf Rarefaction Waves}

There are two families of rarefaction waves associated with the first eigenvalue $\lambda_{1}(U)$ and the third eigenvalue $\lambda_{3}(U)$, respectively. We will focus on that associated with the third eigenvalue $\lambda_{3}(U)$. A single rarefaction wave associated with the third eigenvalue $\lambda_{3}(U)$ is a continuous piecewise smooth self similar solution in variable $\frac{x_1}{t}$, consisting of three pieces separated by hyperplanes $x_{1}=(v_{l}^{1}+c)t$ and $x_{1}=(v_{r}^{1}+c)t$. The middle piece is called the rarefaction wave region, the solution on which takes the form,
\begin{equation}\label{eq: 1-D rarefaction wave}
\begin{cases}
c=\frac{\gamma-1}{\gamma+1}\frac{x}{t}-\frac{\gamma-1}{\gamma+1}v^{1}_{r}+\frac{2}{\gamma+1}c_{r},\\
v^{1}=\frac{2}{\gamma+1}\frac{x}{t}+\frac{\gamma-1}{\gamma+1}v^{1}_{r}-\frac{2}{\gamma+1}c_{r},\\
v^{2}=v^{2}_{r}
\end{cases}
\end{equation}
which implies that 
\[
\begin{cases}
\frac{c_{l}}{\gamma-1}-\frac{1}{2}v_{l}^1=\frac{c_{r}}{\gamma-1}-\frac{1}{2}v_{r}^1,\\
v_{l}^{2}=v_{r}^{2},\\
 \frac{c_{l}}{\gamma-1}+\frac{1}{2}v_{l}^1<\frac{c_{r}}{\gamma-1}+\frac{1}{2}v_{r}^1.
\end{cases}
\]
It can be depicted as follows:
\begin{center}
\begin{tikzpicture}[scale=1]
\draw[thick] (0,0)--(0.5, 1.5);
\draw (4,1) node {$(c_{r}, v^{1}_{r}, v^{2}_{r})$};   
\draw (-1,1) node {$(c_{l}, v^{1}_{l}, v^{2}_{l})$};     
\draw[thin] (-3,0)--(8.5,0);
\draw[thin] (-3,2)--(8.5,2);
\draw[thin] (-2.5, 1.5)--(9, 1.5);
\draw[thin] (-2.5, 3.5)--(9, 3.5);
\draw[color=green, thick, dotted] (0, 0)--(2, 2)--(2.5, 3.5)--(0.5,1.5);
\draw[color=blue, thick, dotted] (0, 0)--(2.5, 2)--(3, 3.5)--(0.5,1.5);
\draw[color=blue, dotted, thick] (0, 0)--(3, 2)--(3.5, 3.5)--(0.5,1.5);
\draw[color=blue, dotted, thick] (0, 0)--(3.5, 2)--(4, 3.5)--(0.5,1.5);
\draw[color=blue, dotted, thick] (0, 0)--(4, 2)--(4.5, 3.5)--(0.5,1.5);
\draw[color=blue, dotted, thick] (0, 0)--(4.5, 2)--(5, 3.5)--(0.5,1.5);
\draw[color=red, dotted, thick] (0, 0)--(5, 2)--(5.5, 3.5)--(0.5,1.5);
\draw[below] (3,0) node {\text{Fig. a single family rarefaction waves associated with the third eigenvalue $\lambda_{3}(U)$.}};
\end{tikzpicture}
\end{center}

\item Vortex sheet

There is a single family of vortex sheet associated with the second eigenvalue $\lambda_{2}(U)$, which is a piecewise constant solution, consisting of two pieces separated by hyperplane $x_{1}=st$, where $s$ is the shock speed. It satisfies the Rankine-Hugoniot condition,
\[
s(U_{r}-U_{l})=F(U_{r})-F(U_{l})
\]
The \emph{linear degeneration} of $\lambda_{2}(U)$ implies that
\begin{equation}\label{eq: data for a single vortex sheet}
\begin{cases}
c_{l}=c_{r},\\
v_{l}^{1}=v^{1}_{r},\\
v_{l}^{2}\not=v_{r}^{2}.
\end{cases}
\end{equation}
It can be depicted as follows:
\begin{center}
\begin{tikzpicture}
\draw (3,1) node {$(c_{r}, v^{1}_{r}, v^{2}_{r})$};   
\draw (-2,1) node {$(c_{l}, v^{1}_{l}, v^{2}_{l})$};     
\draw[thick, color=blue, dashed] (0,0)--(2,2)--(2.5, 3.5)--(0.5, 1.5);
\draw[thick] (0,0)--(0.5, 1.5);
\draw[thin] (-6,0)--(8,0);
\draw[thin] (-6,2)--(8,2);
\draw[thin] (-5.5, 1.5)--(8, 1.5);
\draw[thin] (-5.5, 3.5)--(8, 3.5);
\draw[below] (1,0) node {\text{Fig. a single family rarefaction waves associated with the third eigenvalue $\lambda_{2}(U)$.}};
 \end{tikzpicture}
\end{center}

\end{itemize}
Based on the analysis in phase space, the Riemann Problem for \eqref{eq: Euler equation in the form of conservation law} with general data of form \eqref{eq: constant initial data} can be resolved by one of the following \emph{non-degenerate patterns},
\begin{itemize}
\item shock wave-vortex sheet-shock wave, $S-V-S$, 
\item shock wave-vortex sheet-rarefaction wave, $S-V-S$, 
\item rarefaction wave-vortex sheet-shock wave, $R-V-S$, 
\item rarefaction wave-vortex sheet-rarefaction wave, $R-V-R$, 
\item rarefaction wave-vacuum state-rarefaction wave, $R-V-R$. 
\end{itemize}
The Riemann Problem for \eqref{eq: Euler equation in the form of conservation law} with data of form \eqref{eq: constant initial data} with restriction $v_{l}^{2}=v_{r}^{2}$ can be resolved by one of the following \emph{degenerate patterns},
\begin{itemize}
\item shock wave-shock wave, $S-S$, 
\item shock wave-rarefaction wave, $S-R$, 
\item rarefaction wave-shock wave, $R-S$, 
\item rarefaction wave-rarefaction wave, $R-R$, 
\item rarefaction wave-vacuum state-rarefaction wave, $R-V-R$. 
\end{itemize}
In this work, we focus on the configuration \emph{shock wave-rarefaction wave}\footnote{Due to symmetry, the configuration \emph{rarefaction wave-shock wave} can be studied with the same method.} from \emph{degenerate patterns}, and the configuration \emph{rarefaction wave-vortex sheet-rarefaction wave} from \emph{non-degenerate patterns} from \emph{non-degenerate patterns}. We will demonstrate that local existence of these configurations with $O(\varepsilon)$-perturbation of data  \eqref{eq: constant initial data} under \emph{necessary} compatibility conditions.

\begin{remark}
The relative entropy method, developed by DiPerna \cite{DiPerna79} and Dafermos \cite{Dafermos1}, is a classical tool for studying the uniqueness and stability of entropy solutions to hyperbolic conservation laws in the one-dimensional setting.

For Riemann solutions consisting solely of rarefaction waves, uniqueness in the class of $L^\infty$-entropy weak solutions can be rigorously established via this relative entropy framework; we refer to the work of Chen and Chen \cite{Chen-Chen} for a complete treatment.

For the wave pattern of shock-rarefaction,  it is not difficult to prove the uniqueness in the class of piecewise smooth entropy solutions with the same wave pattern, by the uniform stability estimates for shocks in \cite{MajdaShock2, MajdaShock3} and the uniform estimates without loss of derivatives for rarefaction waves obtained in this paper.  The uniqueness in the less regular class of $L^{\infty}$ admissible entropy solutions of Riemann solutions  involving shock waves or vortex sheets may not hold in stark contrast ( \cite{ChiodaroliDeLellisKreml2015} and  \cite{KrupaSzékelyhidi}).  For multi-dimensional scalar conservation laws, we refer to the foundational work of Kružkov \cite{Kruzkov1970}.
\end{remark}

\subsection{The background solution}

We provide a detailed calculation of the Riemann problem with configuration \emph{shock wave-rarefaction wave} and \emph{rarefaction wave-vortex sheet-rarefaction wave}. These will serve as background solutions for our subsequent analysis of the multi-dimensional case.

In accordance with Reimann's notation, we define the Riemann invariants $(\wb, w)$ relative to the solution $(v^{1}, v^{2}, c)$ with direction $e_{1}=(1, 0)$,
\[
\begin{cases}
\underline{w}:=\frac{c}{\gamma-1}+\frac{1}{2}v^{1},\\
w:=\frac{c}{\gamma-1}-\frac{1}{2}v^{1}.
\end{cases}
\]

\subsubsection{The Riemann problem with configuration shock wave-rarefaction wave}\label{subsubsec: S-R} 
In this case, it is required that for the initial data \eqref{eq: constant initial data} there exists a constant middle state $U_{m}$ such that the $(U_{l}, U_{m})$ can be connected with a single shock wave associated with the first eigenvalue $\lambda_{1}(U)$, $(U_{m}, U_{r})$ can be connected with a single family of rarefaction waves associated with the third eigenvalue $\lambda_{3}(U)$. To be more precise, we require that the nonlinear equation
\[
\begin{cases}
v_{m}^{1}-v_{l}^{1}=\sqrt{\frac{\rho_{l}-\rho_{m}}{\rho_{l}\rho_{m}}\big(p(\rho_{l})-p(\rho_{m})\big)},\\
w_{m}=w_{r},\\
v_{l}^{2}=v^{2}_{m}=v_{r}^{2}.
\end{cases}
\]
has a unique solution $U_{m}$ with $\rho_{m}>0$. The configuration can be depicted as follows:
\begin{center}
\begin{tikzpicture}[scale=1.5]
\coordinate (S) at (0,0);
\draw[thick] (-3,0) -- (3,0); 
\node[right] at (3,0) {$\Sigma_{0}$};
\node[below] at (0,0) {$x_{1}=0$};
\node[above] at (-1,2) {$S$};
\draw[thick] (0,0)--(0,0);
\draw[thick] (0,0)--(-1,2);
\node at (-2, 1.2) {$U_{l}$};
\node at (0.3, 1.2) {$U_{m}$};
 \node at (2.5, 1.2) {$U_{r}$};
\draw[dashed] (0,0)--(2.5,2);
\draw[dashed] (0,0)--(1.8,2);
\draw[dashed] (0,0)--(2.1,2);
\draw[dashed] (0,0)--(1.5,2);
\node[above] at (2.5, 2) {$C_{0}$};
\node[above] at (1.5, 2) {$H$};
\node at (1.25, 1.2) {$U_{R}$};
\end{tikzpicture}
\end{center}

together with its multi-dimensional version:

\begin{center}
\begin{tikzpicture}[scale=1]
  \begin{axis}[axis lines=none]
    
    


     \addplot [domain=0:10, smooth, color=blue, thick, dotted] {0.2*x};
    \addplot [domain=0.5:10.5, smooth, color=blue, thick, dotted] {0.2*(x-0.5)+1.5};
    \addplot [domain=10:10.5, smooth, color=blue, thick, dotted] {3*(x-10)+2};

    \addplot [domain=0:-4, smooth, thick, color=blue] {-0.5*x};
    \addplot [domain=0.5:-3.5, smooth, thick, color=blue] {-0.5*(x-0.5)+1.5};
    \draw [thick] (0, 0)--(0.5,1.5);
    \addplot [domain=-4:-3.5, smooth, thick, color=blue] {3*(x+4)+2};
    
    \draw[thin] (-6,0)--(10.5,0);
    \draw[thin] (-6,2)--(10.5,2);
     \draw[thin] (-5.5, 1.5)--(11, 1.5);
    \draw[thin] (-5.5, 3.5)--(11, 3.5);
    
    \draw[color=blue, thick, dotted] (0, 0)--(6, 2)--(6.5, 3.5)--(0.5,1.5);
    \draw[color=blue, dotted, thick] (0, 0)--(7, 2)--(7.5, 3.5)--(0.5,1.5);
    \draw[color=blue, dotted, thick] (0, 0)--(8, 2)--(8.5, 3.5)--(0.5,1.5);
    \draw[color=blue, dotted, thick] (0, 0)--(9, 2)--(9.5, 3.5)--(0.5,1.5);

  \end{axis}
\end{tikzpicture}
\end{center}

We outline the basic construction of the configuration above, which will be generalized to multi-dimensional case in section \ref{sec: Application to the Riemann problem 2}.
\begin{itemize}
\item {\bf Step 1} The outer boundary $C_{0}$ of the right rarefaction wave region is given by $x_{1}=(v_{r}^{1}+c_{r})t$, which is completely determined by $U_{r}$. 
\item {\bf Step 2} The solution in rarefaction wave region can be constructed as follows
\[
\begin{cases}
c_{R}=\frac{\gamma-1}{\gamma+1}\frac{x_1}{t}-\frac{\gamma-1}{\gamma+1}v^{1}_{r}+\frac{2}{\gamma+1}c_{r},\\
v^{1}_{R}=\frac{2}{\gamma+1}\frac{x_1}{t}+\frac{\gamma-1}{\gamma+1}v^{1}_{r}-\frac{2}{\gamma+1}c_{r},\\
v^{2}_{R}=v^{2}_{r}
\end{cases}
\]
\item {\bf Step 3} The inner boundary $H$ of the rarefaction wave region is given by $x_{1}=kt$, where $k$ is determined by solving the nonlinear equation
\[
v_{R}^{1}-v_{l}^{1}=\sqrt{\frac{\rho_{l}-\rho_{R}}{\rho_{l}\rho_{R}}\big(p(\rho_{l})-p(\rho_{R})\big)},
\]
\item {\bf Step 4} The middle state $U_{m}$ is determined by $U_{m}=U_{R}$ on $H$. 
\item {\bf Step 5} The shock front $S$ is given by $x_{1}=st$, where $s$ is determined by Rankine-Hugoniot condition,
\[
s(U_{l}-U_{m})=F(U_{l})-F(U_{m}).
\]
\end{itemize}
and the Lax-shock condition (the entropy condition). 
\subsubsection{The Riemann problem with configuration rarefaction wave-vortex sheet-rarefaction wave}\label{subsubsec: R-V-R}
In this case, it is required that the initial data \eqref{eq: constant initial data} satisfies
\begin{equation}
\begin{cases}
\wb_{l}<\wb_{r},\\
w_{l}>w_{r},\\
v^{2}_{l}\not=v^{2}_{r}, w_{r}+\wb_{l}>0.
\end{cases}
\end{equation}
The configuration can be depicted as follows:
\begin{center}
\begin{tikzpicture}[scale=1.5]
\coordinate (S) at (0,0);
\draw[thick] (-3,0) -- (3,0); 
\node[right] at (3,0) {$\Sigma_{0}$};
\node[below] at (0,0) {$x_{1}=0$};
\draw[thick] (0,0)--(0,0);
\node at (2.5, 1.2) {$U_{r}$};
\node at (-2.5, 1.2) {$U_{l}$};
\draw[dashed] (0,0)--(2.5,2);
\draw[dashed] (0,0)--(2.2,2);
\draw[dashed] (0,0)--(1.9,2);
\draw[dashed] (0,0)--(1.6,2);
\draw[dashed] (0,0)--(1.3,2);
\draw[dashed] (0,0)--(1,2);
\node at (0.9,1.2) {$U_{R}$};
\node at (0.3,1.2) {$U_{m, r}$};
\draw[dashed] (0,0)--(-2,2);
\draw[dashed] (0,0)--(-1.7,2);
\draw[dashed] (0,0)--(-1.4,2);
\draw[dashed] (0,0)--(-1.1,2);
\draw[dashed] (0,0)--(-0.8,2);
\draw[dotted, thick] (0,0)--(0.1,2);
\node[above] at (0.1, 2) {$V$};
\node at (-0.7,1.2) {$U_{L}$};
\node at (-0.2,1.2) {$U_{m, l}$};
\node[above] at (2.5, 2) {$C_{0}$};
\node[above] at (1, 2) {$H$};
\node[above] at (-2, 2) {$\Cb_{0}$};
\node[above] at (-0.8, 2) {$\Hb$};
\end{tikzpicture}
\end{center}
together with its multi-dimensional version:
\begin{center}
\begin{tikzpicture}[scale=1]
  \begin{axis}[axis lines=none]
    
    \addplot [domain=0:2.5, smooth, thick, dashed, color=blue] {0.8*x };
    \addplot [domain=0.5:3, smooth, thick, dashed, color=blue] {0.8*(x-0.5)+1.5};
    \addplot [domain=3:2.5, smooth, thick, dashed, color=blue] {3*(x-2.5)+2};
    
     
    \addplot [domain=0:-2, smooth, color=blue, dotted, thick] {-x};
    \addplot [domain=0.5:-1.5, smooth, color=blue, dotted, thick] {-x+2};
    \addplot [domain=-2:-1.5, smooth, color=blue, dotted, thick] {3*(x+2)+2};


     \addplot [domain=0:10, smooth, color=blue, dotted, thick] {0.2*x};
    \addplot [domain=0.5:10.5, smooth, color=blue, dotted, thick] {0.2*(x-0.5)+1.5};
    \addplot [domain=10:10.5, smooth, color=blue, dotted, thick] {3*(x-10)+2};

    \addplot [domain=0:-4, smooth, color=blue, dotted, thick] {-0.5*x};
    \addplot [domain=0.5:-3.5, smooth, color=blue, dotted, thick] {-0.5*(x-0.5)+1.5};
    \draw [thick] (0, 0)--(0.5,1.5);
    \addplot [domain=-4:-3.5, smooth, color=blue, dotted, thick] {3*(x+4)+2};
    
    \draw[thin] (-6,0)--(10.5,0);
    \draw[thin] (-6,2)--(10.5,2);
     \draw[thin] (-5.5, 1.5)--(11, 1.5);
    \draw[thin] (-5.5, 3.5)--(11, 3.5);
    
    \draw[color=blue, thick, dotted] (0, 0)--(6, 2)--(6.5, 3.5)--(0.5,1.5);
    \draw[color=blue, dotted, thick] (0, 0)--(7, 2)--(7.5, 3.5)--(0.5,1.5);
    \draw[color=blue, dotted, thick] (0, 0)--(8, 2)--(8.5, 3.5)--(0.5,1.5);
    \draw[color=blue, dotted, thick] (0, 0)--(9, 2)--(9.5, 3.5)--(0.5,1.5);
    
     \draw[color=blue, thick, dotted] (0, 0)--(-5, 2)--(-4.5, 3.5)--(0.5,1.5);
    \draw[color=blue, dotted, thick] (0, 0)--(-3, 2)--(-2.5, 3.5)--(0.5,1.5);
    \draw[color=blue, dotted, thick] (0, 0)--(8, 2)--(8.5, 3.5)--(0.5,1.5);
    \draw[color=blue, dotted, thick] (0, 0)--(9, 2)--(9.5, 3.5)--(0.5,1.5);

  \end{axis}
\end{tikzpicture}
\end{center}

We outline the basic construction of the configuration above, which will be generalized to multi-dimensional case in section \ref{sec: Application to the Riemann problem 3}.
\begin{itemize}
\item {\bf Step 1} The outer boundary $C_{0}$ of the rarefaction wave region is given by $x_{1}=(v_{r}^{1}+c_{r})t$, which is completely determined by $U_{r}$. The outer boundary $\Cb_{0}$ of the left rarefaction wave region is given by $x_{1}=(v_{l}^{1}-c_{l})t$, which is completely determined by $U_{r}$.
\item {\bf Step 2} The solution in the right rarefaction wave region can be constructed as follows
\[
\begin{cases}
c_{R}=\frac{\gamma-1}{\gamma+1}\frac{x_1}{t}-\frac{\gamma-1}{\gamma+1}v^{1}_{r}+\frac{2}{\gamma+1}c_{r},\\
v^{1}_{R}=\frac{2}{\gamma+1}\frac{x_1}{t}+\frac{\gamma-1}{\gamma+1}v^{1}_{r}-\frac{2}{\gamma+1}c_{r},\\
v^{2}_{R}=v^{2}_{r}
\end{cases}
\]
The solution in the left rarefaction wave region can be constructed as follows
\[
\begin{cases}
c_{L}=-\frac{\gamma-1}{\gamma+1}\frac{x_1}{t}+\frac{\gamma-1}{\gamma+1}v^{1}_{l}+\frac{2}{\gamma+1}c_{l},\\
v^{1}_{L}=\frac{2}{\gamma+1}\frac{x_1}{t}+\frac{\gamma-1}{\gamma+1}v^{1}_{l}+\frac{2}{\gamma+1}c_{l},\\
v^{2}_{L}=v^{2}_{l}
\end{cases}
\]

\item {\bf Step 3} The inner boundary $H$ of the right rarefaction wave region is given by $x_{1}=k_{r}t$, where $k_{r}$ is determined by solving the linear equation $\wb_{R}=\wb_{l}$. The inner boundary $\Hb$ of the left rarefaction wave region is given by $x_{1}=k_{l}t$, where $k_{l}$ is determined by solving the linear equation $w_{L}=w_{r}$. 
\item {\bf Step 4} The middle-right state $U_{m, r}$ is determined by $U_{m, r}=U_{R}$ on $H$. The middle-left state $U_{m, l}$ is determined by $U_{m, l}=U_{L}$ on $\Hb$. Moreover, we have $v_{m, l}^{1}=v_{m, r}^{1}$ and we can write $v_{m}^{1}=v_{m, l}^{1}=v_{m, r}^{1}$ without causing any confusion.
\item {\bf Step 5} The vortex sheet $V$ is given by $x_{1}=v_{m}^{1}t$.
\end{itemize}

\subsection{Assumptions on the initial data} For the rest of the work, we will use the three fixed background solution from configurations of \emph{rarefaction wave-rarefaction wave},  \emph{shock wave-rarefaction wave} and \emph{rarefaction wave-vortex sheet-rarefaction wave}.

\subsubsection{The configuration of rarefaction wave-rarefaction wave} In this case, we fix two constant states $\Ur_{l}$ and $\Ur_{r}$, admitting the \emph{rarefaction wave-rarefaction wave configuration} with additional condition $v_{l}^{2}=v_{r}^{2}$. We study the two dimensional isentropic Euler equations \eqref{eq: Euler equations}. We assume the solution is defined on
\[
\Sigma_{0}:=\mathbb{R}\times \mathbb{R}/2\pi\mathbb{Z}=\{(t, x_{1}, x_{2}), t=0, x_{1}\in \mathbb{R}, 0\leq x_{2}\leq 2\pi\}.
\]
The initial data of the equations \eqref{eq: Euler equations} are posed on $\Sigma_{0}$:
\[
U|_{t=0}=\begin{cases}
U_{l}(0, x_{1}, x_{2}), x_{1}\leq 0;\\
U_{r}(0, x_{1}, x_{2}), x_{2}\geq 0.
\end{cases}
\]
We use $\varepsilon$ to quantify the size of the perturbation and we assume the initial data $U|_{t=0}$ is an $\varepsilon-$ perturbated Riemann data, i.e. it satisfies the following definition:

\begin{definition}\label{def:data for R-R} For a given set of initial data $(v,c)\big|_{t=0}$, we say that it is an {\bf $\varepsilon$-perturbed Riemann data} for the Euler equations \eqref{eq: Euler equations} in the sense of admitting the \emph{rarefaction wave-rarefaction wave configuration}, if it satisfies the following assumptions:
	\begin{itemize}
		\item[1)]{\bf (smoothness)} $U_{l}(0, x_{1}, x_{2})$ is smooth on $x_1 \leqslant 0$ (up the boundary); $U_{r}(0, x_{1}, x_{2})$ are smooth on $x_1 \geqslant 0$ (up the boundary).
		\item[2)]{\bf (smallness and localization)} $U|_{t=0}$ is an $O(\varepsilon)$-purterbation of data $\Ur_{l}$ and $\Ur_{r}$ defined as above. In other words, there exists a small constant $\varepsilon>0$ whose size will be determined in the course of the proof and a sufficiently large integer $N_{0}$ so that
		\[\begin{cases}
			\|v^1_l(0,x)-\mathring{v}_l\|_{H^{N_0}(\Sigma_-)}+\|v^2_l(0,x)\|_{H^{N_0}(\Sigma_-)}+\|c_l(0,x)-\mathring{c}_l\|_{H^{N_0}(\Sigma_-)}<\varepsilon,\\
			\|v^1_r(0,x)-\mathring{v}_r\|_{H^{N_0}(\Sigma_+)}+\|v^2_r(0,x)\|_{H^{N_0}(\Sigma_+)}+\|c_r(0,x)-\mathring{c}_r\|_{H^{N_0}(\Sigma_+)}<\varepsilon,
		\end{cases}
		\]
		where $\Sigma_-=\mathbb{R}_{\leqslant 0}\times [0,2\pi]$ and $\Sigma_+=\mathbb{R}_{\geqslant 0}\times [0,2\pi]$.
		\item[3)]{\bf (Compatibility conditions up to order $N_{0}$)} The initial data $U|_{t=0}$ satisfies higher order compatibility condition on $x_{1}=0$,
		\begin{equation}\label{condition: R-R}
\begin{cases}
v_{l}^{2}(0, x_{2})=v_{r}^{2}(0, x_{2}),\\
\frac{1}{\rho_{l}^{k}}\partial_{1}^{k}v_{l}^{2}(0, x_{2})=\frac{1}{\rho_{r}^{k}}\partial_{1}^{k}v_{r}^{2}(0, x_{2})+F_{k}(U_{l}, U_{r}), 1\leq k\leq N_{0}.
\end{cases}
\end{equation}
                where $F_{k}(\cdots)$ represent some smooth function in variables $\partial_{x_1}^{a}\partial_{x_2}^{b}(U_{l}, U_{r})(0, x_{2})$ of size $O(\varepsilon)$ with $a\leq k-1$ and $a+b\leq k$, whose exact form will not be used.
	\end{itemize}
\end{definition}

The main results of our analysis for the Riemann problem with configuration of \emph{rarefaction wave-rarefaction wave} can be summarized as follows:
\begin{center}
\includegraphics[width=3in]{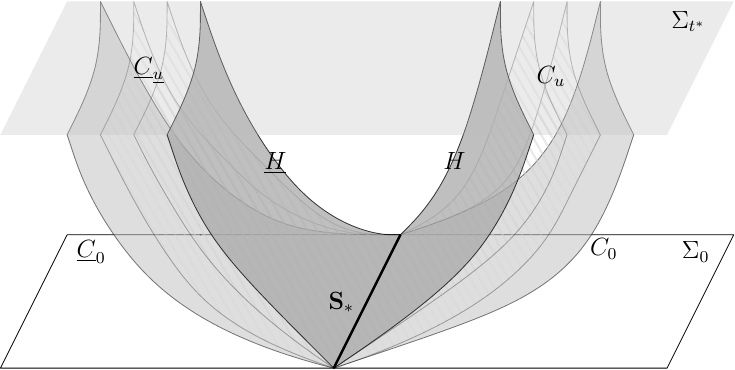}
\end{center}

\medskip

{\bf (Rough statement of the main result for the $R$-$R$ configuration)} \ \ For any positive constant $t^{*}$, the Cauchy problem for Euler equations \eqref{eq: Euler equations} with $\varepsilon-$ perturbated Riemann data $U|_{t=0}$ satisfying definition \ref{def:data for R-R}
admits an entropy solution $U$ on $[0, t^{*}]\times \mathbb{R}\times \mathbb{R}/2\pi\mathbb{Z}$ which is $O(\varepsilon)$ close to the \emph{rarefaction wave-rarefaction wave} configuration with initial data $\Ur=(\Ur_{l}, \Ur_{r})$ providing that $\varepsilon$ is sufficiently small with respect to $\Ur, N_{0}, t^{*}$. Moreover, the solution $U$ connects $U_{l}$ and $U_{r}$ through a smooth superposition of two genuinely nonlinear rarefaction waves, and no other discontinuity appears, which can be depicted as follows:

\subsubsection{The configuration of shock wave-rarefaction wave} In this case, we fix two constant states $\Ur_{l}$ and $\Ur_{r}$, admitting the \emph{shock wave-rarefaction wave configuration} in the sense of \ref{subsubsec: S-R}. We study the two dimensional isentropic Euler equations \eqref{eq: Euler equations}. We assume the solution is defined on
\[
\Sigma_{0}:=\mathbb{R}\times \mathbb{R}/2\pi\mathbb{Z}=\{(t, x_{1}, x_{2}), t=0, x_{1}\in \mathbb{R}, 0\leq x_{2}\leq 2\pi\}.
\]
The initial data of the equations \eqref{eq: Euler equations} are posed on $\Sigma_{0}$:
\[
U|_{t=0}=\begin{cases}
U_{l}(0, x_{1}, x_{2}), x_{1}\leq 0;\\
U_{r}(0, x_{1}, x_{2}), x_{2}\geq 0.
\end{cases}
\]
We use $\varepsilon$ to quantify the size of the perturbation and we assume the initial data $U|_{t=0}$ is an $\varepsilon-$ perturbated Riemann data, i.e. it satisfies the following definition:

\begin{definition}\label{def:data for S-R} For a given set of initial data $(v,c)\big|_{t=0}$, we say that it is an {\bf $\varepsilon$-perturbed Riemann data} for the Euler equations \eqref{eq: Euler equations} in the sense of admitting the \emph{shock wave-rarefaction wave configuration}, if it satisfies the following assumptions:
	\begin{itemize}
		\item[1)]{\bf (smoothness)} $U_{l}(0, x_{1}, x_{2})$ is smooth on $x_1 \leqslant 0$ (up the boundary); $U_{r}(0, x_{1}, x_{2})$ are smooth on $x_1 \geqslant 0$ (up the boundary).
		\item[2)]{\bf (smallness and localization)} $U|_{t=0}$ is an $O(\varepsilon)$-purterbation of data $\Ur_{l}$ and $\Ur_{r}$ defined as above. In other words, there exists a small constant $\varepsilon>0$ whose size will be determined in the course of the proof and a sufficiently large integer $N_{0}$ so that
		\[\begin{cases}
			\|v^1_l(0,x)-\mathring{v}_l\|_{H^{N_0}(\Sigma_-)}+\|v^2_l(0,x)\|_{H^{N_0}(\Sigma_-)}+\|c_l(0,x)-\mathring{c}_l\|_{H^{N_0}(\Sigma_-)}<\varepsilon,\\
			\|v^1_r(0,x)-\mathring{v}_r\|_{H^{N_0}(\Sigma_+)}+\|v^2_r(0,x)\|_{H^{N_0}(\Sigma_+)}+\|c_r(0,x)-\mathring{c}_r\|_{H^{N_0}(\Sigma_+)}<\varepsilon,
		\end{cases}
		\]
		where $\Sigma_-=\mathbb{R}_{\leqslant 0}\times [0,2\pi]$ and $\Sigma_+=\mathbb{R}_{\geqslant 0}\times [0,2\pi]$.
		\item[3)]{\bf (Compatibility conditions up to order $N_{0}$)} The initial data $U|_{t=0}$ satisfies higher order compatibility condition on $x_{1}=0$,
		\begin{equation}\label{condition: S-R}
		\begin{cases}
		v_{l}^{2}(0, x_{2})=v_{r}^{2}(0, x_{2}),\\
                (\frac{v^{1}_{l}-\sigma}{v_{H}^{1}-\sigma})^{k}\partial_{x_1}^{k}v_{l}^{2}(0, x_2)=\frac{\rho_{H}^{k}}{\rho_{r}^{k}}(0, x_{2})\partial_{x_1}^{k}v_{r}^{2}(0, x_{2})+F_{k}(U_{l}, U_{r})(0, x_{2}), 1\leq k\leq N_{0}.
                \end{cases}
                \end{equation}
                where $\rho_{H}(0, x_{2}), \sigma(0, x_{2}), v_{H}^{1}(0, x_{2})$ are determined by $(U_{l}, U_{r})(0, x_{2})$, and $F_{k}(\cdots)$ represent some smooth function in variables $\partial_{x_1}^{a}\partial_{x_2}^{b}(U_{l}, U_{r})(0, x_{2})$ of size $O(\varepsilon)$ with           $a\leq k-1$ and $a+b\leq k$, whose exact form will not be used.

	\end{itemize}
\end{definition}

The main results of our analysis for the Riemann problem with configuration of \emph{shock wave-rarefaction wave} can be summarized as follows:

\medskip

{\bf (Rough statement of the first part of the main result)} \ \ For any positive constant $t^{*}$, the Cauchy problem for Euler equations \eqref{eq: Euler equations} with $\varepsilon-$ perturbated Riemann data $U|_{t=0}$ satisfying definition \ref{def:data for S-R}
admits entropy solution $U$ on $[0, t^{*}]\times \mathbb{R}\times \mathbb{R}/2\pi\mathbb{Z}$ which is $O(\varepsilon)$ close to the \emph{shock wave-rarefaction wave} configuration in \ref{subsubsec: S-R} with initial data $\Ur=(\Ur_{l}, \Ur_{r})$ providing that $\varepsilon$ is sufficiently small with respect to $\Ur, N_{0}, t^{*}$. Moreover, the solution $U$ connects $U_{l}$ and $U_{r}$ through a smooth superposition of a genuinely nonlinear shock wave and a genuinely nonlinear rarefaction wave, and no other discontinuity appears, which can be depicted as follows:
\begin{center}
\begin{tikzpicture}
  \begin{axis}[axis lines=none]
  \addplot [domain=0:5, smooth, color=blue, dotted, thick] {0.4*x+0.002*sin(540*x)*x*(x-5)+0.05*x*(x-5)};
    \addplot [domain=0.5:5.5, smooth, color=blue, dotted, thick] {0.4*(x-0.5)+1.5+0.002*sin(540*x)*(x-0.5)*(x-5.5)+0.04*(x-0.5)*(x-5.5)};
    \draw [thick] (0, 0)--(0.5,1.5);
    \addplot [domain=5:5.5, smooth, color=blue, dotted, thick] {3*(x-5)+2+2*sin(1080*x-120)*(x-5)*(x-5.5)};
   \draw (5, 2.7) node {$H$};
     \addplot [domain=0:10, smooth, color=blue, dotted, thick] {0.2*x+0.0005*sin(360*x)*x*(x-10)+0.01*x*(x-10)};
    \addplot [domain=0.5:10.5, smooth, color=blue, dotted, thick] {0.2*(x-0.5)+1.5+0.0005*sin(360*x)*(x-0.5)*(x-10.5)+0.01*(x-0.5)*(x-10.5)};
    \draw [thick] (0, 0)--(0.5,1.5);
    \addplot [domain=10:10.5, smooth, color=blue, dotted, thick] {3*(x-10)+2+2*sin(1080*x-60)*(x-10)*(x-10.5)};
    \draw (10, 2.7) node {$C_{0}$};

     \addplot [domain=0:-4, smooth, thick, color=blue] {-0.5*x+0.05*x*(x+4)+0.01*sin(1080*x*(x+4))};
    \addplot [domain=0.5:-3.5, smooth, thick, color=blue] {-0.5*(x-0.5)+1.5+0.05*(x-0.5)*(x+3.5)+0.01*sin(1080*(x-0.5)*(x+3.5))};
    \draw [thick] (0, 0)--(0.5,1.5);
    \addplot [domain=-4:-3.5, smooth, thick, color=blue] {3*(x+4)+2+2*sin(1080*x-30)*(x+4)*(x+3.5)};
    \draw (-4, 2.7) node {$S$};
    \draw[thin] (-6,0)--(10.5,0);
    \draw[thin] (-6,2)--(10.5,2);
    
    \draw[thin] (-5.5, 1.5)--(11, 1.5);
    \draw[thin] (-5.5, 3.5)--(11, 3.5);
  \end{axis}
\end{tikzpicture}
\end{center}

\medskip

\subsubsection{The configuration of rarefaction wave-vortex sheet-rarefaction wave} In this case, we fix two constant states $\Ur_{l}$ and $\Ur_{r}$\footnote{It won't cause any confusion since we will treat two cases separately.}, admitting the \emph{rarefaction wave-vortex sheet-rarefaction wave configuration} in the sense of \ref{subsubsec: R-V-R}. Moreover, we assume that the data $(\Ur_{l}, \Ur_{r})$ satisfies the {\bf Super sonic condition}:
\begin{equation}\label{eq: super sonic condition}
|\mathring{v}_{l}^{2}-\mathring{v}_{r}^{2}|>2\sqrt{2}\cdot \frac{\gamma-1}{2}(\mathring{w}_{r}+\mathring{\wb}_{l}),
\end{equation}
We study the two dimensional isentropic Euler equations \eqref{eq: Euler equations}. We assume the solution is defined on
\[
\Sigma_{0}:=\mathbb{R}\times \mathbb{R}/2\pi\mathbb{Z}=\{(t, x_{1}, x_{2}), t=0, x_{1}\in \mathbb{R}, 0\leq x_{2}\leq 2\pi\}.
\]
The initial data of the equations \eqref{eq: Euler equations} are posed on $\Sigma_{0}$:
\[
U|_{t=0}=\begin{cases}
U_{l}(0, x_{1}, x_{2}), x_{1}\leq 0;\\
U_{r}(0, x_{1}, x_{2}), x_{2}\geq 0.
\end{cases}
\]
We use $\varepsilon$ to quantify the size of the perturbation and we assume the initial data $U|_{t=0}$ is an $\varepsilon-$ perturbated Riemann data, i.e. it satisfies the following definition:

\begin{definition}\label{def:data for R-V-R} For a given set of initial data $(v,c)\big|_{t=0}$, we say that it is an {\bf $\varepsilon$-perturbed Riemann data} for the Euler equations \eqref{eq: Euler equations} in the sense of admitting the \emph{rarefaction wave-vortex sheet-rarefaction wave configuration}, if it satisfies the following assumptions:
	\begin{itemize}
		\item[1)]{\bf (smoothness)} $U_{l}(0, x_{1}, x_{2})$ is smooth on $x_1 \leqslant 0$ (up the boundary); $U_{r}(0, x_{1}, x_{2})$ are smooth on $x_1 \geqslant 0$ (up the boundary).
		\item[2)]{\bf (smallness and localization)} $U|_{t=0}$ is an $O(\varepsilon)$-purterbation of data $\Ur_{l}$ and $\Ur_{r}$ defined as above. In other words, there exists a small constant $\varepsilon>0$ whose size will be determined in the course of the proof and a sufficiently large integer $N_{0}$ so that
		\[\begin{cases}
			\|v^1_l(0,x)-\mathring{v}_l\|_{H^{N_0}(\Sigma_-)}+\|v^2_l(0,x)\|_{H^{N_0}(\Sigma_-)}+\|c_l(0,x)-\mathring{c}_l\|_{H^{N_0}(\Sigma_-)}<\varepsilon,\\
			\|v^1_r(0,x)-\mathring{v}_r\|_{H^{N_0}(\Sigma_+)}+\|v^2_r(0,x)\|_{H^{N_0}(\Sigma_+)}+\|c_r(0,x)-\mathring{c}_r\|_{H^{N_0}(\Sigma_+)}<\varepsilon,
		\end{cases}
		\]
		where $\Sigma_-=\mathbb{R}_{\leqslant 0}\times [0,2\pi]$ and $\Sigma_+=\mathbb{R}_{\geqslant 0}\times [0,2\pi]$.
	\end{itemize}
\end{definition}

The main results of our analysis for the Riemann problem with configuration of \emph{rarefaction wave-vortex sheet-rarefaction wave} can be summarized as follows:

\medskip

{\bf (Rough statement of the second part of the main result)} \ \ For any positive constant $t^{*}$, the Cauchy problem for Euler equations \eqref{eq: Euler equations} with $\varepsilon-$ perturbated Riemann data $U|_{t=0}$ satisfying definition \ref{def:data for R-V-R}
admits an entropy solution $U$ on $[0, t^{*}]\times \mathbb{R}\times \mathbb{R}/2\pi\mathbb{Z}$ which is $O(\varepsilon)$ close to the \emph{rarefaction wave-vortex sheet-rarefaction wave} configuration in \ref{subsubsec: R-V-R} with initial data $\Ur=(\Ur_{l}, \Ur_{r})$ providing that $\varepsilon$ is sufficiently small with respect to $\Ur, N_{0}, t^{*}$. Moreover, the solution $U$ connects $U_{l}$ and $U_{r}$ through a smooth superposition of two genuinely nonlinear rarefaction waves and one linearly degenerate vortex sheet, and no other discontinuity appears, which can be depicted as follows:
\begin{center}
\begin{tikzpicture}
  \begin{axis}[axis lines=none]
    
    \addplot [domain=0:2.5, smooth, thick, dashed, color=blue] {0.8*x+0.02*sin(720*x)*x*(x-2.5)+0.05*x*(x-2.5) };
    \addplot [domain=0.5:3, smooth, thick, dashed, color=blue] {0.8*(x-0.5)+1.5+0.02*sin(720*x)*(x-0.5)*(x-3)+0.05*(x-0.5)*(x-3)};
    \draw [thick] (0, 0)--(0.5,1.5);
    \addplot [domain=3:2.5, smooth, thick, dashed, color=blue] {3*(x-2.5)+2+2*sin(1080*x-45)*(x-3)*(x-2.5)};

     \addplot [domain=0:-2, smooth, color=blue, dotted, thick] {-x+0.01*sin(540*x)*x*(x+2)+0.1*x*(x+2)};
    \addplot [domain=0.5:-1.5, smooth, color=blue, dotted, thick] {-x+2+0.01*sin(1080*x)*(x-0.5)*(x+1.5)+0.1*(x-0.5)*(x+1.5)};
    \draw [thick] (0, 0)--(0.5,1.5);
    \addplot [domain=-2:-1.5, smooth, color=blue, dotted, thick] {3*(x+2)+2+2*sin(1080*x)*(x+2)*(x+1.5)};
    
     \draw (-2, 2.7) node {$\Hb_{0}$};

     \addplot [domain=0:5, smooth, color=blue, dotted, thick] {0.4*x+0.002*sin(540*x)*x*(x-5)+0.05*x*(x-5)};
    \addplot [domain=0.5:5.5, smooth, color=blue, dotted, thick] {0.4*(x-0.5)+1.5+0.002*sin(540*x)*(x-0.5)*(x-5.5)+0.08*(x-0.5)*(x-5.5)};
    \draw [thick] (0, 0)--(0.5,1.5);
    \addplot [domain=5:5.5, smooth, color=blue, dotted, thick] {3*(x-5)+2+2*sin(1080*x-120)*(x-5)*(x-5.5)};
    
     \draw (5, 2.7) node {$H_{0}$};

     \addplot [domain=0:10, smooth, color=blue, dotted, thick] {0.2*x+0.0005*sin(360*x)*x*(x-10)+0.01*x*(x-10)};
    \addplot [domain=0.5:10.5, smooth, color=blue, dotted, thick] {0.2*(x-0.5)+1.5+0.0005*sin(360*x)*(x-0.5)*(x-10.5)+0.01*(x-0.5)*(x-10.5)};
    \draw [thick] (0, 0)--(0.5,1.5);
    \addplot [domain=10:10.5, smooth, color=blue, dotted, thick] {3*(x-10)+2+2*sin(1080*x-60)*(x-10)*(x-10.5)};
    
     \draw (10, 2.7) node {$C_{0}$};

     \addplot [domain=0:-4, smooth, color=blue, dotted, thick] {-0.5*x+0.05*x*(x+4)+0.01*sin(1080*x*(x+4))};
    \addplot [domain=0.5:-3.5, smooth, color=blue, dotted, thick] {-0.5*(x-0.5)+1.5+0.05*(x-0.5)*(x+3.5)+0.01*sin(1080*(x-0.5)*(x+3.5))};
    \draw [thick] (0, 0)--(0.5,1.5);
    \addplot [domain=-4:-3.5, smooth, color=blue, dotted, thick] {3*(x+4)+2+2*sin(1080*x-30)*(x+4)*(x+3.5)};

     \draw (-4, 2.7) node {$\Cb_{0}$};

    \draw[thin] (-6,0)--(10.5,0);
    \draw[thin] (-6,2)--(10.5,2);
    
    \draw[thin] (-5.5, 1.5)--(11, 1.5);
    \draw[thin] (-5.5, 3.5)--(11, 3.5);
  \end{axis}
\end{tikzpicture}
\end{center}

\medskip

\subsection{Previous Works}
The compressible Euler equations, the fundamental system governing inviscid compressible fluid motion, lie at the core of hyperbolic conservation laws and continuum mechanics. The Riemann problem—solving the Cauchy problem with piecewise constant initial data—is the cornerstone for understanding nonlinear wave phenomena in hyperbolic systems, with three elementary wave patterns: shock waves, rarefaction waves, and contact discontinuities.

In the one-dimensional setting, the theory of hyperbolic conservation laws is fully mature. Riemann \cite{Riemann} first formulated and solved the Riemann problem for 1D isentropic Euler equations, revealing rarefaction and shock wave formation in compressible flows. Lax \cite{Lax1957,PL} extended this to $n\times n$ strictly hyperbolic conservation laws, establishing the general Riemann problem theory, a complete characterization of elementary wave solutions and their admissibility criteria, and a full solution via the three elementary wave types. A landmark breakthrough came from Glimm \cite{Glimm1965}, who established wave interaction estimates based on the Riemann problem, proved the existence of BV weak solutions, and developed the celebrated random choice method. This method constructs global weak entropy solutions with Riemann solutions as building blocks, yielding not only global existence of BV solutions for strictly hyperbolic systems with small BV initial data, but also a detailed description of the local and large-time behavior of weak solutions \cite{Diperna3,Diperna4,glimmlax,L2,liutp2}. For the 1D perturbed Riemann problem (piecewise continuous initial data with a single discontinuity), Li \& Yu \cite{LiTT} constructed local-in-time piecewise continuous solutions via the free boundary approach. For a comprehensive account of the full 1D theory, we refer to the monographs of Dafermos \cite{Dafermos} and  Smoller \cite{Smoller},  and Liu's recent textbook \cite{Liu2021book}.

In sharp contrast to the well-developed 1D theory, the multi-dimensional setting faces fundamental obstacles, most notably the breakdown of the BV space framework \cite{Rauch}. Following the Riemann-Lax-Glimm paradigm, a core problem is to understand multi-dimensional elementary waves and the Riemann problem, which generally reduce to free boundary problems subject to potential instabilities and coupled initial singularities.

\textbf{Stability of Single Shock Fronts.} The pioneering work of Majda \cite{MajdaShock2,MajdaShock3} established the stability of planar shock fronts in gas dynamics, initiating the systematic study of multi-dimensional elementary waves. For subsequent developments and extensions, we refer to the monographs \cite{Metivier2001book, Benzoni-Gavage-Serre2007book} and references therein.

\textbf{Stability of Rarefaction Waves.} Alinhac \cite{AlinhacWaveRare0, AlinhacWaveRare1, AlinhacWaveRare2} first proved the nonlinear stability of single-family rarefaction waves via Nash-Moser iteration. Luo \& Yu \cite{Luo-YuRare1, Luo-YuRare2} further studied this problem by combining the last slice method from \cite{CK} with the acoustical geometry framework for shock formation developed in \cite{ChristodoulouMiao}, deriving derivative-loss-free energy estimates in the irrotational rarefaction wave region, constructing rarefaction fronts away from vacuum, and applying these results to the stability of $R-R$ type Riemann problems.

\textbf{Stability of Vortex Sheets.} Beyond shocks and rarefaction waves, vortex sheets are another fundamental elementary wave in multi-dimensional gas dynamics, defined as fluid interfaces with discontinuous tangential velocity, and are characteristic discontinuities for the compressible Euler equations. The incompressible Euler vortex sheet problem is ill-posed due to the Kelvin–Helmholtz instability. For the compressible case, Miles \cite{17} showed via modal analysis that 2D compressible vortex sheets are linearly unstable unless Mach-number-dependent stability conditions hold; under these conditions, Coulombel \& Secchi \cite{Coulombel-Secchi1, Coulombel-Secchi2} proved short-time stability of 2D isentropic flat vortex sheets under small initial perturbations. Fejer \& Miles \cite{16} extended this analysis to 3D, proving that 3D vortex sheets are always linearly unstable. Stevens \cite{18} established local well-posedness of the fully non-isentropic compressible vortex sheet problem with surface tension, demonstrating its stabilizing effect. We omit discussion of magnetic field stabilization to remain focused on the compressible Euler equations.

\textbf{Results for Coexisting Multiple Nonlinear Waves.} Alinhac's framework has been adapted by Wang \& Yin \cite{WangYin} to analyze Prandtl-Meyer expansion waves and Li \cite{Li1991} and Chen \& Li \cite{ChenLi} to study combined elementary wave patterns. Qu \& Xiang \cite{QuXiang} also studied 3D steady supersonic Euler flow past a concave cornered wedge with low downstream pressure.

\textbf{Global-in-Time Results.} A core challenge in the global-in-time dynamics of nonlinear waves for multi-dimensional compressible Euler equations is the uniform long-time control of vorticity. Owing to the difficulty of establishing robust global-in-time a priori bounds for vorticity, nearly all existing global stability results are established under the irrotational (potential flow) assumption. Within this framework: Chen, Xin \& Yin \cite{ChenXinYin} proved global stability of stationary supersonic flow past an infinite curved symmetric cone; Ginsberg \& Rodnianski \cite{ginsberg2024stabilityirrotationalshockslandau} established global stability of Landau's two-shock profiles for potential flow; Zhang \cite{ZhangRuotong} obtained local and global-in-time results for 3D rarefaction waves from initial singularities; for flows with vacuum formation at infinity, Wang \cite{Wang} constructed global exterior solutions in $\mathbb{R}^3$ with rarefaction at null infinity, and Xu \& Yin \cite{GangXuYinHuicheng} proved global existence and stability of smooth supersonic flows in a divergent nozzle with vacuum at infinity.

There are also important works on the multi-dimensional stability of nonlinear waves in other hyperbolic systems, such as the relativistic Euler equations and compressible MHD, among others. However, these results are not directly relevant to the present work, and a detailed review is omitted here due to space limitations.

\subsection{The acoustical geometry of Euler equations} We set up the acoustical geometry for the rarefaction wave problem. We refer to \cite{Luo-YuRare1, Luo-YuRare2} for detailed account.

\subsubsection{Riemann invariants, specific vorticity and the wave-transport equations} There are two kinds of infinitesimal waves in Euler equations \eqref{eq: Euler equations},
\begin{itemize}
\item {\bf (1) acoustical wave} The acoustical wave is determined by acoustical metric $g$, which is defined by
\[
g=-c^{2}dt^{2}+\sum_{i=1}^{2}(dx^{i}-v^{i}dt)^{2}.
\]
\item {\bf (2) vorticity wave} The vorticity wave is transported by material vector field $B$, which is defined by
\[
B=\frac{\partial}{\partial t}+\sum_{i=1}^{2}v^{i}\frac{\partial}{\partial x_{i}}
\]
\end{itemize}

We also define
\[
\psi_{1}=-v^{1}, \psi_{2}=-v^{2}
\]
Following Riemann, we define the Riemann invariants with direction $e_{1}=(1, 0)$ as follows,
\begin{equation}\label{eq: Riemann invariants}
w=\frac{1}{\gamma-1}c+\frac{1}{2}\psi_{1}, \wb=\frac{2}{\gamma-1}c-\frac{1}{2}\psi_{1},
\end{equation}
We define the specific vorticity to be
\begin{equation}\label{eq: specific vorticity}
\Omega:=\frac{\partial_{1}v^{2}-\partial_{2}v^{1}}{\rho}.
\end{equation}
Then the Riemann invariants satisfy the following wave-transport equations:

\begin{equation}\label{eq: wave-transport system}
\begin{cases}
\begin{aligned}
\Box_{g}(\underline{w})={}&-c^{-1}\left[\frac{7-\gamma}{4}g(D\underline{w}, D\underline{w})+\frac{\gamma+1}{4}g(Dw,Dw)+\frac{1}{2}g(Dv^{2}, Dv^{2})\right]\\
&+\frac{1}{2}c^{-1}\rho^{2}\Omega^{2}-\frac{1}{2}\rho\partial_{2}\Omega-\frac{2}{\gamma-1}\rho c^{-1}\Omega\partial_{2}c,
\end{aligned}\\[1.5ex]
\begin{aligned}
\Box_{g}(w)={}&-c^{-1}\left[\frac{\gamma+1}{4}g(D\underline{w}, D\underline{w})+\frac{7-\gamma}{4}g(Dw,Dw)+\frac{1}{2}g(Dv^{2}, Dv^{2})\right]\\
&+\frac{1}{2}c^{-1}\rho^{2}\Omega^{2}+\frac{1}{2}\rho\partial_{2}\Omega+\frac{2}{\gamma-1}\rho c^{-1}\Omega\partial_{2}c,
\end{aligned}\\[1.5ex]
\begin{aligned}
\Box_{g}(\psi_{2})={}&-c^{-1}\left[\frac{3-\gamma}{2}g(D\underline{w}, \psi_{2})+\frac{3-\gamma}{2}g(Dw, \psi_{2})\right]\\
&-\rho\partial_{1}\Omega-\frac{4}{\gamma-1}\rho c^{-1}\Omega\partial_{1}c,
\end{aligned}\\[1.5ex]
B\Omega=0.
\end{cases}
\end{equation}

\subsubsection{The acoustical geometry and the construction of the first null frame}\label{section:acoustical geometry}
We use $\mathcal{D}_0$ to denote the future domain of dependence determined by the data $U_{r}|_{t=0}$ on the right. Its boundary $C_0$ is a null hypersurface with respect to the acoustical metric $g$. We also use $U_{r}(t, x_{1}, x_{2})$ to denote the solution in $\mathcal{D}_0$ since it is completely determined by the data $U_{r}|_{t=0}$ on the right.
\begin{center}
\includegraphics[width=2.5in]{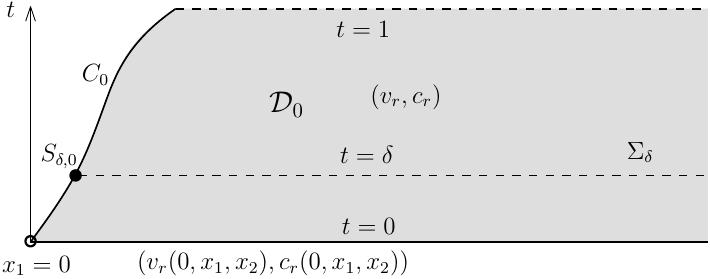}
\end{center}
Since $\varepsilon$ is small, in view of the continuous dependence of the solution on the initial data, the domain $\mathcal{D}_0$ covers up to $t=t^*=1$. For $t_0\in [0,t^*]$, we define $\Sigma_{t_0}=\{(t,x_1,x_2) |t=t_0\}$. For a small parameter $\delta$ (which will go to $0$ in a limiting process), the spatial hypersurface $\mathcal{D}_0\cap \Sigma_\delta$ is depicted as follows:
\begin{center}
\includegraphics[width=2.5in]{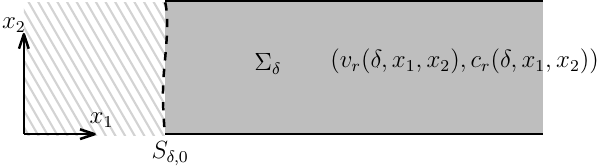}
\end{center}
We define $S_{\delta,0}=\Sigma_\delta\cap C_0$. The restriction of the solution $U_{r}$ on $\Sigma_\delta$ gives data on the right-hand side of $S_{\delta,0}$. The data for the rarefaction waves will be given on the left-hand side of $S_{\delta,0}$ on $\Sigma_\delta$. 

We will choose a specific smooth function $u$ on $\Sigma_\delta$ so that $S_{\delta,0}$ is given by $u=0$ and the lefthand side of $S_{\delta,0}$ on $\Sigma_\delta$ is given by $u>0$, see section 3.1 of \cite{Luo-YuRare2} for the exact construction of $u$ on $\Sigma_{\delta}^{u^*}$.  The data for the front rarefaction waves will be specified for $u \in [0,u^*]$ and the parameter $u^*$, which represents the width of the rarefaction waves, satisfying $u^{*}<\frac{\gamma+1}{\gamma-1}\mathring{c}_{r}$ to stay away from vacuum. Together with the data on $C_0$, the data on $u \in [0,u^*]$ evolves to the development $\mathcal{D}(\delta)$ (the shaded region on the left of the following figure).

\begin{center}
\includegraphics[width=3.3in]{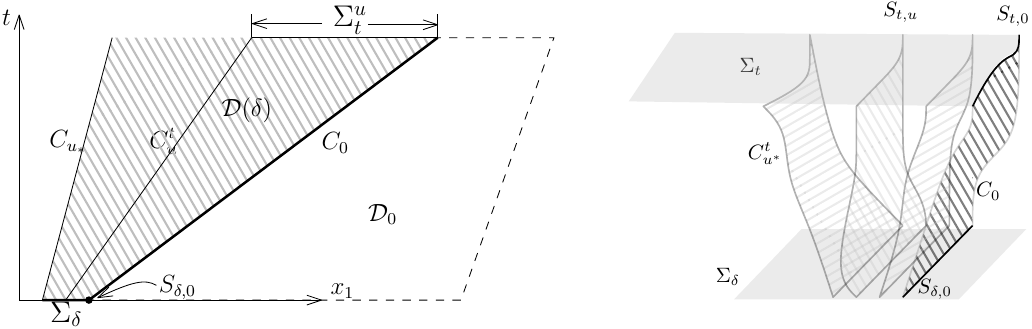}
\end{center}

We recall the definition of the acoustical coordinates $(t,u,\vartheta)$ on $\mathcal{D}(\delta)$, and we refer to \cite{Luo-YuRare1} for the detailed account. The function $t$ is defined as $x_0$ restricted to $\mathcal{D}(\delta)$.  We define $C_{u_0}$ to be the null hypersurfaces consisting of null (future right-going) geodesics emanating from the level set of $u=u_0$ on $\Sigma_\delta$. We define $u$ on $\mathcal{D}(\delta)$ in such a way that $C_u$'s are the level sets of $u$. Let $L$ be the generator of the null geodesics on $C_u$ subject to the normalization that $L(t) = 1$. We define $S_{t,u} = \Sigma_t \cap C_u$. Let $T$ be the vector field  tangential to $\Sigma_t$, orthogonal to $S_{t,u}$ with respect to $g$ and subject to the normalization that $Tu = 1$. We define $\kappa^2=g(T,T)$.

We also have the following notations:
\[\mathcal{D}(\delta)(t^*,u^*) =\!\!\!\!\!\!\bigcup_{(t,u) \in [\delta,t^*]\times [0,u^*]}\!\!\!\!\!\!S_{t,u}, \ \mathcal{D}(\delta)(t,u) =\!\!\!\!\!\!\bigcup_{(t',u') \in [\delta,t]\times [0,u]}\!\!\!\!\!\!S_{t',u'},\ \ \Sigma_t^u=\!\!\!\bigcup_{u' \in [0,u]}\!\!\!S_{t,u'}, \ \ C_u^{t}=\!\!\!\bigcup_{t' \in [\delta,t]}\!\!\!S_{t',u}.\]
For the sake of simplicity, we also use $\Sigma_t$ to denote $\Sigma_t^{u^*}$.

We define $\vartheta$ on $C_0$ through the following ordinary differential equation:
\begin{equation}\label{def: vartheta in acoustical}
L( \vartheta)=0, \ \ \  \vartheta\big|_{S_{0,0}}=x_2\big|_{S_{0,0}}.
\end{equation}
We then extend $\vartheta$ to $\Sigma_\delta$ by solving $T( \vartheta)=0$ with initial data on $S_{\delta,0}$ given by the solution of \eqref{def: vartheta in acoustical}. In particular, this means that  $T\big|_{\Sigma_\delta} = \frac{\partial}{\partial u}$ on $\Sigma_\delta$. Finally, we solve $L(\vartheta)=0$ to extend $\vartheta$ from $\Sigma_\delta$ to the entire spacetime $\mathcal{D}(\delta)$.

In the acoustical coordinates $(t,u,\vartheta)$, we have
\begin{equation}\label{eq: L T in terms of coordinates}
L = \frac{\partial}{\partial t}, \quad T = \frac{\partial}{\partial u} - \Xi \frac{\partial}{\partial \vartheta},
\end{equation}
where $\Xi$ is a smooth function.  We also define $X=\frac{\partial}{\partial \vartheta}$, $\slashed{g}=g(X,X)$ and the unit vector $\widehat{X}=\slashed{g}^{-\frac{1}{2}}X$. Let $\mu=c\kappa$. We then have
\[g(L,T) = -{\mu}, \quad g(L,L) = g(L,\Xh) = g(T,\Xh)=0, \quad g(\Xh,\Xh) = 1.
\]
where $\Th^i$ is the $i$-th component of $\Th$ in the Cartesian coordinates, i.e., $\Th=\sum_{i=1}^2\Th^i \frac{\partial}{\partial x_i}$. Similarly, in Cartesian coordinates, we have $\Xh=\sum_{i=1}^2\Xh^i \partial_i$ and $L=\partial_0+\sum_{i=1}^2L^i \partial_i$. We also remark that $\Th^1=-\Xh^2$ and $\Th^2=\Xh^1$. Let $B=\frac{\partial}{\partial t}+\sum_{i=1}^{2}v^{i}\frac{\partial}{\partial x_{i}}$ be the material vector field. We have $B(t)=1$ and $B$ is $g$-perpendicular to $\Sigma_t$.  We also define the unit vector $\widehat{T} = \kappa^{-1}T$. Thus, $L$ can be expressed as $L =B-c\widehat{T}$.

We define the left-going null vector field $\underline{L} = c^{-1} \kappa L + 2T$. Therefore, we obtain {\bf the first null frame} $(L,\Lb,\Xh)$.

For the following three isometric embeddings $\Sigma_t\hookrightarrow \mathcal{D}$,  $S_{t,u} \hookrightarrow C_u $ and $S_{t,u} \hookrightarrow \Sigma_t$, the corresponding second fundamental forms are denoted as follows:
\[2 c k = \overline{\mathcal{L}}_B g, \ \ 2 \chi = \slashed{\mathcal{L}}_L g, \ \ 2 \kappa \theta = \slashed{\mathcal{L}}_T g.\]
where $\overline{\mathcal{L}}$ and $\slashed{\mathcal{L}}$ denote the projections of Lie derivatives to $\Sigma_t$ and $S_{t,u}$ respectively.
We define the torsion 1-forms $\zeta$ and $\eta$ on $S_{t,u}$ as
\[\zeta(Y) =  g(D_{Y} L, T), \ \ \eta(Y) =  -g(D_{Y} T, L),\]
where $Y$ is any vector field tangent to $S_{t,u}$. We also define the $1$-form $\slashed{\varepsilon}$ as $\kappa \slashed{\varepsilon}(Y)=k(Y,T)$.

Since the $S_{t,u}$'s are 1-dimensional circles, we can represent the tensors by functions. For the sake of simplicity, we use the same symbols to denote the following scalar functions:
\[\chi = \chi(\Xh,\Xh), \  \theta =  \theta(\Xh,\Xh), \ \slashed{k} = k(\Xh,\Xh), \ \zeta = \zeta(\Xh),\ \eta = \eta(\Xh), \ \slashed{\varepsilon}=\slashed{\varepsilon}(\Xh).\]
Since $\slashed{g}=g(X,X)$, we have $L(\slashed{g})=2 \slashed{g} \cdot \chi$. These quantities are related by
\[\chi =  c(\slashed{k} - \theta), \ \ \eta = \zeta + \Xh(\mu), \ \ \zeta=\kappa\big(c\slashed{\varepsilon}-\Xh(c)\big).\]

By using $\Lb$, we can introduce  another second fundamental form $\underline{\chi}$ which is defined by 
\[ 2 \underline{\chi} = \slashed{\mathcal{L}}_{\underline{L}} g. \]
We will also work with its scalar version $\underline{\chi} =\underline{\chi}(\Xh,\Xh)$. It can be represented by $\slashed{k}$ and  $\theta$ through $\underline{\chi} =  \kappa(\slashed{k} + \theta)$.

The wave operator $\Box_g$ can  be decomposed with respect to the null frame $(L,\Lb,\Xh)$:
\begin{equation}\label{eq:wave operator in null frame}
 \Box_{g} (f) = \Xh^2 (f) - \mu^{-1}L\big(\underline{L}(f)\big) - \mu^{-1}\big(\frac{1}{2}\chi\cdot \underline{L}(f) +\frac{1}{2}\chib\cdot L(f)\big)- 2 \mu^{-1}\zeta \cdot \Xh(f).
\end{equation}

The change of $\kappa$ along the characteristic direction $L$ is recorded in the following equation:
\begin{equation}\label{structure eq 1: L kappa}
L \kappa = -T(c)-\hat{T}^{i}T(\psi_{i}),
\end{equation}

\begin{remark}[Einstein summation convention]
The repeated Latin letter indices (say $i,j,k$) indicate the summation over $1,2$. The repeated Greek letter indices (say $\mu,\nu$) indicate the summation over $0,1,2$.  
\end{remark}

The Jacobi matrix of the coordinates transformation $(t,u,\theta)\mapsto (t,x_{1},x_{2})$ is given by
\begin{equation}\label{eq: Jacobi of (t, u, vartheta) to (t, x_1, x_2)}
\begin{pmatrix}
\frac{\partial t}{\partial t}&\frac{\partial t}{\partial u}&\frac{\partial t}{\partial \vartheta}\\
\frac{\partial x_{1}}{\partial t}&\frac{\partial x_{1}}{\partial u}&\frac{\partial x_{1}}{\partial \vartheta}\\
\frac{\partial x_{2}}{\partial t}&\frac{\partial x_{2}}{\partial u}&\frac{\partial x_{2}}{\partial \vartheta}
\end{pmatrix}=
\begin{pmatrix}
1&0&0\\
L^{1}&\kappa\hat{T}^{1}+\Xi\sqrt{\slashed{g}}\hat{X}^{1}&\sqrt{\slashed{g}}\hat{X}^{1}\\
L^{2}&\kappa\hat{T}^{2}+\Xi\sqrt{\slashed{g}}\hat{X}^{2}&\sqrt{\slashed{g}}\hat{X}^{2}\\
\end{pmatrix}
\end{equation}
with determinant $\Delta=-\kappa\sqrt{\slashed{g}}$.


We collect {\bf structure equations of the rarefaction fronts} which will be used to control $\{\hat{T}^{1}, \hat{T}^{2}, \kappa, \slashed{g}\}$,
\begin{equation}\label{eq: formulas to control the geometry}
\begin{cases}
\theta=-\hat{X}(\hat{X}^{2})\hat{X}^{1}+\hat{X}(\hat{X}^{1})\hat{X}^{2},\\
\chi=-\hat{X}^{i}\hat{X}(\psi_{i})-c\theta,\\
\zeta=-\frac{1}{2}\kappa \hat{T}^{j}\hat{X}(\psi_{j})-\frac{1}{2}\hat{X}^{i}T(\psi_{i})-\kappa\hat{X}(c),\\
\eta=-\frac{1}{2}\kappa \hat{T}^{j}\hat{X}(\psi_{j})-\frac{1}{2}\hat{X}^{i}T(\psi_{i})+c\hat{X}(\kappa),\\
[T, \hat{X}]=-\kappa\theta\cdot \hat{X},\\
[L, \hat{X}]=-\chi\cdot \hat{X},\\
[L, T]=-(\zeta+\eta)\hat{X},\\
L(\hat{T}^{i})=\big(\hat{T}^{j}\hat{X}(\psi_{j})+\hat{X}(c) \big),\\
L(\kappa)=-T(c)-\hat{T}^{j}T(\psi_{j}),\\
L(\chi)=2\chi\cdot \slashed{g},\\
L(\Xi)=\frac{1}{\sqrt{\slashed{g}}}(\zeta+\eta).
\end{cases}
\end{equation}

\begin{remark}
The formulas for $\zeta$ and $\eta$ are slightly different from the irrotational case \cite{Luo-YuRare1} due to the non-vanishing of vorticity.
\end{remark}

\subsubsection{Euler equations in the diagonal form}

In terms of the frame $(L,T,\Xh)$,  the Euler equations can be rewritten in terms of Riemann invariants:
\begin{equation}\label{Euler equations:form 2}
\begin{cases}
L (\wb) &= -c \widehat{T}(\wb)(\widehat{T}^1+1)+\frac{1}{2}c \widehat{T}(\psi_2)\widehat{T}^2 +\frac{1}{2}c \Xh(\psi_2)\Xh^2-c\Xh(\wb)\Xh^1,\\
L (w) &= 	c \widehat{T}(w)(\widehat{T}^1-1)	+\frac{1}{2}c \widehat{T}(\psi_2)\widehat{T}^2 +c\Xh(w)\Xh^1+\frac{1}{2}c\Xh(\psi_2)\Xh^2,\\
L (\psi_2) &= -c \widehat{T}(\psi_2)+c\widehat{T}( w+\wb)\widehat{T}^2+c\Xh( w+\wb)
\Xh^2.
\end{cases}
\end{equation}

Let $\{\wb_{g}, w_{g}, \psi^{(\hat{X})}\}$ be the {\bf geometric Riemann invariants} with respect to direction $\hat{T}=(\hat{T}^{1}, \hat{T}^{2})$,
\[
\begin{cases}
\wb_{g}:=\frac{c}{\gamma-1}+\frac{1}{2}\psi^{(\hat{T})},\\
w_{g}:=\frac{c}{\gamma-1}-\frac{1}{2}\psi^{(\hat{T})}.
\end{cases}
\]
\begin{remark}
When $\hat{T}=\hat{\Tr}$, the {geometric Riemann invariants} reduce to the ordinary Riemann invariants \eqref{eq: Riemann invariants}.
\end{remark}
  
 We write Euler equations for Riemann invariants along $C_{0}$.
\begin{equation}\label{Euler equations:form 1}
\begin{cases}
\begin{aligned}
L(\wb_{g})={}&\frac{1}{2}c\hat{X}(\psi^{\hat{X}})+\frac{1}{2}c\theta\cdot \psi^{(\hat{T})}\\
&+\frac{1}{2}\psi^{(\hat{X})}\hat{X}(\psi^{(\hat{T})}+c)-\frac{1}{2}\theta\psi^{(\hat{X})}\psi^{(\hat{X})},
\end{aligned}\\[1.5ex]
\begin{aligned}
L(\psi^{(\hat{X})})={}&-c\hat{T}(\psi^{(\hat{X})})+c\hat{X}_{r}\left(\frac{2c}{\gamma-1}\right)+c\frac{\hat{X}(\kappa)}{\kappa}\psi^{(\hat{T})}\\
&-\psi^{(\hat{T})}\hat{X}(\psi^{(\hat{T})}+c)+\theta\psi^{(\hat{X})}\psi^{(\hat{T})},
\end{aligned}\\[1.5ex]
\begin{aligned}
L(w_{g})={}&-2c\hat{T}(w_{g})+\frac{1}{2}c\hat{X}(\psi^{\hat{X}})+c\frac{\hat{X}(\kappa)}{\kappa}\psi^{(\hat{X})}\\
&+\frac{1}{2}c\theta\cdot \psi^{(\hat{T})}-\frac{1}{2}\psi^{(\hat{X})}\hat{X}(\psi^{(\hat{T})}+c)+\frac{1}{2}\theta\psi^{(\hat{X})}\psi^{(\hat{X})},
\end{aligned}\\[1.5ex]
\textbf{$\hat{T}(\wb_{g})|_{S_{\delta, 0}}$ can be freely prescribed},\\[1ex]
\hat{T}(w_{g})|_{S_{\delta, 0}}\ \text{is determined by data on $C_{0}^{t^{*}}$},\\[1ex]
\hat{T}(\psi^{(\hat{X})})|_{S_{\delta, 0}}\ \text{is determined by data on $C_{0}^{t^{*}}$}.
\end{cases}
\end{equation}
We remark that $\hat{T}(\wb_{g})$ can be freely prescribed, which reflects the fundamental nature of characteristic surface $C_{0}$ and will play an important role in the construction of configurations $R-R$ in section \ref{sec: Application to the Riemann problem 1}, S-R in section \ref{sec: Application to the Riemann problem 2}, and R-V-R in section \ref{sec: Application to the Riemann problem 3}.

\subsubsection{The geometry of the second null frame}
Following \cite{Luo-YuRare1}, in order to derive energy estimates, we introduce
\[\Xr=\partial_2, \ \Trh=-\partial_1,  \ \Lr=\partial_t+v-c\Trh=\partial_t+(v^1+c)\partial_1+v^2\partial_2,\]
and
\[\kappar=t, \ \ \Tr=\kappar \Trh, \ \ \mur = c\kappar.\]
The vector fields satisfy the following metric relations:
\[g(\Lr,\Tr)=-\mur, \ g(\Lr,\Lr)=g(\Lr,\Xr)=0,\  g(\Xr,\Xr)=1, \ g(\Tr,\Tr)=\kappar^2, \ g(\Tr,\Xr)=0.\]
Let $\Lbr= c^{-1} \kappar \Lr + 2\Tr$. We then obtain {\bf the second null frame}
$(\Lr, \Lbr, \Xr)$. One can check that
\[g(\Lr,\Lbr)=-2\mur, \ g(\Lr,\Lr)=g(\Lbr,\Lbr)=g(\Lbr,\Xr)=g(\Lr,\Xr)=0,\  g(\Xr,\Xr)=1.\]
We introduce functions $y$, $\yr$, $z$ and $\zr$ as follows:
\begin{equation}\label{def: yring zring}
y=\Xr(v^1+c), \ \ \yr = \frac{y}{\kappar},  \ \ z=1+\Tr(v^1+c), \ \ \zr = \frac{z}{\kappar}.
\end{equation}
We list the definitions and formulas for the connection coefficients in the second null frame as follows:
\begin{equation*}
	\begin{cases}
		&\chir:=g(D_{\Xr} \Lr,\Xr)=-\Xr(\psi_2), \ \ \chibr :=  g(D_{\Xr} \Lbr, \Xr)=c^{-1}\kappar \chir=-c^{-1}\kappar\Xr(\psi_2), \\ 
		&\zetar:=  g(D_{\Xr} \Lr, \Tr)=-\kappar y, \ \ \etar:=-\Tr(\psi_2), \\
		&\deltasr:=g(D_{\Lr} \Lr,\Xr)=cy,\ \ \deltar:=g(D_{\Lr} \Lr,\Tr)=-\Lr(\mur)+cz.
	\end{cases}
\end{equation*}
The commutators for the new vector fields are collected as follows:
\begin{equation*}
\begin{cases}
&[\Tr, \Xr]=0, \ \ [\Lr, \Xr]=\yr\cdot \Tr-\chir\cdot\Xr, \ \ [\Lr, \Tr]=\zr\cdot \Tr-\etar\Xr,\\
&[\Lbr, \Xr]=-\left(\frac{1}{2}c^{-2}\kappar y+\Xr(c^{-1}\kappar)\right)\Lr-\chibr\cdot\Xr+\frac{1}{2}c^{-1}y\cdot\Lbr, \\
& [\Lr, \Lbr]=\left(\Xr(c^{-1}\kappar)-c^{-1}z\right)\Lr-2\etar\cdot \Xr+\zr\cdot \Lbr.
\end{cases}
\end{equation*}
\begin{remark}
In irrotational case, $ck(\Tr,\Xr)=\Tr(v^{2})$, $\epsilonr$ is taken to be $ck(\Tr,\Xr)=\Tr(v^{2})$. In rotational case, $ck(\Tr,\Xr)=\Tr(v^{2})+\frac{1}{2}\kappar\epsilon_{12}\rho\Omega$, $\epsilonr$ is taken to be $\Tr(v^{2})$ directly. 
\end{remark}

Similar to \eqref{Euler equations:form 1}, we can rewrite the Euler equations in the following form:
\begin{equation}\label{Euler equations:form 1 ringed}
\begin{cases}
\Lr (\frac{2}{\gamma-1}c) &= -c \Trh(\frac{2}{\gamma-1}c)-c \Trh(\psi_1) +c\Xr(\psi_2),\\
\Lr (\psi_1) &= -c \Trh(\psi_1)-c \Trh\left(\frac{2}{\gamma-1}c\right),\\
\Lr (\psi_2) &= -c \Trh(\psi_2)+c\Xr\left(\frac{2}{\gamma-1}c\right).
\end{cases}
\end{equation}
In terms of the Riemann invariants, \eqref{Euler equations:form 1 ringed} reduces to a simple form
\begin{equation}\label{Euler equations:form 3}
\begin{cases}
\Lr (\wb) &= \frac{1}{2}c \Xr(\psi_2),\\
\Lr (w) &= 	-2c \Trh(w)+\frac{1}{2}c\Xr(\psi_2),\\
\Lr (\psi_2) &= -c \Trh(\psi_2)+c\Xr( w+\wb).
\end{cases}
\end{equation}

We also recall the following notations:
A \emph{multi-index} $\alpha$ is a string of numbers $\alpha=(i_1,i_2,\cdots,i_n)$ with {\color{black}$i_j=0$ or $1$} for $1\leqslant j\leqslant n$. The \emph{length} of the multi-index $\alpha$ is defined as $|\alpha|=n$. Given a multi-index $\alpha$ and a smooth function $\psi$, the shorthand notations $Z^\alpha(\psi)$ and  $\Zr^\alpha(\psi)$ denote the following functions:
\[Z^\alpha(\psi)=Z_{(i_N)}\left(Z_{(i_{N-1})}\left(\cdots \left(Z_{(i_1)}(\psi)\right)\cdots\right)\right), \ \Zr^\alpha(\psi)=\Zr_{(i_N)}\big(\cdots \big(\Zr_{(i_1)}(\psi)\big)\cdots \big),\]
where $Z_{(0)}=\Xh$, $Z_{(1)}=T$, $\Zr_{(0)}=\Xr$ and $\Zr_{(1)}=\Tr$. 

\subsection{Reformulation of the Specific Vorticity in the First Null Frame}

In \cite{Luo-YuRare2}, the authors use the spacetime vorticity 1-form
\[
\beta = \big(h+\frac{1}{2}|v|^2\big)dt - v^idx^i,
\]
first defined in Christodoulou‘s work \cite{ChristodoulouShockDevelopment} on shock development, to show that if the flow is irrotational in the region $\mathcal{D}_{0}$ determined by the initial data $U_{r}$ on the right, then the vorticity must vanish across rarefaction fronts $C_{u}, u\geq 0$. In fact, the specific vorticity is non-characteristic with respect to $C_{u}$. The following figure is from Courant and Friedrichs's book \cite{CourantFriedrichs}:
\begin{figure}[ht]
\centering
\includegraphics[width=3in]{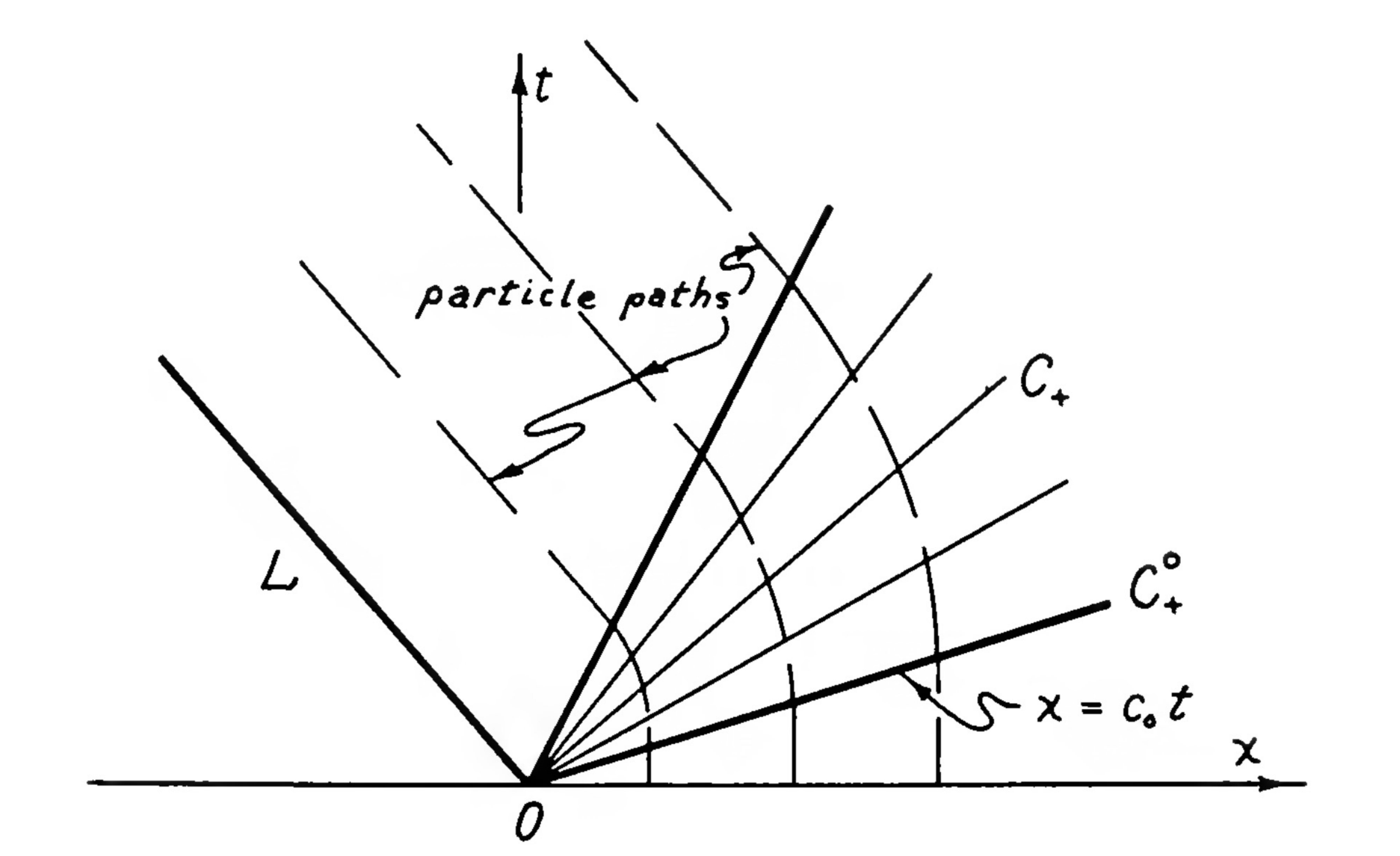}
\caption{Particle paths (integral curves of the material vector field $B$) transverse to rarefaction fronts away from vacuum.}
\end{figure}
Thus, both the specific vorticity $\Omega:=\frac{\omega}{\rho}$ and its derivatives can be expressed in terms of derivatives tangent to $C_{u}$. More precisely,
\begin{equation}\label{eq: reformulation of vorticity in the first null frame}
\begin{cases}
\omega&=\hat{X}^{1}\hat{X}(v^{2})-\hat{X}^{2}\hat{X}(v^{1})+\frac{2}{\gamma-1}\hat{X}(c)+c^{-1}\hat{X}^{1}L(v^{1})+c^{-1}\hat{X}^{2}L(v^{2}),\\
\partial_{1}\Omega&=\hat{X}^{1}\hat{X}(\Omega)+c^{-1}\hat{X}^{2}L(\Omega),\\
\partial_{2}\Omega&=\hat{X}^{2}\hat{X}(\Omega)-c^{-1}\hat{X}^{1}L(\Omega).
\end{cases}
\end{equation}

\begin{proof}
The idea of the proof is inspired by Speck's work \cite{Speck} on null structure of compressible Euler equations. The basic idea is to express the spatial derivatives $\partial_{1}, \partial_{2}$ in terms of the first null frame $\{L, \hat{X}, T\}$. We use
\[
\begin{cases}
\partial_{1}=\hat{X}^{1}\hat{X}-\hat{X}^{2}\hat{T},\\
\partial_{2}=\hat{X}^{2}\hat{X}+\hat{X}^{1}\hat{T}
\end{cases}
\]
to represent $\omega:=\partial_{1}v^{2}-\partial_{2}v^{1}$:
\[
\omega=\hat{X}^{1}\hat{X}(v^{2})-\hat{X}^{2}\hat{T}(v^{2})-\hat{X}^{2}\hat{X}(v^{1})-\hat{X}^{1}\hat{T}(v^{1}).
\]
In view of $Bv^{i}=-\partial_{i}h$ for $i=1,2$,
\[
\begin{cases}
L(v^{1})+c\hat{T}(v^{1})=-\hat{X}^{1}\hat{X}(h)+\hat{X}^{2}\hat{T}(h),\\
L(v^{2})+c\hat{T}(v^{2})=-\hat{X}^{2}\hat{X}(h)-\hat{X}^{1}\hat{T}(h),
\end{cases}
\]
which implies
\[
\begin{cases}
\hat{T}(v^{1})=c^{-1}[-\hat{X}^{1}\hat{X}(h)+\hat{X}^{2}\hat{T}(h)-L(v^{1})],\\
\hat{T}(v^{2})=c^{-1}[-\hat{X}^{2}\hat{X}(h)-\hat{X}^{1}\hat{T}(h)-L(v^{2})].
\end{cases}
\]
Therefore, we have
\[
\begin{split}
\omega&=\hat{X}^{1}\hat{X}(v^{2})-\hat{X}^{2}\hat{X}(v^{1})-\hat{X}^{2}\hat{T}(v^{2})-\hat{X}^{1}\hat{T}(v^{1})\\[0.5ex]
&=\hat{X}^{1}\hat{X}(v^{2})-\hat{X}^{2}\hat{X}(v^{1})
  -\hat{X}^{2}c^{-1}\left[-\hat{X}^{2}\hat{X}(h)-\hat{X}^{1}\hat{T}(h)-L(v^{2})\right]\\
&\quad-\hat{X}^{1}c^{-1}\left[-\hat{X}^{1}\hat{X}(h)+\hat{X}^{2}\hat{T}(h)-L(v^{1})\right]\\[0.5ex]
&=\hat{X}^{1}\hat{X}(v^{2})-\hat{X}^{2}\hat{X}(v^{1})+\frac{2}{\gamma-1}\hat{X}(c)
  +c^{-1}\hat{X}^{1}L(v^{1})+c^{-1}\hat{X}^{2}L(v^{2}).
\end{split}
\]

This shows that $\Omega=\frac{\omega}{\rho}$ is determined by $(c, v^{1}, v^{2})|_{C_{0}}=(c_{r}, v_{r}^{1}, v_{r}^{2})|_{C_{0}}$. Hence,
\[
\Omega|_{C_{0}}=\Omega_{r}|_{C_{0}}.
\]
In view of $B\Omega=0$, we have $L(\Omega)+c\hat{T}(\Omega)=0$. It follows that
\[
\begin{cases}
\partial_{1}\Omega=\hat{X}^{1}\hat{X}(\Omega)-\hat{X}^{2}\hat{T}(\Omega)=\hat{X}^{1}\hat{X}(\Omega)+c^{-1}\hat{X}^{2}L(\Omega),\\
\partial_{2}\Omega=\hat{X}^{2}\hat{X}(\Omega)+\hat{X}^{1}\hat{T}(\Omega)=\hat{X}^{2}\hat{X}(\Omega)-c^{-1}\hat{X}^{1}L(\Omega),
\end{cases}
\]
which implies
\[
\partial_{1}\Omega|_{C_{0}}=\partial_{1}\Omega_{r}|_{C_{0}}, \quad \partial_{2}\Omega|_{C_{0}}=\partial_{2}\Omega_{r}|_{C_{0}}.
\]
\end{proof}

\section{Main theorems}

\subsection{Energy identities for linear waves}

\subsubsection{Energy identities for linear acoustical waves} Let $\varrho$ be a source function, we derive energy identities for the linear wave equation $\Box_{g}\psi=\varrho$. Following \cite{Luo-YuRare1}, the energy momentum tensor associated to $\psi$ is defined as 
\[
T_{\mu\nu}=\partial_{\mu}\psi\partial_{\nu}\psi-g^{\alpha\beta}\partial_{\alpha}\psi\partial_{\beta}\psi g_{\mu\nu},
\]
In the first null frame $\{L, \Lb, \hat{X}\}$, the components $T_{\mu\nu}$ are list as follows
\begin{equation}
\begin{split}
T_{LL}&=(L\psi)^{2},T_{\underline{L}\underline{L}}=(\underline{L}\psi)^{2},T_{L\underline{L}}=\mu\hat{X}(\psi)^{2},T_{L\hat{X}}=L(\psi)\hat{X}(\psi),T_{\underline{L}\hat{X}}=\underline{L}(\psi)\hat{X}(\psi),\\
T_{\hat{X}\hat{X}}&=\frac{1}{2}\hat{X}(\psi)^{2}+\frac{1}{2\mu}L(\psi)\underline{L}(\psi),
\end{split}
\end{equation}
For a vector field $J$, its energy current field is defined $P^{\mu}:=-T^{\mu}{}_{\nu}J^{\nu}$. Therefore, 
\[
D_{\mu}P^{\mu}=-\Box_{g}\psi\cdot J(\psi)-\frac{1}{2}T^{\mu\nu}\cdot{}^{(J)}\pi_{\mu\nu}.
\]

For $(t, u)\in [\delta, t^{*}]\times [0, u^{*}]$ and a smooth function $f$ defined on $\mathcal{D}(\delta)(t, u)$, we use the following notations to denote the integrals:
\begin{equation}
\begin{split}
\int_{\Sigma_{t}^{u}} f:&=\int_{\mathbb{T}}\int_{0}^{u} f(t,u^{'},\theta)\sqrt{\slashed{g}}du^{'} d\theta,\ \ \int_{C_{u}^{t}} f:=\int_{\mathbb{T}}\int_{\delta}^{t} f(t^{'},u,\theta)\sqrt{\slashed{g}}dt^{'} d\theta,\\
\int_{D(\delta)(t,u)} f:&=\int_{\mathbb{T}}\int_{0}^{u}\int_{\delta}^{t} f(t^{'},u^{'},\theta)\sqrt{\slashed{g}}dt^{'}du^{'} d\theta.
\end{split}
\end{equation}

We have two choices for $J$. This leads to the following two energy identities:

\begin{itemize}
\item {\bf (1)$ J=\hat{L}:=c^{-1}\kappa L$.} We will use $\mathcal{E}_{w}(\psi)$ to denote the associated energy and $\mathcal{F}_{w}(\psi)$ to denote the corresponding flux,
\begin{equation}
\begin{cases}
\mathcal{E}_{w}(\psi)(t,u):&=\int_{\Sigma_{t}^{u}}\frac{1}{2}\kappa^{2}(|\slashed{\nabla}\psi|^{2}+c^{-2}|L\psi|^{2}),\\
\mathcal{F}_{w}(\psi)(t,u):&=\int_{C_{u}^{t}} c^{-1}\kappa|L\psi|^{2}.
\end{cases}
\end{equation}
Integrate $D_{\mu}P^{\mu}$ over $\mathcal{D}(\delta)(t, u)$, we have
\begin{align}\label{vor: energy estimate for wave equation with multiplier hat(L)}
\begin{split}
\mathcal{E}_{w}(\psi)(t,u)+\mathcal{F}_{w}(\psi)(t,u)&=\mathcal{E}_{w}(\psi)(\delta,u)+\mathcal{F}_{w}(\psi)(t,0)+ \underbrace{\int_{D(t,u)} -\mu \varrho \hat{L}\psi}_{Q_{0}}\\
&+\underbrace{\int_{\mathcal{D}(\delta)(t,u)} \frac{1}{2}L(\kappa^{2})\hat{X}(\psi)^{2}}_{Q_{1}}+\underbrace{\int_{D(t,u)} T(c^{-1}\kappa)|L\psi|^{2}}_{Q_{2}}\\
&+\underbrace{\int_{\mathcal{D}(\delta)(t,u)} ((\eta+\zeta)c^{-1}\kappa-\mu \hat{X}(c^{-1}\kappa))L(\psi)\hat{X}(\psi)}_{Q_{3}}\\
&\underbrace{-\int_{\mathcal{D}(\delta)(t,u)} c^{-1}\kappa[\frac{1}{2}\mu \chi\hat{X}(\psi)^{2}+\frac{1}{2}\chi L(\psi)\underline{L}(\psi)]}_{Q_{4}}
\end{split}
\end{align}

\item {\bf (2)$ J=\Lb$.} We will use $\underline{\mathcal{E}}_{w}(\psi)$ to denote the associated energy and $\underline{\mathcal{F}}_{w}(\psi)$ to denote the corresponding flux,
\begin{equation}
\begin{cases}
\underline{\mathcal{E}}_{w}(\psi)(t,u):&= \int_{\Sigma_{t}^{u}}\frac{1}{2}(|\underline{L}\psi|^{2}+\kappa^{2}|\slashed{\nabla}\psi|^{2}),\\
\underline{\mathcal{F}}_{w}(\psi)(t,u):&=\int_{C_{u}^{t}} \mu |\slashed{\nabla}\psi|^{2}.
\end{cases}
\end{equation}
Integrate $D_{\mu}P^{\mu}$ over $\mathcal{D}(\delta)(t, u)$, we have
\begin{align}\label{vor: energy estimate for wave equation with multiplier underline(L)}
\begin{split}
\underline{\mathcal{E}}_{w}(\psi)(t,u)+\underline{\mathcal{F}}_{w}(\psi)(t,u)&=\underline{\mathcal{E}}_{w}(\psi)(\delta,u)+\underline{\mathcal{F}}_{w}(\psi)(t,0)+\underbrace{\int_{D(t,u)} -\mu \varrho \underline{L}\psi}_{\underline{Q}_{0}}\\
&+\underbrace{\int_{\mathcal{D}(\delta)(t,u)}\frac{1}{2}(\underline{L}\mu+\mu L(c^{-1}\kappa))\hat{X}(\psi)^{2}}_{\underline{Q}_{1}}\\
&\underbrace{-\int_{\mathcal{D}(\delta)(t,u)}(\zeta+\eta)\underline{L}(\psi) \hat{X}(\psi)}_{\underline{Q}_{2}}\underbrace{-\int_{D(t,u)}\mu \hat{X}(c^{-1}\kappa)\hat{X}(\psi)L(\psi)}_{\underline{Q}_{3}}\\
&\underbrace{-\int_{\mathcal{D}(\delta)(t,u)} (\frac{1}{2}\mu\underline{\chi}\hat{X}(\psi)^{2}+\frac{1}{2}L(\psi)\underline{L}(\psi)\underline{\chi})}_{\underline{Q}_{4}}
\end{split}
\end{align}
\end{itemize}

\subsubsection{Energy identities for linear vorticity wave} Let $F$ be a source term, we derive energy identities for the linear transport equation for the $B\Omega=F$. Following \cite{LukSpeck2D}, we write $P^{\mu}=\Omega^{2}B^{\mu}$, then 
\[
D_{\mu}B^{\mu}=\kappa^{-1}(L\kappa+\underline{L}c)+\frac{1}{2}\chi+\frac{1}{2}c\kappa^{-1}\underline{\chi},
\]
Integrate $D_{\mu}P^{\mu}$ over $\mathcal{D}(\delta)(t, u)$, we have
\begin{equation}\label{vor: energy estimate for transport equation}
\begin{split}
\mathcal{E}_{t}(\Omega)(t,u)+ \mathcal{E}_{t}(\Omega)(t,u)&=\ \mathcal{E}_{t}(\Omega)(\delta,u)+ \mathcal{E}_{t}(\Omega)(t,0)\\
&+\underbrace{\int_{\mathcal{D}(\delta)(t,u)} [c(L\kappa+\underline{L}c)+\frac{1}{2}\mu \chi+\frac{c^{2}}{2}\underline{\chi}]\Omega^{2}}_{V_{1}}+\underbrace{\int_{\mathcal{D}(\delta)(t,u)} 2\mu F\Omega}_{V_{0}}.
\end{split}
\end{equation}

where we use $\mathcal{E}_{transport}(\Omega)(t,u)$ to denote the energy and $\mathcal{F}_{transport}(\Omega)(t,u)$ to denote the flux
\begin{equation}
\begin{cases}
\mathcal{E}_{transport}(\Omega)(t,u)=\int_{\Sigma_{t}^{u}}\mu \Omega^{2},\\
\mathcal{F}_{transport}(\Omega)(t,u)=\int_{C_{u}^{t}}c^{2} \Omega^{2}.
\end{cases}
\end{equation}

\subsection{The existence of the initial data on $\Sigma_{\delta}$ and $C_{0}$}

\subsubsection{Formulation for the data on $\Sigma_\delta$}

The initial data for $U$ or equivalently $(\wb,w,\psi_2)$ is already prescribed on the acoustical null hypersurface $C_0^{t^{*}}$. This is because $C_0^{t^{*}}$ is the future boundary of the domain of dependence $\mathcal{D}_0$ of the solution $U_{r}$ associated to the data $U_{r}|_{t=0}$ given on $x_1\geqslant 0$. The trace of $U_{r}$ on $C_0$ is well-defined. It suffices to prescribe initial data on $\Sigma_\delta^{u^{*}}$. This is depicted in the following picture:
\begin{center}
\includegraphics[width=4.5in]{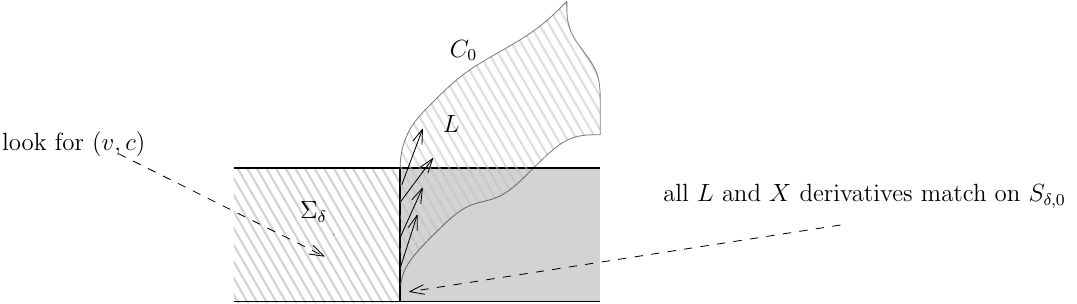}
\end{center}
In particular, since $L$ and $X$ are tangential to $C_0$, for all $m$ and $n$, $L^mX^n (U)$ are already determined by the solution on $\mathcal{D}_0$. We remark that $X=\frac{\partial}{\partial \vartheta}$ and $[L,X]=0$. We also remark that on $S_{\delta,0}$, the vector field $L=B-c\widehat{T}$ is completely determined by the solution on $\mathcal{D}_0$. This is because {\color{black}$\widehat{T}$ is the unit normal of $S_{\delta,0}$ in $\Sigma_{\delta}^{u^{*}}$.}

To summarize, we will look for the data $U$ on $\Sigma_\delta^{u^{*}}$ so that the following compatibility conditions are satisfied:
\begin{itemize}
\item[C1)](Smoothness in acoustical coordinates) We require that $U$ are smooth functions in $(u, \vartheta)$ for $(u,\vartheta)\in [0,u^*]\times [0,2\pi]$. 
\item[C2)](Continuity across $S_{\delta,0}$) We require that  $U|_{S_{\delta,0}}=U_{r}|_{S_{\delta,0}}$. 
\item[C3)](Compatibility in higher order derivatives) We require that higher jets of $U$ along $C_0$ are compatible with those of $U_{r}$ at $S_{\delta,0}$, i.e., for a fixed positive integer $N$, for all nonnegative integers $m$ and $n$ with $m+n\leqslant N$, we have
\begin{equation}\label{compatibility condition 3}L^mX^n(U)|_{S_{\delta,0}}=L^mX^n(U)|_{S_{\delta,0}}.
  \end{equation}
\end{itemize}

\subsubsection{The assumptions on the initial data in the rarefaction wave region} Given $\psi\in \{\wb, w, \psi_{2}\}$ on $\mathcal{D}(t, u)$, for a multi-index $\alpha$, for all $(t, u)\in [\delta, t^*]\times [0, u^*]$, we define the total wave energy and total wave flux associated to $\Zr^{\alpha}(\psi)$ as follows:
\[
\begin{cases}
\mathscr{E}_{w}(\Zr^{\alpha}(\psi))(t, u):=\mathcal{E}_{w}(\Zr^{\alpha}(\psi))(t, u)+\underline{\mathcal{E}}_{w}(\Zr^{\alpha}(\psi))(t, u),\\
\mathscr{F}_{w}(\Zr^{\alpha}(\psi))(t, u):=\mathcal{F}_{w}(\Zr^{\alpha}(\psi))(t, u)+\underline{\mathcal{F}}_{w}(\Zr^{\alpha}(\psi))(t, u).
\end{cases}
\]
For all $n\leq \Ntop$ and $\psi\in \{w, \psi_{2}\}$, we define
\[
\mathscr{E}_{w, \leq n}(\psi)(t, u)=\sum_{|\alpha|\leq n}\mathscr{E}_{w}(\Zr^{\alpha}(\psi)),  \mathscr{F}_{w, \leq n}(\psi)(t, u)=\sum_{|\alpha|\leq n}\mathscr{F}_{w}(\Zr^{\alpha}(\psi)), 
\]
while for $\psi=\wb$, we define
\[
\mathscr{E}_{w, \leq n}(\wb)(t, u)=\mathring{\mathscr{E}}_{w}(\wb)(t, u)+\sum_{1\leq |\alpha|\leq n}\mathscr{E}_{w}(\Zr^{\alpha}(\psi)),  \mathscr{F}_{w, \leq n}(\psi)(t, u)=\mathring{\mathscr{F}}_{w}(\wb)(t, u)+\sum_{1\leq |\alpha|\leq n}\mathscr{F}_{w}(\Zr^{\alpha}(\psi)), 
\]
where
\begin{equation}
\mathring{\mathscr{E}}_{w}(\wb)(t, u)=\frac{1}{2}\int_{\Sigma_{t}^{u}} c^{-2}\kappa^{2}(L\wb)^{2}+\kappa^{2}(\Xr\underline{w})^{2},  \mathring{\mathscr{F}}_{w}(\psi)(t, u)=\int_{C_{u}^{t}} c^{-1}\kappa (L\wb)^{2}+c\kappa (\Xr\wb)^{2}.
\end{equation}

Given $\Omega$ on $\mathcal{D}(t, u)$, for a multi-index $\alpha$, for all $(t, u)\in [\delta, t*]\times [0, u*]$, the energy and flux associated to $\Zr^{\alpha}(\Omega)$ is given by
\[
\begin{cases}
\mathcal{E}_{t}(\Zr^{\alpha}(\Omega))(t, u),\\
\mathcal{E}_{t}(\Zr^{\alpha}(\Omega))(t, u).
\end{cases}
\]
For all $n\leq \Ntop+1$, we define
\[
\mathscr{E}_{t, \leq n}(\psi):=\sum_{|\alpha|\leq n}\mathcal{E}_{t}(\Zr^{\alpha}(\Omega)), \mathscr{F}_{t, \leq n}(\psi):=\sum_{|\alpha|\leq n}\mathcal{F}_{t}(\Zr^{\alpha}(\Omega)).
\]
In order to state a prior energy estimates of the paper, we need precise estimates on the initial data posed on $\Sigma_{\delta}^{u*}$ and $C_{0}^{t^{*}}$. It consists of three sets of assumptions $(\mathbf{I}_{0}), (\mathbf{I}_{2})$ and $(\mathbf{I}_{\infty})$. The assumption are listed as follows:

\begin{equation}
(\mathbf{I}_{0}) \ \ u^{*}=r_{0}\frac{\gamma+1}{\gamma-1}\mathring{c}_{0}, r_{0}\in (0, 1).
\end{equation}

\begin{equation}{(\mathbf{I}_{2})}
\begin{cases}
\sum_{|\alpha|=n}\mathscr{E}_{w}(\Zr^{\alpha}(\psi))(\delta,u^{*})\leq C_{0}\varepsilon^{2}\delta^{2}, \psi\in \{\underline{w},w,\psi_{2}\}, 1\leq n\leq N_{top},\\
\sum_{|\alpha|=n}\mathscr{F}_{w}(\Zr^{\alpha}(\psi))(t,0)\leq C_{0}\varepsilon^{2}t^{2}, \psi\in \{\underline{w},w,\psi_{2}\}, 1\leq n\leq N_{top}, \delta\leq t\leq t^{*},\\
\mathcal{E}_{w}(\delta,u^{*})+\underline{\mathcal{E}}_{w}(\psi)(\delta,u^{*})\leq C_{0}\varepsilon^{2}\delta^{2}, \psi\in\{w,\psi_{2}\},\\
\mathcal{F}_{w}(\psi)(t,0)+\underline{\mathcal{F}}_{w}(\psi)(t,0)\leq C_{0}\varepsilon^{2}t^{2}, \psi\in\{w,\psi_{2}\}, \delta\leq t\leq t^{*}\\
\mathscr{E}_{t, \leq \Ntop+1}(\Omega)(\delta, 0)\leq C_{0}\varepsilon^{2}\delta,\\
\mathscr{F}_{t, \leq \Ntop+1}(\Omega)(t,0)\leq C_{0}\varepsilon^{2}t, \delta\leq t\leq t^{*}
\end{cases}
\end{equation}

\begin{equation}{(\mathbf{I}_{\infty})}
\begin{cases}
|LZ^{\alpha}(\underline{w})|+|LZ^{\alpha}(w)|+|LZ^{\alpha}(\psi_{2})|\leq C\varepsilon,\\
|\hat{X}Z^{\alpha}(\psi)|\leq C\varepsilon, \psi\in \{\wb, w, \psi_{2}\},\\
|T(w)|+|T(\psi_{2})|+|T(\underline{w})+\frac{2}{\gamma+1}|\leq C\varepsilon \delta,\\
|TZ^{\alpha+1}(\psi)|\leq C\varepsilon \delta,\\
|Z^{\alpha}(\hat{T}^{1}+1)|\leq C\varepsilon^{2}\delta^{2}, |Z^{\alpha}(\hat{T}^{2})|\leq C\varepsilon \delta,\\
|Z^{\alpha}(\kappa-\delta)|\leq C\varepsilon \delta^{2}, |Z^{\alpha}(\slashed{g})|\leq C\varepsilon \delta.
\end{cases}, |\alpha|\leq \Ntop+10.
\end{equation}

\begin{remark}\label{rem: away from vacuum on Sigma_delta}
In view of $(\mathbf{I}_{\infty})$ implies that $Tc+\frac{\gamma-1}{\gamma+1}\lesssim \varepsilon\delta$ on $\Sigma_{\delta}^{u*}$. In view of $(\mathbf{I}_{0})$ and $T|_{\Sigma_{\delta}^{u*}}=\frac{\partial}{\partial u}$, we may assume on $\Sigma_{\delta}^{u^{*}}$
\[
\frac{1}{2}r_{0}\mathring{c}_{r}\leq c\leq 2\mathring{c}_{r}.
\]
providing that $\varepsilon$ is sufficiently small.
\end{remark}

\begin{definition}\label{def:rare data}
	Given $(v,c)$ on $\Sigma_\delta$ which satisfies the above condition C1),C2) and C3), we say that $(v,c)$ is a {\bf $C^N$ data}. We say that it is a \textbf{$C^N$ data for rarefaction waves}, if it satisfies in addition the rarefaction ansatz $\mathbf{(I_0)}, \mathbf{(I_2)}$ and $\mathbf{(I_\infty)}$. 
\end{definition}

\begin{proposition}
Given $\Ur_{r}$ and sufficiently small $\varepsilon$. For any smooth $U_{r}|_{t=0}$ defined on the right $x_{1}\geq 0$ up to the boundary with 
\[
\|\Ur_{r}-U_{r}\|_{H^{N_{0}}([0, +\infty)\times [0, 2\pi]}\leq \varepsilon.
\]
and any $\delta\in (0, \frac{1}{2})$, there exists \textbf{$C^N$ data for rarefaction waves}.
\end{proposition}

Under irrotational condition, the existence of \textbf{$C^N$ data for rarefaction waves} has been constructed in \cite{Luo-YuRare2}. Moreover, the construction in \cite{Luo-YuRare2} is essentially independent of the irrotational condition and can be modified slightly to handle the case {\bf without irrotational condition}, see section \ref{sec: existence of initial data} for the proof.

\subsection{A prior energy estimate of the paper}

\subsubsection{Main energy estimate} We now state the first theorem of the paper:
\begin{theorem}{(A priori Energy estimate)}\label{theorem: A priori Energy estimate}
Assume the initial data posed on $\Sigma_{\delta}^{u^*}$ and $C_{0}^{t^{*}}$ satisfying $(\mathbf{I}_{0}), (\mathbf{I}_{2})$ and $(\mathbf{I}_{\infty})$. For $N_{top}$ sufficiently large, there exists a constant $\varepsilon_{0}>0$, so that for all $0< \delta< \frac{1}{2}$, for all $\varepsilon\leq \varepsilon_{0}, \mathcal{D}(\delta)(t, u^*)\subset \mathcal{D}(\delta)(t^* , u^*)$, there exists a constant $C>0$ such that for all $t\in [\delta, 1]$, we have
\begin{equation}\label{eq: main estimates}
\begin{cases}
\mathcal{E}_{w}(\psi)(t, u^{*})+\underline{\mathcal{E}}_{w}(\psi)(t, u^{*})\leq C\varepsilon^{2}t^{2}, \psi\in \{w, \psi_{2}\},\\
\sum_{|\alpha|=n}\mathscr{E}_{w}(\Zr^{\alpha}(\psi))(t, u^{*})\leq C\varepsilon^{2}t^{2}, \psi\in \{\wb, w, \psi_{2}\}, 1\leq n\leq \Ntop,\\
\mathscr{E}_{t, \leq \Ntop+1}(\Omega)(t, u^{*})\leq C\varepsilon^{2}t.
\end{cases}
\end{equation}
\end{theorem}

\subsubsection{The bootstrap argument and the ansatz} We use the method of continuity to prove the main estimates \eqref{eq: main estimates}. We propose a set of energy ansatz $(\mathbf{B}_{2})$, $L^{\infty}$ ansatz $(\mathbf{B}_{\infty})$ and geometry ansatz $(\mathbf{B}_{\text{diffeomorphism}})$. Then we will run a bootstrap argument to close them on $\mathcal{D}(\delta)(t^{*}, u^{*})$ to close them in section \ref{section4}, \ref{section5} and \ref{section6}.

The ansatz $(\mathbf{B}_{2})$ is as follows: we assume there exists a constant $M>0$, so that for all $(t, u)\in [\delta, t^{*}]\times [0, u^{*}]$, the following inequalities hold:
\begin{equation}{(\mathbf{B}_{2})}\label{assumption: B2}
\begin{cases}
\sum_{|\alpha|=n}\Big(\mathscr{E}_{w}(\Zr^{\alpha}(\psi))(t,u)+\mathscr{F}_{w}(\Zr^{\alpha}(\psi))(t,u)\Big)\leq M\varepsilon^{2}t^{2},\psi\in\{\underline{w},w,\psi_{2}\},1\leq n\leq N_{top},\\
\mathcal{E}_{w}(\psi)(t,u)+\underline{\mathcal{E}}_{w}(\psi)(t,u)+\mathcal{F}_{w}(\psi)(t,u)+\underline{\mathcal{F}}_{w}(\psi)(t,u)\leq M\varepsilon^{2}t^{2}, \psi\in\{w,\psi_{2}\},\\
\mathscr{E}_{t, \leq \Ntop+1}(\Omega)(t, u)+\mathscr{F}_{t, \leq \Ntop+1}(\Omega)(t, u)\leq M\varepsilon t^{2}.
\end{cases}
\end{equation}

The ansatz $(\mathbf{B}_{\infty})$ is as follows: we assume there exists a constant $M>0$, so that for all $(t, u)\in [\delta, t^{*}]\times [0, u^{*}]$ and $\psi\in \{\wb, w, \psi_{2}\}$, the following inequalities hold:
\begin{equation}{(\mathbf{B}_{\infty})}\label{assumption: B_infty}
\begin{cases}
\|L(\wb)\|_{L^{\infty}(\Sigma_{t}^{u})}+\|L(w)\|_{L^{\infty}(\Sigma_{t}^{u})}+\|L(\psi_{2})\|_{L^{\infty}(\Sigma_{t}^{u})}\leq M\varepsilon,\\
\|LZ^{\alpha}(\psi)\|_{L^{\infty}(\Sigma_{t}^{u})}+\|\hat{X}Z^{\alpha}(\psi)\|_{L^{\infty}(\Sigma_{t}^{u})}\leq M\varepsilon, |\alpha|\leq \Ninf-2,\\
\|T(w)\|_{L^{\infty}(\Sigma_{t}^{u})}+\|T(\psi_{2})\|_{L^{\infty}(\Sigma_{t}^{u})}\leq M\varepsilon t,\\
\|TZ^{\alpha}(\psi)\|_{L^{\infty}(\Sigma_{t}^{u})}\leq M\varepsilon t,  1\leq|\beta|\leq \Ninf-2.
\end{cases}
\end{equation}

The geometry ansatz $(\mathbf{B}_{\text{diffeomorphism}})$ is as follows: for all for all $(t, u)\in [\delta, t^{*}]\times [0, u^{*}]$, the following inequalities hold:
\begin{equation}{(\mathbf{B}_{\text{diffeomorphism}})}\label{assump: B_geoemtry}
\begin{cases}
|\hat{T}^{1}+1|\leq \frac{1}{2026},\\
|\frac{\partial u}{\partial x_{1}}+\frac{1}{t}|\leq \frac{1}{2026}, |\frac{\partial u}{\partial x_{2}}|\leq \frac{1}{2026},\\
|\frac{\partial \vartheta}{\partial x_{1}}|\leq \frac{1}{2026}t, |\frac{\partial \vartheta}{\partial x_{2}}-1|\leq \frac{1}{2026}t.
\end{cases}
\end{equation}

\subsection{Existence of rarefaction waves connected to the data on the right}
For a constant state on the right in one-dimensional case,  it can be connected by a front rarefaction wave. We show that  there is an analogue in multi-dimensional cases. We let $\mathring{u}^{*}$ be the minimal slope of background rarefaction fronts opened up to $u^{*}$ with initial data $\Ur_{r}$ given on $x_{1}\geq 0$,
\[
\mathring{u}^{*}= (\mathring{v}_{r}^{1}+\mathring{c})-u^*
\]
and define the following region:
\[\mathcal{W}=\big\{(t,x_1,x_2)\in \mathbb{R}^3-\mathcal{D}_0\big| 0<t\leqslant t^*,  \frac{x_1}{t}\geqslant \mathring{u}^{*}+\varepsilon_0\big\}.\]
\begin{center}
\includegraphics[width=3.2in]{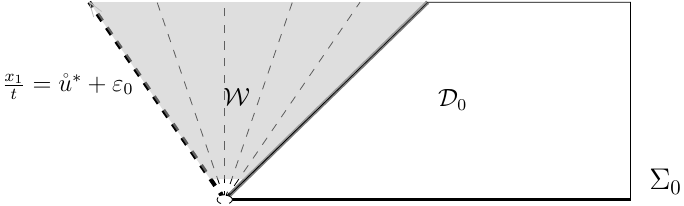}
\end{center}
where $\varepsilon_{0}$ is account for the $O(\varepsilon)$ perturbation and we refer to \eqref{def: varepsilon0} for definition.
\begin{theorem}[Existence of rarefaction waves  connected to the data on the right]\label{theorem: existence-rarefaction}
	There exists an solution $U$ to the Euler equations \eqref{eq: Euler equations} defined on $\mathcal{W}$ so that 
\begin{itemize}
\item[1)] $U|_{C_0^{t^{*}}}=U_{r}|_{C_0}$;
\item[2)] $U\in C^{0} \big((0,t^*]; C^k(\Sigma_t\cap \overline{\mathcal{W}}) \big)$ for all $k\leqslant N_{3}$;
\item[3)] For all $t\in (0,t^*]$, we have $\big|\frac{\partial\wb}{\partial x_1}-\frac{2}{\gamma+1}\frac{1}{t}\big|\lesssim \varepsilon$. Therefore,  the continuous solution to the Euler equations defined by $U$ on $\mathcal{D}_0$ and by $U_{r}$ on $\mathcal{W}$ is not a $C^1$ solution.
\item[4)] Take $\ub=u^{*}+\varepsilon_{0}$ and define\footnote{We remark that $\overline{\mathcal{W}}$ is a subset of $\mathcal{W}$ but not the closure of $\mathcal{W}$.}
\[
\overline{\mathcal{W}}=\cup_{u\in [0, \ub]}C_{u}^{t^{*}}.
\]
Then the initial singularity $\mathbf{S}_{*}$ 
\[\mathbf{S}_*:=\big\{(t,x_1,x_2)\big|t=0,x_1=0\big\},\]
can be resolved by the acoustical coordinate $(t, u, \vartheta)\in [0, t^{*}]\times [0, \ub]\times \mathbb{R}/2\pi\mathbb{Z}$ and $U$ can be extended to be a smooth solution up to $t=0$ (in acoustical coordinate),
\[
U\in C^{N_{4}}([0, t^{*}]\times [0, \ub]\times [0, 2\pi]).
\]
\end{itemize}
\end{theorem}
The proof the theorem \ref{theorem: existence-rarefaction} is essentially the same as the irrotational case \cite{Luo-YuRare2} except for the step of {\bf Retrieving uniform $L^{k}$ bounds of solution} since we need to use wave-transport system to deriving ODE systems for $\psi$ and $\Omega$ simultaneously. We refer to section 7.2 \ref{section 7.2} for details.
\begin{remark}
In \cite{Luo-YuRare2}, the {\bf canonical canonical foliation} of characteristic hypersurfaces in the rarefaction wave region $\mathcal{W}$ has been constructed, we refer to section 5 of \cite{Luo-YuRare2} for detailed account.
\end{remark}

\subsection{Applications to Riemann problem}

\begin{theorem}\label{thm: R-R}
We use $\mathbf{S}_*$ to denote the singularity:
\[\mathbf{S}_*:=\big\{(t,x_1,x_2)\big|t=0,x_1=0\big\}.\]
There exists a constant $\varepsilon_*>0$, for all $\varepsilon<\varepsilon_*$ and for any given data $(v,c)\big|_{t=0}$ in Definition \ref{def:data for R-R}, there exists a continuous solution to the Euler equations \eqref{eq: Euler equations} on $[0,t^*]\times \mathbb{R}^2-\mathbf{S}_*$.

\begin{center}
\includegraphics[width=3in]{II-Pic19.pdf}
\end{center}

The solution is piecewisely $C^{N_7}$ in the following sense:
\begin{itemize}
\item[1)] There are four characteristic hypersurfaces emanating from the singular set $\mathbf{S}_*$ and we denote them as $\Cb_0, \Hb, H$ and $C_0$ from left to right (as $x_1$ increases). 
\begin{itemize}
\item The hypersurfaces $\Cb_0$  and $C_0$ are the characteristic boundaries of the domain of dependence of the data $U_{l}|_{t=0}$ posed on $x_1\leqslant 0$ and  $U_{r}|_{t=0}$ posed on $x_1\geqslant 0$ respectively;
\item The hypersurfaces $\Hb$ and $H$ are uniquely determined by the data $U|_{t=0}$.
\end{itemize}
\item[2)] The solution is of class $C^N$ for all points $(t,x_1,x_2)$ with $t\geqslant 0$ on the left of $\Cb_0$ or on the right of $C_0$.
\item[3)] The solution is of class $C^N$ for all points $(t,x_1,x_2)$ with $t>0$ between $\Cb_0$ and $\Hb$ or between  $H$ and $C_0$. These two regions are the back and front rarefaction wave regions respectively. Moreover, the solution is not of class $C^1$ on $C_0,\Cb_0,H$ and $\Hb$.
\item[4)] The solution is of class $C^N$ for all points $(t,x_1,x_2)$ with $t\geqslant 0$ (including the singularity) between $\Hb$ and $H$.
\end{itemize}
\end{theorem}

\begin{theorem}
We use $\mathbf{S}_*$ to denote the singularity:
\[\mathbf{S}_*:=\big\{(t,x_1,x_2)\big|t=0,x_1=0\big\}.\]
There exists a constant $\varepsilon_*>0$, for all $\varepsilon<\varepsilon_*$ and for any given data $(v,c)\big|_{t=0}$ in Definition \ref{def:data for S-R}, there exists a continuous solution to the Euler equations \eqref{eq: Euler equations} on $[0,t^*]\times \mathbb{R}^2-\mathbf{S}_*$.  

\begin{center}
\begin{tikzpicture}
  \begin{axis}[axis lines=none]
  \addplot [domain=0:5, smooth, color=blue, dotted, thick] {0.4*x+0.002*sin(540*x)*x*(x-5)+0.05*x*(x-5)};
    \addplot [domain=0.5:5.5, smooth, color=blue, dotted, thick] {0.4*(x-0.5)+1.5+0.002*sin(540*x)*(x-0.5)*(x-5.5)+0.04*(x-0.5)*(x-5.5)};
    \draw [thick] (0, 0)--(0.5,1.5);
    \addplot [domain=5:5.5, smooth, color=blue, dotted, thick] {3*(x-5)+2+2*sin(1080*x-120)*(x-5)*(x-5.5)};
   \draw (5, 2.7) node {$H$};
     \addplot [domain=0:10, smooth, color=blue, dotted, thick] {0.2*x+0.0005*sin(360*x)*x*(x-10)+0.01*x*(x-10)};
    \addplot [domain=0.5:10.5, smooth, color=blue, dotted, thick] {0.2*(x-0.5)+1.5+0.0005*sin(360*x)*(x-0.5)*(x-10.5)+0.01*(x-0.5)*(x-10.5)};
    \draw [thick] (0, 0)--(0.5,1.5);
    \addplot [domain=10:10.5, smooth, color=blue, dotted, thick] {3*(x-10)+2+2*sin(1080*x-60)*(x-10)*(x-10.5)};
    \draw (10, 2.7) node {$C_{0}$};

     \addplot [domain=0:-4, smooth, thick, color=blue] {-0.5*x+0.05*x*(x+4)+0.01*sin(1080*x*(x+4))};
    \addplot [domain=0.5:-3.5, smooth, thick, color=blue] {-0.5*(x-0.5)+1.5+0.05*(x-0.5)*(x+3.5)+0.01*sin(1080*(x-0.5)*(x+3.5))};
    \draw [thick] (0, 0)--(0.5,1.5);
    \addplot [domain=-4:-3.5, smooth, thick, color=blue] {3*(x+4)+2+2*sin(1080*x-30)*(x+4)*(x+3.5)};
    \draw (-4, 2.7) node {$S$};
    \draw[thin] (-6,0)--(10.5,0);
    \draw[thin] (-6,2)--(10.5,2);
    
    \draw[thin] (-5.5, 1.5)--(11, 1.5);
    \draw[thin] (-5.5, 3.5)--(11, 3.5);
  \end{axis}
\end{tikzpicture}
\end{center}

The solution is piecewisely $C^{N_{7}}$ in the following sense:
\begin{itemize}
\item[1)] There are one non-characteristic hypersurface and two characteristic hypersurfaces emanated from the singular set $\mathbf{S}_*$ and we denote them as $S, H$ and $C_0$ from left to right (as $x_1$ increases). 
\begin{itemize}
\item The hypersurfaces $C_0$ is the characteristic boundary of the domain of dependence of the data $U_{r}|_{t=0}$ posed on $x_1\leqslant 0$;
\item The hypersurfaces $S$ and $H$ are determined by the data $U\big|_{t=0}$.
\end{itemize}
\item[2)] The solution is of class $C^N$ for all points $(t,x_1,x_2)$ with $t\geqslant 0$ on the right of $C_0$.
\item[2)] The solution is of class $C^N$ for all points $(t,x_1,x_2)$ with $t\geqslant 0$ on the left of $S$ or between $S$ and $H$. Moreover, they satisfy the {\bf Rankine-Hugoniot jump condition for shock} and {\bf entropy condition} crossing $S$.
\item[3)] The solution is of class $C^N$ for all points $(t,x_1,x_2)$ with $t>0$ between  $H$ and $C_0$. This is front rarefaction wave region. Moreover, the solution is not of class $C^1$ crossing $C_0$ and $H$.
\end{itemize}
\end{theorem}

\begin{theorem}
We use $\mathbf{S}_*$ to denote the singularity:
\[\mathbf{S}_*:=\big\{(t,x_1,x_2)\big|t=0,x_1=0\big\}.\]
There exists a constant $\varepsilon_*>0$, for all $\varepsilon<\varepsilon_*$ and for any given data $(v,c)\big|_{t=0}$ in Definition \ref{def:data for R-V-R}, there exists a continuous solution to the Euler equations \eqref{eq: Euler equations} on $[0,t^*]\times \mathbb{R}^2-\mathbf{S}_*$.

\begin{center}
\begin{tikzpicture}
  \begin{axis}[axis lines=none]
    
    \addplot [domain=0:2.5, smooth, thick, dashed, color=blue] {0.8*x+0.02*sin(720*x)*x*(x-2.5)+0.05*x*(x-2.5) };
    \addplot [domain=0.5:3, smooth, thick, dashed, color=blue] {0.8*(x-0.5)+1.5+0.02*sin(720*x)*(x-0.5)*(x-3)+0.05*(x-0.5)*(x-3)};
    \draw [thick] (0, 0)--(0.5,1.5);
    \addplot [domain=3:2.5, smooth, thick, dashed, color=blue] {3*(x-2.5)+2+2*sin(1080*x-45)*(x-3)*(x-2.5)};

     \addplot [domain=0:-2, smooth, color=blue, dotted, thick] {-x+0.01*sin(540*x)*x*(x+2)+0.1*x*(x+2)};
    \addplot [domain=0.5:-1.5, smooth, color=blue, dotted, thick] {-x+2+0.01*sin(1080*x)*(x-0.5)*(x+1.5)+0.1*(x-0.5)*(x+1.5)};
    \draw [thick] (0, 0)--(0.5,1.5);
    \addplot [domain=-2:-1.5, smooth, color=blue, dotted, thick] {3*(x+2)+2+2*sin(1080*x)*(x+2)*(x+1.5)};
    
     \draw (-2, 2.7) node {$\Hb_{0}$};

     \addplot [domain=0:5, smooth, color=blue, dotted, thick] {0.4*x+0.002*sin(540*x)*x*(x-5)+0.05*x*(x-5)};
    \addplot [domain=0.5:5.5, smooth, color=blue, dotted, thick] {0.4*(x-0.5)+1.5+0.002*sin(540*x)*(x-0.5)*(x-5.5)+0.08*(x-0.5)*(x-5.5)};
    \draw [thick] (0, 0)--(0.5,1.5);
    \addplot [domain=5:5.5, smooth, color=blue, dotted, thick] {3*(x-5)+2+2*sin(1080*x-120)*(x-5)*(x-5.5)};
    
     \draw (5, 2.7) node {$H_{0}$};
     \draw (2.8, 2.7) node {$V$};

     \addplot [domain=0:10, smooth, color=blue, dotted, thick] {0.2*x+0.0005*sin(360*x)*x*(x-10)+0.01*x*(x-10)};
    \addplot [domain=0.5:10.5, smooth, color=blue, dotted, thick] {0.2*(x-0.5)+1.5+0.0005*sin(360*x)*(x-0.5)*(x-10.5)+0.01*(x-0.5)*(x-10.5)};
    \draw [thick] (0, 0)--(0.5,1.5);
    \addplot [domain=10:10.5, smooth, color=blue, dotted, thick] {3*(x-10)+2+2*sin(1080*x-60)*(x-10)*(x-10.5)};
    
     \draw (10, 2.7) node {$C_{0}$};

     \addplot [domain=0:-4, smooth, color=blue, dotted, thick] {-0.5*x+0.05*x*(x+4)+0.01*sin(1080*x*(x+4))};
    \addplot [domain=0.5:-3.5, smooth, color=blue, dotted, thick] {-0.5*(x-0.5)+1.5+0.05*(x-0.5)*(x+3.5)+0.01*sin(1080*(x-0.5)*(x+3.5))};
    \draw [thick] (0, 0)--(0.5,1.5);
    \addplot [domain=-4:-3.5, smooth, color=blue, dotted, thick] {3*(x+4)+2+2*sin(1080*x-30)*(x+4)*(x+3.5)};

     \draw (-4, 2.7) node {$\Cb_{0}$};

    \draw[thin] (-6,0)--(10.5,0);
    \draw[thin] (-6,2)--(10.5,2);
    
    \draw[thin] (-5.5, 1.5)--(11, 1.5);
    \draw[thin] (-5.5, 3.5)--(11, 3.5);
  \end{axis}
\end{tikzpicture}
\end{center}

The solution is piecewisely $C^{N_{7}}$ in the following sense:
\begin{itemize}
\item[1)] There are five characteristic hypersurfaces all emanated from the singular set $\mathbf{S}_*$ and we denote them as $\Cb_0, \Hb, H, V$ and $C_0$ from left to right (as $x_1$ increases). 
\begin{itemize}
\item The hypersurfaces $\Cb_0$  and $C_0$ are the characteristic boundaries of the domain of dependence of the data $U_{l}|_{t=0}$ posed on $x_1\leqslant 0$ and  $U_{r}|_{t=0}$ posed on $x_1\geqslant 0$ respectively;
\item The hypersurfaces $\Hb$ and $H$ are uniquely determined by the data $U|_{t=0}$.
\end{itemize}
\item[2)] The solution is of class $C^N$ for all points $(t,x_1,x_2)$ with $t\geqslant 0$ on the left of $\Cb_0$ or on the right of $C_0$.
\item[3)] The solution is of class $C^N$ for all points $(t,x_1,x_2)$ with $t>0$ between $\Cb_0$ and $\Hb$ or between  $H$ and $C_0$. These two regions are the back and front rarefaction wave regions respectively. Moreover, the solution is not of class $C^1$ on $C_0,\Cb_0,H$ and $\Hb$.
\item[4)] The solution is of class $C^N$ for all points $(t,x_1,x_2)$ with $t\geqslant 0$ (including the singularity) between $\Hb$ and $V$ or between $H$ and $V$. Moreover, they satisfy the {\bf Rankine-Hugoniot jump condition for vortex sheet}.
\end{itemize}
\end{theorem}

\subsection{Ideas of the Proof}

\begin{itemize}
    \item \textbf{Exploitation of the extremal property of acoustical wave eigenvalues (Luo–Yu’s framework)}
    
    The construction of rarefaction fronts is inherently a free boundary problem, which generically suffers from a loss of derivatives. In their seminal work \cite{Luo-YuRare1, Luo-YuRare2}, Luo and Yu exploit the extremal property of acoustical wave eigenvalues to construct all rarefaction waves $C_u$ connected to the right state $U_r$ (away from vacuum), as well as all rarefaction waves $C_{\ub}$ connected to the left state $U_l$ (away from vacuum). We emphasize that these two constructions can be carried out independently.

    More precisely, they adopt the last-slice method to construct approximate initial data on $\Sigma_{\delta}$, and use the structure of the acoustical wave equation to propagate all normal jets from $S_{\delta, 0}:=\Sigma_{\delta}\cap C_{0}$. The rarefaction wave region is then constructed via classical energy estimates for the acoustical wave equation, and the rarefaction fronts are no longer a free boundary problem: they are directly determined by the initial conditions on $\Sigma_{\delta}$.

    Finally, they again exploit the extremal property of acoustical wave eigenvalues to couple the rarefaction waves $C_u$ (connected to $U_r$) and $C_{\ub}$ (connected to $U_l$). The inner boundaries $H$ and $\Hb$ are determined via simple algebraic relations, which reduces the Riemann problem in the R-R regime to a Goursat problem for the acoustical wave equations with data prescribed on $H$ and $\Hb$.

    \item \textbf{Non-characteristic nature of $\Omega$ with respect to rarefaction fronts $C_u$ (lifting the irrotationality assumption)}
    
    To lift the irrotationality assumption in Luo–Yu’s work \cite{Luo-YuRare1, Luo-YuRare2}, we employ the wave-transport system derived in Luk–Speck’s work \cite{LukSpeck2D} on stable shock formation with vorticity. This system allows us to perform classical energy estimates for the acoustical waves and the vorticity simultaneously. However, the initial data constructed in \cite{Luo-YuRare2} do not have sufficient regularity for this purpose: in view of \eqref{eq: reformulation of vorticity in the first null frame}, $\hat{T}(\Omega)$ can be expressed as $C_u$-tangential derivatives of $(c, v^1, v^2)$, and in particular $T(\Omega)$ should satisfy the additional vanishing property $|T(\Omega)|=O(\varepsilon \delta)$ on $S_{\delta, 0}$. The construction in \cite{Luo-YuRare2} only ensures $|T(\Omega)|=O(\varepsilon)$ on $\Sigma_{\delta}$, which forces us to close our energy estimates under the weaker bootstrap assumption $|T(\Omega)|=O(\varepsilon)$ on $\Sigma_{\delta}$.

    Nevertheless, the structure of the wave-transport system and a refined Gronwall inequality enable us to close the energy estimates even under this weaker assumption. Once the energy estimates are closed, we recover the sharp bound for $\Omega$ away from $C_0$ by exploiting the non-characteristic nature of $\Omega$ with respect to the rarefaction fronts $C_u$ (away from vacuum).

    We remark on a key comparison between our setting and the stable shock formation with vorticity in \cite{LukSpeck2D}: in the latter, the authors design a descent scheme to control the blow-up of higher derivatives of $\Omega$, while in our work we obtain uniform bounds for all higher derivatives of $\Omega$. We attribute this difference to the fact that the pre-shock in \cite{LukSpeck2D} is of spacetime codimension 1, while the initial singularity of rarefaction waves is of spacetime codimension 2, as explained in the canonical foliation construction for rarefaction waves in \cite{Luo-YuRare2}.

    \item \textbf{Coupling of rarefaction waves with other nonlinear waves}
    
    In general, two distinct nonlinear waves cannot be directly coupled, due to their intrinsic nonlinear interactions. For the 2D isentropic Euler equations for a polytropic gas, we follow Luo–Yu’s approach \cite{Luo-YuRare1, Luo-YuRare2} to independently construct the rarefaction waves $C_u$ (connected to $U_r$) and $C_{\ub}$ (connected to $U_l$), both away from vacuum. This allows us to reduce the Riemann problem in the S-R regime to a Goursat problem admitting a single shock front, as illustrated below:
    \begin{center}
    \begin{tikzpicture}[scale=0.8]
        \begin{axis}[axis lines=none]
            \addplot [domain=0:5, smooth, color=blue, thick] {0.4*x+0.002*sin(540*x)*x*(x-5)+0.05*x*(x-5)};
            \addplot [domain=0.5:5.5, smooth, color=blue, thick] {0.4*(x-0.5)+1.5+0.002*sin(540*x)*(x-0.5)*(x-5.5)+0.04*(x-0.5)*(x-5.5)};
            \draw [thick] (0, 0)--(0.5,1.5);
            \addplot [domain=5:5.5, smooth, color=blue, thick] {3*(x-5)+2+2*sin(1080*x-120)*(x-5)*(x-5.5)};
            \addplot [domain=-5:0, smooth, color=blue, thick] {0};
            \draw (5, 2.7) node {$H$};
            \draw(-4, 1) node{$x_{1}\leq 0$};
            \draw[thin] (-6,0)--(10.5,0);
            \draw[thin] (-6,2)--(10.5,2);
            \draw[thin] (-5.5, 1.5)--(11, 1.5);
            \draw[thin] (-5.5, 3.5)--(11, 3.5);
        \end{axis}
    \end{tikzpicture}
    \end{center}
    with data prescribed on $x_1\leq 0$ and $H$. It would be of interest to investigate how to adapt Majda’s method \cite{MajdaShock2, MajdaShock3} for the Cauchy problem to this Goursat setting.

    Similarly, the Riemann problem in the R-V-R regime reduces to a Goursat problem admitting a single vortex sheet, as illustrated below:
    \begin{center}
    \begin{tikzpicture}
        \begin{axis}[axis lines=none]
            \addplot [domain=0:-2, smooth, color=blue, thick] {-x+0.01*sin(540*x)*x*(x+2)+0.1*x*(x+2)};
            \addplot [domain=0.5:-1.5, smooth, color=blue, thick] {-x+2+0.01*sin(1080*x)*(x-0.5)*(x+1.5)+0.1*(x-0.5)*(x+1.5)};
            \draw [thick] (0, 0)--(0.5,1.5);
            \addplot [domain=-2:-1.5, smooth, color=blue, thick] {3*(x+2)+2+2*sin(1080*x)*(x+2)*(x+1.5)};
            \draw (-2, 2.7) node {$\Hb$};
            \addplot [domain=0:5, smooth, color=blue, thick] {0.4*x+0.002*sin(540*x)*x*(x-5)+0.05*x*(x-5)};
            \addplot [domain=0.5:5.5, smooth, color=blue, thick] {0.4*(x-0.5)+1.5+0.002*sin(540*x)*(x-0.5)*(x-5.5)+0.04*(x-0.5)*(x-5.5)};
            \draw [thick] (0, 0)--(0.5,1.5);
            \addplot [domain=5:5.5, smooth, color=blue, thick] {3*(x-5)+2+2*sin(1080*x-120)*(x-5)*(x-5.5)};
            \draw (5, 2.7) node {$H$};
            \draw[thin] (-6,0)--(10.5,0);
            \draw[thin] (-6,2)--(10.5,2);
            \draw[thin] (-5.5, 1.5)--(11, 1.5);
            \draw[thin] (-5.5, 3.5)--(11, 3.5);
        \end{axis}
    \end{tikzpicture}
    \end{center}
    with data prescribed on $H$ and $\Hb$. It would similarly be of interest to adapt Coulombel–Secchi’s method \cite{Coulombel-Secchi1, Coulombel-Secchi2} for the Cauchy problem to this Goursat setting.

    Instead, we proceed via a different strategy to show that \textbf{the Goursat problem for a single shock wave or a single vortex sheet can be reduced to corresponding Cauchy problem}. In view of the extremal nature of $H$ and $\Hb$, we can smoothly extend the solutions on $H$ and $\Hb$ in a controlled manner, via three key observations:
    \begin{enumerate}
        \item The characteristic nature of $H$ and $\Hb$ allows us to freely prescribe certain normal jets;
        \item By the work of Wang–Yu–Yu \cite{wang2025constructingcharacteristicinitialdata}, the special structure of the wave-transport system allows us to propagate the normal jets from $\mathbf{S}_{*}$ along $H$ and $\Hb$;
        \item The classical Goursat problem is well-posed, by the work of Speck and Yu \cite{SpeckYu}.
    \end{enumerate}

    We note that the loss of derivatives in this work does not occur in the energy estimates, but rather in the propagation of normal jets: this arises from the structure of the acoustical wave equation $L(T(\psi))=\frac{1}{2}\mu\hat{X}^{2}(\psi)+\cdots$, where two tangential derivatives are required to estimate a single normal derivative.
\end{itemize}

\subsection{Remark on numerical constants}
We fix an arbitrary positive constant $t^{*}$ since our analysis is restricted to local-in-time existence. Let $\mathring{c}_{r}$ denote the sound speed associated with the background solution, and let $u^{*}$ be an arbitrary constant satisfying $0<u^{*}<\frac{\gamma+1}{\gamma-1}\mathring{c}_{r}$. The parameter $\varepsilon$ is a sufficiently small positive constant that measures the magnitude of the perturbation.

We adopt the standard convention that $C$ denotes a generic positive constant greater than or equal to 1, whose value may change from one occurrence to another. The sequence of regularity exponents $N_{0}>N_{1}>\cdots>N_{7}$ satisfies either $N_{i+1}=N_{i}-C$ (resulting from Sobolev embeddings or limiting procedures) or $N_{i+1}=\frac{N_{i}}{C}$ (resulting from propagation estimates of normal derivatives via ODE systems obtained from the acoustical wave equation).

We also stress that there is \textbf{no loss of derivates} in our \textbf{main energy estimate} in rarefaction wave region.

\section{The existence of initial data on $\Sigma_{\delta}$ and $C_{0}$}\label{sec: existence of initial data}
In this section, we explain how to modify the construction in section 3 of \cite{Luo-YuRare2} to obtain the existence of \textbf{$C^N$ data for rarefaction waves} \emph{without irrotational condition}, see definition \ref{def:rare data}.

\begin{itemize}
\item {\bf (1)} The construction of acoustical function $u$ and $\vartheta$ on $\Sigma_{\delta}^{u*}$ and the verification of $(\mathbf{I}_{\infty})$ for $\hat{T}^{1}, \hat{T}^{2}, \kappa, \slashed{g}$.  In particular, it is proved \cite{Luo-YuRare2} that
\begin{proposition}\label{appendix: initial geometry}
For $|\alpha|\leq N_{1}$ and $Z\in \{T, \hat{X}\}$, we have
\begin{equation}
\begin{cases}
\|Z^{\alpha}(\hat{T}^{1}+1)\|_{L^{\infty}(\Sigma_{\delta}^{u*})}\lesssim \varepsilon^{2}\delta^{2}, \|Z^{\alpha}(\hat{T}^{2})\|_{L^{\infty}(\Sigma_{\delta}^{u*})}\lesssim \varepsilon\delta,\\
\|Z^{\alpha}(\frac{\kappa}{\delta})\|_{L^{\infty}(\Sigma_{\delta}^{u*})}+\|Z^{\alpha}(\slashed{g}-1)\|_{L^{\infty}(\Sigma_{\delta}^{u*})}\lesssim \varepsilon\delta.
\end{cases}
\end{equation}
\end{proposition}
We remark that this step is {\bf the same as the irrotational case}.
\item {\bf (2)} Since it is required that $(c, v^{1}, v^{2})=(c_{r}, v_{r}^{1}, v^{2}_{r})$ along $C_{0}^{t*}$, all the normal jets of $(c, v^{1}, v^{2})$ on $S_{\delta, 0}:=\Sigma_{\delta}^{u*}\cap C_{0}^{t*}$ can be determined inductively by characteristic system of Euler equation along $C_{0}$. In view of \eqref{Euler equations:form 1}, for each $1\leq k\leq N_{1}$, we have freedom to prescribe $T^{k}(\frac{1}{\gamma-1}c+\frac{1}{2}\psi^{(\hat{T})})$, while $T^{k}(\frac{1}{\gamma-1}c-\frac{1}{2}\psi^{(\hat{T})})$ and $T^{k}(\psi^{(\hat{X})})$ are automatically determined. In particular, the choice for $T^{k}(\frac{1}{\gamma-1}c+\frac{1}{2}\psi^{(\hat{T})})$ on $S_{\delta, 0}$ is given by
\[
\begin{cases}
T(\frac{1}{\gamma-1}c+\frac{1}{2}\psi^{(\hat{T})})=-\frac{2}{\gamma+1},\\
T^{k}(\frac{1}{\gamma-1}c+\frac{1}{2}\psi^{(\hat{T})})=0, \ 2\leq k\leq N_{1}.
\end{cases}
\]
{\bf This step is exactly the same with the irrotational case.}
\item {\bf (3)} $(\frac{1}{\gamma-1}c+\frac{1}{2}\psi^{(\hat{T})}), (\frac{1}{\gamma-1}c-\frac{1}{2}\psi^{(\hat{T})})$ and $\psi^{(\hat{X})}$ are constructed on whole $\Sigma_{\delta}^{u*}$ by Taylor expansion in $u$. In particular, it is proved \cite{Luo-YuRare2} that
\begin{proposition}\label{appendix: initial data bound from LuoYu2}
For $|\alpha|\leq N_{1}, 1\leq |\beta|\leq N_{1}, \psi\in \{\wb, w, \psi_{2}\}, Z\in \{T, \hat{X}\}$ and $\Zr\in \{\Tr, \Xr\}$, we have
\begin{equation}
\begin{cases}
\|T(w)\|_{L^{\infty}(\Sigma_{\delta}^{u*})}+\|T(\psi_{2})\|_{L^{\infty}(\Sigma_{\delta}^{u*})}+\|T(\wb)+\frac{2}{\gamma+1}\|_{L^{\infty}(\Sigma_{\delta}^{u*})}\lesssim \varepsilon \delta,\\
\|LZ^{\alpha}(\psi)\|_{L^{\infty}(\Sigma_{\delta}^{u*})}+\|LZ^{\alpha}(\psi)\|_{L^{\infty}(\Sigma_{\delta}^{u*})}\lesssim \varepsilon,\\
\|TZ^{\beta}(\psi)\|_{L^{\infty}(\Sigma_{\delta}^{u*})}\lesssim \varepsilon\delta,\\
\|\Tr(w)\|_{L^{\infty}(\Sigma_{\delta}^{u*})}+\|\Tr(\psi_{2})\|_{L^{\infty}(\Sigma_{\delta}^{u*})}+\|\Tr(\wb)+\frac{2}{\gamma+1}\|_{L^{\infty}(\Sigma_{\delta}^{u*})}\lesssim \varepsilon \delta,\\
\|L\Zr^{\alpha}(\psi)\|_{L^{\infty}(\Sigma_{\delta}^{u*})}+\|\hat{X}\Zr^{\alpha}(\psi)\|_{L^{\infty}(\Sigma_{\delta}^{u*})}+\|\Xr\Zr^{\alpha}(\psi)\|_{L^{\infty}(\Sigma_{\delta}^{u*})}\lesssim \varepsilon,\\
\|T\Zr^{\beta}(\psi)\|_{L^{\infty}(\Sigma_{\delta}^{u*})}+\|\Tr\Zr^{\beta}(\psi)\|_{L^{\infty}(\Sigma_{\delta}^{u*})}\lesssim \varepsilon\delta.
\end{cases}
\end{equation}
\end{proposition}
We now show that {\bf the bound for specific vorticity $\Omega$ follows form \ref{appendix: initial data bound from LuoYu2} directly}. For $|\gamma|\leq N_{1}-1$,
\[
\Zr^{\gamma}(\Omega)=\Zr^{\gamma}(\frac{\partial_{1}v^{2}-\partial_{2}v^{1}}{\rho})=\frac{1}{\kappar}\Zr^{\gamma}\big(\frac{\Tr(\psi_{2})}{(k_{0}^{-\frac{1}{2}}\gamma^{-\frac{1}{2}}c)^{\frac{2}{\gamma-1}}}\big)+\Zr^{\gamma}\big(\frac{\Xr(\psi_{1})}{(k_{0}^{-\frac{1}{2}}\gamma^{-\frac{1}{2}}c)^{\frac{2}{\gamma-1}}}\big).
\]
In view of prop \ref{appendix: initial data bound from LuoYu2} and $c\approx \mathring{c}$, we can conclude that 
\begin{corollary}\label{cor: initial bound for Omega}
For $|\gamma|\leq N_{1}-1$ and $\Zr\in \{\Tr, \Xr\}$, 
\[
\|\Zr^{\gamma}(\Omega)\|_{L^{\infty}(\Sigma_{\delta}^{u*})}\lesssim \varepsilon.
\]
\end{corollary}
In summary, prop \ref{appendix: initial geometry}, prop \ref{appendix: initial data bound from LuoYu2}, together with corollary \ref{cor: initial bound for Omega} imply that the data posed on $\Sigma_{\delta}^{u*}$ satisfies conditions $(\mathbf{I}_{0})$, $(\mathbf{I}_{2})$ and $(\mathbf{I}_{\infty})$.
\item {\bf (4)} To derive the initial flux bound on $C_{0}^{t*}$, the ODE systems for normal jets of $\{\wb, w, \psi_{2}, \hat{T}^{1}, \hat{T}^{2}, \kappa\}$ have been derived in \cite{Luo-YuRare2} by rewriting the acoustical wave equation $\Box_{g}\psi=\cdots$ as
\begin{equation}\label{appendix: transport equation for T(psi)}
\begin{split}
L(T(\psi))&=\kappa\cdot \big(\text{ \fbox{$C_{0}^{t^{*}}$ tangential derivates of $\{\wb, w, \psi_{2}, \hat{T}^{1}, \hat{T}^{2}\}$}}\big)\\
&+T(\psi)\cdot \big(\text{ \fbox{$C_{0}^{t^{*}}$ tangential derivates of $\{\wb, w, \psi_{2}, \hat{T}^{1}, \hat{T}^{2}\}$}}\big),\\
\end{split}
\end{equation}
and writing the transport equation for $\kappa$ as
\begin{equation}\label{appendix: transport equation for T(kappa)}
\begin{split}
L(\kappa)&=\kappa\cdot \big(\text{ \fbox{$C_{0}^{t^{*}}$ tangential derivates of $\{\wb, w, \psi_{2}, \hat{T}^{1}, \hat{T}^{2}\}$}}\big)\\
&+T(\psi)\cdot \big(\text{ \fbox{$C_{0}^{t^{*}}$ tangential derivates of $\{\wb, w, \psi_{2}, \hat{T}^{1}, \hat{T}^{2}\}$}}\big),\\
\end{split}
\end{equation}
and writing the transport equation for $T(\hat{T}^{i}), i=1,2$ as
\begin{equation}\label{appendix: transport equation for T(hat{T})}
\begin{split}
L(\hat{T}^{i})&=\sum_{i=1}^{2}T(\hat{T}^{i})\cdot \big(\text{ \fbox{$C_{0}^{t^{*}}$ tangential derivates of $\{\wb, w, \psi_{2}, \hat{T}^{1}, \hat{T}^{2}\}$}}\big)\\
&+\big(\text{ \fbox{terms already known on $C_{0}^{t^{*}}$ and $C_{0}^{t^{*}}$ tangential derivates of $\{\wb, w, \psi_{2}, \hat{T}^{1}, \hat{T}^{2}, T(\wb), T(w), T(\psi_{2}, \kappa)\}$ }}\big).
\end{split}
\end{equation}
In view of \eqref{eq: reformulation of vorticity in the first null frame}, $\Omega, \partial_{1}\Omega, \partial_{2}\Omega$ can be written as  $C_{0}^{t^{*}}$ tangential derivates of $\{\wb, w, \psi_{2}, \hat{T}^{1}, \hat{T}^{2}\}$. Therefore, the appearance of specific vorticity and its spatial derivates won't effect the structure of \eqref{appendix: transport equation for T(psi)}, \eqref{appendix: transport equation for T(kappa)}, and \eqref{appendix: transport equation for T(hat{T})}. Commuting $T^{k}$ with them, we obtain ODE systems for normal jets of $\{\wb, w, \psi_{2}, \hat{T}^{1}, \hat{T}^{2}, \kappa\}$ with the same structure and same initial data on $S_{\delta, 0}$ as \cite{Luo-YuRare2} and we have
\begin{proposition}\label{appendix: initial flux bound from LuoYu2}
For $|\alpha|\leq N_{2}, 1\leq |\beta|\leq N_{2}, \psi\in \{\wb, w, \psi_{2}\}, Z\in \{T, \hat{X}\}, \Zr\in \{\Tr, \Xr\}$ and $\delta\leq t\leq t^{*}$, 
\begin{equation}
\begin{cases}
\|T(w)\|_{L^{\infty}(S_{t, 0})}+\|T(\psi_{2})\|_{L^{\infty}(S_{t, 0})}+\|T(\wb)+\frac{2}{\gamma+1}\|_{L^{\infty}(S_{t, 0})}\lesssim \varepsilon t,\\
\|LZ^{\alpha}(\psi)\|_{L^{\infty}(S_{t, 0})}+\|LZ^{\alpha}(\psi)\|_{L^{\infty}(S_{t, 0})}\lesssim \varepsilon,\\
\|TZ^{\beta}(\psi)\|_{L^{\infty}(S_{t, 0})}\lesssim \varepsilon t,\\
\|\Tr(w)\|_{L^{\infty}(S_{t, 0})}+\|\Tr(\psi_{2})\|_{L^{\infty}(S_{t, 0})}+\|\Tr(\wb)+\frac{2}{\gamma+1}\|_{L^{\infty}(S_{t, 0})}\lesssim \varepsilon t,\\
\|L\Zr^{\alpha}(\psi)\|_{L^{\infty}(S_{t, 0})}+\|\hat{X}\Zr^{\alpha}(\psi)\|_{L^{\infty}(S_{t, 0})}+\|\Xr\Zr^{\alpha}(\psi)\|_{L^{\infty}(S_{t, 0})}\lesssim \varepsilon,\\
\|T\Zr^{\beta}(\psi)\|_{L^{\infty}(S_{t, 0})}+\|\Tr\Zr^{\beta}(\psi)\|_{L^{\infty}(S_{t, 0})}\lesssim \varepsilon t.
\end{cases}
\end{equation}
\end{proposition}
We now show that {\bf the bound for specific vorticity $\Omega$ follows form \ref{appendix: initial flux bound from LuoYu2} directly}. For $|\gamma|\leq N_{2}-1$,
\[
\Zr^{\gamma}(\Omega)=\Zr^{\gamma}(\frac{\partial_{1}v^{2}-\partial_{2}v^{1}}{\rho})=\frac{1}{\kappar}\Zr^{\gamma}\big(\frac{\Tr(\psi_{2})}{(k_{0}^{-\frac{1}{2}}\gamma^{-\frac{1}{2}}c)^{\frac{2}{\gamma-1}}}\big)+\Zr^{\gamma}\big(\frac{\Xr(\psi_{1})}{(k_{0}^{-\frac{1}{2}}\gamma^{-\frac{1}{2}}c)^{\frac{2}{\gamma-1}}}\big).
\]
In view of prop \ref{appendix: initial flux bound from LuoYu2} and $c\approx \mathring{c}$, we can conclude that 
\begin{corollary}\label{cor: initial flux for Omega}
For $|\gamma|\leq N_{2}-1$ and $\Zr\in \{\Tr, \Xr\}$, 
\[
\|\Zr^{\gamma}(\Omega)\|_{L^{\infty}(S_{t, 0})}\lesssim \varepsilon.
\]
\end{corollary}
In summary, prop \ref{appendix: initial geometry}, prop \ref{appendix: initial flux bound from LuoYu2}, together with corollary \ref{cor: initial flux for Omega} imply that the data posed on $C_{0}^{t*}$ satisfies conditions $(\mathbf{I}_{0})$, $(\mathbf{I}_{2})$ and $(\mathbf{I}_{\infty})$.
\end{itemize}

\section{Preparations for the energy estimate and the closing of $(\mathbf{B}_{\text{diffeomorphism}})$}\label{section4}
In view of $(\mathbf{B}_{\text{diffeomorphism}})$, we can work in acoustical coordinate $(t, u, \vartheta)$. In particular, we can use the null vector field $L$ to propagate the bound for $\{\hat{T}^{1}, \hat{T}^{2}, \kappa\}$ from $\Sigma_{\delta}^{u}$ to the whole $\mathcal{D}(\delta)(t, u)$. We will use $\Mr$ to denote a generic polynomial of $M$.

\subsection{Control of sound speed $c$ and $T(\wb)$}
In view of $(\mathbf{I}_{\infty})$ and remark \ref{rem: away from vacuum on Sigma_delta}, we have on $\Sigma_{\delta}^{u^{*}}$,
\[
\frac{1}{2}r_{0}\mathring{c}_{r}\leq c\leq 2\mathring{c}_{r}.
\]
In view of $(\mathbf{B}_{\infty})$, we have
\[
|c(t, u, \vartheta)-c(\delta, u, \vartheta)|=|\int_{\delta}^{t} L(c)(t', u, \vartheta) dt'|\leq Mt^{*}\varepsilon=\Mr\varepsilon.
\]
Thus, we can conclude that
\[
\frac{1}{4}r_{0}\mathring{c}_{r}\leq c\leq 4\mathring{c}_{r}.
\]
providing that $M\varepsilon$ is sufficiently small. 

In view of $(\mathbf{I}_{\infty})$ and $(\mathbf{I}_{\infty})$ for $T(\wb)$, we have
\[
|T(\wb)(t, u, \vartheta)-T(\wb)(\delta, u, \vartheta)|=|\int_{\delta}^{t} LT(\wb)(t', u, \vartheta)dt'|\leq M\varepsilon t.
\]
Thus, we can conclude that,
\begin{equation}\label{eq: bound for T(wb)}
\|T(\wb)+\frac{2}{\gamma+1}\|_{L^{\infty}(\Sigma_{t}^{u})}\leq \Mr\varepsilon t, \|T(v^{1}+c)+\frac{2}{\gamma+1}\|_{L^{\infty}(\Sigma_{t}^{u})}\leq \Mr\varepsilon t.
\end{equation}

\subsection{Control of $\hat{T}^{1}, \hat{T}^{2}$, and $\kappa$}

\begin{proposition}
For any $(t, u)\in [\delta, t^{*}]\times [0, u^{*}]$, we have
\begin{equation}\label{eq: 0-th control of the frame}
\begin{cases}
|\big(\hat{T}^{1}+1\big)(t, u)|\leq P(M)\varepsilon^{2}t^{2},\\
|\hat{T}^{2}|(t, u)\leq P(M)\varepsilon t.\\
|\kappa-\kappar|(t, u)\leq P(M)\varepsilon t^{2}.
\end{cases}.
\end{equation}
\end{proposition}

\begin{proof}
In view of the transport equation \eqref{eq: formulas to control the geometry} for $\hat{T}^{i}$ and $(\mathbf{B}_{\infty})$, we have
\[
|L(\hat{T}^{i})|=|\big(\hat{T}^{j}\hat{X}(\psi_{j})+\hat{X}(c) \big)\hat{X}^{i}|\leq \Mr\varepsilon,
\]
Integrating from $\delta$ to $t$, we have $|\hat{T}^{1}+1|\leq \Mr\varepsilon t, |\hat{T}^{2}|\leq \Mr\varepsilon t$ in view of $(\mathbf{I}_{\infty})$. Then $|\hat{X}^{1}|=|\hat{T}^{2}|\leq \Mr\varepsilon t$ implies that 
\[
|L(\hat{T}^{1})|=|\big(\hat{T}^{j}\hat{X}(\psi_{j})+\hat{X}(c) \big)\hat{X}^{1}|\leq \Mr\varepsilon^{2}t,
\]
Integrating from $\delta$ to $t$, we have $|\hat{T}^{1}+1|\leq \Mr\varepsilon^{2} t^{2}$ in view of $(\mathbf{I}_{\infty})$. In view of the $(\mathbf{B}_{\infty})$, $|T(v^{1}+c)+1|\leq \Mr\varepsilon$. In view of the transport equation \eqref{eq: formulas to control the geometry} for $\kappa$, we have
\[
|L(\kappa-\kappar)|=|-\big(T(c+v^{1})+1\big)-(\hat{T}^{1}+1)T(\psi_{1})-\hat{T}^{2}T(\psi_{2})|\leq \Mr\varepsilon t.
\]
Integrating from $\delta$ to $t$, we have $|\kappa-\kappar|\leq \Mr\varepsilon t^{2}$ in view of $(\mathbf{I}_{\infty})$.
\end{proof}

\begin{proposition}
For any $(t, u)\in [\delta, t^{*}]\times [0, u^{*}]$, and $Z\in \{T, \hat{X}\}$, we have
\begin{equation}\label{eq: 1-th control of the frame}
\begin{cases}
|Z\big(\hat{T}^{1}+1\big)(t, u)|\leq \Mr\varepsilon^{2}t^{2},\\
|Z\big(\hat{T}^{2}\big)|(t, u)\leq \Mr\varepsilon t.\\
|Z\big(\kappa-\kappar\big)|(t, u)\leq \Mr\varepsilon t^{2}.
\end{cases}.
\end{equation}
\end{proposition}

\begin{proof}
\begin{itemize}
\item The case $Z=\hat{X}$. We commute $\hat{X}$ with the transport equation \eqref{eq: formulas to control the geometry} of $\hat{T}^{i}$, 
\[
\begin{split}
L(\hat{X}(\hat{T}^{i}))&=\hat{X}L(\hat{T}^{i})-\chi\cdot \hat{X}(\hat{T}^{i}),\\
&=\hat{X}\Big(\big(\hat{T}^{j}\hat{X}(\psi^{j})+\hat{X}(c)\big)\hat{X}^{i}\Big)-(-\hat{X}^{j}\hat{X}(\psi_{j})+c\hat{X}^{1}\hat{X}(\hat{X}^{2})-c\hat{X}^{2}\hat{X}(\hat{X}^{1}))\cdot\hat{X}(\hat{T}^{i}).
\end{split}
\]
In view of $(\mathbf{B}_{\infty})$ and \eqref{eq: 0-th control of the frame}, we have for $i=1, 2$,
\[
|L(\hat{X}(\hat{T}^{i}))|\leq \Mr\varepsilon \big(|\hat{X}(\hat{T}^{1})|+|\hat{X}(\hat{T}^{2})|\big)+C(|\hat{X}(\hat{T}^{1})|+|\hat{X}(\hat{T}^{2})|\big)^{2}+\Mr\varepsilon.
\]
Integrating from $\delta$ to $t$ with Gronwall type argument\footnote{The quadratic term $C(|\hat{X}(\hat{T}^{1})|+|\hat{X}(\hat{T}^{2})|\big)^{2}$ won't cause any difficulty since the initial sizes of $\hat{X}(\hat{T}^{i})$ are of $O(\varepsilon\delta)$ and $t\in [\delta, t^{*}]$.} , we have $|\hat{X}(\hat{T}^{i})|\leq P(M)\varepsilon t$ in view of $(I_{\infty})$. In view of $|\hat{X}^{1}|\leq P(M)\varepsilon t$,
\[
|L(\hat{X}(\hat{T}^{1}))|\leq P(M)\varepsilon^{2} t.
\]
Integrating from $\delta$ to $t$, we have $|\hat{X}(\hat{T}^{1})|\leq P(M)\varepsilon^{2} t^{2}$ in view of $(I_{\infty})$. 

We commute $\hat{X}$ with transport equation \eqref{eq: formulas to control the geometry} of $\kappa$,
\[
\begin{split}
L(\hat{X}(\kappa))&=\hat{X}L(\kappa)-\chi\cdot\hat{X}(\kappa)\\
&=\hat{X}\big(-T(c)-\hat{T}^{j}T(\psi_{j})\big)-(-\hat{X}^{j}\hat{X}(\psi_{j})+c\hat{X}^{1}\hat{X}(\hat{X}^{2})-c\hat{X}^{2}\hat{X}(\hat{X}^{1}))\cdot\hat{X}(\kappa).
\end{split}
\]
In view of \eqref{eq: formulas to control the geometry}, for $\psi\in \{\wb, w, \psi_{2}\}$,
\[
\hat{X}T(\psi)=T\hat{X}(\psi)+\kappa\theta\cdot \hat{X}(\psi)=T\hat{X}(\psi)+\kappa(-\hat{X}(\hat{X}^{2})\hat{X}^{1}+\hat{X}(\hat{X}^{1})\hat{X}^{2})\cdot \hat{X}(\psi),
\]
we have $|\hat{X}T(\psi)|\leq P(M)\varepsilon t$. In view of estimate of $\hat{X}(\hat{T}^{i})$ and $B_{\infty}$,
\[
|L(\hat{X}(\kappa))|\leq P(M)\varepsilon t+P(M)\varepsilon |\hat{X}(\kappa)|.
\]
Integrating from $\delta$ to $t$ with Gronwall type argument, we have $|\hat{X}(\kappa-\kappar)|\leq P(M)\varepsilon t^{2}$ in view of $(I_{\infty})$.

\item The case $Z=T$. We commute $T$ with the transport equation \eqref{eq: formulas to control the geometry} of $\hat{T}^{i}$, 
\[
\begin{split}
L(T(\hat{T}^{i}))&=TL(\hat{T}^{i})-(\zeta+\eta)\cdot\hat{X}(\hat{T}^{i})\\
&=T\Big(\big(\hat{T}^{j}\hat{X}(\psi^{j})+\hat{X}(c)\big)\hat{X}^{i}\Big)-\big(-\kappa \hat{T}^{j}\hat{X}(\psi_{j})-\hat{X}^{i}T(\psi_{i})-\kappa\hat{X}(c)+c\hat{X}(\kappa)\big)\hat{X}(\hat{T}^{i})
\end{split}
\]
We remark the term $\hat{X}^{i}T(\psi_{i})$ (which is absent in irrotational case) can be bounded by
\[
|\hat{X}^{1}T(\psi_{1})+\hat{X}^{2}T(\psi_{2})|\leq \Mr\varepsilon t.
\]
In view of $(\mathbf{B}_{\infty})$ and \eqref{eq: 0-th control of the frame}, we have for $i=1, 2$,
\[
|L(T(\hat{T}^{i}))|\leq \Mr\varepsilon \big(|T(\hat{T}^{1})|+|T(\hat{T}^{2})|\big)+\Mr\varepsilon t.
\]
Integrating from $\delta$ to $t$ with Gronwall type argument, we have $|T(\hat{T}^{i})|\leq P(M)\varepsilon t$ in view of $(\mathbf{I}_{\infty})$. In view of $|\hat{X}^{1}|\leq \Mr\varepsilon t$,
\[
|L(T(\hat{T}^{1}))|\leq \Mr\varepsilon^{2} t.
\]
Integrating from $\delta$ to $t$, we have $|T(\hat{T}^{1})|\leq \Mr\varepsilon^{2} t^{2}$ in view of $(\mathbf{I}_{\infty})$. 

We commute $T$ with transport equation \eqref{eq: formulas to control the geometry} of $\kappa$,
\[
\begin{split}
L(T(\kappa))&=TL(\kappa)-(\zeta+\eta)\cdot T(\kappa)\\
&=T\big(-T(c)-\hat{T}^{i}T(\psi_{i})\big)-\big(-\kappa \hat{T}^{j}\hat{X}(\psi_{j})-\hat{X}^{i}T(\psi_{i})-\kappa\hat{X}(c)+c\hat{X}(\kappa)\big)\cdot \hat{X}(\kappa).
\end{split}
\]
In view of estimate of $\hat{X}(\hat{T}^{i}), T(\hat{T}^{i})$ and $(\mathbf{B}_{\infty})$,
\[
|L(T(\kappa))|\leq \Mr\varepsilon t.
\]
Integrating from $\delta$ to $t$, we have $|T(\kappa-\kappar)|\leq \Mr\varepsilon t^{2}$ in view of $(\mathbf{I}_{\infty})$.

\end{itemize}
\end{proof}

\begin{proposition}
For any $(t, u)\in [\delta, t^{*}]\times [0, u^{*}]$, $|\alpha|\leq \Ninf-1$, and $Z\in \{T, \hat{X}\}$, we have
\begin{equation}\label{eq: higher order control of the frame}
\begin{cases}
|Z^{\alpha}\big(\hat{T}^{1}+1\big)(t, u)|\leq \Mr\varepsilon^{2}t^{2},\\
|Z^{\alpha}\big(\hat{T}^{2}\big)|(t, u)\leq \Mr\varepsilon t.\\
|Z^{\alpha}\big(\kappa-\kappar\big)|(t, u)\leq \Mr\varepsilon t^{2}.
\end{cases}, |\alpha|\leq \Ninf-1.
\end{equation}
\end{proposition}

\begin{proof}
Let $n, 0\leq n\leq \Ninf-1$ denote the number $|\alpha|$. Let $m, 0\leq m\leq n$ denote the number $T$ appearing in $Z^{\alpha}$. We will prove by induction on $(n, m)$ with order
\[
(0, 0)\rightarrow (1, 0)\rightarrow (1, 1)\rightarrow \cdots \rightarrow (n, n)\rightarrow (n+1, 0)\rightarrow \cdots \rightarrow (n+1, n+1)\rightarrow \cdots \rightarrow (\Ninf-1, \Ninf-1).
\]
The base case $(n,m)=(0, 0), (1,0), (1, 1)$ follows from \eqref{eq: 0-th control of the frame} and \eqref{eq: 1-th control of the frame}. 
\begin{itemize}
\item From the case $(n, n)$ to the case $(n+1, 0)$ with $n\geq 1$, we rewrite the transport equation \eqref{eq: formulas to control the geometry} for $\hat{T}^{i}$ and $\kappa$ as follows,
\[
\begin{cases}
L(\hat{T}^{i})=\hat{X}(v^{1}+c)\hat{X}^{i}+[(\hat{T}^{1}+1)\hat{X}(\psi_{1})+\hat{T}^{2}\hat{X}(\psi_{2})]\hat{X}^{i},\\
L(\kappa)=-T(v^{1}+c)-[(\hat{T}^{1}+1)T(\psi_{1})+\hat{T}^{2}T(\psi_{2})].
\end{cases}
\]
and we commute $\hat{X}^{n}$ with them,
\[
\begin{cases}
L(\hat{X}^{n+1}(\hat{T}^{i}))=\hat{X}^{n}L(T^{i})+[L, \hat{X}^{n+1}](\hat{T}^{i}),\\
L(\hat{X}^{n+1}(\kappa))=\hat{X}^{n}L(\kappa)+[L, \hat{X}^{n+1}](\kappa).
\end{cases}
\]
In view of \eqref{eq: formulas to control the geometry}, we have
\[
\begin{split}
[L,\hat{X}^{n+1}](\hat{T}^{i})&=\sum_{j+k=n}\hat{X}^{k}[L, \hat{X}]\hat{X}^{j}(\hat{T}^{i})=\sum_{j+k=n}\hat{X}^{k}(-\chi)\hat{X}^{j+1}(\hat{T}^{i})\\
&=\sum_{j+k=n}\hat{X}^{k}\big(\hat{X}^{l}\hat{X}(\psi_{l})+c\hat{X}^{1}\hat{X}(\hat{X}^{2})-c\hat{X}^{2}\hat{X}(\hat{X}^{1})\big)\hat{X}^{j+1}(\hat{T}^{i})\\
\end{split}
\]
By induction hypothesis and $(\mathbf{B}_{\infty})$, we have
\[
|L(\hat{X}^{n+1}(\hat{T}^{i}))|\leq \Mr\varepsilon\big(|\hat{X}^{n+1}(\hat{T}^{1})|+|\hat{X}^{n+1}(\hat{T}^{2})|\big)+\Mr\varepsilon.
\]
Integrating from $\delta$ to $t$ with Gronwall type argument, we have $|\hat{X}^{n+1}(\hat{T}^{i})|\leq \Mr\varepsilon t$ in view of $(\mathbf{I}_{\infty})$. In view of $|\hat{X}^{1}|\leq\Mr\varepsilon$, we have
\[
|L(\hat{X}^{n+1}(\hat{T}^{1}))|\leq \Mr\varepsilon^{2}t.
\]
Integrating from $\delta$ to $t$, we have $|\hat{X}^{n+1}(\hat{T}^{1})|\leq \Mr\varepsilon^{2} t^{2}$ in view of $(\mathbf{I}_{\infty})$.

In view of \eqref{eq: formulas to control the geometry}, we also have
\[
\begin{split}
[L,\hat{X}^{n+1}](\kappa)&=\sum_{j+k=n}\hat{X}^{k}[L, \hat{X}]\hat{X}^{j}(\kappa)=\sum_{j+k=n}\hat{X}^{k}(-\chi)\hat{X}^{j+1}(\kappa)\\
&=\sum_{j+k=n}\hat{X}^{k}(\hat{X}^{l}\hat{X}(\psi_{l})+c\hat{X}^{1}\hat{X}(\hat{X}^{2})-c\hat{X}^{2}\hat{X}(\hat{X}^{1}))\hat{X}^{j+1}(\kappa)\\
\end{split}
\]
by induction hypothesis we have
\[
|[L, \hat{X}^{n+1}](\kappa)|\leq \Mr\varepsilon |\hat{X}^{n+1}(\kappa)|+\Mr\varepsilon^{2} t.
\]
In view of \eqref{eq: formulas to control the geometry}, for $\psi\in \{\wb, w, \psi_{2}\}$,
\[
\begin{split}
\hat{X}^{n}T(\psi)&=T\hat{X}^{n}(\psi)+\sum_{k_{1}+k_{2}=n-1}\hat{X}^{k_{1}}[\hat{X}, T]\hat{X}^{k_2}(\psi)\\
&=T\hat{X}^{n}(\psi)+\sum_{k_{1}+k_{2}=n-1}\hat{X}^{k_{1}}\Big(\kappa\big(-\hat{X}(\hat{X}^{2})\hat{X}^{1}+\hat{X}(\hat{X}^{1})\hat{X}^{2}\big)\Big)\hat{X}^{k_2}(\psi)\\
\end{split}
\]
By induction  hypothesis we have $|\hat{X}^{n}T(\psi)|\leq \Mr\varepsilon t$. Therefore, we have
\[
L(\hat{X}^{n+1}(\kappa))\leq \Mr\varepsilon |\hat{X}^{n+1}(\kappa)|+\Mr\varepsilon t.
\]
Integrating from $\delta$ to $t$ with Gronwall type argument, we have $|\hat{X}^{n+1}(\kappa-\kappar)|\leq \Mr\varepsilon t^{2}$ in view of $(\mathbf{I}_{\infty})$.

\item From the case $(n, m)$ to the case $(n, m+1)$ with $m\leq n-1$, we rewrite the transport equation \eqref{eq: formulas to control the geometry} for $\hat{T}^{i}$ and $\kappa$ as follows,
\[
\begin{cases}
L(\hat{T}^{i})=\hat{X}(v^{1}+c)\hat{X}^{i}+[(\hat{T}^{1}+1)\hat{X}(\psi_{1})+\hat{T}^{2}\hat{X}(\psi_{2})]\hat{X}^{i},\\
L(\kappa)=-T(v^{1}+c)-[(\hat{T}^{1}+1)T(\psi_{1})+\hat{T}^{2}T(\psi_{2})].
\end{cases}
\]
and we commute $Z^{\alpha}, |\alpha|=n$ where the number of $T$ appearing in $Z^{\alpha}$ is $m+1$,
\[
\begin{cases}
LZ^{\alpha}(\hat{T}^{i})=Z^{\alpha}L(\hat{T}^{i})+[L, Z^{\alpha}](\hat{T}^{i}),\\
LZ^{\alpha}(\kappa)=Z^{\alpha}L(\kappa)+[L, Z^{\alpha}](\kappa).
\end{cases}
\] 
In view of \eqref{eq: formulas to control the geometry}, we have
\[
\begin{split}
[L, Z^{\alpha}](\hat{T}^{i})&=\sum_{\alpha_{1}+\alpha_{2}<\alpha}Z^{\alpha_{1}}[L, \hat{X}]Z^{\alpha_{2}}(\hat{T}^{i})+Z^{\alpha_{1}}[L, T]Z^{\alpha_{2}}(\hat{T}^{i}),\\
&=\sum_{\alpha_{1}+\alpha_{2}<\alpha}Z^{\alpha_{1}}(-\chi)\hat{X}Z^{\alpha_{2}}(\hat{T}^{i})+\sum_{\alpha_{1}+\alpha_{2}<\alpha}Z^{\alpha_{1}}(-\zeta-\eta)\hat{X}Z^{\alpha_{2}}(\hat{T}^{i}),\\
&=\sum_{\alpha_{1}+\alpha_{2}<\alpha}Z^{\alpha_{1}}(\hat{X}^{l}\hat{X}(\psi_{l})+\hat{X}^{1}\hat{X}(\hat{X}^{2})-\hat{X}^{2}\hat{X}(\hat{X}^{1}))\hat{X}Z^{\alpha_{2}}(\hat{T}^{i})\\
&+\sum_{\alpha_{1}+\alpha_{2}<\alpha}\underbrace{Z^{\alpha_{1}}\big(-\kappa \hat{T}^{j}\hat{X}(\psi_{j})-\hat{X}^{i}T(\psi_{i})-\kappa\hat{X}(c)+c\hat{X}(\kappa)\big)\hat{X}Z^{\alpha_{2}}(\hat{T}^{i})}_{I_{1}}
\end{split}
\]
Note that the numbers of $T$ in $Z^{\alpha_{1}}(\kappa)$ and $Z^{\alpha_{1}}\hat{X}(\kappa)$ from terms $I_{1}$ are no more than $m$, and we can use induction argument to bound them. Thus, 
\[
L(Z^{\alpha}(\hat{T}^{i}))\leq P(M)\varepsilon\sum_{|\beta|=n, \text{the number of $T$ in $Z^{\beta}$ is no more than $m+1$}}(|Z^{\beta}(\hat{T}^{1}+1)|+|Z^{\beta}(\hat{T}^{2})|)+\Mr\varepsilon.
\]
Integrating from $\delta$ to $t$ with Gronwall type argument, we have $|Z^{\alpha}(\hat{T}^{i})|\leq \Mr\varepsilon t$ in view of $(\mathbf{I}_{\infty})$. In view of $|\hat{X}^{1}|\leq \Mr\varepsilon$, we have
\[
|L(Z^{\alpha}(\hat{T}^{1}))|\leq \Mr\varepsilon^{2}t.
\]
Integrating from $\delta$ to $t$, we have $|Z^{\alpha}(\hat{T}^{1})|\leq \Mr\varepsilon^{2} t^{2}$ in view of $(\mathbf{I}_{\infty})$.

In view of \eqref{eq: formulas to control the geometry}, we also have
\[
\begin{split}
[L, Z^{\alpha}](\kappa)&=\sum_{\alpha_{1}+\alpha_{2}<\alpha}Z^{\alpha_{1}}[L, \hat{X}]Z^{\alpha_{2}}(\kappa)+Z^{\alpha_{1}}[L, T]Z^{\alpha_{2}}(\kappa),\\
&=\sum_{\alpha_{1}+\alpha_{2}<\alpha}Z^{\alpha_{1}}(-\chi)\hat{X}Z^{\alpha_{2}}(\kappa)+\sum_{\alpha_{1}+\alpha_{2}<\alpha}Z^{\alpha_{1}}(-\zeta-\eta)\hat{X}Z^{\alpha_{2}}(\kappa),\\
&=\sum_{\alpha_{1}+\alpha_{2}<\alpha}Z^{\alpha_{1}}(\hat{X}^{l}\hat{X}(\psi_{l})+\hat{X}^{1}\hat{X}(\hat{X}^{2})-\hat{X}^{2}\hat{X}(\hat{X}^{1}))\hat{X}Z^{\alpha_{2}}(\kappa)\\
&+\sum_{\alpha_{1}+\alpha_{2}<\alpha}Z^{\alpha_{1}}\big(-\kappa \hat{T}^{j}\hat{X}(\psi_{j})-\hat{X}^{i}T(\psi_{i})-\kappa\hat{X}(c)+c\hat{X}(\kappa)\big)\hat{X}Z^{\alpha_{2}}(\kappa)
\end{split}
\]
by induction hypothesis we have
\[
|[L, Z^{\alpha}](\kappa)|\leq \Mr\varepsilon \sum_{|\beta|=n, \text{the number of $T$ in $Z^{\beta}$ is no more than $m+1$}}|Z^{\beta}(\kappa)|+\Mr\varepsilon^{2} t.
\]
In view of \eqref{eq: formulas to control the geometry}, for $\psi\in \{\wb, w, \psi_{2}\}$,
\[
\begin{split}
Z^{\alpha}T(\psi)&=TZ^{\alpha}(\psi)+\sum_{\alpha_{1}+\alpha_{2}<\alpha}Z^{\alpha_{1}}[\hat{X}, T]Z^{\alpha_2}(\psi)\\
&=TZ^{\alpha}(\psi)+\sum_{\alpha_{1}+\alpha_{2}<\alpha}Z^{\alpha_{1}}\Big(\kappa\big(-\hat{X}(\hat{X}^{2})\hat{X}^{1}+\hat{X}(\hat{X}^{1})\hat{X}^{2}\big)\Big)Z^{\alpha_2}(\psi)\\
\end{split}
\]
By induction  hypothesis we have $|Z^{\alpha}T(\psi)|\leq P(M)\varepsilon t$. Therefore, we have
\[
L(Z^{\alpha}(\kappa))\leq \Mr\varepsilon \sum_{|\beta|=n, \text{the number of $T$ in $Z^{\beta}$ is no more than $m+1$}}|Z^{\beta}(\kappa)|+\Mr\varepsilon t.
\]
Integrating from $\delta$ to $t$ with Gronwall type argument, we have $|Z^{\alpha}(\kappa-\kappar)|\leq \Mr\varepsilon t^{2}$ in view of $(\mathbf{I}_{\infty})$. 
\end{itemize}
which closes the induction argument.
\end{proof}

\begin{corollary}
For $|\beta|\leq \Ninf-2$ and $Z\in \{T, \hat{X}\}$, 
\begin{equation}
\begin{cases}
|Z^{\beta}\big(\slashed{g}-1\big)|(t, u)\leq \Mr\varepsilon t,\\
|Z^{\beta}|\big(\Xi\big)(t, u)\leq \Mr\varepsilon t^{2}.
\end{cases}
\end{equation}
\end{corollary}

\begin{proof}
From the proof of the previous proposition \eqref{eq: higher order control of the frame}, we have obtained that for $|\beta|\leq \Ninf-2$ and $Z\in \{T, \hat{X}\}$,
\[
\begin{cases}
|Z^{\beta}\big(\chi\big)|(t, u)\leq \Mr\varepsilon,\\
|Z^{\beta}\big(\zeta\big)|(t,u)+|Z^{\alpha}\big(\eta\big)|(t,u)\leq \Mr\varepsilon t.
\end{cases}
\]
Commuting $Z^{\beta}$ with the transport equation \eqref{eq: formulas to control the geometry} of $\slashed{g}$, we have
\[
\begin{split}
L(Z^{\beta}(\slashed{g}))&=Z^{\beta}(L(\slashed{g}))+[L, Z^{\beta}](\slashed{g})\\
&=Z^{\beta}(2\chi\cdot \slashed{g})+\sum_{\beta_{1}+\beta_{2}<\beta}Z^{\beta_{1}}[L, Z]Z^{\beta_{2}}(\slashed{g})\\
&=Z^{\beta}(2\chi\cdot \slashed{g})+\sum_{\beta_{1}+\beta_{2}<\beta}Z^{\beta_{1}}(-\chi)\hat{X}Z^{\beta_{2}}(\slashed{g})+\sum_{\beta_{1}+\alpha_{2}<\beta}Z^{\beta_{1}}(-\zeta-\eta)\hat{X}Z^{\beta_{2}}(\slashed{g})
\end{split}
\]
Therefore, we have
\[
|LZ^{\beta}(\slashed{g})|\leq \Mr\varepsilon \sum_{|\gamma|\leq |\beta|}|Z^{\gamma}(\slashed{g})|.
\]
Integrating from $\delta$ to $t$ with Gronwall type argument, we have $|Z^{\alpha}(\slashed{g}-1)|\leq \Mr\varepsilon t$ in view of $(\mathbf{I}_{\infty})$. 

Commuting $Z^{\beta}$ with the transport equation \eqref{eq: formulas to control the geometry} of $\Xi$, we have
\[
\begin{split}
L(Z^{\beta}(\Xi))&=Z^{\beta}\big(L(\Xi)\big)+[L, Z^{\beta}](\Xi)\\
&=Z^{\beta}\big(\frac{1}{\sqrt{\slashed{g}}}(\zeta+\eta)\big)+\sum_{\beta_{1}+\beta_{2}<\beta}Z^{\beta_{1}}[L, Z]Z^{\beta_{2}}(\Xi)\\
&=Z^{\beta}\big(\frac{1}{\sqrt{\slashed{g}}}(\zeta+\eta)\big)+\sum_{\beta_{1}+\beta_{2}<\beta}Z^{\beta_{1}}(-\chi)\hat{X}Z^{\beta_{2}}(\Xi)+\sum_{\beta_{1}+\alpha_{2}<\beta}Z^{\beta_{1}}(-\zeta-\eta)\hat{X}Z^{\beta_{2}}(\Xi)
\end{split}
\]
Therefore, we have
\[
|LZ^{\beta}(\Xi)|\leq \Mr\varepsilon \sum_{|\gamma|\leq |\beta|}|Z^{\gamma}(\Xi)|+\Mr\varepsilon t.
\]
Integrating from $\delta$ to $t$ with Gronwall type argument, we have $|Z^{\alpha}(\Xi)|\leq \Mr\varepsilon t^{2}$ in view of $\Xi(\delta, u, \vartheta)\equiv 0$ on $\Sigma_{\delta}^{u^{*}}$. 
\end{proof}

 In summary, we have
\begin{proposition}
For $|\alpha|\leq \Ninf-1, |\beta|\leq \Ninf-2$, and $Z\in \{T, \hat{X}\}$, we have
\begin{equation}\label{eq: L infty control of the geometry}
\begin{cases}
|Z^{\alpha}\big(\hat{T}^{1}+1\big)(t, u)|\leq \Mr\varepsilon^{2}t^{2},\\
|Z^{\alpha}\big(\hat{T}^{2}\big)|(t, u)\leq \Mr\varepsilon t.\\
|Z^{\alpha}\big(\kappa-\kappar\big)|(t, u)\leq \Mr\varepsilon t^{2},\\
|Z^{\beta}\big(\chi\big)|(t, u)\leq \Mr\varepsilon,\\
|Z^{\beta}\big(\zeta\big)|(t,u)+|Z^{\alpha}\big(\eta\big)|(t,u)\leq \Mr\varepsilon t,\\
|Z^{\beta}\big(\slashed{g}-1\big)|(t, u)\leq \Mr\varepsilon t,\\
|Z^{\beta}|\big(\Xi\big)(t, u)\leq \Mr\varepsilon t^{2}.
\end{cases}
\end{equation}
\end{proposition}

\subsection{Close of $(\mathbf{B}_{\text{diffeomorphism}})$}
In view of \eqref{eq: Jacobi of (t, u, vartheta) to (t, x_1, x_2)}, we have
\[
\begin{pmatrix}
\frac{\partial u}{\partial x_1}&\frac{\partial u}{\partial x_2}\\
\frac{\partial \vartheta}{\partial x_1}&\frac{\partial \vartheta}{\partial x_2}
\end{pmatrix}=\begin{pmatrix}
\frac{\hat{T}^{1}}{\kappa}&\frac{\hat{T}^{2}}{\kappa}\\
\frac{\hat{X}^{1}}{\sqrt{\slashed{g}}}+\frac{\Xi}{\kappa}\slashed{g}\hat{X}^{2}&\frac{\hat{X}^{2}}{\sqrt{\slashed{g}}}-\frac{\Xi}{\kappa}\slashed{g}\hat{X}^{1}
\end{pmatrix}
\]
In view of \eqref{eq: L infty control of the geometry}, we can conclude that
\[
\begin{cases}
|\hat{T}^{1}+1|\leq \Mr\varepsilon^{2}t^{2},\\
|\frac{\partial u}{\partial x_{1}}-\frac{1}{t}|\leq \Mr\varepsilon, |\frac{\partial u}{\partial x_{2}}|\leq \Mr\varepsilon,\\
|\frac{\partial \vartheta}{\partial x_{1}}|\leq \Mr\varepsilon t, |\frac{\partial \vartheta}{\partial x_{2}}-1|\leq \Mr\varepsilon t.
\end{cases}
\]
which closes the geometry assumption $(\mathbf{B}_{\text{diffeomorphism}})$ \eqref{assump: B_geoemtry}.

\subsection{Sobolev inequality and $L^{\infty}$ estimates on acoustical waves and vorticity}

Recall that the Jacobi matrix of the coordinates transformation $(t,u,\theta)\mapsto (t,x_{1},x_{2})$ is given by
\[
\begin{pmatrix}
\frac{\partial t}{\partial t}&\frac{\partial t}{\partial u}&\frac{\partial t}{\partial \vartheta}\\
\frac{\partial x_{1}}{\partial t}&\frac{\partial x_{1}}{\partial u}&\frac{\partial x_{1}}{\partial \vartheta}\\
\frac{\partial x_{2}}{\partial t}&\frac{\partial x_{2}}{\partial u}&\frac{\partial x_{2}}{\partial \vartheta}
\end{pmatrix}=
\begin{pmatrix}
1&0&0\\
L^{1}&\kappa\hat{T}^{1}+\Xi\sqrt{\slashed{g}}\hat{X}^{1}&\sqrt{\slashed{g}}\hat{X}^{1}\\
L^{2}&\kappa\hat{T}^{2}+\Xi\sqrt{\slashed{g}}\hat{X}^{2}&\sqrt{\slashed{g}}\hat{X}^{2}\\
\end{pmatrix}
\]
In view of \eqref{eq: L infty control of the geometry}, we can conclude that for $|\beta|\leq \Ninf-2$ and $Z\in \{T, \hat{X}\}$,
\begin{equation}\label{eq: estimate for Jacobi}
|Z^{\beta}(\frac{\partial x_{1}}{\partial u}+t)|+|Z^{\beta}(\frac{\partial x_{2}}{\partial u})|\leq \Mr\varepsilon t^{2}, |Z^{\beta}(\frac{\partial x_{1}}{\partial \vartheta})|+ |Z^{\beta}(\frac{\partial x_{2}}{\partial \vartheta}-1)|\leq \Mr\varepsilon t.
\end{equation}

In particular, we have
\begin{equation}\label{eq: control of coordinate transform}
\begin{cases}
|\frac{\partial x_{1}}{\partial u}+t|+|\frac{\partial x_{2}}{\partial u}|\leq \Mr\varepsilon t^{2}, |\frac{\partial x_{1}}{\partial \vartheta}|+ |\frac{\partial x_{2}}{\partial \vartheta}-1|\leq \Mr\varepsilon t,\\
\sum_{i=1}^{2}\Big(|\frac{\partial^{2}x^{i}}{\partial u^{2}}|+|\frac{\partial^{2} x^{i}}{\partial \vartheta\partial u}|\Big)\leq \Mr\varepsilon t^{2}, \sum_{i=1}^{2}|\frac{\partial^{2} x^{i}}{\partial \vartheta^{2}}|\leq \Mr\varepsilon t. 
\end{cases}
\end{equation}

Therefore, we have the following Sobolev inequality, whose proof is the same as \cite{Luo-YuRare1},

\begin{lemma}
Under the bootstrap assumption $(\mathbf{B}_{\infty})$, if $\Mr\varepsilon$ is sufficiently small, for all $t\in [\delta,t^{*}]$, for any smooth function $f$ defined on $\Sigma_{t}^{u^{*}}$, we have
\begin{equation}\label{eq: Sobolev inequality}
\|f\|_{L^{\infty}(\Sigma_{t}^{u^{*}})}\lesssim \sum_{k+l\leq 2}\|\Xr^{k}\Tr^{l}(f)\|_{L^{2}(\Sigma_{t}^{u^{*}})}.
\end{equation}
\end{lemma}

In view of $(\mathbf{B}_{2})$, $L^{\infty}$ estimates of $\{\hat{T}^{1}, \hat{T}^{2}, \kappa\}$ \eqref{eq: L infty control of the geometry} and Sobolev inequality \eqref{eq: Sobolev inequality}. For all multi-index $|\alpha|\leq \Ninf$, for all $\psi\in \{\wb, w, \psi_{2}\}$, except for the case $\Zr^{\alpha}\psi=T(\wb)$ and $\Omega$, we have
\begin{equation}\label{eq: L infty bound for waves}
\|\Zr^{\alpha}(\psi)\|_{L^{\infty}(\Sigma_{t}^{u^{*}})}\lesssim \begin{cases}
\Mr\varepsilon, \ \ \text{if}\ \Zr^{\alpha}=\Xr^{\alpha},\\
\Mr\varepsilon t, \ \ \text{otherwise}.
\end{cases}, \|\Zr^{\alpha}(\Omega)\|_{L^{\infty}(\Sigma_{t}^{u^{*}})}\lesssim \Mr\varepsilon.
\end{equation}

\section{Energy estimate and the closing of $(\mathbf{B}_{2})$}\label{section5}

\subsection{A refined Gronwall type inequality} To handle the the potential dangerous "linear" terms (are of size $\Mr\varepsilon$) in bulk terms, we need a refined Gronwall type inequality which has been introduce in \cite{Luo-YuRare1} and we refer to section 5 of \cite{Luo-YuRare1} for a proof.
\begin{lemma}\label{lem: refined Gronwall}
Let $E(t,u)$ and $F(t, u)$ be two smooth non-negative function defined on $\mathcal{D}(\delta)(t^{*}, u^{*})$ such that 
\[
E(t, u')\leq E(t, u), 0\leq u'\leq u\leq u^{*},\ \ F(t', u)\leq F(t, u), \delta\leq t'\leq t\leq t^{*}.
\]
We assume that there exist positive constants $A, B, C$ such that for all $(t, u)\in [\delta, t^{*}]\times [0, u^{*}]$, we have the following inequality:
\[
E(t, u)+F(t, u)\leq At^{2}+B\int_{0}^{u}F(t, u')+C\int_{\delta}^{t}\frac{E(t', u)}{t'}dt'
\]
Then, if $e^{Bu^{*}}C\leq 1$, we have the following inequality for all $(t, u)\in [\delta, t^{*}]\times [0, u^{*}]$:
\[
E(t, u)+F(t, u)\leq 3Ae^{Bu}t^{2}.
\]
\end{lemma}

\begin{remark}
In applications, it can happen that $C<1$ and $B$ is very large such that $e^{Bu^{*}}C>1$. We can divide $[0, u^{*}]$ into $\cup_{j=0}^{L}[u_{j-1}, u_{j}]$ with equal size such that
\[
L=\frac{2Bu^{*}}{ln(\frac{1}{C})}+1\implies e^{B\frac{u^{*}}{L}}C\leq \sqrt{C}<1.
\]
Then we can apply lemma \ref{lem: refined Gronwall} inductively on $[\delta, t^{*}]\times [u_{j-1}, u_{j}]$ with $1\leq j\leq L$ to obtain that
\[
E(t, u)+F(t, u)\leq \Big(3Ae^{Bu^{*}}\Big)^{\frac{2Bu^{*}}{ln(\frac{1}{C})}+1}t^{2}.
\]
\end{remark}

\subsection{The fundamental energy inequality}

In view of the \eqref{eq: L infty control of the geometry}, the bulk terms $Q_{i}, 1\leq i\leq 4$ in \eqref{vor: energy estimate for wave equation with multiplier hat(L)} together with the bulk terms $\underline{Q}_{i}, 1\leq i\leq 4$ in \eqref{vor: energy estimate for wave equation with multiplier underline(L)} can be bounded by
\[
\sum_{i=1}^{4}\Big(|Q_{i}|+|\underline{Q}_{i}|\Big)\lesssim \int_{0}^{u} \mathscr{F}_{w}(\psi)(t, u') du'+\int_{\delta}^{t}\mathscr{E}_{w}(t', u)dt'.
\]
The bulk terms $V_{1}$ in \eqref{vor: energy estimate for transport equation} can be bounded by
\[
|V_{1}|\lesssim \int_{0}^{u}\mathcal{F}_{t}(\Omega)(t,u^{'})du^{'}.
\]

To be summarized, we have
\begin{itemize}
\item The fundamental energy inequality for acoustical waves
\begin{equation}\label{eq: fundamental inequalities for acoustical wave}
\begin{cases}
\mathscr{E}_{w}(\psi)(t,u)+\mathscr{F}_{w}(\psi)(t,u)=\mathscr{E}_{w}(\psi)(t,u)+\mathscr{F}_{w}(\psi)(t,u)+\mathcal{N}_{w}(\psi)(t,u)+\underline{\mathcal{N}}_{w}(\psi)(t,u)+\text{\bf Err}_{w},\\
\mathcal{N}_{w}(\psi)(t,u)=-\int_{D(t,u)}\mu\varrho\cdot\hat{L}(\psi),\ \ \underline{\mathcal{N}}_{w}(\psi)=-\int_{D(t,u)}\mu\varrho\cdot\underline{L}(\psi),\\
|\text{\bf Err}_{w}|\lesssim \int_{0}^{u}\mathscr{F}_{w}(\psi)(t,u^{'})du^{'}+\int_{\delta}^{t}\mathscr{E}_{w}(\psi)(t^{'},u)dt^{'}.
\end{cases}
\end{equation}
\item The fundamental energy inequality for vorticity
\begin{equation}\label{eq: fundamental inequalities for vorticity}
\begin{cases}
\mathcal{E}_{t}(\Omega)(t,u)+ \mathcal{F}_{t}(\Omega)(t,u)=\ \mathcal{E}_{t}(\Omega)(\delta,u)+ \mathcal{F}_{t}(\Omega)(t,0)+\mathcal{N}_{t}(\Omega)(t,u)+\text{\bf Err}_{t},\\
\mathcal{N}_{t}(\Omega)(t,u)=\int_{D(t,u)} 2\mu F\Omega \\
|\text{\bf Err}_{t}|\lesssim \int_{0}^{u}\mathcal{F}_{t}(\Omega)(t,u^{'})du^{'}.
\end{cases}
\end{equation}
\end{itemize}

\subsection{Energy estimate for vorticity} 
This section is devoted to bound $\Zr^{\alpha}\Omega$ in $L^{2}$ where $|\alpha|\leq N_{top}+1$.  We commute $\Zr^{\alpha}$ with $B\Omega=0$ and apply the fundamental energy inequality for transport equation.

For multi-index $\alpha$ with $|\alpha|=n, n\leq N_{top}+1$, we use $\Omega_{n}$ to denote $\Zr^{\alpha}\Omega$. If we set $\kappar B\Omega_{n}=F_{n}$ and $\Omega_{n}=\Zr(\Omega_{n-1})$. We have $F_{n}=\Zr(F_{n-1})+[\kappar B, \Zr]\Omega_{n-1}$. The commutators $[\kappar B,\Zr], \Zr\in \mathring{\mathcal{L}}=\{\Xr,\Tr\}$ are given by
\begin{align}
\begin{cases}
[\kappar B,\Xr]&=\Xr(v^{1})\Tr-\kappar \Xr(v^{2})\Xr,\\
[\kappar B,\Tr]&=(\Tr(v^{1})+1)\Tr-\kappar\Tr(v^{2})\Xr.
\end{cases}
\end{align}
Therefore, the recursion relations are given by
\begin{equation}
F_{n}=\begin{cases}
\Xr(F_{n-1})+\Xr(v^{1})\Tr(\Omega_{n-1})-\kappar \Xr(v^{2})\Xr(\Omega_{n-1}), \ \ \Zr=\Xr,\\
\Tr(F_{n-1})+\big(\Tr(v^{1})+1\big)\Tr(\Omega_{n-1})-\kappar \Tr(v^{2})\Xr(\Omega_{n-1}), \ \ \Zr=\Tr.
\end{cases}
\end{equation}
with $F_{0}=0$. For $n\leq N_{top}+1$, $F_{n}$ is the combination of terms of form 
\[
\begin{cases}
\Zr^{\beta}\Big(\Xr(v^{1})\Tr\big(\Zr^{\gamma}(\Omega)\big)-\kappar\Xr(v^{2})\Xr\big(\Zr^{\gamma}(\Omega)\big)\Big), \ \ |\beta|+|\gamma|\leq n-1,\\
\Zr^{\beta}[(\Tr(v^{1})+1)\Tr(\Zr^{\gamma}\Omega)-\kappar\Tr(v^{2})\Xr(\Zr^{\gamma}(\Omega))], \ \ |\beta|+|\gamma|\leq n-1.
\end{cases}
\]

Therefore, the corresponding contribution for $\mathcal{N}_{t}(\Omega_{n})$ in \eqref{eq: fundamental inequalities for vorticity} can be divided into four cases:
\begin{itemize}
\item $\mathcal{N}_{t, 1}(\Omega_{n})=\int_{\mathcal{D}(\delta)(t, u)} \frac{2\mu}{\kappar} \Zr^{\beta_{1}}\Xr(v^{1})\Zr^{\beta_{2}}\Tr(\Omega)\Omega_{n}$,
\item $\mathcal{N}_{t, 2}(\Omega_{n})=\int_{\mathcal{D}(\delta)(t, u)} \frac{2\mu}{\kappar} \kappar\Zr^{\beta_{1}}\Xr(v^{2})\Zr^{\beta_{2}}\Xr(\Omega)\Omega_{n}$,  
\item $\mathcal{N}_{t, 3}(\Omega_{n})=\int_{\mathcal{D}(\delta)(t, u)} \frac{2\mu}{\kappar} \kappar\Zr^{\beta_{1}}\Tr(v^{2})\Zr^{\beta_{2}}\Xr(\Omega)\Omega_{n}$,  
\item $\mathcal{N}_{t, 4}(\Omega_{n})=\int_{\mathcal{D}(\delta)(t, u)} \frac{2\mu}{\kappar} \Zr^{\beta_{1}}\big(\Tr(v^{1})+1\big)\Zr^{\beta_{2}}\Tr(\Omega)\Omega_{n}$,  
\end{itemize}
with $|\beta_{1}|+|\beta_{2}|\leq n-1$. In view of $(\mathbf{B}_{2})$, \eqref{eq: L infty bound for waves} and $\frac{\mu}{\kappar}\lesssim 1$,
\[
\sum_{j=1}^{3}\mathcal{N}_{t, j}\lesssim \Mr\varepsilon^{3} t
\]
As for $\mathcal{N}_{t, 4}(\Omega_{n})$, if $\beta_{1}\not=0$, it can be bounded exactly the same way as $\mathcal{N}_{t, j}, 1\leq j\leq 3$. If $\beta_{1}=0$, 
\[
\mathcal{N}_{t, 4}(\Omega_{n})\lesssim \int_{\mathcal{D}(\delta)(t, u)} \Omega_{n}^{2}\lesssim \int_{0}^{u} \mathcal{F}_{t}(\Omega_{n})(t, u^{'}) du'.
\]

In summary, we have for $n\leq \Ntop+1$,
\[
\begin{split}
&\sum_{|\alpha|=n}\Big(\mathcal{E}_{t}\big(\Zr^{\alpha}(\Omega)\big)(t,u)+\mathcal{F}_{t}\big(\Zr^{\alpha}(\Omega)\big)(t,u)\Big)\leq \sum_{|\alpha|=n}\Big(\mathcal{E}_{t}\big(\Zr^{\alpha}(\Omega)\big)(\delta,u)+\mathcal{F}_{t}\big(\Zr^{\alpha}(\Omega)\big)(t,0)\Big)\\
&+\Mr\varepsilon^{3}t+C\int_{0}^{u^{'}}\sum_{|\alpha|=n}\Big(\mathcal{F}_{}\big(\Zr^{\alpha}(\Omega)\big)(t,u^{'}) \Big)du^{'}.
\end{split}
\]
In view of $(\mathbf{I}_{2})$ and Gronwall' s inequality, we have
\begin{equation}\label{close of L2 for vorticity}
\mathscr{E}_{t, \leq \Ntop+1}(\Omega)+\mathscr{F}_{t, \leq \Ntop+1}(\Omega)\lesssim \varepsilon^{2}t+\Mr\varepsilon^{3}t\leq \varepsilon^{2}t.
\end{equation}
providing that $\Mr\varepsilon$ is sufficiently small, which closes the third estimates of bootstrap assumption $(\mathbf{B}_{2})$ \eqref{assumption: B2}. We also have the corollary,
\begin{corollary}
For multiindex $|\alpha|\leq \Ninf$, 
\[
\|\Zr^{\alpha}\Omega\|_{L^{\infty}(\Sigma_{t}^{u})}\lesssim \varepsilon.
\]
\end{corollary}

\subsection{Lowest energy estimate for acoustical wave }

\subsubsection{The lowest energy estimate for $w$ and $\psi_{2}$}
We take $\psi=w$ or $\psi_{2}$. The structure of source terms of wave-transport system \eqref{eq: wave-transport system} can be represented
\[
\{c^{-1}g(Df_{1},Df_{2}), \rho \partial\Omega, \rho c^{-1}\Omega\partial c, c^{-1}\rho^{-2}\Omega^{2}\}.
\]
where $\partial$ means spatial derivates $\partial_{x_1}, \partial_{x_2}$ and $f_{1}, f_{2}\in \{\wb, w, \psi_{2}\}$.

Therefore, the corresponding contribution for $\mathcal{N}_{w}(\psi)$ and $\underline{\mathcal{N}}_{w}(\psi)$ in \eqref{eq: fundamental inequalities for acoustical wave} can be divided into eight cases:
\begin{itemize}
\item $\mathcal{N}_{w, 1}(\psi)=-\int_{\mathcal{D}(t,u)} \mu c^{-1}g(Df_{1},Df_{2}) \hat{L}(\psi)$
\item $\mathcal{N}_{w, 2}(\psi)=-\int_{\mathcal{D}(t,u)}\mu \rho \partial_{i}\Omega \hat{L}(\psi)$
\item $\mathcal{N}_{w, 3}(\psi)=-\int_{\mathcal{D}(t,u)}\mu \rho c^{-1}\Omega\partial_{i}c\hat{L}(\psi)$
\item $\mathcal{N}_{w, 4}(\psi)=-\int_{\mathcal{D}(t,u)} \mu c^{-1}\rho^{-2}\Omega^{2}\hat{L}(\psi)$
\item $\underline{\mathcal{N}}_{w, 1}(\psi)=-\int_{\mathcal{D}(t,u)} \mu c^{-1}g(Df_{1},Df_{2}) \underline{L}(\psi)$
\item $\underline{\mathcal{N}}_{w, 2}(\psi)=\int_{\mathcal{D}(t,u)}\mu \rho \partial_{i}\Omega \underline{L}(\psi)$
\item $\underline{\mathcal{N}}_{w, 3}(\psi)=\int_{\mathcal{D}(t,u)}\mu \rho c^{-1}\Omega\partial_{i}c\underline{L}(\psi)$
\item $\underline{\mathcal{N}}_{w, 4}(\psi)=-\int_{\mathcal{D}(t,u)} \mu c^{-1}\rho^{-2}\Omega^{2}\underline{L}(\psi)$
\end{itemize}
The self interaction of acoustical waves $\mathcal{N}_{w, 1}(\psi)$ and $\underline{\mathcal{N}}_{w, 1}(\psi)$ can be exactly the same way as the irrotational case by
\[
\mathcal{N}_{w, 1}(\psi)+\underline{\mathcal{N}}_{w, 1}(\psi)\leq \Mr\varepsilon^{3}t^{2}+C\int_{0}^{u}\sum_{\psi\in \{w, \psi_{2}\}}\mathscr{F}_{w}(\psi)(t,u^{'})du^{'}+\frac{1}{8}\int_{\delta}^{t}\sum_{\psi\in \{w, \psi_{2}\}}\frac{\mathscr{E}_{w}(\psi)(t^{'},u)}{t^{'}}dt^{'}.
\]
and we refer to section 5 of \cite{Luo-YuRare1} for details. 

In view of $(\mathbf{B}_{2})$, \eqref{eq: L infty bound for waves}, $\frac{\mu}{\kappar}\lesssim 1$ and Cauchy-Schwartz, the terms $\mathcal{N}_{w, j}(\psi), j=2, 3, 4$ and $\underline{\mathcal{N}}_{w, j}(\psi), j=2, 3, 4$ can be bounded by
\[
\sum_{j=2}^{4}\Big(\mathcal{N}_{w, j}(\psi)+ \underline{\mathcal{N}}_{w, j}(\psi)\Big)\leq C\varepsilon^{2}t^{2}+\frac{1}{8}\int_{\delta}^{t} \frac{\mathscr{E}_{w}(\psi)(t', u)}{t'}dt'.
\]

In summary, we have
\[
\begin{split}
&\sum_{\psi\in \{w, \psi_{2}\}}\Big(\mathscr{E}_{w}(\psi)(t,u)+\mathscr{F}_{w}(\psi)(t,u)\Big)\leq \sum_{\psi\in \{w, \psi_{2}\}}\Big(\mathscr{E}_{w}(\psi)(\delta,u)+\mathscr{F}_{w}(\psi)(t,0)\Big)\\
&+\Mr\varepsilon^{3}t^{2}+C\varepsilon^{2}t^{2}+C\int_{0}^{u}\sum_{\psi\in \{w, \psi_{2}\}}\mathscr{F}_{w}(\psi)(t,u^{'})du^{'}+C\int_{0}^{u}\sum_{\psi\in \{w, \psi_{2}\}}\mathscr{E}_{w}(\psi)(t,u^{'})du^{'}\\
&+\frac{1}{2}\int_{\delta}^{t}\sum_{\psi\in \{w, \psi_{2}\}}\frac{\mathscr{E}_{w}(\psi)(t^{'},u)}{t^{'}}dt^{'}.
\end{split}
\]
In view of $(\mathbf{I}_{2})$ and refined Gronwall's inequality \ref{lem: refined Gronwall}, we can conclude that
\begin{equation}
\sum_{\psi\in \{w, \psi_{2}\}}\Big(\mathscr{E}_{w}(\psi)(t,u)+\mathscr{F}_{w}(\psi)(t,u)\Big)\lesssim \varepsilon^{2}t^{2}.
\end{equation}
providing that $\Mr\varepsilon$ is sufficiently small. This closes the second estimate of the bootstrap assumption $(\mathbf{B}_{2})$ \eqref{assumption: B2}.

\subsubsection{The lowest energy bound for $\underline{w}$}
Recall that
\begin{align*}
\mathring{\mathscr{E}}_{0}(\underline{w})(t,u)=\frac{1}{2}\int_{\Sigma_{t}^{u}}c^{-2}\kappa^{2}(L\underline{w})^{2}+\kappa^{2}(\Xr\underline{w})^{2},\mathring{\mathscr{F}}_{0}(\underline{w})(t,u)=\int_{C_{u}^{t}} c^{-1}\kappa(L\underline{w})^{2}+c\kappa(\Xr\underline{w})^{2}.
\end{align*}
are not included in the bootstrap assumption $(B_{2})$. To handle $L(\wb)$,  we need use the Euler equation in the first null frame \eqref{Euler equations:form 2},
\[
L (\wb) = -c \widehat{T}(\wb)(\widehat{T}^1+1)+\frac{1}{2}c \widehat{T}(\psi_2)\widehat{T}^2 +\frac{1}{2}c \Xh(\psi_2)\Xh^2-c\Xh(\wb)\Xh^1,
\]
which implies that $|L(\wb)-\frac{1}{2}c\hat{X}(\psi_{2})\hat{X}^{2}|\lesssim \Mr\varepsilon^{2}t$. Therefore, for all $(t, u)\in [\delta, t^{*}]\times [0, u^{*}]$, we have
\[
\int_{\Sigma_{t}^{u}}c^{-2}\kappa^{2}(L\underline{w})^{2}+\int_{C_{u}^{t}} c^{-1}\kappa(L\underline{w})^{2}\lesssim \varepsilon^{2}t^{2}
\]
providing that $\Mr\varepsilon$ is sufficiently small. 

To handle $\Xr(\wb)$,  we need use the Euler equation in the second null frame \eqref{Euler equations:form 1 ringed},
\[
\Lr (\psi_2)= -c \Trh(\psi_2)+c\Xr( w+\wb).
\]
Note that for $\psi\in\{w, \psi_{2}\}$,
\[
\begin{cases}
\hat{\Tr}(\psi)-\hat{T}(\psi)=\frac{\hat{T}^{1}+1}{\kappar}\hat{\Tr}(\psi)-\hat{T}^{2}\hat{X}(\psi),\\
\Xr(\psi)-\hat{X}(\psi)=\frac{\hat{X}^{1}}{\kappar}\Tr(\psi)+(\hat{T}^{1}+1)\Xr(\psi).
\end{cases}\implies |\hat{\Tr}(\psi)-\hat{T}(\psi)|+|\Xr(\psi)-\hat{X}(\psi)|\lesssim \Mr\varepsilon^{2}.
\]
Therefore,  for all $(t, u)\in [\delta, t^{*}]\times [0, u^{*}]$, we have
\[
\int_{\Sigma_{t}^{u}} \kappa^{2}(\Xr\underline{w})^{2}+\int_{C_{u}^{t}}c\kappa(\Xr\underline{w})^{2}\lesssim \varepsilon^{2}t^{2}
\]
providing that $\Mr\varepsilon$ is sufficiently small. 

We summarize the zero order energy estimates as
\begin{equation}\label{eq: close of lowest energy}
\sum_{\psi\in\{w,\psi_{2}\}}[\mathscr{E}_{w}(\psi)(t,u)+\mathscr{F}_{w}(\psi)(t,u)]+\mathring{\mathscr{E}}_{0}(\underline{w})(t,u)+\mathring{\mathscr{F}}_{0}(\underline{w})(t,u)\lesssim \varepsilon^{2}t^{2}.
\end{equation}

\subsection{Higher energy estimate for acoustical wave}
To do higher order estimate for wave variables $\Psi\in \{\underline{w},w,\psi_{2}\}$, we will commute $\Zr^{\alpha}, |\alpha|=n, 1\leq n\leq \Ntop$ with the wave-transport system \eqref{eq: wave-transport system} and we can derive the recursion formula for $\Zr^{\alpha}\Psi$ as follows,
\begin{align}
\begin{cases}
\mur\Box_{g}\Psi_{n}=\varrho_{n},\\
\Psi_{n}=\Zr\Psi_{n-1},\\
\varrho_{n}=\Zr(\varrho_{n-1})+{}^{(\Zr)}\sigma_{n-1},\\
{}^{(\Zr)}\sigma_{n-1}=\mur div_{g} {}^{(\Zr)}J_{n-1},\\
{}^{(\Zr)}J_{n-1}^{\mu}=\frac{1}{2}(g^{\mu\alpha}g^{\nu\beta}+g^{\mu\beta}g^{\nu\alpha}-g^{\mu\nu}g^{\alpha\beta}){}^{(\Zr)}\pi_{\alpha\beta}\partial_{\nu}\Psi_{n-1},\\
\Psi_{0}\in \{\underline{w},w,\psi_{2}\},\\
\varrho_{0}=\mur c^{-1}g(Df_{1},Df_{2})+\mur\rho \partial\Omega+\mur\rho c^{-1}\Omega\partial c+\mur c^{-1}\rho^{-2}\Omega^{2}.
\end{cases}
\end{align}
where $\partial$ means spatial derivates $\partial_{1}, \partial_{2}$ and we have neglected unimportant numerical constants. Writing $\Psi_{n}=\Zr_{n}(\Zr_{n-1}(...(\Zr(\Psi_{0}))...))$, we have
\begin{align}
\varrho_{n}=\Zr_{n}(...(\Zr(\varrho_{0}))...)+\sum_{i=0}^{n-1}\Zr_{n}(...(\Zr_{i+2}({}^{(\Zr_{i+1})}\sigma_{i}))...).
\end{align}

Therefore, $\varrho_{n}$ is a sum of the following two types of terms:
\begin{itemize}
\item \text{\bf Type I}: $\Zr^{\alpha}(\varrho_{0}),|\alpha|=n$;
\item \text{\bf Type II}: $\Zr^{\beta}({}^{(\Zr_{i+1})}\sigma_{i}), |\beta|=n-i-1$.
\end{itemize}

The estimation of the \textbf{Type II} contributions to $\mathcal{N}_{w}(\Psi_{n})$ and $\underline{\mathcal{N}}_{w}(\Psi_{n})$ constitutes the most challenging and interesting part of \cite{Luo-YuRare1}. This analysis relies crucially on the structure of the Euler equations in the second null frame \eqref{Euler equations:form 1 ringed}, as well as the {\bf additional vanishing mechanism} of $\Xr(v^{1}+c)$. Moreover, this proof carries over verbatim to our setting (without the irrotational assumption), and we refer the reader to Sections 7 and 8 of \cite{Luo-YuRare1} for full details. It thus remains for us only to handle the \textbf{Type I} terms.

\subsubsection{Energy estimates on \text{\bf Type I} terms.} Since $\frac{\mu}{\mur}\lesssim 1$, it suffices to bound $\mathcal{N}_{w}(\Psi_{n})$ and $\underline{\mathcal{N}}_{w}(\Psi_{n})$ in the following form,
\begin{align*}
\mathcal{N}_{w}(\Psi_{n})=\int_{\mathcal{D}(t,u)}|\Zr^{\alpha}(\varrho_{0})||\hat{L}(\Psi_{n})|,\ \ \underline{\mathcal{N}}_{w}(\Psi_{n})=\int_{\mathcal{D}(t,u)} |\Zr^{\alpha}(\varrho_{0})||\underline{L}(\Psi_{n})|.
\end{align*}
The contribution for $\mathcal{N}_{w}(\psi)$ and $\underline{\mathcal{N}}_{w}(\psi)$ in \eqref{eq: fundamental inequalities for acoustical wave} can be divided into eight cases:
\begin{itemize}
\item $\mathcal{N}_{w, 1}(\Psi_{n})=-\int_{\mathcal{D}(t,u)} \frac{\kappa}{\kappar}\Zr^{\alpha}\Big(\kappar g(Df_{1},Df_{2})\Big) \hat{L}(\Psi_{n})$
\item $\mathcal{N}_{w, 2}(\Psi_{n})=-\int_{\mathcal{D}(t,u)}\frac{\kappa}{\kappar} \Zr^{\alpha}\Big(c^{\frac{\gamma+1}{\gamma-1}}\kappar\partial_{i}\Omega\Big)\hat{L}(\Psi_{n})$
\item $\mathcal{N}_{w, 3}(\Psi_{n})=-\int_{\mathcal{D}(t,u)}\frac{\kappa}{\kappar} \Zr^{\alpha}\Big(\Omega\kappar\partial_{i}(c^{\frac{\gamma+1}{\gamma-1}})\Big)\hat{L}(\Psi_{n})$
\item $\mathcal{N}_{w, 4}(\Psi_{n})=-\int_{\mathcal{D}(t,u)} \frac{\kappa}{\kappar}\Zr^{\alpha}\Big(\kappar c^{-\frac{4}{\gamma-1}}\Omega^{2}\Big)\hat{L}(\Psi_{n})$
\item $\underline{\mathcal{N}}_{w, 1}(\Psi_{n})=-\int_{\mathcal{D}(t,u)} \frac{\kappa}{\kappar}\Zr^{\alpha}\Big(\kappar g(Df_{1},Df_{2})\Big) \Lb(\Psi_{n})$
\item $\underline{\mathcal{N}}_{w, 2}(\Psi_{n})=\int_{\mathcal{D}(t,u)}\frac{\kappa}{\kappar} \Zr^{\alpha}\Big(c^{\frac{\gamma+1}{\gamma-1}}\kappar\partial_{i}\Omega\Big)\Lb(\Psi_{n})$
\item $\underline{\mathcal{N}}_{w, 3}(\Psi_{n})=\int_{\mathcal{D}(t,u)}\frac{\kappa}{\kappar} \Zr^{\alpha}\Big(\Omega\kappar\partial_{i}(c^{\frac{\gamma+1}{\gamma-1}})\Big)\Lb(\Psi_{n})$
\item $\underline{\mathcal{N}}_{w, 4}(\Psi_{n})=-\int_{\mathcal{D}(t,u)} \frac{\kappa}{\kappar}\Zr^{\alpha}\Big(\kappar c^{-\frac{4}{\gamma-1}}\Omega^{2}\Big)\Lb(\Psi_{n})$
\end{itemize}
The self interaction of acoustical waves $\mathcal{N}_{w, 1}(\Psi_{n})$ and $\underline{\mathcal{N}}_{w, 1}(\Psi_{n})$ can be exactly the same way as the irrotational case. To be more precise, the total contribution of $\mathcal{N}_{w, 1}(\Psi_{n})$, $\underline{\mathcal{N}}_{w, 1}(\Psi_{n})$ and \textbf{Type II} can be bounded by
\begin{equation}\label{eq: total contribution of acoustical wave}
\mathcal{N}_{w, 1}(\Psi_{n})+\underline{\mathcal{N}}_{w, 1}(\Psi_{n})\leq \Mr\varepsilon^{3}t^{2}+\frac{C}{a_{0}}\int_{0}^{u}\sum_{|\beta|\leq n}\mathscr{F}_{w}\big(\Zr^{\beta}(\psi)\big)(t,u^{'})du^{'}+a_{0}\int_{\delta}^{t}\sum_{|\beta|\leq n}\frac{\mathscr{E}_{w}\big(\Zr^{\beta}(\psi)\big)(t^{'},u)}{t^{'}}dt^{'}.
\end{equation}
where $a_{0}$ is small constant which will be chosen later, and we refer to section 8 of \cite{Luo-YuRare1} for details.

All nonlinear terms admit a unified treatment, making a case-by-case exposition unnecessary; the primary difficulty lies instead with the linear terms. Indeed, in view of $(\mathbf{B}_{2})$, \eqref{eq: L infty bound for waves}, the bound $\frac{\mu}{\kappar}\lesssim 1$, the Cauchy–Schwarz inequality, and the fact that we have closed the bootstrap assumption on $\Omega$ in \ref{close of L2 for vorticity}, the terms $\mathcal{N}_{w, j}(\Psi_{n})$ for $j=2,3,4$ and $\underline{\mathcal{N}}_{w, j}(\Psi_{n})$ for $j=2,3,4$ can be bounded by
\[
\sum_{j=2}^{4}\Big(\mathcal{N}_{w, j}(\Psi_{n})+ \underline{\mathcal{N}}_{w, j}(\Psi_{n})\Big)\leq \frac{C}{a_{0}}\varepsilon^{2}t^{2}+\Mr\varepsilon^{3}t^{2}+a_{0}\int_{\delta}^{t} \frac{\mathscr{E}_{w}(\Psi_{n})(t', u)}{t'}dt',
\]
where $a_{0}$ is a small constant to be chosen later. To illustrate this, consider $\mathcal{N}_{w, 2}(\Psi_{n})$: the most critical scenario arises when all $\Zr^{\alpha}$ derivatives act on $\partial_{1}\Omega$, which we handle as follows:
\[
\begin{split}
\int_{\mathcal{D}(t, u)} \Zr^{\alpha+1}(\Omega)\Lb(\Psi_{n})&\leq \frac{1}{4a_{0}} \int_{\mathcal{D}(t, u)} |\Zr^{\alpha+1}(\Omega)|^{2}\kappar +a_{0}\int_{\mathcal{D}(t, u)} \frac{|\Lb(\Psi_{n})|^{2}}{\kappar}\\
&\leq \frac{C}{a_{0}}\varepsilon^{2}t^{2}+a_{0}\int_{\delta}^{t}\frac{\mathscr{E}_{w}(\Psi_{n})(t', u)}{t'}dt'.
\end{split}
\]
All other cases are analogous or more straightforward.

In summary, upon an appropriate choice of $a_{0}$, we obtain for all $1\leq n\leq \Ntop$,
\[
\begin{split}
&\sum_{|\alpha|=n}\sum_{\psi\in \{\wb, w, \psi_{2}\}}\Big(\mathscr{E}_{w}\big(\Zr^{\alpha}(\psi)\big)(t,u)+\mathscr{F}_{w}(\Zr^{\alpha}\big(\psi)\big)(t,u)\Big)\leq \sum_{|\alpha|=n}\sum_{\psi\in \{\wb, w, \psi_{2}\}}\Big(\mathscr{E}_{w}\big(\Zr^{\alpha}(\psi)\big)(\delta,u)+\mathscr{F}_{w}\big(\Zr^{\alpha}(\psi)\big)(t,0)\Big)\\
&+\Mr\varepsilon^{3}t^{2}+C\varepsilon^{2}t^{2}+C\sum_{|\alpha|\leq n}\sum_{\psi\in \{\wb, w, \psi_{2}\}}\Big(\int_{0}^{u} \mathscr{F}_{w}\big(\Zr^{\alpha}(\psi)\big)(t,u^{'})du^{'}+C\int_{0}^{u}\mathscr{E}_{w}\big(\Zr^{\alpha}(\psi)\big)(t,u^{'})du^{'}\Big)\\
&+\frac{1}{2}\sum_{|\alpha|\leq n}\sum_{\psi\in \{\wb, w, \psi_{2}\}}\int_{\delta}^{t}\frac{\mathscr{E}_{w}\big(\Zr^{\alpha}(\psi)\big)(t^{'},u)}{t^{'}}dt^{'}.
\end{split}
\]
In view of $(\mathbf{I}_{2})$ and applying refined Gronwall's inequality \ref{lem: refined Gronwall} inductively from $n=1$ to $n=\Ntop$, we can conclude that for any $1\leq n\leq \Ntop$,
\begin{equation}
\sum_{|\alpha|=n}\sum_{\psi\in \{\wb, w, \psi_{2}\}}\Big(\mathscr{E}_{w}\big(\Zr^{\alpha}(\psi)\big)(t,u)+\mathscr{F}_{w}\big(\Zr^{\alpha}(\psi)\big)(t,u)\Big)\lesssim \varepsilon^{2}t^{2}.
\end{equation}
providing that $\Mr\varepsilon$ is sufficiently small. This closes the bootstrap assumption $(\mathbf{B}_{2})$ \eqref{assumption: B2} together with \eqref{close of L2 for vorticity} and \eqref{eq: close of lowest energy}.

In view of Sobolev inequality \eqref{eq: Sobolev inequality}, we have 
\begin{corollary}
For all multi-index $|\alpha|\leq \Ninf+1$, for all $\psi\in \{\wb, w, \psi_{2}\}$, except for the case $\Zr^{\alpha}\psi=T(\wb)$ and $\Omega$, we have
\begin{equation}\label{eq: improved L infty bound for waves}
\|\Zr^{\alpha}(\psi)\|_{L^{\infty}(\Sigma_{t}^{u})}\lesssim \begin{cases}
\varepsilon, \ \ \text{if}\ \Zr^{\alpha}=\Xr^{\alpha},\\
\varepsilon t, \ \ \text{otherwise}.
\end{cases}, \|\Zr^{\alpha}(\Omega)\|_{L^{\infty}(\Sigma_{t}^{u})}\lesssim \varepsilon.
\end{equation}
Moreover, $\Xr(v^{1}+c)$ and $\Tr(v^{1}+c)+1$ has the following additional vanishing property: for any $|\beta|\leq \Ninf-1$, 
\begin{equation}\label{eq: additional vanishing}
\begin{cases}
\|\Zr^{\beta}\Xr(v^{1}+c)\|_{\Sigma_{t}^{u}}\lesssim \varepsilon t,\\
\|\Zr^{\beta}\Big(T(v^{1}+c)+1\Big)\|_{\Sigma_{t}^{u}}\lesssim \varepsilon t.
\end{cases}
\end{equation}
\end{corollary}
We remark that \eqref{eq: additional vanishing} is a byproduct of estimation of the \textbf{Type II} in \cite{Luo-YuRare1} and we refer to section 7 of \cite{Luo-YuRare1} for details.

\section{The closing of $(\mathbf{B}_{\infty})$ assumption}\label{section6}

\subsection{Closing $(\mathbf{B}_{\infty})$ for $w, \psi_{2}$}

\begin{itemize}
\item The case $\hat{X}^{n}(\psi), \psi\in \{w, \psi_{2}\}, n\leq \Ninf$. Since $\hat{X}=-\hat{T}^{1}\Xr-\frac{\hat{T}^{2}}{\kappar}\Tr$
We can rewrite $\hat{X}^{n}(\psi), \psi\in \{w, \psi_{2}\}, n=|\alpha|\leq \Ninf$ as
\[
(\hat{X}^{\Tr}\Tr+\hat{X}^{\Xr}\Xr)^{n}\psi
\]
For $f\in \{\hat{X}^{\Xr}, \hat{X}^{\Tr}\}$, we have
\[
\Zr^{\alpha}(f)\lesssim \begin{cases}
1, \Zr^{\alpha}(f)=\hat{X}^{\Xr},\\
\Mr\varepsilon, \ \ \text{otherwise}.
\end{cases}
\]
Therefore, we can conclude that
\[
|Z^{\alpha}(\psi)|\lesssim |\Xr^{n}(\Omega)|+\Mr\varepsilon^{2}\lesssim \varepsilon.
\]
providing that $\Mr\varepsilon$ is sufficiently small.

\item The case $Z^{\alpha}(\psi), \psi\in \{w, \psi_{2}\}, n=|\alpha|\leq \Ninf, Z^{\alpha}\not=\hat{X}^{n}$. In view of commutator for $[T, \hat{X}]$, it suffices to close the bound for $TZ^{\beta}\psi$ with $m=|\beta|\leq \Ninf-1$. Since
\[
\begin{cases}
\hat{X}=-\hat{T}^{1}\Xr-\frac{\hat{T}^{2}}{\kappar}\Tr,\\
T=\kappa \hat{T}^{2}\Xr-\frac{\kappa}{\kappar}\hat{T}^{1}\Tr.
\end{cases}
\]
We can rewrite $TZ^{\beta}(\psi), m=|\beta|\leq \Ninf-1$ as
\begin{equation}
(T^{\Xr}\Xr+T^{\Tr}\Tr)(Z_{1}^{\Tr}\Tr+Z_{1}^{\Xr}\Xr)\cdots (Z_{m}^{\Tr}\Tr+Z_{m}^{\Xr}\Xr)\psi
\end{equation}
For $f\in \{\hat{X}^{\Xr}, \hat{X}^{\Tr}, T^{\Xr}, T^{\Tr}\}$, we have
\[
\Zr^{\gamma}(f)\lesssim \begin{cases}
1, \Zr^{\gamma}(f)=\hat{X}^{\Xr}\ \ \text{or}\ \ \Zr^{\gamma}(f)=T^{\Tr},\\
\Mr\varepsilon, \Zr^{\gamma}(f)=\Zr^{\gamma}(\hat{X}^{\Tr}),\\
\Mr \varepsilon t,\ \ \text{otherwise}.
\end{cases}, |\gamma|\leq \Ninf-1.
\]

Therefore, we can conclude that
\[
|TZ^{\beta}(\psi)|\lesssim |\Tr\Zr_{1}\cdots \Zr_{m}(\psi)|+\Mr\varepsilon^{2}t\lesssim \varepsilon t.
\]
providing that $\Mr\varepsilon$ is sufficiently small.
\item The case $LZ^{\alpha}(\psi), \psi\in \{\hat{X}, T\}, |\alpha|\leq \Ninf-1$. Since
\[
\begin{cases}
\hat{X}=-\hat{T}^{1}\Xr-\frac{\hat{T}^{2}}{\kappar}\Tr,\\
T=\kappa \hat{T}^{2}\Xr-\frac{\kappa}{\kappar}\hat{T}^{1}\Tr.
\end{cases}
\]
We can rewrite $LZ^{\alpha}(\psi), m=|\alpha|\leq \Ninf-1$ as
\begin{equation}
L(Z_{1}^{\Tr}\Tr+Z_{1}^{\Xr}\Xr)\cdots (Z_{m}^{\Tr}\Tr+Z_{m}^{\Xr}\Xr)\psi
\end{equation}
For $f\in \{\hat{X}^{\Xr}, \hat{X}^{\Tr}, T^{\Xr}, T^{\Tr}\}$, we have
\[
L\Zr^{\gamma}(f)\lesssim \begin{cases}
t^{-1}, L\Zr^{\gamma}(f)=L\Zr^{\gamma}(T^{\Tr})\ \ \text{or}\ \ L\Zr^{\gamma}(f)=L\Zr^{\gamma}(\hat{X}^{\Tr}),\\
\Mr \varepsilon ,\ \ \text{otherwise}.
\end{cases}, |\gamma|\leq \Ninf-1.
\]
Therefore, we can conclude that
\[
|LZ^{\alpha}(\psi)|\lesssim |L\Tr\Zr_{1}\cdots \Zr_{m}(\psi)|+\Mr\varepsilon^{2}\lesssim \varepsilon t.
\]
providing that $\Mr\varepsilon$ is sufficiently small.
\end{itemize}

\subsection{The closing $(\mathbf{B}_{\infty})$ for $\wb$}
It suffices to bound the maximal characteristic speed $v^{1}+c$ in place of $\underline{w}$ since $v^{1}+c=\frac{\gamma+1}{2}\wb+\frac{\gamma-3}{2}w$

\begin{itemize}
\item The case $\hat{X}^{n}(v^{1}+c), n\leq \Ninf$. Since $\hat{X}=-\hat{T}^{1}\Xr-\frac{\hat{T}^{2}}{\kappar}\Tr$, we can rewrite $\hat{X}^{n}(v^{1}+c)$ as
\[
(\hat{X}^{\Tr}\Tr+\hat{X}^{\Xr}\Xr)^{n}(v^{1}+c)
\]
The summation of which consists of terms like
\[
\Zr_{1}^{\beta_{1}}(f_{1})\cdots \Zr^{\beta_{m}}_{m}(f_{m})\Zr^{\alpha_{1}}(v^{1}+c), \sum_{k=1}^{m}|\beta_{k}|+|\alpha_{1}|=n, |\alpha_{1}|\geq 1.
\]
where $f_{i}\in \{\hat{X}^{\Xr}, \hat{X}^{\Tr}\}$. The terms can be divided into three cases:
\begin{itemize}
\item main term: $\beta_{1}=...=\beta_{n}=0, f_{1}=\cdots=f_{n}=\hat{X}^{\Xr}$ 
\[
|(\hat{X}^{\Xr})^{n}\mathring{Z}_{1}\cdots \mathring{Z}_{n}(v^{1}+c)|\lesssim \varepsilon t.
\]
\item {\bf linear error term}: $\hat{X}^{n-1}(\hat{X}^{\Tr})\Tr(v^{1}+c)$. 
\item nonlinear error terms: Note that $\Zr^{\alpha_{1}}(v^{1}+c)\not=\Tr(v^{1}+c)$ and we have
\[
|\Zr_{1}^{\beta_{1}}(f_{1})\cdots \Zr^{\beta_{n}}_{n}(f_{n})\Zr^{\alpha_{1}}(v^{1}+c)|\lesssim \Mr\varepsilon^{2} t.
\]
\end{itemize}
Note that the rough bound $\|\hat{X}^{n-1}(\hat{X}^{\Tr})\Tr(v^{1}+c)\|_{L^{\infty}(\Sigma_{t}^{u})}\lesssim \Mr\varepsilon$ for {\bf linear error term} is insufficient to close the $(\mathbf{B}_{\infty})$ assumption for $v^{1}+c$. In view of \eqref{eq: additional vanishing}, $|\Tr(v^{1}+c)+1|\lesssim \varepsilon t$, we can rewrite the {\bf linear error term} as 
\[
\hat{X}^{n-1}(\hat{X}^{\Tr})\Tr(v^{1}+c)=-\hat{X}^{n-1}(\hat{X}^{\Tr})+\Mr\varepsilon^{2}t.
\]
and we have
\[
\hat{X}^{n}(v^{1}+c)+\hat{X}^{n-1}(\hat{X}^{\Tr})\lesssim \varepsilon t+\Mr\varepsilon^{2}t .
\]

In view of the transport equation for $\hat{T}^{2}$ \eqref{eq: formulas to control the geometry}, we also have
\[
L\hat{X}^{n-1}(\hat{T}^{2})=\hat{X}^{n}(v^{1}+c)+\Mr\varepsilon^{2}t,
\]
integrating $L\hat{X}^{n-1}(\hat{T}^{2})$ from $\delta$ to $t$,
\[
\hat{X}^{n-1}(\hat{T}^{1})(t,u,\theta)-\hat{X}^{n-1}(\hat{T}^{2})(\delta, u, \theta)-\int_{\delta}^{t}\hat{X}^{n}(v^{1}+c)(\tau, u, \theta)d\tau=\Mr\varepsilon^{2}t^{2}.
\]
Since $\hat{X}^{\Tr}=-\frac{\hat{T}^{2}}{\kappar}$, we have
\[
|\hat{X}^{n}(v^{1}+c)|(t,u,\theta)\leq \frac{|\hat{X}^{n-1}(\hat{T}^{2})|(\delta, u, \theta)}{t}+\frac{1}{t}\int_{\delta}^{t}|\hat{X}^{n}(v^{1}+c)|(\tau, u, \theta)d\tau+\Mr\varepsilon^{2}t+\varepsilon t.
\]
By Gronwall's inequality, we have
\[
|\hat{X}^{n}(v^{1}+c)|(t,u,\theta)\leq \frac{|\hat{X}^{n-1}(\hat{T}^{2})|(\delta, u, \theta)}{\delta}+\Mr\varepsilon^{2}t+\varepsilon t.
\]
which implies that $|Z^{\beta}T(v^{1}+c)|\lesssim \varepsilon+\Mr\varepsilon^{2}t+\varepsilon t\lesssim \varepsilon $ providing that $\varepsilon$ is sufficiently small.

\item The case $Z^{\alpha}(v^{1}+c), Z^{\alpha}\not=\hat{X}^{\alpha}, 2\leq |\alpha|\leq \Ninf$. In view of commutator formula \eqref{eq: formulas to control the geometry} for $[T, \hat{X}]$,
\[
\begin{split}
Z^{\alpha_{1}}TZ^{\alpha_{2}}(v^{1}+c)&=Z^{\alpha_{1}+\alpha_{2}}T(v^{1}+c)+\sum_{\beta_{1}+\beta_{2}<\alpha_{2}}Z^{\beta_{1}}(\zeta+\eta)\hat{X}Z^{\beta_{2}}T(v^{1}+c)\\
&=Z^{\alpha_{1}+\alpha_{2}}T(v^{1}+c)+\Mr\varepsilon^{2}t^{2}.
\end{split}
\]
Therefore it suffices to close the bound for $Z^{\beta}T(v^{1}+c)$ with $m=|\beta|\leq \Ninf-1$. Since
\[
\begin{cases}
\hat{X}=-\hat{T}^{1}\Xr-\frac{\hat{T}^{2}}{\kappar}\Tr,\\
T=\kappa \hat{T}^{2}\Xr-\frac{\kappa}{\kappar}\hat{T}^{1}\Tr.
\end{cases}
\]
we can rewrite $Z^{\beta}T(v^{1}+c)$ as
\[
(Z_{1}^{\Tr}\Tr+Z_{1}^{\Xr}\Xr)\cdots (Z_{m}^{\Tr}\Tr+Z_{m}^{\Xr}\Xr)(T^{\Xr}\Xr+T^{\Tr}\Tr)(v^{1}+c),
\]
the summation of which consists of terms like
\[
\Zr_{1}^{\beta_{1}}(f_{1})\cdots \Zr^{\beta_{n}}_{n}(f_{n})\Zr^{\alpha_{1}}(v^{1}+c), \sum_{k=1}^{n}|\beta_{k}|+|\alpha_{1}|=n, |\alpha_{1}|\geq 1.
\]
where $f_{i}\in \{\hat{X}^{\Xr}, \hat{X}^{\Tr}, T^{\Xr}, T^{\Tr}\}$. The terms can be divided into three cases:
\begin{itemize}
\item main term: $\beta_{1}=...=\beta_{m}=0, f_{1}=Z_{1}^{\Zr_{1}}, ..., f_{m}=Z_{m}^{\Zr_{m}}$ 
\[
|Z_{1}^{\mathring{Z}_{1}}\cdots Z_{m}^{\mathring{Z}_{m}}T^{\Tr}\mathring{Z}_{1}\cdots \mathring{Z}_{m}\Tr(v^{1}+c)|\lesssim \varepsilon t.
\]
\item {\bf linear error term}: $Z^{\beta}(T^{\Tr})\Tr(v^{1}+c)$. 
\item nonlinear error terms:  Note that $\Zr^{\alpha_{1}}(v^{1}+c)\not=\Tr(v^{1}+c)$ and we have
\[
|\Zr_{1}^{\beta_{1}}(f_{1})\cdots \Zr^{\beta_{m}}_{m}(f_{m})\Zr^{\alpha_{1}}(v^{1}+c)|\lesssim \Mr\varepsilon^{2} t.
\]
\end{itemize}
Note that the rough bound $\|Z^{\beta}(T^{\Tr})\Tr(v^{1}+c)\|_{L^{\infty}(\Sigma_{t})}\lesssim \Mr\varepsilon t$ for {\bf linear error term} is insufficient to close the $(\mathbf{B}_{\infty})$ assumption for $v^{1}+c$. In view of \eqref{eq: additional vanishing}, $|\Tr(v^{1}+c)+1|\lesssim \varepsilon t$, we can write {\bf linear error term}  as
\[
Z^{\beta}(T^{\Tr})\Tr(v^{1}+c)=-Z^{\beta}(T^{\Tr})+\Mr\varepsilon^{2}t,
\]
and we have
\[
Z^{\beta}T(v^{1}+c)+Z^{\beta}(T^{\Tr})\lesssim \varepsilon t+ \Mr\varepsilon^{2}t.
\]
In view of the transport equation for $\kappa$ and $\hat{T}^{1}$ \eqref{eq: formulas to control the geometry}, we also have
\[
L(Z^{\beta}(\kappa\hat{T}^{1}))=Z^{\beta}T(v^{1}+c)+\Mr\varepsilon^{2}t, Z^{\beta}(\kappa\hat{T}^{1})=-Z^{\beta}(\kappa)+\Mr\varepsilon^{2}t^{2}.
\]
Integrating $LZ^{\beta}(\kappa\hat{T}^{1})$ from $\delta$ to $t$,
\[
Z^{\beta}(\kappa\hat{T}^{1})(t,u,\theta)-Z^{\beta}(\kappa\hat{T}^{1})(\delta, u, \theta)-\int_{\delta}^{t}Z^{\beta}T(v^{1}+c)(\tau, u, \theta)d\tau=\Mr\varepsilon^{2}t^{2}.
\]
Since $T^{\Tr}=-\frac{\kappa}{\kappar}\hat{T}^{1}$, we can obtain
\[
|Z^{\beta}T(v^{1}+c)|(t,u,\theta)\leq \frac{|Z^{\beta}(\kappa)|(\delta, u, \theta)}{t}+\frac{1}{t}\int_{\delta}^{t}|Z^{\beta}T(v^{1}+c)|(\tau, u, \theta)d\tau+\Mr\varepsilon^{2}t+\varepsilon t.
\]
By Gronwall's inequality, we have
\[
|Z^{\beta}T(v^{1}+c)|(t,u,\theta)\leq \frac{|Z^{\beta}(\kappa)|(\delta, u, \theta)}{\delta}+\Mr\varepsilon^{2}t+\varepsilon t.
\]
which implies that $|Z^{\beta}T(v^{1}+c)|\lesssim \Mr\varepsilon^{2}t+\varepsilon t\lesssim \varepsilon t$ providing that $\varepsilon$ is sufficiently small.

\item The case $LZ^{\alpha}(\wb), |\alpha|\leq \Ninf-1$. We commute $Z^{\alpha}, 1\leq |\alpha|\leq \Ninf-1$ with the equation 
\[L(\underline{w})=-c\frac{T(\underline{w})}{\kappa}(\hat{T}^{1}+1)+\frac{1}{2}c\frac{T(\psi_{2})}{\kappa}+\frac{1}{2}c\hat{X}(\psi_{2})-c\hat{X}(\underline{w})\hat{X}^{1}\]
to conclude that $\|LZ^{\alpha}(\underline{w})\|_{L^{\infty}(\Sigma_{t})}\lesssim \varepsilon$ providing that $\varepsilon$ is sufficiently small.
\end{itemize}

{\bf This closes the $(\mathbf{B}_{\infty})$ assumption \eqref{assumption: B_infty} and we have finished the proof of Theorem \ref{theorem: A priori Energy estimate}}. 

In summary, we have 
\begin{proposition}
For any multi-index $\alpha$ with $|\alpha|\leqslant \Ninf - 1, |\beta|\leq \Ninf, |\gamma|\leq \Ninf-2$, for all $t\in [\delta,t^*]$, for all $\psi,\psi'\in \{\wb,w,\psi_2\},  Z\in \{\hat{X}, T\}$ except for the case $(\alpha, \psi')=(0,\wb)$, we have (independent of $\delta$)
\begin{equation}\label{eq:basic L estimates}
\begin{cases}
\|T(\wb)+\frac{2}{\gamma+1}\|_{L^{\infty}(\Sigma_{t}^{u^{*}})}\lesssim \varepsilon t, 
\|L Z^\alpha \psi \|_{L^\infty(\Sigma_t^{u^{*}})}+\|\Xh Z^\alpha \psi \|_{L^\infty(\Sigma_t^{u^{*}})}+t^{-1}\|T Z^\alpha \psi' \|_{L^\infty(\Sigma_t^{u^{*}})}\lesssim \varepsilon,\\
\|Z^{\alpha}(\hat{T}^{1}+1)\|_{L^{\infty}(\Sigma_{t}^{u^{*}})}\lesssim \varepsilon^{2}t^{2}, \|Z^{\alpha}(\hat{T}^{2})\|_{L^{\infty}(\Sigma_{t}^{u^{*}})}\lesssim \varepsilon t, \|Z^{\alpha}(\kappa-t)\|_{L^{\infty}(\Sigma_{t}^{u^{*}})}\lesssim \varepsilon t^{2},\\
\|Z^{\gamma}\big(\chi\big)\|_{L^{\infty}(\Sigma_{t}^{u^{*}})}\leq \Mr\varepsilon, \|Z^{\gamma}\big(\slashed{g}-1\big)\|_{L^{\infty}(\Sigma_{t}^{u^{*}})}\lesssim \varepsilon t,\\
\|Z^{\gamma}\big(\zeta\big)\|_{L^{\infty}(\Sigma_{t}^{u^{*}})}+\|Z^{\alpha}\big(\eta\big)\|_{L^{\infty}(\Sigma_{t}^{u^{*}})}\leq \Mr\varepsilon t, \|Z^{\gamma}\|\big(\Xi\big)\|_{L^{\infty}(\Sigma_{t}^{u^{*}})}\leq \Mr\varepsilon t^{2},\\
\|\Zr^{\alpha}\Xr(v^{1}+c)\|_{\Sigma_{t}^{u}}\lesssim \varepsilon t, \|\Zr^{\alpha}\Big(T(v^{1}+c)+1\Big)\|_{\Sigma_{t}^{u}}\lesssim \varepsilon t,\\
\|Z^{\beta}\Omega\|_{L^{\infty}(\Sigma_{t}^{u^{*}})}\lesssim \varepsilon.
\end{cases}
\end{equation}
\end{proposition}

\section{The existence of rarefaction waves connected to the data on the right}
For all $\delta>0$ and the data constructed on $\Sigma_\delta^{u^{*}}$, we have a unique solution $U_{\delta}$ defined on $\mathcal{D}(\delta)$. For the sake of simplicity, we drop the dependence on $\delta$ and write the solution as $U$.

\subsection{The region of convergence}\label{section:region of convergence}

For the constant state {\color{black}$\Ur_{r}$} on $x_1\geqslant 0$, the corresponding characteristic boundary of the future development is a flat hypersurface ${C}^{\rm cst}_{0}$ inside $\mathbb{R}^3$, which is given by
\[
\frac{x_{1}}{t}=\mathring{v}_{r}^{1}+\mathring{c}.
\]


We compute the null vector field $L^{\rm cst}$ and the acoustical function $u^{\rm cst}$ associated to the constant state {\color{black}$(\mathring{v}_r,\mathring{c}_r)$}. In view of the associated 1-dimensional rarefaction wave computed in \eqref{eq: 1-D rarefaction wave}, we have
\[L^{\rm cst}=\partial_t+\frac{x_1}{t}\partial_1+\mathring{v}_{r}^{2}\partial_{2}, \quad u^{\rm cst}=(\mathring{v}_{r}^{1}+\mathring{c})-\frac{x_1}{t}. \] 
Moreover, in this case, the acoustical function $u$ on $\Sigma_\delta$ is defined by $u^{\rm cst}=(\mathring{v}_{r}^{1}+\mathring{c})-\frac{x_1}{\delta}$.

We now compare $L$ and $L^{\rm cst}$. Since $L^i = L(x^i)=v^i-c\Th^i$, we can compare $L$ and $L^{\rm cst}$ as follows
\[L-L^{\rm cst}=\big(v^1+c-\frac{x_1}{t}\big)\partial_1+\big(v^2-\mathring{v}_{r}^{2}-c\Th^2\big)\partial_2-c\big(\Th^1+1\big)\partial_1.\]

Let $\slashed{x}_{1}$ be the restriction of $x_{1}$ on $C_{0}^{t^{*}}$, then in acoustical coordinate $(t, \vartheta)$, 
\[
\slashed{x}_{1}=\int_{0}^{t}(v^{1}-c\hat{T}^{1})(\tau, \vartheta)d\tau.
\]
Therefore, we have
\[ \big(v^1+c - \frac{\slashed{x}_1}{t}\big)(t,\vartheta) = \frac{1}{t}\int_0^t (v^1+c)(t,\vartheta) - (v^1+c)(\tau,\vartheta)d\tau + {\color{black}\frac{1}{t}\int_0^t \big(c\cdot (\Th^1+1)\big)(\tau,\vartheta)d\tau}. \]
Therefore, by $|\hat{T}^{1}+1|\lesssim \varepsilon^{2}t^{2}$ on $C_{0}^{u^{*}}$ and the intermediate value theorem  we have 
\[ \|v^1+c-\frac{x_1}{t}\|_{L^\infty(C_0 \cap \Sigma_t)}\lesssim \varepsilon t. \]
Note that in acoustical coordinate $(t, u, \vartheta)$, $\frac{\partial}{\partial u}=T+\Xi\slashed{g}\hat{X}$, in view of \eqref{eq:basic L estimates} and $\Tr\big(v^1+c-\frac{x_1}{t}\big)=\Tr(v^1+c)+1$, we have 
\[\|\frac{\partial}{\partial u}\big(v^1+c-\frac{x_1}{t}\big)\|_{L^\infty(\Sigma_t^{u^{*}})}\lesssim \varepsilon t.\]
We can integrate the above inequality in $u$ form $C_0$ and this leads to
\[\|v^1+c-\frac{x_1}{t}\|_{L^\infty(\Sigma_t^{u^{*}})}\lesssim \varepsilon t.\]
In view of the pointwise bounds on $\psi_2, \Th^2$ and $\Th^1+1$ in \eqref{eq:basic L estimates},
we conclude that
\begin{equation}\label{eq: compare L Lcst}
	\|(L-L^{\rm cst})(x_1)\|_{L^\infty(\Sigma_t)} \lesssim \varepsilon t, \quad \|(L-L^{\rm cst})(x_2)\|_{L^\infty(\Sigma_t)} \lesssim \varepsilon,
\end{equation}
on the entire $\mathcal{D}(\delta)(t^{*}, u^{*})$.

Now we compare the function $u$ with its counterpart $u^{\rm cst}$. The exact construction of $u$ on $\Sigma_{\delta}^{u^{*}}$ is given by in 3.1 of \cite{Luo-YuRare2},
\[
u(x_{1}, x_{2})=-\frac{x_{1}}{\delta}+\frac{\slashed{x}_{1}(\delta, \vartheta(\delta, x_{2}))}{\delta}=-\frac{x_{1}}{\delta}+\int_{0}^{\delta} (v^{1}-c\hat{T}^{1})(\tau, \vartheta(\delta, x_{2}))d\tau
\]

First, since $u\big|_{S_{\delta,0}}\equiv 0$ and the data is $O(\varepsilon)$-close to the constant state {\color{black}$(\mathring{v}_r,\mathring{c}_r)$},  we have on $S_{\delta, 0}$
\[ 
\begin{split}
u^{\rm cst}-u&= (\mathring{v}_{r}^{1}+\mathring{c}_{r})-\frac{x_{1}}{\delta}-(-\frac{x_{1}}{\delta}-\frac{1}{\delta}\int_{0}^\delta I(\tau,x_2)d\tau)\\
&= (\mathring{v}_{r}-v(\delta, \vartheta(\delta, x_{2})))+(\mathring{c}_{r}-c(\delta, \vartheta(\delta, x_{2})))+\frac{1}{\delta}\int_{0}^{\delta}\Big((v^{1}+c)(\delta, \vartheta(\delta, x_{2}))-(v^{1}+c)(\tau, \theta(\delta, x_{2}))\Big)d\tau\\
&+\frac{1}{\delta}\int_0^\delta \big(c\cdot (\Th^1+1)\big)(\tau,\vartheta)d\tau\lesssim \varepsilon+\varepsilon t\lesssim \varepsilon.
\end{split}
\]
In view of \eqref{eq:basic L estimates}, we have on $\Sigma_{\delta}^{u^{*}}$ that
\begin{equation}
\nabla\big[u-\big(k^{\rm cst}-\frac{x_1}{\delta}\big)\big] = -\frac{1}{\delta}\big(0,\int_{0}^{\delta}X(\psi_{1}+c\cdot\hat{T}^{1})(\tau, \theta(\delta, x_{2}))\frac{\partial \theta}{\partial x_{2}}(\delta, x_{2})d\tau\big) \lesssim O(\varepsilon).
\end{equation}
for all $u\in [0, u^*]$.  
Furthermore, since $L^{\rm cst}(u^{\rm cst}) = 0$, by \eqref{eq: compare L Lcst}  we have 
\begin{equation}
	L(u^{\rm cst}) = (L-L^{\rm cst})\Big((\mathring{v}_{r}^{1}+\mathring{c})-\frac{x_1}{t}\Big)=-\frac{(L-L^{\rm cst})(x_{1})}{t} \lesssim O(\varepsilon).
\end{equation}
In view of these estimates, We conclude that
\begin{equation}\label{eq: compare u ucst}
	\|u - u^{\rm cst}\|_{L^\infty(\Sigma_t^{u^{*}})} \lesssim \varepsilon.
\end{equation}

\begin{center}
\includegraphics[width=3.2in]{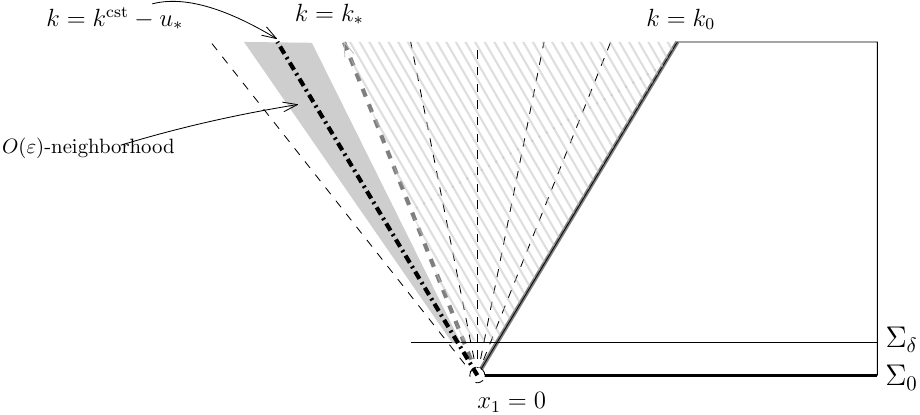}
\end{center}

By definition, $C_{u^*}$ is the characteristic hypersurface emanated from the level set  $u=u^*$ inside $\Sigma_\delta$ along the direction of $L$. {\color{black}More precisely, we consider the
	vector field $L$ along $S_{\delta,u^*}$ and extend each vector to a geodesic to generate $C_{u^*}$. Similarly, the hypersurface $C^{\rm cst}_{u^*}$ is generated by extending each vector $L^{\rm cst}$ along the surface $u^{\rm cst}=u^*$, i.e., the circle $\frac{x^1}{\delta} = \mathring{v}_{r}^{1}+\mathring{c}_{r} - u^*$, to a geodesic of the background solution which corresponds to the constant state $(\mathring{v}_r,\mathring{c}_r)$ on the right.}
In the above picture, it is  depicted as the dash-dotted line and it is the center in the grey region. It can be written down explicitly as follows:
 \[C^{\rm cst}_{u^*}=\big\{(t,x_1,x_2)\big|\frac{x_1}{t}=\mathring{v}_{r}^{1}+\mathring{c}_{r} - u^*, t\in [\delta, t^*], x_2\in [0,2\pi]\big\}.\]

Note that \eqref{eq: compare L Lcst} and \eqref{eq: compare u ucst} imply that $C_{u^*}$ is in the $O(\varepsilon t)$-neighborhood (depicted as the grey region in the picture) of $C^{\rm cst}_{u^*}$, i.e., on each $\Sigma_t$, the distance between $C^{\rm cst}_{u^*}$ and {\color{black}$C_{u^*}$} is bounded above by {\color{black}$A_0 \varepsilon t$} where $A_0$ is a universal constant.
We define 
\begin{equation}\label{def: varepsilon0}
\varepsilon_0= A_0 \varepsilon. 
\end{equation} 
We also define
\[k_*=(\mathring{v}_{r}^{1}+\mathring{c})-u^*+\varepsilon_0.\]
For all $\delta>0$, we then consider the following region\footnote{While it is more natural to consider the region $\mathcal{D}(\delta)$ directly, its left boundary depends on $\delta$. To apply the Arzelà–Ascoli type argument on a fixed domain, we therefore define $\mathcal{W}$ as follows.}:
\[\mathcal{W}_\delta=\big\{(t,x_1,x_2)\in \mathcal{D}(\delta)\big| \delta\leqslant t\leqslant t^*,  \frac{x_1}{t}\geqslant k_*\big\}.\]
This is the shaded region in the above picture. It is clear that $\mathcal{W}_\delta$ is a compact domain. We also have $\mathcal{W}_\delta\subset \mathcal{D}_{\delta'}$\footnote{Note that $u=u^{*}\implies u^{\rm cst}\geq u^{*}-\varepsilon_{0}$, which ensures that $C_{u}, 0\leq u\leq u^{*}$ are contained $\mathcal{W}_{\delta}$.} for all $\delta\geqslant \delta'$. 
We define
\[\mathcal{W}=\bigcup_{0<\delta<\frac{1}{2}, \delta\ \text{is dyadic}}\mathcal{W}_\delta.\]
 In the rest of the section, we will construct centered rarefaction waves on $\mathcal{W}$ associated to the solution $U_{r}$ on $\mathcal{D}_0$. Since for any dyadic $\delta$, we have derived $\delta$ independent estimates \eqref{eq:basic L estimates} for $U_{\delta}$. By repeating Arzelà–Ascoli theorem and diagonal argument as \cite{Luo-YuRare2}, there exists $N_{3}\leq \Ntop$ such that 
 \[
 \{c_{\delta}, v^{1}_{\delta}, v^{2}_{\delta}, \Omega_{\delta}, \hat{T}^{1}_{\delta}, \hat{T}^{2}_{\delta}, \kappa_{\delta}, \chi_{\delta}, \zeta_{\delta}, \eta_{\delta}, \Xi_{\delta}, u_{\delta}, \vartheta_{\delta}\}
 \]
converges to $C^{0}\Big((0,t^{*}]; C^{N_{3}}(\mathcal{W}\cap \Sigma_{t})\Big)$ functions 
\[
 \{c, v^{1}, v^{2}, \Omega, \hat{T}^{1}, \hat{T}^{2}, \kappa, \chi, \zeta, \eta, \Xi, u, \vartheta\}
 \]
defined on $\mathcal{W}$. In particular we have,
\begin{proposition}
For the solution $U$ on $\mathcal{W}$ together with associated geometric quantities, for any multi-index $\alpha$ with $|\alpha|\leqslant N_{3} - 1, |\beta|\leq N_{3}, |\gamma|\leq N_{3}$, for all $t\in [\delta,t^*]$, for all $\psi,\psi'\in \{\wb,w,\psi_2\},  Z\in \{\hat{X}, T\}$ except for the case $(\alpha, \psi')=(0,\wb)$, we have 
\begin{equation}\label{eq: L estimates in W}
\begin{cases}
\|T(\wb)+\frac{2}{\gamma+1}\|_{L^{\infty}(\Sigma_{t}^{u^{*}})}\lesssim \varepsilon t, 
\|L Z^\alpha \psi \|_{L^\infty(\Sigma_t^{u^{*}})}+\|\Xh Z^\alpha \psi \|_{L^\infty(\Sigma_t^{u^{*}})}+t^{-1}\|T Z^\alpha \psi' \|_{L^\infty(\Sigma_t^{u^{*}})}\lesssim \varepsilon,\\
\|Z^{\alpha}(\hat{T}^{1}+1)\|_{L^{\infty}(\Sigma_{t}^{u^{*}})}\lesssim \varepsilon^{2}t^{2}, \|Z^{\alpha}(\hat{T}^{2})\|_{L^{\infty}(\Sigma_{t}^{u^{*}})}\lesssim \varepsilon t, \|Z^{\alpha}(\kappa-t)\|_{L^{\infty}(\Sigma_{t}^{u^{*}})}\lesssim \varepsilon t^{2},\\
\|Z^{\gamma}\big(\chi\big)\|_{L^{\infty}(\Sigma_{t}^{u^{*}})}\leq \Mr\varepsilon, \|Z^{\gamma}\big(\slashed{g}-1\big)\|_{L^{\infty}(\Sigma_{t}^{u^{*}})}\lesssim \varepsilon t,\\
\|Z^{\gamma}\big(\zeta\big)\|_{L^{\infty}(\Sigma_{t}^{u^{*}})}+\|Z^{\alpha}\big(\eta\big)\|_{L^{\infty}(\Sigma_{t}^{u^{*}})}\leq \Mr\varepsilon t, \|Z^{\gamma}\|\big(\Xi\big)\|_{L^{\infty}(\Sigma_{t}^{u^{*}})}\leq \Mr\varepsilon t^{2},\\
\|\Zr^{\alpha}\Xr(v^{1}+c)\|_{\Sigma_{t}^{u}}\lesssim \varepsilon t, \|\Zr^{\alpha}\Big(T(v^{1}+c)+1\Big)\|_{\Sigma_{t}^{u}}\lesssim \varepsilon t,\\
\|Z^{\beta}\Omega\|_{L^{\infty}(\Sigma_{t}^{u^{*}})}\lesssim \varepsilon.
\end{cases}
\end{equation}
\end{proposition}

To resolve the initial singularity $\mathbf{S}_{*}$ with acoustical coordinate $(t, u, \vartheta)$, we still have to derive uniform estimates on $L^n(U)$ for sufficiently large $n$, which is absent in \eqref{eq: L estimates in W}.

\subsection{Retrieving the uniform bounds on $L$-derivatives}\label{Section:The precise bounds on L derivatives}

\subsubsection{An ODE systems for $L^{k+1}(\psi)$ and $L^{k}(\Omega)$}\label{section 7.2}

 The main idea in this subsection is to use the acoustical wave equations for $\psi\in \{\wb,w,\psi_2\}$ in the null frame $(\Lb,L,\Xh)$ and transport equation $B\Omega=0$ for $\Omega$. We will integrate along the $\Lb$ direction to retrieve the uniform bounds on $L^k(\psi)$ from $C_0$ where  $L^k(\psi)$ are uniformly bounded from the continuity across $C_0$. We will integrate along $c^{-1}\kappa B$ direction to retrieve the uniform bounds on $L^{k+1}(\Omega)$ from $C_{0}$ where $L^{k+1}(\Omega)$ are uniformly bounded from the continuity across $C_{0}$.

 Let $Z\in \{T, \hat{X}\}, \psi\in \{\wb, w, \psi_{2}\}, f\in \{\hat{T}^{1}, \hat{T}^{2}, \kappa\}$. We define the set $\mathbf{W}_{n,m}, n\geq 1, m\geq 1$ to be the set
 \[
 \mathbf{W}_{n,m}=\{Z^{\alpha}L^{k}(\psi), |\alpha|+k\leq n, k\leq m\}\cup \{Z^{\alpha}(f), |\alpha|\leq n\}.
 \]
 
 $\mathbf{V}_{n,m}, n\geq 1, m\geq 1$ to be the set
 \[
 \mathbf{V}_{n,m}=\mathbf{W}_{n,m}\cup \{Z^{\alpha}L^{k-1}(\Omega), k+|\alpha|=n, k\leq m\}.
 \]
 
We will use $\mathbf{P}_{n,m}$ to denote a general polynomial form the ring $\mathbb{R}[\mathbf{V}_{n,m}, c^{-1}, \rho, \rho^{-1}]$. We will use $\mathbf{P}^{w}_{n,m}$ to denote a general polynomial form the ring $\mathbb{R}[\mathbf{W}_{n,m}, c^{-1}, \rho, \rho^{-1}]$. Note that $\mathbf{P}_{n, m}^{w}$ can be written as $\mathbf{P}_{n,m}$, $\mathbf{P}_{n,m}$ can't be written as $\mathbf{P}_{n,m}^{w}$ in general.

The following examples help to elucidate the definition of the symbol $\mathbf{P}^{w}_{n,m}$ and $\mathbf{P}_{n,m}$.
\begin{itemize}
\item $\theta=-\hat{X}^{1}\hat{X}(\hat{X}^{2})+\hat{X}^{2}\hat{X}(\hat{X}^{1})\implies \theta=\mathbf{P}_{1,1}^{w}$.
\item $\chi=\hat{X}^{i}\hat{X}(v^{i})-c\hat{X}^{2}\hat{X}(\hat{X}^{1})+c\hat{X}^{1}\hat{X}(\hat{X}^{2})\implies \chi=\mathbf{P}_{1,1}^{w}$.
\item $\zeta=\kappa\Big(\frac{1}{2}\big(\hat{T}^{j}\hat{X}(v^{j})+\hat{X}^{i}\hat{T}(v^{i})\big)-\hat{X}(c)\Big)\implies \zeta=\mathbf{P}_{1,1}^{w}$.
\item $\eta=\frac{1}{2}\kappa\big(\hat{T}^{j}\hat{X}(v^{j})+\hat{X}^{i}\hat{T}(v^{i})\big)+c\hat{X}(\kappa)\implies \zeta=\mathbf{P}_{1,1}^{w}$.
\item $[L, \hat{X}]=-\chi\hat{X}, [L, T]=-(\zeta+\eta)\hat{X}\implies [L, Z]=\mathbf{P}_{1,1}^{w}Z, Z\in \{T, \hat{X}\}$.
\item $L(\hat{T}^{k})=\big(\hat{T}^{j}\hat{X}(\psi_{j})+\hat{X}(c)\big)\hat{X}^{k}\implies L(\hat{T}^{k})=\mathbf{P}_{1,1}^{w}$.
\item $L(\kappa)=-T(c)-\hat{T}^{i}T(\psi_{i})\implies L(\kappa)=\mathbf{P}_{1,1}^{w}$.
\item On the one hand, $\Omega=\mathbf{P}$ by definition. On the other hand, $\Omega=\mathbf{P}_{1,1}^{w}$ in view of \eqref{eq: reformulation of vorticity in the first null frame}.
\item On the one hand, $\mu\partial_{i}\Omega=\mathbf{P}_{1,1}$ by definition. On the other hand, $\mu\partial_{i}\Omega=\mathbf{P}_{2, 2}^{w}$ n view of \eqref{eq: reformulation of vorticity in the first null frame}.
\end{itemize}

 Roughly speaking, for both $\mathbf{P}_{n, m}$ and $\mathbf{P}_{n,m}^{w}$, $m$ counts the number of $L$ derivates. To be more precise, we have
 \begin{lemma}
 For $Z\in \{T, \hat{X}\}$, we have
 \begin{equation}\label{Formulas for P_n,m}
 \begin{cases}
 LZ^{\alpha}L^{k}(\psi)=Z^{\alpha}L^{k+1}(\psi)+\mathbf{P}^{w}_{|\alpha|+k, k},\\
 LZ^{\alpha}L^{k-1}(\Omega)=Z^{\alpha}L^{k}(\Omega)+\mathbf{P}_{|\alpha|+k-1, k},\\
 LZ^{\alpha}(f)=\mathbf{P}^{w}_{|\alpha|+1, 1}.
 \end{cases}
 \end{equation}
 which implies that
 \[
 L(\mathbf{P}^{w}_{n,m})=\mathbf{P}^{w}_{n+1, m+1}, Z(\mathbf{P}_{n,m}^{w})=\mathbf{P}^{w}_{n+1, m}.
 \]
 \end{lemma}

 \begin{proof}
 From the definition of $\mathbf{V}_{n,m}$, it is clear that $Z(\mathbf{P}_{n,m})=\mathbf{P}_{n+1, m}$. Then we can check terms in \eqref{Formulas for P_n,m} one by one.
 \begin{itemize}
 \item The first term $LZ^{\alpha}L^{k}(\psi)$ can be treated as follows,
 \[
 \begin{split}
 LZ^{\alpha}L^{k}(\psi)&=Z^{\alpha}L^{k+1}(\psi)+\sum_{|\alpha_{1}|+|\alpha_{2}|\leqslant |\alpha|-1}Z^{\alpha_{1}}[L, Z]Z^{\alpha_{2}}L^{k}(\psi)\\
 &=Z^{\alpha}L^{k+1}(\psi)+\sum_{|\alpha_{1}|+|\alpha_{2}|\leqslant |\alpha|-1}Z^{\alpha_{1}}\mathbf{P}^{w}_{1,1}Z^{\alpha_{2}+1}L^{k}(\psi)\\
 &=Z^{\alpha}L^{k+1}(\psi)+\mathbf{P}^{w}_{|\alpha|+k, k}.
 \end{split}
 \]
 \item The second term can be treated exactly the same way as the first term.
 \item The third term $LZ^{\alpha}(f), f\in \{\hat{T}^{1}, \hat{T}^{2}, \kappa\}$ can be treated as follows,
 \[
 \begin{split}
 LZ^{\alpha}(f)&=Z^{\alpha}L(f)+\sum_{|\alpha_{1}|+|\alpha_{2}|\leqslant |\alpha|-1}Z^{\alpha_{1}}[L,Z]Z^{\alpha_{2}}(f)\\
 &=Z^{\alpha}\mathbf{P}_{1,1}+\sum_{|\alpha_{1}|+|\alpha_{2}|\leqslant |\alpha|-1}Z^{\alpha_{1}}\mathbf{P}^{w}_{1,1}Z^{\alpha_{2}+1}(f)\\
 &=\mathbf{P}^{w}_{|\alpha|+1, 1}.
 \end{split}
 \]
 \end{itemize}
 \end{proof}

 Now we can derive the ODE systems for $L^{k+1}(\psi)$ and $L^{k}(\Omega)$.
 \begin{lemma}
 Formally, for $k\geq 1$, we have the following system:
 \begin{equation}\label{ODE system for L^k+1(psi) and L^k(Omega)-basic}
 \begin{cases}
 c^{-1}\kappa B(L^{k}(\Omega))=-k\cdot L(c^{-1}\kappa)L^{k}(\Omega)+\mathbf{P}_{k,k},\\
 \Lb(L^{k+1}(\psi))=\mathbf{P}_{1, 1}L^{k+1}(\psi)+\mathbf{P}_{1,1}\sum_{|\beta|\leq 1}Z^{\beta}L^{k}(\Omega)+\kappa c \hat{X}^{2}L^{k}(\psi)+\mathbf{P}_{k+1, k}.
 \end{cases}
 \end{equation}
 \end{lemma}
 
 \begin{remark}
On the one hand, the term $\kappa c\hat{X}^{2}L^{k}(\psi)$ has total derivate $k+2$, which will lead to loss of derivates in the procedure of retrieving the bounds of $L^{k+1}(\psi)$ and $L^{k}(\Omega)$. On the other hand, the term $\kappa c\hat{X}^{2}L^{k}(\psi)$ vanishes on initial singularity $\mathbf{S}_{*}$ will won't lead to  loss of derivates in the procedure of deriving the asymptotic behavior of solutions.
 \end{remark}
 
 \begin{proof}
 We treat the $L^{k}(\Omega)$ part and $L^{k+1}(\psi)$ separately by induction argument.
 \begin{itemize}
 \item To deal with $L^{k}(\Omega)$, we commute $L$ the equation $c^{-1}\kappa B\Omega=0$ to obtain
 \[
 c^{-1}\kappa B(L(\Omega))=-L(c^{-1}\kappa)L\Omega+\underbrace{(\zeta+\eta)\hat{X}(\Omega)}_{\mathbf{P}_{1,1}}
 \]
 which confirms the case $k=1$. From the case $k$ to $k+1$, we commute $L$ with first equation of \eqref{ODE system for L^k+1(psi) and L^k(Omega)-basic} to obtain
 \[
 c^{-1}\kappa B(L^{k+1}\Omega)=-(k+1)\cdot L(c^{-1}\kappa)L^{k+1}(\Omega)+\underbrace{(\zeta+\eta)\hat{X}L^{k}(\Omega)-k\cdot L^{2}(c^{-1}\kappa)L^{k}(\Omega)+L(\mathbf{P}_{k,k})}_{\mathbf{P}_{k+1, k+1}}.
 \]
 which closes the induction argument.
 \item To deal with $L^{k+1}(\psi)$, we write the acoustical wave equation $\Box_{g}(\psi)=\varrho$ in the null frame $\{L, \Lb, \hat{X}\}$ as follows:
 \begin{equation}\label{Transport equation for L(psi)}
 \Lb(L(\psi))=-(L(c^{-1}\kappa)-c^{-1}\kappa\hat{X}^{j}\hat{X}(\psi_{j}))L(\psi)+\mu \hat{X}^{2}(\psi)-\chi T(\psi)+2\eta\hat{X}(\psi)-\mu \varrho.
 \end{equation}
 where the term $\mu \varrho$ is linear combination of terms
 \[
 \{\kappa, \rho, c^{-1}, \hat{T}^{1}, \hat{T}^{2}\}\times \{L(\psi)^{2}, L(\psi)Z(\psi), Z(\psi)^{2}, Z(\psi)\Omega, Z(\Omega), \Omega^{2}\}.
 \]
 We commute $L$ with \eqref{Transport equation for L(psi)} and check all the possible terms one by one
 \begin{itemize}
 \item $[L, \Lb](L\psi)=-2(\zeta+\eta)\hat{X}L\psi+L(c^{-1}\kappa)L^{2}(\psi)=\mathbf{P}_{1, 1}L^{2}(\psi)+\mathbf{P}_{2, 1}$,
 \item $(L(c^{-1}\kappa)-c^{-1}\kappa\hat{X}^{j}\hat{X}(\psi_{j}))L^{2}(\psi)=\mathbf{P}_{1, 1}L^{2}(\psi)$,
 \item $L(L(c^{-1}\kappa)-c^{-1}\kappa\hat{X}^{j}\hat{X}(\psi_{j}))L(\psi)=-\kappa c^{-2}L(\psi)L^{2}(c)+\mathbf{P}_{2,1}=\mathbf{P}_{1, 1}L^{2}(\psi)+\mathbf{P}_{2, 1}$,
 \item $L(\mu\hat{X}^{2}(\psi))=c\kappa \hat{X}^{2}L(\psi)+\mathbf{P}_{2,1}$.
 \item $L(-\chi T(\psi)+2\eta\hat{X}(\psi))=\mathbf{P}_{2,1}$.
 \item $L(L(\psi)^{2})=2L(\psi)L^{2}(\psi)=\mathbf{P}_{1, 1}L^{2}(\psi)$,
 \item $L(L(\psi)Z(\psi))=Z(\psi)L^{2}(\psi)+LZ(\psi)L(\psi)=\mathbf{P}_{1, 1}L^{2}(\psi)+\mathbf{P}_{2,1}$,
 \item $L(Z(\psi)^{2})=\mathbf{P}_{2,1}$
 \item $L(Z(\psi)\Omega)=LZ(\psi)\Omega+Z(\psi)L\Omega=\mathbf{P}_{1,1}L\Omega+\mathbf{P}_{2,1}$,
 \item $L(Z(\Omega))=ZL(\Omega)+\mathbf{P}_{2,1}$,
 \item $L(\Omega^{2})=2\Omega L(\Omega)=\mathbf{P}_{1, 1}L(\Omega)$.
 \end{itemize}
 Therefore, we have
 \[
 \Lb(L^{2}(\psi))=\mathbf{P}_{1,1}L^{2}(\psi)+\mathbf{P}_{1,1}\sum_{|\beta|\leq 1}Z^{\beta}L\Omega+\kappa c \hat{X}^{2}L(\psi)+\mathbf{P}_{2,1}.
 \]
 which confirms the case $k=1$. To go from case $k$ to $k=1$, we commute $L$ with the second equation of  \eqref{ODE system for L^k+1(psi) and L^k(Omega)-basic} to obtain
 \[
 \begin{split}
 \Lb(L^{k+2}(\psi))&=\underbrace{-L(c^{-1}\kappa)}_{\mathbf{P}_{1,1}}L^{k+2}(\psi)+\underbrace{2(\zeta+\eta)\hat{X}L^{k+1}(\psi)+L(\mathbf{P}_{1, 1})L^{k+1}(\psi)++L(\mathbf{P}_{1,1})\sum_{|\beta|\leq 1}Z^{\beta}L^{k}(\Omega)}_{\mathbf{P}_{k+2,k+1}}\\
 &+\mathbf{P}_{1, 1}L^{k+2}(\psi)+\mathbf{P}_{1,1}\sum_{|\beta|\leq 1}Z^{\beta}L^{k+1}(\Omega)+\underbrace{\mathbf{P}_{1,1}\sum_{|\beta|\leq 1}[L, Z^{\beta}]L^{k}(\Omega)}_{\mathbf{P}_{k+2, k+1}}+\underbrace{L(\kappa c) \hat{X}^{2}L^{k}(\psi)}_{\mathbf{P}_{k+2, k+1}}\\
 &+\kappa c \hat{X}^{2}L^{k+1}(\psi)+\underbrace{\kappa c[L, \hat{X}^{2}]L^{k}(\psi)+L(\mathbf{P}_{k+1, k})}_{\mathbf{P}_{k+2,k+1}}\\
 &=\mathbf{P}_{1,1}L^{k+2}(\psi)+\mathbf{P}_{1,1}\sum_{|\beta|\leq 1}Z^{\beta}L^{k+1}(\Omega)+\kappa c\hat{X}^{2}L^{k+1}(\psi)+\mathbf{P}_{k+2,k+1}.
 \end{split}
 \]
 which closes the induction argument.
 \end{itemize}
 \end{proof}

\begin{remark}\label{remark: ODE without loss of derivates}
In the deriving of transport equation of $L^{2}(\psi)$, if we use \eqref{eq: reformulation of vorticity in the first null frame} for vorticity appearing in the source terms of wave equations $\Box_{g}\psi=\varrho$, we will obtain
\[
\Lb(L^{2}\psi)=\mathbf{P}^{w}_{1,1}L^{2}(\psi)+\mathbf{P}^{w}_{2,1}+\kappa c\hat{X}^{2}L(\psi)+\kappa \mathbf{P}_{3,3}^{w}.
\]
Therefore, we can obtain
\[
\Lb(L^{k+1}(\psi))=\mathbf{P}^{w}_{1,1}L^{k+1}(\psi)+\mathbf{P}_{k+1, k}^{w}+\kappa c\hat{X}^{2}L^{k}(\psi)+\kappa \mathbf{P}_{k+2, k+2}^{w}.
\]
which will be used in the deriving higher jets of solutions on the initial singularity $\mathbf{S}_{*}$.
\end{remark}

 \subsubsection{An ODE systems for $L^{k+1}Z^{\alpha-1}(\psi)$ and $L^{k}Z^{\alpha}(\Omega)$.} Set the set $\tilde{\mathbf{V}}_{n, m}$ to be the set
 \[
 \tilde{\mathbf{V}}_{n,m}:=\mathbf{V}_{n, m}\cup \{\frac{Z^{\alpha}(c^{-1}\kappa)}{c^{-1}\kappa}, |\alpha|\leq n\}
 \]
 We will use $\mathbf{Q}_{n,m}$ to denote a general polynomial form the ring $\mathbb{R}[\tilde{\mathbf{V}}_{n,m}, c^{-1}, \rho, \rho^{-1}]$. Note that $Z(\mathbf{Q}_{n,m})=\mathbf{Q}_{n+1, m}$.  The following examples help to elucidate the definition of the symbol $\mathbf{Q}_{n,m}$.
 
 \begin{itemize}
 \item The commutators $[\Lb, \hat{X}]$ and $[\Lb, T]$
 \[
 \begin{cases}
 [\Lb, \hat{X}]=-\kappa(c^{-1}\chi+\theta)\hat{X}-\frac{\hat{X}(c^{-1}\kappa)}{c^{-1}\kappa}\Lb+2\frac{\hat{X}(c^{-1}\kappa)}{c^{-1}\kappa}T,\\
 [\Lb, T]=-c^{-1}\kappa(\zeta+\eta)\hat{X}-\frac{T(c^{-1}\kappa)}{c^{-1}\kappa}\Lb+2\frac{T(c^{-1}\kappa)}{c^{-1}\kappa}T,
 \end{cases}\implies [\Lb, Z]=\mathbf{Q}_{1, 1}\Lb+\mathbf{Q}_{1,1}Z.
 \]
 \item The commutators $[c^{-1}\kappa B, \hat{X}]$ and $[c^{-1}\kappa B, T]$
 \[
 \begin{cases}
 [c^{-1}\kappa B, \hat{X}]=-\kappa(c^{-1}\chi+\theta)\hat{X}-\frac{\hat{X}(c^{-1}\kappa)}{c^{-1}\kappa}c^{-1}\kappa B+\frac{\hat{X}(c^{-1}\kappa)}{c^{-1}\kappa}T,\\
 [c^{-1}\kappa B, T]=-c^{-1}\kappa(\zeta+\eta)\hat{X}-\frac{T(c^{-1}\kappa)}{c^{-1}\kappa}c^{-1}\kappa B+\frac{T(c^{-1}\kappa)}{c^{-1}\kappa}T,
 \end{cases}\implies [c^{-1}\kappa B, Z]=\mathbf{Q}_{1, 1}c^{-1}\kappa B+\mathbf{Q}_{1,1}Z.
 \]
 \end{itemize}
 
 Now we can derive the ODE systems for $L^{k+1}Z^{\alpha-1}(\psi)$ and $L^{k}Z^{\alpha}(\Omega)$.
 \begin{lemma}
 Formally, for $k\geq 1$, we have the following system:
 \begin{equation}\label{ODE system for L^k+1(psi) and L^k(Omega)-commuted}
 \begin{cases}
 c^{-1}\kappa B(Z^{\alpha}L^{k}(\Omega))=\mathbf{Q}_{|\alpha|+1, 1}Z^{\leq |\alpha|}L^{k}(\Omega)+\mathbf{Q}_{k+|\alpha|,k}, |\alpha|\geq 0\\
 \Lb(Z^{\alpha-1}L^{k+1}(\psi))=\mathbf{Q}_{|\alpha|, 1}Z^{\leq |\alpha|-1}L^{k+1}(\psi)+\mathbf{Q}_{|\alpha|,1}Z^{\leq |\alpha|}L^{k}(\Omega)+\kappa \mathbf{Q}_{|\alpha|, 1} Z^{\alpha-1}\hat{X}^{2}L^{k}(\psi)+\mathbf{Q}_{k+|\alpha|, k}, |\alpha|\geq 1.
 \end{cases}
 \end{equation}
 Note that the term $\kappa \mathbf{Q}_{|\alpha|, 1} Z^{\alpha-1}\hat{X}^{2}L^{k}(\psi)=\mathbf{Q}_{k+|\alpha|+1, k}$. We also have a rough version:
 \[
  \begin{cases}
 c^{-1}\kappa B(Z^{\alpha}L^{k}(\Omega))=\mathbf{Q}_{|\alpha|+1, 1}Z^{\leq |\alpha|}L^{k}(\Omega)+\mathbf{Q}_{k+|\alpha|,k}, |\alpha|\geq 0\\
 \Lb(Z^{\alpha-1}L^{k+1}(\psi))=\mathbf{Q}_{|\alpha|, 1}Z^{\leq |\alpha|-1}L^{k+1}(\psi)+\mathbf{Q}_{|\alpha|,1}Z^{\leq |\alpha|}L^{k}(\Omega)+\mathbf{Q}_{k+|\alpha|+1, k}, |\alpha|\geq 1.
 \end{cases}
 \]
 \end{lemma}

\begin{proof}
We treat the $Z^{\alpha}L^{k}(\Omega)$ part and $Z^{\alpha-1}L^{k+1}(\psi)$ separately by induction on $|\alpha|$. 
\begin{itemize}
\item The base case $|\alpha|=0$ of the first equation in \eqref{ODE system for L^k+1(psi) and L^k(Omega)-commuted} follows from the equation for $L^{k}(\Omega)$ of \eqref{ODE system for L^k+1(psi) and L^k(Omega)-commuted}. From the case $|\alpha|$ to case $|\alpha|+1$, we commute $Z$ with the first equation of \eqref{ODE system for L^k+1(psi) and L^k(Omega)-commuted} to obtain
\[
\begin{split}
c^{-1}\kappa B(Z^{\alpha+1}L^{k}(\Omega))&=\underbrace{\mathbf{Q}_{1,1}c^{-1}\kappa B(Z^{\alpha}L^{k}(\Omega))}_{\mathbf{Q}_{|\alpha|+2,1}Z^{\leq |\alpha|+1}L^{k}(\Omega)+\mathbf{Q}_{k+|\alpha|+1, k}}\\
&+\underbrace{\mathbf{Q}_{1,1}Z^{\alpha+1}L^{k}(\Omega)+\mathbf{Q}_{|\alpha|+1, 1}Z^{\leq |\alpha|+1}L^{k}(\Omega)+Z(\mathbf{Q}_{|\alpha|+1, 1})Z^{\leq |\alpha|}L^{k}(\Omega)}_{\mathbf{Q}_{|\alpha|+2,1}Z^{\leq |\alpha|+1}L^{k}(\Omega)}\\
&+\underbrace{Z(\mathbf{Q}_{k+|\alpha|,k})}_{\mathbf{Q}_{k+|\alpha|+1, k}}
\end{split}
\]
which closes the induction argument.
\item The base case $|\alpha|=1$ of second equation in \eqref{ODE system for L^k+1(psi) and L^k(Omega)-commuted} follows from the equation for $L^{k+1}(\psi)$ of \eqref{ODE system for L^k+1(psi) and L^k(Omega)-commuted}. From the case $|\alpha|$ to case $|\alpha|+1$, we commute $Z$ with the second equation of \eqref{ODE system for L^k+1(psi) and L^k(Omega)-commuted} to obtain
\[
\begin{split}
\Lb(Z^{\alpha}L^{k+1}(\psi))&=\underbrace{\mathbf{Q}_{1,1}\Lb(Z^{\alpha-1}L^{k+1}\psi)}_{\mathbf{Q}_{|\alpha|+1, 1}Z^{\leq |\alpha|}L^{k+1}(\psi)+\mathbf{Q}_{|\alpha|+1,1}Z^{\leq |\alpha|+1}L^{k}(\Omega)+\kappa \mathbf{Q}_{|\alpha|+1, 1} Z^{\alpha}\hat{X}^{2}L^{k}(\psi)+\mathbf{Q}_{k+|\alpha|+1, k}}\\
&+\underbrace{\mathbf{Q}_{1,1}Z^{|\alpha|}L^{k+1}(\psi)+Z(\mathbf{Q}_{|\alpha|, 1})Z^{\leq |\alpha|-1}L^{k+1}(\psi)+\mathbf{Q}_{|\alpha|, 1}Z^{\leq |\alpha|}L^{k+1}(\psi)}_{\mathbf{Q}_{|\alpha|+1, 1}Z^{\leq |\alpha|}L^{k+1}(\psi)}\\
&+\underbrace{Z(\mathbf{Q}_{|\alpha|,1})Z^{\leq |\alpha|}L^{k}(\Omega)+\mathbf{Q}_{|\alpha|,1}Z^{\leq |\alpha|+1}L^{k}(\Omega)}_{\mathbf{Q}_{|\alpha|+1,1}Z^{\leq |\alpha|+1}L^{k}(\Omega)}\\
&+\kappa \mathbf{Q}_{\alpha, 1} Z^{\alpha}\hat{X}^{2}L^{k}(\psi)+\underbrace{Z(\kappa \mathbf{Q}_{|\alpha|+1}) Z^{\alpha-1}\hat{X}^{2}L^{k}(\psi)+Z(\mathbf{Q}_{k+|\alpha|, k})}_{\mathbf{Q}_{k+1+|\alpha|,k}}
\end{split}
\]
which closes the induction argument.
\end{itemize}
\end{proof}

\subsubsection{The uniform bounds on $L$-derivatives in effective region $\mathcal{D}^{+}(\delta)$}
In acoustical coordinate $(t, u, \theta)$, the region $\mathcal{D}(\delta)$ is given by $(t, u, \theta)\in [\delta, 1]\times [0, u^{*}]\times \mathbb{T}^{1}$, the bottom boundary $\Sigma_{\delta}$ is given by $(t, u, \theta)\in {\delta}\times [0, u^{*}]\times \mathbb{T}^{1}$. the right boundary $C_{0}$ is given by $(t, u, \theta)\in [\delta, 1]\times {0}\times \mathbb{T}^{1}$. The intersection of bottom boundary and right boundary $S_{\delta, 0}$ is given by $(t, u, \theta)\in {\delta}\times {0}\times \mathbb{T}^{1}$. By $c^{-1}\kappa B=T+c^{-1}\kappa L$, we have $c^{-1}\kappa B(u)=1$ and $c^{-1}\kappa B(t)=c^{-1}\kappa$. Therefore, the region $\mathcal{D}(\delta)$ can be divided into two parts:
\begin{itemize}
\item Effective region:  $\mathcal{D}^{+}(\delta)$ to be union of integral curves of $c^{-1}\kappa B$ emanated from $C_{0}$.
\item Irrelevant region: $\mathcal{D}^{-}(\delta)$ to be union of integral curves of $c^{-1}\kappa B$ emanated from $\Sigma_{\delta}$.
\end{itemize}
Moreover, the inner boundary $B_{\delta}$ of both $\mathcal{D}^{+}(\delta)$ and $\mathcal{D}^{-}(\delta)$ is given by the union of integral curves of $c^{-1}\kappa B$ emanated from $S_{\delta, 0}$

By $\Lb=2T+c^{-1}\kappa L$, we have $\Lb(u)=1<2=c^{-1}\kappa B(u)$ and $\Lb(t)=c^{-1}\kappa=c^{-1}\kappa B(t)$. Therefore, the effective region $\mathcal{D}^{+}(\delta)$ must be a subset of union of integral curves of $\Lb$ emanated from $C_{0}$, which make it possible to apply transport system \eqref{ODE system for L^k+1(psi) and L^k(Omega)-commuted} to retrieve uniform bounds on $L$-derivates in $\mathcal{D}^{+}(\delta)$ from $C_{0}$.

\begin{center}
\begin{tikzpicture}
\draw[thick] (0, 0)--(1, 0)--(3, 4)--(0,4)--(0,0);
\draw[thick, dotted, ->] (1, 0)--(0.3, 0.7);
\draw[thick, dotted] (0.3, 0.7)--(0, 1);
\draw[thick, dotted, ->] (1.5, 1)--(0.75, 1.75);
\draw[thick, dotted] (0.75, 1.75)--(0, 2.5);
\draw[thick, dashed, ->] (1.1, 0.2)--(0.8, 0.35);
\draw[thick, dashed] (0.8, 0.35)--(0.5, 0.5);
\draw[thick, dashed, ->] (1.2, 0.4)--(0.6, 0.7);
\draw[thick, dashed] (0.6, 0.7)--(0, 1);
\draw[thick, dashed, ->] (1.5, 1)--(0.75, 1.375);
\draw[thick, dashed] (0.75, 1.375)--(0, 1.75);
\draw (0, 1)--(1.5, 1);
\node[left] at (0,0) {$\Sigma_{\delta}$};
\node[left] at (0,1) {$\Sigma_{C\dot\delta}$};
\node[right] at (1,0) {$S_{\delta, 0}$};
\end{tikzpicture}
\end{center}
We remark that the dotted directed line represents the integral curve of $c^{-1}\kappa B$, the dashed directed line represents the integral curve of $\Lb$.

We give an estimate of the size of the irrelevant domain $\mathcal{D}^-(\delta)$. By $c^{-1}\kappa B(u)=1$ we can use $u$ to parameterize such an integral curve. We need to bound the maximal possible time. Since $c^{-1}\kappa B=c^{-1}\kappa\frac{\partial}{\partial t}+\frac{\partial}{\partial u}-\Xi\frac{\partial}{\partial \theta}$, we can parametrize the integral curve by $(t(u), u, \theta(u)), u\in [a, u^{*}]$ where $a\geq 0$ is chosen such that $(t(a), a, \theta(a))\in \Sigma_{\delta}$ represents the initial point of the integral curve. We divide the interval $[a, u^{*}]$ in to 
\[
[a, u^{*}]=\cup_{k=0}^{n} [u_{k}, u_{k+1}],
\]
where $u_{0}=a$ and $u_{k}$ is chosen such that $t(u^{k})=2^{k}\delta$. On the one hand, on each interval $[u_{k}, u_{k+1}]$, we have
\[
2^{k+1}\delta-2^{k}\delta=\int_{u_{k}}^{u_{k+1}} c^{-1}\kappa c^{-1}\kappa B(t)du=\int_{u_{k}}^{u_{k+1}} c^{-1}\kappa du\approx (u_{k+1}-u_{k})2^{k}\delta
\]
which implies that $u_{k+1}-u_{k}$ has a universal positive lower bound. On the other hand, we have
\[
\sum_{k=0}^{n}(u_{k+1}-u_{k})= u^{*},
\] 
which implies that $n$  has a universal upper bound. Hence, there exists a universal constant $C$ so that the maximal time on $\mathcal{D}^{-}(\delta)$ is at most $t_{\delta}=C\cdot \delta$.  Therefore, the limit of  the effective domains $\mathcal{D}^+(\delta)$ also gives $\mathcal{W}$.

\begin{proposition}\label{prop: estimates on Lk}
One effective domain $\mathcal{D}^{+}(\delta)$, the terms in $\mathbf{V}_{N_{3}+1-k, k}, 2\leq 2k\leq N^{*}+1$ are bounded by $C\cdot \varepsilon$ where $C$ is a universal constant. In particular, we have for $\psi\in \{\underline{w}, w, \psi_{2}\}$ and $Z\in \{\hat{X}, T\}$,
\[
\begin{cases}
\|L^{k}Z^{\alpha}\psi\|_{L^{\infty}(\mathcal{D}^{+}(\delta))}\lesssim \varepsilon, \\
\|L^{k}Z^{\alpha}\Omega\|_{L^{\infty}(\mathcal{D}^{+}(\delta))}\lesssim \varepsilon 
\end{cases}, 2k+|\alpha|\leq N_{3}.
\]
\end{proposition}

\begin{proof}
We prove by induction. The $L^{\infty}$ bounds of the solution form proposition \eqref{eq: L estimates in W} implies that the base case $k=1$ and $|\frac{Z^{\alpha}(c^{-1}\kappa)}{c^{-1}\kappa}|\lesssim 1, Z\in \{T, \hat{X}\}, |\alpha|\leq N_{3}$.

 From the case $k$ to $k+1$, we use the rough version transport system \eqref{ODE system for L^k+1(psi) and L^k(Omega)-commuted} for $Z^{\alpha}L^{k}(\Omega)$ and $Z^{\alpha-1}L^{k+1}(\psi)$ with $|\alpha|\leq N_{3}-2k$.\footnote{To estimate terms in $\mathbf{V}_{n-k, k+1}$, it suffices to estimate $Z^{\alpha}L^{k}(\Omega)$ and $Z^{\alpha-1}L^{k+1}(\psi)$ with $|\alpha|-1\leq n-k-(k+1)\iff |\alpha|\leq n-2k$.}
\[
 \begin{cases}
 c^{-1}\kappa B(Z^{\alpha}L^{k}(\Omega))=\mathbf{Q}_{N_{3}, 1}Z^{\leq |\alpha|}L^{k}(\Omega)+\mathbf{Q}_{N_{3}+1-k,k}, |\alpha|\geq 0,\\
 \Lb(Z^{\alpha-1}L^{k+1}(\psi))=\mathbf{Q}_{N_{3}, 1}Z^{\leq |\alpha|-1}L^{k+1}(\psi)+\mathbf{Q}_{N_{3},1}Z^{\leq |\alpha|}L^{k}(\Omega)+\mathbf{Q}_{N_{3}+1-k, k}, |\alpha|\geq 1.
 \end{cases}
\]
By induction hypothesis, terms $\mathbf{Q}_{N_{3}, 1}$ appearing in above system must be of size $O(1)$, terms $\mathbf{Q}_{N_{3}+1-k, k}$ appearing in above system must be of size $O(\varepsilon)$.\footnote{Comparing with the case $\varepsilon=0$, there are no monomials of $\approx 1$ size.} Therefore, the Gronwall’s inequality gives the desired estimates for all terms in $\mathbf{V}_{N_{3}-k, k+1}$ which closes the induction argument.
\end{proof}

Since the limit of  the effective domains $\mathcal{D}^+(\delta_l)$ also gives $\mathcal{W}$. By passing to the limit, Proposition \ref{prop: estimates on Lk}
gives the following proposition:
\begin{proposition}\label{prop: estimates on Lk final}
For the solution $U$ on $\mathcal{W}$, for any positive integer $k\geqslant 1$ and all multi-indices $\alpha$ with $k+|\alpha|\leqslant \frac{N_{3}}{2}$, for all $\psi \in \{\wb,w,\psi_2\}$ and for all $Z\in \{\Xh, T\}$, we have
\begin{equation}\label{eq: Lk estimates  on effective domain}
\|L^k Z^\alpha (\psi) \|_{L^\infty(\mathcal{W})}+\|L^k Z^\alpha (\Omega) \|_{L^\infty(\mathcal{W})} \lesssim \varepsilon.
\end{equation}
\end{proposition}

\subsection{Additional vanishing of transversal derivatives of $\Omega$}

In view of the equivalence $B\Omega = 0 \iff T(\Omega) = c^{-1}\kappa L(\Omega)$ and the bound \eqref{eq: Lk estimates on effective domain}, we deduce that the transversal derivatives of $\Omega$ enjoy an \textbf{additional vanishing property} in the rarefaction wave region. More precisely, we have:

\begin{proposition}\label{prop:additional_vanishing_Omega}
For $|\alpha|\leq \frac{1}{2}N_{3}-1$, $Z\in \{T, \hat{X}\}$ and $\Zr\in \{\Tr, \Xr\}$, it holds that
\begin{equation}\label{eq: addition vanishing of transversal derivate of Omega}
\begin{cases}
\|TZ^{\alpha}(\Omega)\|_{L^{\infty}(\Sigma_{t}^{u^{*}})}\lesssim \varepsilon t,\\
\|\Tr\Zr^{\alpha}(\Omega)\|_{L^{\infty}(\Sigma_{t}^{u^{*}})}\lesssim \varepsilon t.
\end{cases}
\end{equation}
\end{proposition}

\begin{remark}\label{rem:additional_vanishing_not_ansatz}
It is an interesting fact that, although \eqref{eq: addition vanishing of transversal derivate of Omega} can be obtained as a corollary of the energy estimates, it seems difficult to establish the \textbf{additional vanishing property} for the transversal derivatives of $\Omega$ in the rarefaction wave region during the construction of the data on $\Sigma_{\delta}^{u^{*}}$. Therefore, this additional vanishing property cannot be part of the ansatz. Instead, we can only make the weaker ansatz:
\[
\|Z^{\alpha}(\Omega)\|_{L^{\infty}(\Sigma_{t}^{u^{*}})}\lesssim \varepsilon.
\]
Nevertheless, the structure of the wave-transport system and the refined Gronwall's inequality enable us to close the energy estimates even with this weaker assumption.
\end{remark}

\subsection{The resolution of initial singularity $\mathbf{S}_{*}$ with acoustical coordinate $(t, u, \vartheta)$}

\subsubsection{The construction of coordinate system for the initial singularity $\mathbf{S}_{*}$}
In view of the domains defined in Section \ref{section:region of convergence}, there exists a $\overline{u}$ so that $0\leqslant u^*-\overline{u}\leqslant \varepsilon_0$ and for each $u \in [0,\overline{u}]$, the characteristic hypersurface $C_u$ is complete in the following sense: for all $0<t\leq t^{*}$, $C_u\cap \Sigma_t$ is a complete circle. We then define 
\[\overline{\mathcal{W}}=\bigcup_{u\in [0,\overline{u}]}C_u.\]

In view of \eqref{eq: L estimates in W} and \eqref{eq: Lk estimates  on effective domain}, the solution $(\wb,w, \psi_2)$ are now $C^{\frac{N_{3}}{2}}$ functions on $\mathcal{W}$ with respect to acoustical coordinate $(t, u, \vartheta)\in (0, 1]\times [0, \ub]\times \mathbb{R}/2\pi\mathbb{Z}$. At the same time, acoustical coordinate coordinate $(t, u, \vartheta)$ are $C^{\frac{N_{3}}{2}}$ functions of Euclidean coordinate $(t, x_{1}, x_{2})$. Therefore, the region $\overline{\mathcal{W}}$ can be parametrized by acoustical coordinate $(t, u, \vartheta)\in (0, 1]\times [0, \ub]\times \mathbb{R}/2\pi\mathbb{Z}$.

Recall that Jacobi matrix of the coordinate transformation $(t,u,\theta)\mapsto (x_{0}, x_{1}, x_{2})$ is given by
\[
\begin{pmatrix}
\frac{\partial x^{0}}{\partial t}&\frac{\partial x^{0}}{\partial u}&\frac{\partial x^{0}}{\partial \theta}\\
\frac{\partial x^{1}}{\partial t}&\frac{\partial x^{1}}{\partial u}&\frac{\partial x^{1}}{\partial \theta}\\
\frac{\partial x^{2}}{\partial t}&\frac{\partial x^{2}}{\partial u}&\frac{\partial x^{2}}{\partial \theta}
\end{pmatrix}=\begin{pmatrix}
1&0&0\\
L^{1}&\kappa \hat{T}^{1}+\Xi\sqrt{\slashed{g}}\hat{X}^{1}&\sqrt{\slashed{g}}\hat{X}^{1}\\
L^{2}&\kappa \hat{T}^{2}+\Xi\sqrt{\slashed{g}}\hat{X}^{2}&\sqrt{\slashed{g}}\hat{X}^{2}
\end{pmatrix}.
\]
For all $(t,u,\vartheta)\in (0,t^*]\times [0,\overline{u}]\times [0,2\pi]$, for all multi-indices $\alpha, \beta, k$ with $|\alpha|+|\beta|+k\leq \frac{N_{3}}{2}$, we have
\[
|L^{k}T^{\alpha}X^{\beta}(\frac{\partial x_{1}}{\partial \theta})|+|L^{k}T^{\alpha}X^{\beta}(\frac{\partial x_{1}}{\partial u}+t)|+|L^{k}T^{\alpha}X^{\beta}(\frac{\partial x_{2}}{\partial \theta}-1)|+|L^{k}T^{\alpha}X^{\beta}(\frac{\partial x_{2}}{\partial u})|\lesssim \varepsilon.
\]
which implies $x_{1}, x_{2}$ (parametrized by $(t, u, \vartheta)\in (0, t^{*}]\times [0, \ub]\times \mathbb{R}/2\pi\mathbb{Z}$) can be extend smoothly up to $t=0$ such that $x_{1}, x_{2}\in C^{N_{4}}\Big([0, t^{*}]\times [0, \overline{u}]\times [0, 2\pi]\Big)$. Moreover, we have
\[
\begin{cases}
x_{2}(0, u, \vartheta)=\vartheta, \\
x_{1}(0, u, \vartheta)=0.
\end{cases}, \forall (u, \vartheta)\in [0, \ub]\times \mathbb{R}/2\pi\mathbb{Z}.
\]
which resolves the initial singularity $\mathbf{S}_{*}$.

According to \eqref{eq: L estimates in W} and \eqref{eq: Lk estimates  on effective domain}, for all $\psi \in \{\wb,w,\psi_2\}$ for all $(t,u,\vartheta)\in (0,t^*]\times [0,\overline{u}]\times [0,2\pi]$, for all multi-indices $\alpha, \beta$ and for all $k\geqslant 1$, if $|\alpha|+|\beta|+k\leqslant \frac{N_{3}}{2}$, we have
\[
\big|\big(L^k T^\alpha X^\beta \psi\big)(t,u,\vartheta)\big|\lesssim \varepsilon, \big|(L^{k}T^{\alpha}{X}^{\beta}\Omega)(t,u,\vartheta)\big|\lesssim \varepsilon.
\]
which implies $(\wb, w, \psi_{2}, \Omega)$ (parametrized by $(t, u, \vartheta)\in (0, t^{*}]\times [0, \ub]\times \mathbb{R}/2\pi\mathbb{Z}$) can be extend smoothly up to $t=0$ such that   $\big(\wb, w,\psi_2, \Omega\big)\in C^{N_{4}}\big([0,t^*]\times [0,\overline{u}]\times [0,2\pi]\big)$.

\subsubsection{The behavior of $(\wb, w, \psi_{2}, \Omega, \hat{T}^{1}, \hat{T}^{2}, \kappa, \slashed{g}, \Xi)$ on the initial singularity $\mathbf{S}_{*}$}

By Proposition \eqref{eq: L estimates in W}, we have:

\begin{proposition}\label{prop:data at the singularity}
On $\mathbf{S}_{*}$, the solution $U$ (expressed in terms of Riemann invariants) and the associated geometric quantities $(\hat{T}^{1}, \hat{T}^{2}, \kappa, \slashed{g}, \Xi)$ assume a very simple form:
\begin{equation}\label{eq:data at the singularity}
\begin{cases}
\wb(u,\vartheta)= \wb_r(0,\vartheta)-\frac{2}{\gamma+1}u,\\
w(u,\vartheta)=w_r(0,\vartheta),\\
\psi_2(u,\vartheta)=-v^2_r(0,\vartheta),\\
\hat{T}^{1}(u, \vartheta)=-1,\\
\hat{T}^{2}(u, \vartheta)=0,\\
\kappa(u, \vartheta)=0,\\
\slashed{g}(u, \vartheta)=1,\\
\Xi(u, \vartheta)=0,
\end{cases}
\quad (u, \vartheta)\in [0, \ub]\times \mathbb{R}/2\pi\mathbb{Z}.
\end{equation}
\end{proposition}

\subsubsection{The behavior of $(\Omega, L(\wb), L(w), L(\psi_{2}), L(\Omega))$ on the initial singularity $\mathbf{S}_{*}$}

\begin{itemize}
\item $\boldsymbol{\Omega}$.

In view of the equation $T(\Omega) = -c^{-1}\kappa L(\Omega)$ and Proposition \ref{prop: estimates on Lk final}, we have $T(\Omega) = 0$, and thus
\[
\Omega(u,\vartheta) = \Omega_{r}(0,\vartheta) \quad \text{for } (u, \vartheta)\in [0, \ub]\times \mathbb{R}/2\pi\mathbb{Z}.
\]

\item $\boldsymbol{L(\underline{w})}$.

Using the Euler equation for $\underline{w}$ extended to $\mathbf{S}_{*}$, we obtain
\[
L(\underline{w})(u, \vartheta) = -\frac{1}{2}\Bigl(c_{r}(0, \vartheta)-\frac{\gamma-1}{\gamma+1}u\Bigr)\partial_{2}v^{2}_{r}(0, \vartheta) \quad \text{for } (u, \vartheta)\in [0, \ub]\times \mathbb{R}/2\pi\mathbb{Z}.
\]

\item $\boldsymbol{L(w)}$.

Using the wave equation for $w$, we derive
\[
\begin{cases}
2\partial_{u}L(w)(u, \vartheta) + c^{-1}L(w)(u,\vartheta) = \frac{1}{2}\partial_{\vartheta}\psi_{2}(u, \vartheta),\\
L(w)(0, \vartheta) = 2(c_{r}\partial_{1}w_{r})(0, \vartheta) + \frac{1}{2}(c_{r}\partial_{2}\psi_{2})(0,\vartheta),
\end{cases}
\]
Solving this ODE for $(u, \vartheta)\in [0, \ub]\times \mathbb{R}/2\pi\mathbb{Z}$ yields
\[
\begin{split}
L(w)(u, \vartheta) &= \frac{\gamma+1}{2(3-\gamma)}\Bigl(c_{r}-\frac{\gamma-1}{\gamma+1}u\Bigr)\partial_{2}\psi_{r2}(0,\vartheta) \\
&\quad - \frac{\gamma+1}{2(3-\gamma)}\Bigl(\frac{c}{c_{r}}\Bigr)^{\frac{1}{2}\frac{\gamma+1}{\gamma-1}}c_{r}\partial_{2}\psi_{r2}(0, \vartheta) \\
&\quad + \Bigl(\frac{c}{c_{r}}\Bigr)^{\frac{1}{2}\frac{\gamma+1}{\gamma-1}}L(w_{r})(0, \vartheta).
\end{split}
\]

\item $\boldsymbol{L(\psi_{2})}$.

Using the wave equation for $\psi_{2}$, we derive
\[
\begin{cases}
2\partial_{u}L(\psi_{2})(u, \vartheta) + \frac{4}{\gamma+1}c^{-1}L(\psi_{2})(u, \vartheta) = \frac{4}{\gamma+1}\rho \Omega,\\
L(\psi_{2})(0, \vartheta) = L(\psi_{2r})(0, \vartheta) = (c_{r}\partial_{1}\psi_{2})(0, \vartheta) + c_{r}\partial_{2}(\underline{w}_{r} + w_{r})(0, \vartheta).
\end{cases}
\]
Solving this ODE for $(u, \vartheta)\in [0, \ub]\times \mathbb{R}/2\pi\mathbb{Z}$ gives
\[
L(\psi_{2})(u, \vartheta) = \frac{2u}{\gamma+1}\Bigl(c_{r}-\frac{\gamma-1}{\gamma+1}u\Bigr)^{\frac{2}{\gamma-1}}\Omega_{r}(0,\vartheta) + \Bigl(\frac{c}{c_{r}}\Bigr)^{\frac{2}{\gamma-1}}L(\psi_{2r})(0,\vartheta).
\]

\item $\boldsymbol{L(\Omega)}$.

Commuting $L$ with the equation $c^{-1}\kappa B(\Omega) = 0$, we obtain $c^{-1}\kappa B(L(\Omega)) = -L(c^{-1}\kappa)L(\Omega) + (\zeta+\eta)\hat{X}\Omega$. Using the transport equation for $L(\Omega)$,
\[
\begin{cases}
\partial_{u}L(\Omega)(u, \vartheta) = -c^{-1}(u, \vartheta)L(\Omega)(u, \vartheta),\\
L(\Omega)(0,\vartheta) = L(\Omega_{r})(0, \vartheta),
\end{cases}
\]
solving this ODE for $(u, \vartheta)\in [0, \ub]\times \mathbb{R}/2\pi\mathbb{Z}$ yields
\[
L(\Omega)(u,\vartheta) = \Bigl(\frac{c(u,\vartheta)}{c_{r}(0,\vartheta)}\Bigr)^{\frac{\gamma+1}{\gamma-1}}L(\Omega_{r})(0,\vartheta).
\]
\end{itemize}

\begin{remark}\label{remark: 1-th L derivates on singularity}
Since $\hat{X} = X = \frac{\partial}{\partial \vartheta}$ and $T = \frac{\partial}{\partial u}$ on $\mathbf{S}_{*}$, terms of the form $Z^{\alpha}L(\psi)$ on $\mathbf{S}_{*}$ can be computed by applying $\frac{\partial}{\partial \vartheta}$ and $\frac{\partial}{\partial u}$ to $L(\psi)(u, \vartheta)$. As a result, on $\mathbf{S}_{*}$,
\[
Z^{\alpha}L(\psi)(u, \vartheta) \quad \text{for } (u, \vartheta)\in [0, \ub]\times \mathbb{R}/2\pi\mathbb{Z}
\]
are smooth functions of $u$ and $\partial^{i}_{1}\partial_{2}^{j}\psi_{r}(0, \vartheta)$ for $i\leq 1$ and $i+j\leq |\alpha|+1$.
\end{remark}

\subsubsection{The behavior of $(L^{k+1}(\wb), L^{k+1}(w), L^{k+1}(\psi_{2}), L^{k}(\Omega))$ on the initial singularity $\mathbf{S}_{*}$}

To resolve the \emph{compatibility issues} arising in the applications to the Riemann problem \textbf{in the regime of $R-R$} in section \ref{sec: Application to the Riemann problem 1} and \textbf{in the regime of $S-R$} in section \ref{sec: Application to the Riemann problem 2}, we must characterize the dependence of $L^{k-1}(\Omega)(u, \vartheta)$ on $\mathbf{S}_{*}$ on the normal derivatives $\partial^{k}_{1}U_{r}$.

Roughly speaking, the following proposition establishes that the $n$-jets of the solution $(\wb, w, \psi_{2})(u, \vartheta)$ on the initial singularity $\mathbf{S}_{*}$ depend only on $u$ and the $n$-jets of the initial data $(c_{r}, v^{1}_{r}, v^{2}_{r})$ at $x_{1}=0$, in accordance with our intuitive expectation.

\begin{proposition}\label{prop:jet_dependence_singularity}
Terms of the form $Z^{\alpha}L^{k}(\psi)(u, \vartheta)$ for $k+|\alpha|\leq N_{4}$ are smooth functions on $\mathbf{S}_{*}$ of $u$ and $\partial^{i}_{1}\partial_{2}^{j}\psi_{r}(0,\vartheta)$ for $i\leq k$ and $i+j\leq n$. Moreover, the coefficients $\mathbf{P}^{w}_{n, k}(u, \vartheta)$ for $n\leq N_{4}$ are smooth functions of $u$ and $\partial^{i}_{1}\partial_{2}^{j}\psi_{r}(0, \vartheta)$ for $i\leq k$ and $i+j\leq n$.
\end{proposition}

\begin{proof}
The base case $k=1$ follows directly from Remark \ref{remark: 1-th L derivates on singularity}. For the inductive step from $k$ to $k+1$, we employ the transport system for $L^{k+1}(\Omega)$ from Remark \ref{remark: ODE without loss of derivates}, restricted to the initial singularity $\mathbf{S}_{*}$. Noting that the terms $\kappa c\hat{X}^{2}L^{k}(\psi)$ and $\kappa \mathbf{P}_{k+2, k+2}^{w}$ vanish identically on $\mathbf{S}_{*}$, we obtain
\[
2\frac{\partial}{\partial u}\bigl(L^{k+1}(\psi)\bigr)(u, \vartheta)=\mathbf{P}^{w}_{1,1}(u, \vartheta)L^{k+1}(\psi)(u, \vartheta)+\mathbf{P}_{k+1, k}^{w}(u, \vartheta).
\] 
By the induction hypothesis, the coefficients on the right-hand side are smooth functions of $u$ and $\partial^{i}_{1}\partial_{2}^{j}\psi_{r}(0, \vartheta)$ for $i\leq k$ and $i+j\leq k+1$. The initial data $L^{k+1}(\psi)(0, \vartheta)$ are likewise smooth functions of $\partial^{i}_{1}\partial_{2}^{j}\psi_{r}(0,\vartheta)$ for $i\leq k+1$ and $i+j\leq k+1$. Solving this linear ODE yields that $L^{k+1}(\psi)(u, \vartheta)$ on $\mathbf{S}_{*}$ are smooth functions of $u$ and $\partial^{i}_{1}\partial_{2}^{j}\psi_{r}(0,\vartheta)$ for $i\leq k+1$ and $i+j\leq k+1$. Since $\hat{X}=X=\frac{\partial}{\partial \vartheta}$ and $T=\frac{\partial}{\partial u}$ on $\mathbf{S}_{*}$, it follows that $Z^{\alpha}L^{k+1}(\psi)(u, \vartheta)$ for $k+1+|\alpha|\leq N_{4}$ are smooth functions of $u$ and $\partial^{i}_{1}\partial_{2}^{j}\psi_{r}(0,\vartheta)$ for $i\leq k+1$ and $i+j\leq k+1+|\alpha|$.
\end{proof}

In the next proposition, we establish that $L^{n-1}(\Omega)(u, \vartheta)$ on the initial singularity $\mathbf{S}_{*}$ exhibits a controlled, explicit dependence on the top-order normal derivative $\partial_{1}^{n}v_{r}^{2}(0, \vartheta)$.

\begin{proposition}\label{prop:top_order_dependence_Omega}
For $1\leq n\leq N_{4}$, $L^{n-1}(\Omega)(u, \vartheta)$ admits the decomposition
\begin{equation}\label{eq:Omega_top_order_decomposition}
L^{n-1}(\Omega)(u, \vartheta)=\frac{c^{\frac{\gamma+1}{\gamma-1}(n-1)}}{\rho_{r}^{n}}\partial_{1}^{n}v_{r}^{2}(0, \vartheta)+F_{n, n-1}(u, U_{r}(0, \vartheta)),
\end{equation}
where $F_{n, n-1}(u, U_{r}(0, \vartheta))$ denotes a function depending smoothly on $u$ and $\partial^{i}_{1}\partial_{2}^{j}\psi_{r}(0,\vartheta)$ for $i\leq n-1$ and $i+j\leq n$.
\end{proposition}

\begin{proof}
We begin with the transport equation for $L^{k}(\Omega)$ from \eqref{ODE system for L^k+1(psi) and L^k(Omega)-basic}, specialized to the case $k = n-1$:
\[
\frac{\partial}{\partial u} \bigl(L^{n-1}(\Omega)\bigr)(u, \vartheta)=-(n-1)c^{-1}\cdot L^{n-1}(\Omega)(u, \vartheta)+\mathbf{P}_{n-1,n-1}(u, \vartheta).
\]
In view of \eqref{eq: reformulation of vorticity in the first null frame}, the term $\mathbf{P}_{n-1,n-1}(u, \vartheta)$, when restricted to $\mathbf{S}_{*}$, is a smooth function of $u$ and $\partial^{i}_{1}\partial_{2}^{j}\psi_{r}(0,\vartheta)$ for $i\leq n-1$ and $i+j\leq n$. Solving this linear ODE thus yields
\[
\begin{split}
L^{n-1}(\Omega)(u, \vartheta)&=\Bigl(\frac{c}{c_{r}}\Bigr)^{\frac{\gamma+1}{\gamma-1}(n-1)}L^{n-1}(\Omega_{r})(0, \vartheta)+F_{n, n-1}(u, U_{r}(0, \vartheta))\\
&=\frac{c^{\frac{\gamma+1}{\gamma-1}(n-1)}}{\rho_{r}^{n}}\partial_{1}^{n}v_{r}^{2}(0, \vartheta)+F_{n, n-1}(u, U_{r}(0, \vartheta)),
\end{split}
\]
as required.
\end{proof}

\section{Structural Stability of the Rarefaction Wave-Rarefaction Wave Configuration with Vorticity}\label{sec: Application to the Riemann problem 1}

In their second paper \cite{Luo-YuRare2}, Luo and Yu established the existence and uniqueness of solutions to the Riemann problem in the irrotational $R$-$R$ regime. In this section, we prove the existence part of the same problem without assuming irrotationality. There are two primary motivations for this extension:
\begin{enumerate}
\item To clarify the crucial role played by vorticity for the compatibility conditions in the $R$-$R$ regime.
\item To serve as a warm-up for the more complex applications in subsequent sections.
\end{enumerate}
We also remark that the uniqueness of solutions in the $R$-$R$ regime, which was proven via the relative entropy method in \cite{Luo-YuRare2}, is independent of the irrotationality assumption and can be applied verbatim to the rotational case.

\subsection{The determination of inner boundary $H$ and $\Hb$}\label{Section: app to Riemann 2}
We now consider the Cauchy problem with the $\varepsilon$-perturbed Riemann data in the regime of $R-R$, see Definition \ref{def:data for R-R}. We have already proved that,  for the data $(v_r,c_r)$ given on $x_1>0$, we can construct a family of rarefaction waves connecting to it. It corresponds to {\bf the front rarefaction waves} and is depicted on the right of the following picture.
\begin{center}
\includegraphics[width=3.5in]{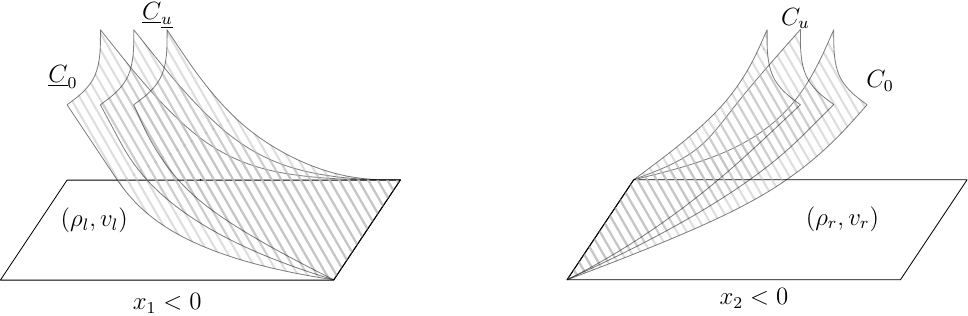}
\end{center}
For the data $(v_l,c_l)$ given on $x_1<0$, we can construct also a family of rarefaction waves connecting to it. This is {\bf the back rarefaction waves} and is depicted on the left of the above picture. These two families of rarefactions are associated to different families of characteristic hypersurfaces. We use $\Cb_{\ub}$ and $C_u$ to denote the back rarefaction wave fronts and  the front rarefaction wave fronts respectively. For a given front rarefaction wave front $C_u$, it cuts the singularity or equivalently the limiting surface $\mathbf{S}_*$ at $S_{0,u}$. We have shown that we can at least open up the front rarefaction waves up to $u=\overline{u}$, i.e., the solution exists for $u\in [0,\overline{u}]$. This is depicted in the right part of the following picture.
\begin{center}
\includegraphics[width=3.5in]{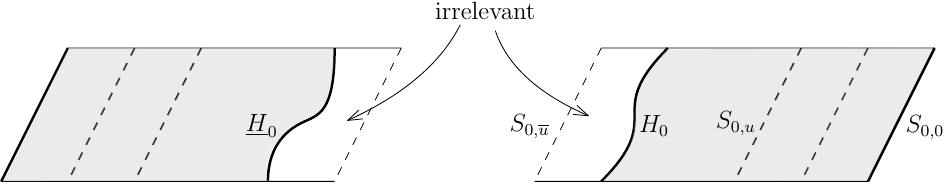}
\end{center}

We will show that there exists a curve $H_0$ between $S_{0,0}$ and $S_{0,\overline{u}}$ so that the region bounded by $H_0$ and $S_{0,\overline{u}}$ is not relevant to the perturbed Riemann problem. Let $H$ be the union of all the null geodesics emanated from $H_0$ along the $L$ direction. The  physically relevant front rarefaction wave region is the spacetime domain bounding $H$ and $C_0$. We also have a similar picture for back rarefaction waves.

We now define the inner boundary $H$ for the front rarefaction waves. According to \eqref{eq:data at the singularity}, we have $\wb(0,u,\vartheta)= \wb_r(0,\vartheta)-\frac{2}{\gamma+1}u$. In particular, $\wb$ decreases as $u$ increases. The curve $H_0$ consists of those points where $\wb$ decreases to $\wb_l$, $\wb_l$ being the Riemann invariants defined by the data on $x_1<0$, i.e., $\wb=\wb_l$. More precisely, we define
\begin{equation}\label{eq:H_0}
	H_0:=\big\{(u,\vartheta)\big|u=\frac{\gamma+1}{2}\big(\wb_r(0,\vartheta)-\wb_l(0,\vartheta)\big)\big\},
\end{equation}
where $\vartheta=x_2$. Since the solution on $\mathbf{S}_*$ is also  $O(\varepsilon)$-close to the one dimensional picture, this shows the existence of $H_0$.

\begin{center}
\includegraphics[width=2.5in]{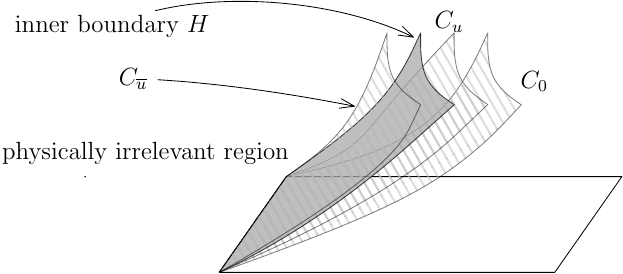}
\end{center}

For the back rarefaction waves, we can define $\Hb$ in a similar manner. The initial curve $\Hb_0$ is given by
\[\Hb_0:=\big\{(\ub,\vartheta)\big|\ub=\frac{\gamma+1}{2}\big(w_l(0,\vartheta)-w_r(0,\vartheta)\big)\big\},\]
where $\ub$ is the acoustical functions defined in the back rarefaction wave region. The $\Hb$ is defined as the null hypersurface emanated from $\Hb_0$. 

The above constructions give two characteristic hypersurfaces $H$ and $\Hb$. They are all emanated from the singularity
\[\mathbf{S}_*:=\big\{(t,x_1,x_2)\big|t=0,x_1=0\big\}\]

\subsection{Reduction to the Goursat problem with data on $H$ and $\Hb$}

Recall that the \textbf{front rarefaction wave region} $\overline{\mathcal{W}}$ is parametrized by $(t, u, \theta)\in [0, t^{*}]\times [0, \ub]\times \mathbb{R}/2\pi\mathbb{Z}$. The inner boundary $H\subset \overline{\mathcal{W}}_{right}$ can thus be parametrized by $(t, \vartheta)\in [0, t^{*}]\times \mathbb{R}/2\pi\mathbb{Z}$ as follows:
\[
H=\{\Big(t, \frac{\gamma+1}{2}\big(\wb_{r}(0, \vartheta)-\wb_{l}(0, \vartheta)\big), \vartheta\Big): (t, \theta)\in [0, t^{*}]\times \mathbb{R}/2\pi\mathbb{Z}\}.
\]
We also have the $H$-tangent vector fields
\[
\begin{cases}
{}^{(H)}L:=\frac{\partial}{\partial t}=L,\\
{}^{(H)}X:=(1+\frac{\partial u(\theta)}{\partial \theta}\cdot\Xi)X+\frac{\partial u(\theta)}{\partial \theta}\cdot T. 
\end{cases}
\]
and the $H$-intrinsic data
\[
\begin{cases}
\psi^{(H)}(t, \vartheta)= \Big((\wb(t, u(\vartheta),\vartheta), w(t, u(\vartheta),\vartheta), \psi_{2}(t, u(\vartheta),\vartheta))\Big),\\
\Omega^{(H)}(t, \theta)=\Omega(t, u(\vartheta),\vartheta).
\end{cases}
\]
For simplicity, we will omit the ${}^{(H)}$ notation from now on, as we will only discuss the vector fields $\{L, X\}$ tangent to $H$ and the data $(c, v^{1}, v^{2}, \Omega)$ intrinsic to $H$. In view of \eqref{eq: Lk estimates  on effective domain}, \eqref{eq:Omega_top_order_decomposition} and Proposition \ref{prop:jet_dependence_singularity}, we have:

\begin{proposition}\label{prop: data on H}
On the inner boundary $H$ of the \textbf{front rarefaction waves}, the data $(c, v^{1}, v^{2})$ satisfies the following three properties:
\begin{itemize}
\item The data $(c, v^{1}, v^{2})|_{H}$ satisfies the Euler equations along $H$, with
\[
\|L^{k}X^{s}\psi\|_{L^{\infty}(H)}\lesssim \varepsilon, 1\leq k+s\leq N_{4}.
\]
\item 
\[L^{k}X^{s}(\psi)(0, \vartheta)=F_{k+s, k}\big((U_{l}, U_{r})\big)(\vartheta),\] 
where $F_{k+s, k}\big((U_{l}, U_{r})\big)(\vartheta)$ denotes a generic smooth function of $\partial_{1}^{i}\partial_{2}^{j}\psi_{r}(0, \vartheta), \partial_{1}^{i}\partial_{2}^{j}\psi_{r}(0, \vartheta), i\leq k, i+j\leq k+s$.
\item $L^{n-1}(\Omega)(0, \vartheta)$ has good dependence on $\partial_{1}^{n}v_{r}^{2}(0, \vartheta)$:
\begin{equation}\label{eq: L^n-1 Omega on H}
L^{n-1}(\Omega)(0, \vartheta)=\frac{c^{\frac{\gamma+1}{\gamma-1}(n-1)}}{\rho_{r}^{n}}\partial_{1}^{n}v_{r}^{2}(0, \vartheta)+F_{n, n-1}\big((U_{l}, U_{r})\big)(\vartheta).
\end{equation}
\end{itemize}
\end{proposition}

A corresponding result holds for the \textbf{back rarefaction waves}:

\begin{proposition}\label{prop: data on Hb}
On the inner boundary $\Hb$ of the \textbf{back rarefaction waves}, the data $(\cb, \vb^{1}, \vb^{2})$ satisfies the following four properties:
\begin{itemize}
\item The data $(\cb, \vb^{1}, \vb^{2})|_{\Hb}$ satisfies the Euler equations along $\Hb$.
\item The data $(\cb, \vb^{1}, \vb^{2})$ on $\Hb$ are $O(\varepsilon)$-close to constant states, with
\[
\|\Lb^{k}\Xb^{s}\psib\|_{L^{\infty}(\Hb)}\lesssim \varepsilon, 1\leq k+s\leq N_{4}.
\]
\item 
\[
\Lb^{k}\hat{\Xb}^{s}(\psi)(0, \vartheta)=F_{k+s, k}\big((U_{l}, U_{r})\big)(\vartheta),
\] 
where $F_{k+s, k}\big((U_{l}, U_{r})\big)(\vartheta)$ denotes a generic smooth function of $\partial_{1}^{i}\partial_{2}^{j}\psi_{r}(0, \vartheta), \partial_{1}^{i}\partial_{2}^{j}\psi_{r}(0, \vartheta), i\leq k, i+j\leq k+s$.
\item $\Lb^{n-1}(\Omegab)(0, \vartheta)$ has good dependence on $\partial_{1}^{n}v_{l}^{2}(0, \vartheta)$:
\begin{equation}\label{eq: L^n-1 Omega on Hb}
\Lb^{n-1}(\Omegab)(0, \vartheta)=(-1)^{n-1}\frac{\cb^{\frac{\gamma+1}{\gamma-1}(n-1)}}{\rho_{l}^{n}}\partial_{1}^{n}v_{l}^{2}(0, \vartheta)+F_{n, n-1}\big((U_{l}, U_{r})\big)(\vartheta).
\end{equation}
\end{itemize}
\end{proposition}

It remains to construct the solution to the Euler equations in the region bounded by $H$, $\Hb$ and $\Sigma_{t^*}$. These three hypersurfaces are depicted in grey in the figure below.
\begin{center}
\includegraphics[width=2.2in]{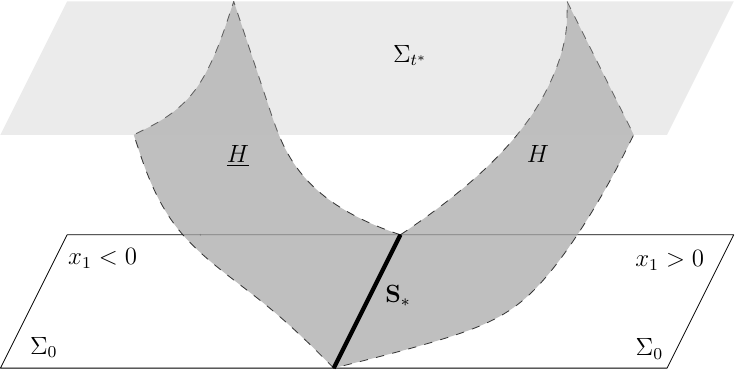}
\end{center}

To obtain a smooth solution, it is necessary to impose \textbf{compatibility conditions} up to order $N_{4}$ on $\mathbf{S}_{*}=H\cap \Hb$. We show that these compatibility conditions can be constructed inductively.

\subsubsection{The $0$-th order compatibility condition}

To obtain a smooth solution, it is necessary to require that 
\[
(c,v^{1},v^{2})=(\cb,\vb^{1},\vb^{2})\ \ \text{on}\ \ \mathbf{S}_{*}.
\]
Note that the conditions $c=\cb$ and $v^{1}=\vb^{1}$ are automatically satisfied by the definition of $H$ and $\Hb$ in Section \ref{Section: app to Riemann 2}. It is thus necessary to impose the $0$-th order compatibility condition on $\mathbf{S}_{*}$:
\[
v^{2}(0, \vartheta)=\vb^{2}(0, \vartheta), \quad \vartheta\in \mathbb{R}/2\pi\mathbb{Z}.
\]

\subsubsection{The higher order compatibility conditions}

To address the higher-order compatibility conditions, we first write the Euler equations as a characteristic system adapted to the geometry of $H$ and $\Hb$.

We first write the Euler equations as a characteristic system along $H$:
\begin{equation}\label{eq: characteristic system along H}
\begin{cases}
\begin{aligned}
L(\wb)={}&\frac{1}{2}c\hat{X}(\psi^{\hat{X}})+\frac{1}{2}c\theta\cdot \psi^{(\hat{T})}\\
&+\frac{1}{2}\psi^{(\hat{X})}\hat{X}(\psi^{(\hat{T})}+c)-\frac{1}{2}\theta\psi^{(\hat{X})}\psi^{(\hat{X})},
\end{aligned}\\[1.5ex]
\begin{aligned}
L(\psi^{(\hat{X})})={}&-c\hat{T}(\psi^{(\hat{X})})+c\hat{X}_{r}\left(\frac{2c}{\gamma-1}\right)+c\frac{\hat{X}(\kappa)}{\kappa}\psi^{(\hat{T})}\\
&-\psi^{(\hat{T})}\hat{X}(\psi^{(\hat{T})}+c)+\theta\psi^{(\hat{X})}\psi^{(\hat{T})},
\end{aligned}\\[1.5ex]
\begin{aligned}
L(w)={}&-2c\hat{T}(w)+\frac{1}{2}c\hat{X}(\psi^{\hat{X}})+c\frac{\hat{X}(\kappa)}{\kappa}\psi^{(\hat{X})}\\
&+\frac{1}{2}c\theta\cdot \psi^{(\hat{T})}-\frac{1}{2}\psi^{(\hat{X})}\hat{X}(\psi^{(\hat{T})}+c)+\frac{1}{2}\theta\psi^{(\hat{X})}\psi^{(\hat{X})},
\end{aligned}\\[1.5ex]
\textbf{$\hat{T}(\wb)|_{\mathbf{S}_{*}}=-\partial_{1}\left(\frac{1}{\gamma-1}c+\frac{1}{2}v^{1}\right)$, can be freely prescribed},\\[1ex]
\hat{T}(w)|_{\mathbf{S}_{*}}=-\partial_{1}\left(\frac{1}{\gamma-1}c-\frac{1}{2}v^{1}\right), \text{is determined by data on $H$},\\[1ex]
\hat{T}(\psi^{(\hat{X})})|_{\mathbf{S}_{*}}=-\partial_{1}(-v^{2}), \text{is determined by data on $H$}.
\end{cases}
\end{equation}
This system is expressed in terms of the geometric Riemann invariants associated with the direction $-\hat{T}$:
\[
\wb:=\frac{1}{2}(\frac{2}{\gamma-1}c+\psi^{(\hat{T})}), w:=\frac{1}{2}(\frac{2}{\gamma-1}c-\psi^{(\hat{T})}), \psi^{(\hat{X})}.
\]

We similarly write the Euler equations as a characteristic system along $\Hb$:
\begin{equation}\label{eq: characteristic system along Hb}
\begin{cases}
\begin{aligned}
\Lb(w)={}&\frac{1}{2}\cb\cdot\hat{\Xb}(\psib^{\hat{\Xb}})+\frac{1}{2}\cb\cdot\thetab\cdot \psib^{(\hat{\Tb})}\\
&+\frac{1}{2}\psib^{(\hat{\Xb})}\hat{\Xb}(-\psib^{(\hat{\Tb})}+\cb)+\frac{1}{2}\thetab\cdot\psib^{(\hat{\Xb})}\psib^{(\hat{\Xb})},
\end{aligned}\\[1.5ex]
\begin{aligned}
\Lb(\psib^{(\hat{\Xb})})={}&\cb\cdot\hat{\Tb}(\psib^{(\hat{\Xb})})+\cb\cdot\hat{\Xb}\left(\frac{2\cb}{\gamma-1}\right)-\cb\frac{\hat{\Xb}(\kappab)}{\kappab}\psib^{(\hat{\Tb})}\\
&-\psib^{(\hat{\Tb})}\hat{\Xb}(-\psib^{(\hat{\Tb})}+\cb)+\thetab\psib^{(\hat{\Xb})}\psib^{(\hat{\Tb})},
\end{aligned}\\[1.5ex]
\begin{aligned}
\Lb(\wb)={}&2\cb\cdot\hat{\Tb}(\wb)+\frac{1}{2}\cb\cdot\hat{\Xb}(\psib^{\hat{\Xb}})+\cb\frac{\hat{\Xb}(\kappab)}{\kappab}\psib^{(\hat{\Xb})}\\
&+\frac{1}{2}\cb\cdot\thetab\cdot \psib^{(\hat{\Tb})}-\frac{1}{2}\psib^{(\hat{\Xb})}\hat{\Xb}(-\psib^{(\hat{\Tb})}+\cb)-\frac{1}{2}\thetab\psib^{(\hat{\Xb})}\psib^{(\hat{\Xb})},
\end{aligned}\\[1.5ex]
\hat{\Tb}(\wb)|_{\mathbf{S}_{*}}=-\partial_{1}\left(\frac{1}{\gamma-1}\cb+\frac{1}{2}\vb^{1}\right), \text{is determined by data on $\Hb$},\\[1ex]
\textbf{$\hat{\Tb}(w)|_{\mathbf{S}_{*}}=-\partial_{1}\left(\frac{1}{\gamma-1}\cb-\frac{1}{2}\vb^{1}\right)$, can be freely prescribed},\\[1ex]
\hat{\Tb}(\psib^{(\hat{\Xb})})|_{\mathbf{S}_{*}}=-\partial_{1}(-\vb^{2}), \text{is determined by data on $\Hb$}.
\end{cases}
\end{equation}
This system is expressed in terms of the geometric Riemann invariants associated with the direction $\hat{\Tb}$:
\[
\wb:=\frac{1}{2}(\frac{2}{\gamma-1}\cb+\psib^{(\hat{\Tb})}), w:=\frac{1}{2}(\frac{2}{\gamma-1}\cb-\psi^{(\hat{\Tb})}), \psib^{(\hat{\Xb})}.
\]

In view of \eqref{eq: characteristic system along H} and \eqref{eq: characteristic system along Hb}, for each $1\leq n\leq N_{4}$, assume that all compatibility conditions up to order $n-1$ have been imposed to ensure that 
\[
(\partial_{1}^{n-1}c, \partial_{1}^{n-1}v^{1}, \partial_{1}^{n-1}v^{2})=(\partial_{1}^{n-1}\cb, \partial_{1}^{n-1}\vb^{1}, \partial_{1}^{n-1}\vb^{2})\ \ \text{on}\ \ \mathbf{S}_{*}.
\]
Then the conditions $\partial_{1}^{n}c=\partial_{1}^{n}\cb$ and $\partial_{1}^{n}v^{1}=\partial_{1}^{n}\vb^{1}$ on $\mathbf{S}_{*}$ are automatically satisfied by setting
\[
\hat{T}^{n}(\wb):=\hat{\Tb}^{n}(\wb), \hat{\Tb}^{n}(w):=\hat{T}^{n}(w)\ \ \text{on}\ \ \mathbf{S}_{*},
\]
since $\hat{T}^{n}(\wb)|$ and $\hat{\Tb}^{n}(w)$ on $\mathbf{S}_{*}$ can be freely prescribed. Once $\hat{T}^{n}(\wb)|$ and $\hat{\Tb}^{n}(w)$ on $\mathbf{S}_{*}$ have been freely prescribed, $\hat{\Tb}^{n}(\psi^{(\hat{\Xb})})$ and $\hat{T}^{n}(\psi^{(\hat{X})})$ are uniquely determined in view of the non-characteristic property of the tangential velocity $\psi^{(\hat{X})}$ with respect to $H$ and $\Hb$. It is thus necessary to impose the $n$-th order compatibility condition a priori:
\begin{equation}\label{eq: q1}
\hat{\Tb}^{n}(\psi^{(\hat{\Xb})})=\hat{T}^{n}(\psi^{(\hat{X})})\ \ \text{on}\ \ \mathbf{S}_{*}.
\end{equation}

Moreover, \eqref{eq: q1} can be reformulated in terms of $\Omega$ in place of the tangential velocity $\psi^{(\hat{X})}$:
\[
\begin{cases}
\varrho\hat{T}^{n-1}(\Omega)+\sum_{1\leq k\leq n-1}\frac{(n-1)!}{k!(n-1-k)!}\hat{T}^{k}(\varrho)\hat{T}^{n-1-k}(\Omega)=\hat{T}^{n}(\psi^{(\hat{X})})-\hat{T}^{n-1}\hat{X}(\psi^{(\hat{T})}),\\
\underline{\varrho}\hat{\Tb}^{n-1}(\Omegab)+\sum_{1\leq k\leq n-1}\frac{(n-1)!}{k!(n-1-k)!}\hat{T}^{k}(\underline{\varrho})\hat{T}^{n-1-k}(\Omega)=\hat{\Tb}^{n}(\psib^{(\hat{\Xb})})-\hat{\Tb}^{n-1}\hat{\Xb}(\psib^{(\hat{\Tb})}),
\end{cases}\ \ \text{on}\ \ \mathbf{S}_{*},
\]
which implies that
\begin{equation}\label{eq: compatibility of Omega}
\hat{\Tb}^{n}(\psi^{(\hat{\Xb})})=\hat{T}^{n}(\psi^{(\hat{X})})\ \ \text{on}\ \ \mathbf{S}_{*}\iff \hat{\Tb}^{n-1}(\Omegab)=\hat{T}^{n-1}(\Omega) \ \ \text{on}\ \ \mathbf{S}_{*}.
\end{equation}

We may use the equations $B\Omega=0$ and $B\Omegab=0$ to relate $\hat{T}^{n-1}(\Omega)$ to $L^{n-1}(\Omega)$, and $\hat{\Tb}^{n-1}(\Omegab)$ to $\Lb^{n-1}(\Omegab)$, respectively:
\[
\begin{cases}
L^{n-1}(\Omega)=(-c)^{n-1}\hat{T}^{n-1}\Omega+F_{n, n-1}\big((\underline{U}, U)\big)(0, \vartheta),\\
\Lb^{n-1}(\Omegab)=\cb^{n-1}\hat{\Tb}^{n-1}\Omegab+F_{n, n-1}\big((\underline{U}, U)\big)(0, \vartheta).
\end{cases}
\]
where $F_{n, n-1}\big((\underline{U}, U)\big)(0, \vartheta)$ denotes a generic smooth function of $\partial_{1}^{i}\partial_{2}^{j}U(0, \vartheta)$ and $\partial_{1}^{i}\partial_{2}^{j}\underline{U}(0, \vartheta)$ with $i+j\leq n$ and $i\leq n-1$. The $n$-th order compatibility condition can thus be written in terms of $L^{n-1}(\Omega)$ and $\Lb^{n-1}(\Omegab)$ as
\begin{equation}\label{eq: L^n-1 Omega and L^n-1 Omegab}
\Lb^{n-1}(\Omegab)(0, \vartheta)=(-1)^{n-1}L^{n-1}(\Omega)(0, \vartheta)+F_{n, n-1}\big((\underline{U}, U)\big)(0, \vartheta)\ \ \text{on}\ \ \mathbf{S}_{*}.
\end{equation}

Combining \eqref{eq: L^n-1 Omega on H}, \eqref{eq: L^n-1 Omega on Hb} and \eqref{eq: L^n-1 Omega and L^n-1 Omegab}, we summarize the \textbf{$n$-th order compatibility condition for the $R$-$R$ problem} as
\begin{equation}\label{eq: n-th compatibility condition for R-R}
\begin{cases}
v_{l}^{2}(0, x_{2})=v_{r}^{2}(0, x_{2}),\\
\frac{1}{\rho_{l}^{n}}\partial_{1}^{n}v_{l}^{2}(0, x_{2})=\frac{1}{\rho_{r}^{n}}\partial_{1}^{n}v_{r}^{2}(0, x_{2})+F\big((U_{l}, U_{r})\big)(\vartheta), 1\leq n\leq N_{4}.
\end{cases}
\end{equation}

In summary, for sufficiently large $N_{4}$, we may apply the $H^{N_{4}}$ well-posedness theory for the classical Goursat problem for the Euler equations and we refer to Speck-Yu \cite{SpeckYu} for details. For sufficiently small $\varepsilon$, this allows us to construct a solution to the equations in the region bounded by $H$, $\Hb$ and $\Sigma_{t^*}$. This completes the construction of the solution to the Riemann problem, and hence the proof of Theorem \ref{thm: R-R}.

\begin{remark}
The uniqueness of entropy solutions to the $R$-$R$ problem can be established via the relative entropy method, in exactly the same manner as in \cite{Luo-YuRare2}.
\end{remark}

\begin{remark}
In the irrotational setting, the 0-th order compatibility condition
\[
v_{l}^{2}(0, x_{2})=v_{r}^{2}(0, x_{2}), \quad x_{2}\in \mathbb{R}/2\pi\mathbb{Z},
\]
for the $R$-$R$ problem automatically implies all higher-order compatibility conditions for any $n\geq 1$, as \eqref{eq: compatibility of Omega} holds trivially in the irrotational setting. This explains why only the 0-th order compatibility condition is required in \cite{Luo-YuRare2}.
\end{remark}

\begin{remark}
There exists a large class of initial data satisfying Definition \ref{def:data for R-R}. In fact, for any initial data satisfying Definition \ref{def:data for R-R} up to the compatibility conditions, we may make $O(\varepsilon)$ modifications to
\[
v_{l}^{2}(0, x_{2}), \partial_{1}v_{l}^{2}(0, x_{2}), \cdots \partial_{1}^{n}v_{l}^{2}(0, x_{2}), \cdot, 0\leq n\leq N_{4}.
\]
inductively, in view of the structure of the $n$-th order compatibility condition for the $R$-$R$ problem given in \eqref{eq: n-th compatibility condition for R-R}.
\end{remark}

\section{Structural Stability of the Shock Wave-Rarefaction Wave Configuration}\label{sec: Application to the Riemann problem 2}

\subsection{A short review of Majda's results on nonlinear stability of a single shock front}

Let the initial data be an $O(\varepsilon)$ perturbation of the piecewise constant initial data \eqref{data: a single shock}, which admits a single shock wave:
\[
U(0, x_{1}, x_{2})=\begin{cases}
U_{l}, & x_{1}\leq 0,\\
U_{r}, & x_{1}\geq 0.
\end{cases}
\]
Let the curved shock front be parametrized by $x_{1}=\phi(t, x_{2})$, with
\[
U(t, x_{1}, x_{2})=\begin{cases}
U_{-}(t, x_{1}, x_{2}), & x_{1}\leq \phi(t, x_{2}),\\
U_{+}(t, x_{1}, x_{2}), & x_{1}\geq \phi(t, x_{2}).
\end{cases}
\]

To better discuss the compatibility conditions for shock waves, we introduce the following geometric quantities: $\hat{T}$, the unit spatial normal vector field to the shock front; $\hat{X}$, the unit spatial tangent vector field to the shock front; $\sigma$, the normalized shock speed; and $\slashed{g}$, the first fundamental form of the shock front, given by
\begin{equation}
\begin{cases}
\hat{T}=(\frac{1}{\sqrt{1+(\partial_{2}\phi)^{2}}}, \frac{-\partial_{2}\phi}{\sqrt{1+(\partial_{2}\phi)^{2}}}),\\
\hat{X}=(\frac{\partial_{2}\phi}{\sqrt{1+(\partial_{2}\phi)^{2}}}, \frac{1}{\sqrt{1+(\partial_{2}\phi)^{2}}}),\\
\sigma=\frac{\partial_{t}\phi}{\sqrt{1+(\partial_{2}\phi)^{2}}},\\
\slashed{g}:=\sqrt{1+(\partial_{2}\phi)^{2}}.
\end{cases}
\end{equation}

On the shock front $x_{1}=\phi(t, x_{2})$, the Rankine–Hugoniot conditions for the shock front, together with the entropy condition, imply that
\begin{equation}\label{eq: R-H for shock}
\begin{cases}
 v_{+}^{(\hat{T})}=v_{-}^{(\hat{T})}-\sqrt{\frac{P_{+}-P_{-}}{(\rho_{+}-\rho_{-})\rho_{+}\rho_{-}}}(\rho_{+}-\rho_{-}),\\
v_{+}^{(\hat{X})}=v_{-}^{(\hat{X})}.
\end{cases}
\end{equation}
The shock speed along the front is given by
\[
\sigma=\frac{\rho_{+}v_{+}^{(\hat{T})}-\rho_{-}v_{-}^{(\hat{T})}}{\rho_{+}-\rho_{-}}.
\]

By flattening the boundary via the coordinate transformation
\[
\begin{cases}
t=s,\\
x_{1}=y_{1}+\phi(s, y_{2}),\\
x_{2}=y_{2},
\end{cases}, (s, y_{1}, y_{2})\in [0, t^{*}]\times \mathbb{R}\times \mathbb{T}.
\]
It can be shown that the necessary $k$-th order compatibility condition for the existence of a piecewise smooth solution with a single shock front takes the form
\begin{equation}\label{eq: compatibility condition for a single shock front}
\begin{cases}
\partial_{x_1}^{k}(\frac{c_{r}}{\gamma-1}+\frac{1}{2}v_{r}^{1})(0, x_{2})=G\Big(\partial_{x_1}^{k}(\frac{c_{l}}{\gamma-1}+v_{l}^{1}), F_{k, k-1}\big(U_{l}, U_{r})(x_{2})\big)\Big),\\
\partial_{x_1}^{k}v_{r}^{2}(0, x_{2})=(\frac{v_{l}^{1}-\sigma}{v_{r}^{1}-\sigma})^{k}\partial_{x_1}^{k}v_{l}^{2}(0, x_{2})+F_{k, k-1}\big(U_{l}, U_{r})\big)(x_{2}),
\end{cases}
\end{equation}
where $F_{k, k-1}\big(U_{l}, U_{r})(x_{2})$ denotes a generic smooth function of $\partial^{i}_{1}\partial^{j}_{2}U_{l}(0, x_{2})$ and $\partial^{i}_{1}\partial^{j}_{2}U_{r}(0, x_{2})$ of size $O(\varepsilon)$ with $i+j\leq k$ and $i\leq k-1$. We refer to \cite{MajdaShock2} for full details, as these works establish how to construct the compatibility conditions in a far more general setting than the 2D isentropic Euler equations.

We summarize Majda's results \cite{MajdaShock2, MajdaShock3} for the 2D isentropic Euler equations for a polytropic gas as follows. Let $(U_{r}, U_{l})$ be an $\varepsilon$-localized perturbation of the piecewise constant initial data \eqref{data: a single shock}, which admits a single shock wave. Suppose further that $(U_{r}, U_{l})$ satisfies the Rankine–Hugoniot conditions and the entropy condition \eqref{eq: R-H for shock} along $x_{1}=0$, as well as the $k$-th order compatibility conditions \eqref{eq: compatibility condition for a single shock front} for $1\leq k\leq N$ and sufficiently large $N$. Then the Euler equations \eqref{eq: Euler equations} admit a piecewise smooth solution with a single shock front on $[0, t^{*}]\times \mathbb{R}\times \mathbb{R}/2\pi\mathbb{Z}$.

\subsection{Construction of the $S-R$ configuration}

 \begin{itemize}
\item  We construct a classical solution $U_{r}$ on the region completely determined by the initial data $U_{r}$ posed on $\Sigma_{0}^{\geqslant 0} = \{x_{1}\geqslant 0\}$, up to a finite time $t^{*}$. We denote the parametrization of the outermost 3rd-family characteristic surface $C_{0}$ by $x_{1}=\phi_{C_{0}}(t, x_{2})$.

\begin{center}
\begin{tikzpicture}[scale=1.5]
\coordinate (O) at (0,0);
\fill (O) circle (0.8pt);

\draw[thick] (-3,0) -- (3,0);
\draw[thick] (-3,2) -- (3,2);

\node[right] at (3,0) {$\Sigma_{0}$};
\node[right] at (3,2) {$\Sigma_{t^{*}}$};
\node[below] at (O) {$x_{1}=0$};

\node[above] at (1.5,0) {$U_{r}$ on $\Sigma_{0}^{\geqslant 0}$};
\node[above] at (-1.5,0) {$U_{l}$ on $\Sigma_{0}^{\leqslant 0}$};

\draw[thick,dashed] (O) -- (2.5,2);
\node[above] at (2.5,2) {$C_{0}$};

\node at (2.2,1) {$U_{r}$};
\end{tikzpicture}
\end{center}
 
\item
In view of Theorem \ref{theorem: existence-rarefaction}, we can construct a family of rarefaction fronts emanating from the singular surface $\mathbf{S}_{*} = \{t=0, x_{1}=0\}$ and connecting to $U_{r}$ along $C_{0}$. Moreover, the rarefaction wave region can be parametrized by acoustical coordinates $(t, u, \vartheta)\in [0, t^{*}]\times[0, \ub]\times \mathbb{T}$ which resolve the initial singularity at $\mathbf{S}_{*}$.

Furthermore, in acoustical coordinates $(t, u, \vartheta)$, the solution $U_{R}$ exhibits a one-dimensional profile along the $u$-direction at $t=0$:
\begin{equation}
\begin{cases}
c_{R}(0, u, \vartheta)=c_{r}(0, \vartheta)-\frac{\gamma-1}{\gamma+1}u,\\
v_{R}^{1}(0, u, \vartheta)=v_{r}^{1}(0, \vartheta)-\frac{2}{\gamma+1}u,\\
v_{R}^{2}(0, u, \vartheta)=v_{r}^{2}(0, \vartheta).
\end{cases}
\end{equation}
For higher-order $C_{u}$-tangent jets at $t=0$, in view of Proposition \ref{prop:jet_dependence_singularity} and Proposition \ref{prop:top_order_dependence_Omega}, we have
\[
\begin{cases}
L^{n}U_{R}(0, u, \vartheta)=F_{n}\big(u, U_{r}(0, \vartheta)\big),\\
L^{n-1}(\Omega_{R})(u, \vartheta)=\frac{c_{R}^{\frac{\gamma+1}{\gamma-1}(n-1)}}{\rho_{r}^{n}}\partial_{1}^{n}v_{r}^{2}(0, \vartheta)+F_{n, n-1}(u, U_{r}(0, \vartheta)).
\end{cases}
\]

\begin{center}
\begin{tikzpicture}[scale=1.5]
\coordinate (O) at (0,0);
\fill (O) circle (0.8pt);

\draw[thick] (-3,0) -- (3,0);
\draw[thick] (-3,2) -- (3,2);

\node[right] at (3,0) {$\Sigma_{0}$};
\node[right] at (3,2) {$\Sigma_{t^{*}}$};
\node[below] at (O) {$x_{1}=0$};

\node[above] at (1.5,0) {$U_{r}$ on $\Sigma_{0}^{\geqslant 0}$};
\node[above] at (-1.5,0) {$U_{l}$ on $\Sigma_{0}^{\leqslant 0}$};

\foreach \x in {0.6,0.9,1.2,1.5,1.8,2.1} {
  \draw[dashed] (O) -- (\x,2);
}

\draw[thick,dashed] (O) -- (2.5,2);
\node[above] at (2.5,2) {$C_{0}$};

\node at (2.2,1) {$U_{r}$};
\node at (1.25,1.2) {$U_{R}$};
\end{tikzpicture}
\end{center}

\item To determine the inner boundary $H$ of the rarefaction wave region, for each $x_{2}\in \mathbb{R}/2\pi\mathbb{Z}$, we let $u=g(x_{2})$ denote the unique solution to the Rankine–Hugoniot condition:
\[
v_{R}^{1}=v_{l}^{1}-\sqrt{\frac{\Big(P(\rho_{R})-P(\rho_{l})\Big)(\rho_{R}-\rho_{l})}{\rho_{R}\rho_{l}}}.
\]
Then the inner boundary $H$ is defined by the level set $u(t, x_{1}, x_{2})=g(x_{2})$. By the implicit function theorem, $H$ can be parametrized by $x_{1}=\phi_{H}(t, x_{2})$.

\begin{center}
\begin{tikzpicture}[scale=1.5]
\coordinate (O) at (0,0);
\fill (O) circle (0.8pt);

\draw[thick] (-3,0) -- (3,0);
\draw[thick] (-3,2) -- (3,2);

\node[right] at (3,0) {$\Sigma_{0}$};
\node[right] at (3,2) {$\Sigma_{t^{*}}$};
\node[below] at (O) {$x_{1}=0$};

\node[above] at (1.5,0) {$U_{r}$ on $\Sigma_{0}^{\geqslant 0}$};
\node[above] at (-1.5,0) {$U_{l}$ on $\Sigma_{0}^{\leqslant 0}$};

\foreach \x in {1.8,2.1} {
  \draw[dashed] (O) -- (\x,2);
}

\draw[thick,dashed] (O) -- (2.5,2);
\node[above] at (2.5,2) {$C_{0}$};
\draw[thick,dashed] (O) -- (1.5,2);
\node[above] at (1.5,2) {$H$};

\node at (2.2,1) {$U_{r}$};
\node at (1.25,1.2) {$U_{R}$};
\end{tikzpicture}
\end{center}
 
\item Let $U_{H}$ be the restriction of $U_{R}$ to $H$. We then need to prescribe\footnote{In view of \eqref{eq: characteristic system along H}, the normal jets $\partial_{1}^{k}\left(\frac{c_{H}}{\gamma-1}+\frac{1}{2}v_{H}^{1}\right)$ on $\mathbf{S}_{*}$ can be freely prescribed.} normal jets of $U_{H}$ on $\mathbf{S}_{*}$ inductively to ensure that $U_{l}$ and $U_{H}$ are compatible on $\mathbf{S}_{*}$ up to order $N_{4}$ in the sense of admitting a single 1st-family shock front.

Assume that $\partial_{x_1}^{j}\left(\frac{c_{H}}{\gamma-1}+\frac{1}{2}v_{H}^{1}\right)(0, x_{2})$ for $1\leq j\leq k-1$ have been prescribed to ensure compatibility up to order $k-1$. Then $\partial_{x_1}^{k}\left(\frac{c_{H}}{\gamma-1}-\frac{1}{2}v_{H}^{1}\right)(0, x_{2})$ and $\partial_{x_1}^{k}v_{H}^{2}(0, x_{2})$ are determined by $\partial_{x_1}^{a}\partial_{x_2}^{b}U_{l, r}(0, x_{2})$ for $a+b\leq k$.

On the one hand, the compatibility condition of the shock-rarefaction configuration \eqref{condition: S-R} ensures that
\begin{equation}\label{local: SR1}
\partial_{x_1}^{k}v_{H}^{2}(0, x_{2})=\left(\frac{v^{1}_{l}-\sigma}{v_{H}^{1}-\sigma}\right)^{k}\partial_{x_1}^{k}v_{l}^{2}(0, x_2)+F_{k, k-1}\big((U_{l}, U_{H})\big)(0, x_{2}).
\end{equation}
On the other hand, we can freely prescribe $\partial_{1}^{k}\left(\frac{c_{H}}{\gamma-1}+\frac{1}{2}v_{H}^{1}\right)(0, x_{2})$ by setting
\begin{equation}\label{local: SR2}
\partial_{x_1}^{k}\left(\frac{c_{H}}{\gamma-1}+\frac{1}{2}v_{H}^{1}\right)(0, x_{2})=G\Big(\partial_{x_1}^{k}\left(\frac{c_{l}}{\gamma-1}+v_{l}^{1}\right)(0, x_{2}),F_{k, k-1}\big((U_{l}, U_{H})\big)\Big)(0, x_{2}).
\end{equation}
Therefore, \eqref{local: SR1} and \eqref{local: SR2} imply that $U_{l}$ and $U_{H}$ are compatible on $\mathbf{S}_{*}$ up to order $N_{4}$ in the sense of admitting a single shock front.

\item Since the normal jets of $U_{H}$ are determined on $\mathbf{S}_{*}$, they are uniquely determined by solving ODE systems along $H$ and we refer to \cite{wang2025constructingcharacteristicinitialdata} for details. Moreover, the size of the normal jets of $U_{H}$ is of $O(\varepsilon)$. By Taylor expansion, we can construct auxiliary data $U_{au, r}(t^{*}, x_{1}, x_{2})$ on $\Sigma_{t^{*}}^{x_{1} \geqslant \phi_{H}(t^{*}, x_{2})}$ to ensure that:
\begin{itemize}
\item The auxiliary data $U_{au}$ are smoothly compatible with $U_{H}$ on $H\cap \Sigma_{t^{*}}$ up to order $N_{5}$\footnote{To obtain all the normal jets of $U_{H}$ along $H$, we need to use the wave equation $L(TU_{H})=\mu\hat{X}^{2}(U_{H})+\text{lower order terms}$, which causes loss of derivatives since we need to use two tangential derivatives to estimate one normal derivative.},
\item The auxiliary data $U_{au}$ have compact support\footnote{Note that we can use a cut-off function away from $\mathbf{S}_{*}$.},
\item The auxiliary data $U_{au}$ are of size $O_{C^{N_{5}}}(\varepsilon)$.
\end{itemize}

\begin{center}
\begin{tikzpicture}[scale=1.5]
\coordinate (O) at (0,0);
\fill (O) circle (0.8pt);

\draw[thick] (-3,0) -- (5,0);
\draw[thick] (-3,2) -- (5,2);

\node[right] at (5,0) {$\Sigma_{0}$};
\node[right] at (5,2) {$\Sigma_{t^{*}}$};
\node[below] at (O) {$x_{1}=0$};

\node[above] at (-1.5,0) {$U_{l}$ on $\Sigma_{0}^{\leqslant 0}$};

\draw[thick,dashed] (O) -- (1.5,2);
\node[above] at (1.5,2) {$H$};

\node at (0.75,1) {$U_{H}$ on $H$};
\node[above] at (3,2) {$U_{au, r}$ on $\Sigma_{t^{*}}^{x_{1}\geqslant \phi_{H}(t^{*}, x_{2})}$};  
\end{tikzpicture}
\end{center}

\item{\bf Given initial data $U_{H}$ on H and $U_{au}$ on $\Sigma_{t^{*}}^{x_{1}\geqslant \phi_{H}(t^{*}, x_{2})}$}, we can construct a $C^{N_{6}}$ solution $U_{au}$ on the region $\{(t, x_{1}, x_{2}): 0\leq t\leq t^{*}, x_{2}\in \mathbb{T}, x_{1}\geq \phi_{H}(t, x_{2})\}$ and we refer to Speck-Yu \cite{SpeckYu} for details on the local well-posedness of the Goursat problem for general compressible Euler equations. In particular, we obtain $C^{N_{6}}$ auxiliary data $U_{au, r}(0, x_{1}, x_{2})$ on $\Sigma_{0}^{\geqslant 0}$. Moreover, $U_{l}$ and $U_{au, r}$ are compatible up to order $N_{6}$ in the sense of admitting a single shock front.

\begin{center}
\begin{tikzpicture}[scale=1.5]
\coordinate (O) at (0,0);
\fill (O) circle (0.8pt);

\draw[thick] (-3,0) -- (5,0);
\draw[thick] (-3,2) -- (5,2);

\node[right] at (5,0) {$\Sigma_{0}$};
\node[right] at (5,2) {$\Sigma_{t^{*}}$};
\node[below] at (O) {$x_{1}=0$};

\node[above] at (-1.5,0) {$U_{l}$ on $\Sigma_{0}^{\leqslant 0}$};
\node[above] at (2,0) {$U_{au}(0, x_{1}, x_{2})$ on $\Sigma_{0}^{\geqslant 0}$};  

\draw[thick,dashed] (O) -- (1.5,2);
\node[above] at (1.5,2) {$H$};
\end{tikzpicture}
\end{center}

\item We can apply Majda's results \cite{MajdaShock2,MajdaShock3} directly to the Cauchy data $(U_{l}, U_{au, r})$ to construct a single 1st-family shock front $S: x_{1}=\phi_{S}(t, x_{2})$ together with a piecewise $C^{N_{7}}$ smooth solution
\[
\begin{cases}
U_{-}(t, x_{1}, x_{2}), & x_{1}\leq \phi_{S}(t, x_{2}),\\
U_{+}(t, x_{1}, x_{2}), & x_{1}\geq \phi_{S}(t, x_{2}),
\end{cases}
\quad (t, x_{1}, x_{2})\in [0, t^{*}]\times \mathbb{R}\times \mathbb{R}/2\pi \mathbb{Z},
\]
provided that $N_{6}$ is sufficiently large.
 
\begin{center}
\begin{tikzpicture}[scale=1.5]
\coordinate (O) at (0,0);
\fill (O) circle (0.8pt);

\draw[thick] (-3,0) -- (4,0);
\draw[thick] (-3,2) -- (4,2);

\node[right] at (4,0) {$\Sigma_{0}$};
\node[right] at (4,2) {$\Sigma_{t^{*}}$};
\node[below] at (O) {$x_{1}=0$};

\node[above] at (-1.5,0) {$U_{l}$ on $\Sigma_{0}^{\leqslant 0}$};
\node[above] at (2,0) {$U_{au}(0, x_{1}, x_{2})$ on $\Sigma_{0}^{\geqslant 0}$};  

\draw[thick,solid] (O) -- (-1,2);
\node[above] at (-1,2) {$S$};

\node at (-2,1) {$U_{-}$};
\node at (0.5,1) {$U_{+}$};
\end{tikzpicture}
\end{center}

\item Finally, we verify that
\[
 U(t, x_{1}, x_{2}):=\begin{cases}
 U_{-}(t, x_{1}, x_{2}), & x_{1}\leq \phi_{S}(t, x_{2}),\\
U_{+}(t, x_{1}, x_{2}), & \phi_{S}(t, x_{2})\leq x_{1}\leq \phi_{H}(t, x_{2}),\\
U_{R}(t, x_{1}, x_{2}), & \phi_{H}(t, x_{2})\leq x_{1}\leq \phi_{C_{0}}(t, x_{2}),\\
U_{r}(t, x_{1}, x_{2}), & \phi_{C_0}(t, x_{2})\leq x_{1},
 \end{cases}
 \quad (t, x_{1}, x_{2})\in [0, t^{*}]\times \mathbb{R}\times \mathbb{T},
 \]
is a piecewise smooth entropy solution with initial data $(U_{l}, U_{r})$. In fact, it suffices to note that $U_{R}=U_{+}$ on $H$ by construction.

\begin{center}
\begin{tikzpicture}[scale=1.5]
\coordinate (O) at (0,0);
\fill (O) circle (0.8pt);

\draw[thick] (-3,0) -- (3,0);
\draw[thick] (-3,2) -- (3,2);

\node[right] at (3,0) {$\Sigma_{0}$};
\node[right] at (3,2) {$\Sigma_{t^{*}}$};
\node[below] at (O) {$x_{1}=0$};

\node[above] at (-1.5,0) {$U_{l}$ on $\Sigma_{0}^{\leqslant 0}$};
\node[above] at (1.5,0) {$U_{r}$ on $\Sigma_{0}^{\geqslant 0}$};  

\draw[thick,solid] (O) -- (-1,2);
\node[above] at (-1,2) {$S$};
\draw[thick,dashed] (O) -- (1.5,2);
\node[above] at (1.5,2) {$H$};
\draw[thick,dashed] (O) -- (2.5,2);
\node[above] at (2.5,2) {$C_{0}$};

\foreach \x in {1.8,2.1} {
  \draw[dashed] (O) -- (\x,2);
}

\node at (-2,1) {$U_{-}$};
\node at (0.3,1) {$U_{+}$};
\node at (1.25,1.2) {$U_{R}$};
\node at (2.5,1) {$U_{r}$};
\end{tikzpicture}
\end{center}

 \end{itemize}

\section{Structural Stability of the Rarefaction Wave-Vortex Sheet-Rarefaction Wave Configuration}\label{sec: Application to the Riemann problem 3}

\subsection{A short review of Coulombel and Secchi's results on nonlinear stability of a single vortex sheet}

Let the initial data be an $O(\varepsilon)$ perturbation of the piecewise constant initial data \eqref{eq: data for a single vortex sheet}, which admits a single vortex sheet:
\[
U(0, x_{1}, x_{2})=\begin{cases}
U_{l}, & x_{1}\leq 0,\\
U_{r}, & x_{1}\geq 0.
\end{cases}
\]
Let the curved vortex sheet be parametrized by $x_{1}=\phi(t, x_{2})$, with
\[
U(t, x_{1}, x_{2})=\begin{cases}
U_{-}(t, x_{1}, x_{2}), & x_{1}\leq \phi(t, x_{2}),\\
U_{+}(t, x_{1}, x_{2}), & x_{1}\geq \phi(t, x_{2}).
\end{cases}
\]

Following the Coulombel–Secchi approach, we introduce eikonal functions $\Phi_{\pm}(t, x_{1}, x_{2})$ for $(t, x_{1}, x_{2})\in [0, \infty)\times (0, \infty)\times \mathbb{T}$ to fix the boundary:
\[
\begin{cases}
-\partial_{t}\Phi_{\pm}+v_{\pm}^{1}-v_{\pm}^{2}\partial_{x_{2}}\Phi_{\pm}=0,\\
\Phi_{\pm}(0, x_{1}, x_{2})=\pm x_{1}.
\end{cases}
\]
In particular, we have
\[
\Phi_{\pm}(t, 0, x_{2})=\phi(t, x_{2}).
\]

To better discuss the compatibility conditions for vortex sheets, we introduce the following geometric quantity: $\hat{T}_{\pm}$, the unit spatial normal vector field to the vortex sheet, given by
\begin{equation}
\hat{T}_{\pm}=(-\frac{1}{\sqrt{1+(\partial_{2}\Phi_{\pm})^{2}}}, \frac{\partial_{2}\Phi_{\pm}}{\sqrt{1+(\partial_{2}\Phi_{\pm})^{2}}}).
\end{equation}

On the vortex sheet $x_{1}=\phi(t, x_{2})$, the Rankine–Hugoniot conditions for the vortex sheet imply that
\begin{equation}\label{eq: R-H for vortex sheet}
\begin{cases}
 v_{+}^{(\hat{T}_{+})}=v_{-}^{(\hat{T}_{-})},\\
c_{+}=c_{-}.
\end{cases}
\end{equation}

By flattening the boundary via the coordinate transformation
\[
\begin{cases}
t=s,\\
x_{1}=\Phi_{\pm}(s, y_{1}, y_{2}),\\
x_{2}=y_{2},
\end{cases}
\quad (s, y_{1}, y_{2})\in \mathbb{R}_{+}\times \mathbb{R}_{+}\times \mathbb{T},
\]
it has been shown in \cite{Coulombel-Secchi2} that the necessary $k$-th order compatibility condition for the existence of a piecewise smooth solution with a single vortex sheet takes the form
\begin{equation}\label{eq: compatibility condition for a single vortex sheet}
\begin{cases}
\partial_{x_1}^{k}\big(c_{r}-v_{r}^{1}+c_{l}-v_{l}^{1}\big)=F_{k, k-1}\big((U_{l}, U_{r})\big)(0, x_{2}),\\
\partial_{x_1}^{k}\big(c_{r}+v_{r}^{1}+c_{l}+v_{l}^{1}\big)=F_{k, k-1}\big((U_{l}, U_{r})\big)(0, x_{2}),
\end{cases}
\end{equation}
where $F_{k, k-1}\big((U_{l}, U_{r})\big)(x_{2})$ denotes a generic smooth function of $\partial^{i}_{1}\partial^{j}_{2}U_{l}(0, x_{2})$ and $\partial^{i}_{1}\partial^{j}_{2}U_{r}(0, x_{2})$ of size $O(\varepsilon)$ with $i+j\leq k$ and $i\leq k-1$.

We summarize Coulombel and Secchi's results \cite{Coulombel-Secchi1, Coulombel-Secchi2} for the 2D isentropic Euler equations for a polytropic gas as follows. Let $(U_{r}, U_{l})$ be an $\varepsilon$-localized perturbation of the piecewise constant initial data \eqref{eq: data for a single vortex sheet}, which admits a single vortex sheet. Moreover, we require the piecewise constant initial data \eqref{eq: data for a single vortex sheet} to satisfy the \textbf{supersonic condition} \eqref{eq: super sonic condition}. Suppose further that $(U_{r}, U_{l})$ satisfies the Rankine–Hugoniot conditions \eqref{eq: R-H for vortex sheet} along $x_{1}=0$, as well as the $k$-th order compatibility conditions \eqref{eq: compatibility condition for a single vortex sheet} for $1\leq k\leq N$ and sufficiently large $N$. Then the Euler equations \eqref{eq: Euler equations} admit a piecewise smooth solution with a single vortex sheet on $[0, t^{*}]\times \mathbb{R}\times \mathbb{R}/2\pi\mathbb{Z}$.

\subsection{Construction of the $R-V-R$ configuration}

\begin{itemize}
\item  We construct classical solutions $U_{r}$ and $U_{l}$ on the regions completely determined by the initial data $U_{r}$ posed on $\Sigma_{0}^{\geqslant 0} = \{x_{1}\geqslant 0\}$ and $U_{l}$ posed on $\Sigma_{0}^{\leqslant 0} = \{x_{1}\leqslant 0\}$, respectively, up to a finite time $t^{*}$. We denote the parametrizations of the outermost 3rd-family characteristic surface $C_{0}$ and the outermost 1st-family characteristic surface $\Cb_{0}$ by $x_{1}=\phi_{C_{0}}(t, x_{2})$ and $x_{1}=\phi_{\Cb_{0}}(t, x_{2})$, respectively.

\begin{center}
\begin{tikzpicture}[scale=1.5]
\coordinate (O) at (0,0);
\fill (O) circle (0.8pt);

\draw[thick] (-3,0) -- (3,0);
\draw[thick] (-3,2) -- (3,2);

\node[right] at (3,0) {$\Sigma_{0}$};
\node[right] at (3,2) {$\Sigma_{t^{*}}$};
\node[below] at (O) {$x_{1}=0$};

\node[above] at (1.5,0) {$U_{r}$ on $\Sigma_{0}^{\geqslant 0}$};
\node[above] at (-1.5,0) {$U_{l}$ on $\Sigma_{0}^{\leqslant 0}$};

\draw[thick,dashed] (O) -- (2.5,2);
\node[above] at (2.5,2) {$C_{0}$};
\draw[thick,dashed] (O) -- (-2,2);
\node[above] at (-2,2) {$\Cb_{0}$};

\node at (2.2,1) {$U_{r}$};
\node at (-2.2,1) {$U_{l}$};
\end{tikzpicture}
\end{center}
 
\item
By Theorem \ref{theorem: existence-rarefaction}, we can construct families of rarefaction fronts emanating from the singular surface $\mathbf{S}_{*} = \{t=0, x_{1}=0\}$, with solutions $U_{R}$ and $U_{L}$ connecting to $U_{r}$ along $C_{0}$ and to $U_{l}$ along $\Cb_{0}$, respectively. The right rarefaction wave region can be parametrized by acoustical coordinates $(t, u, \vartheta)\in [0, t^{*}]\times[0, u^{*}]\times \mathbb{T}$, and the left rarefaction wave region by $(t, \ub, \vartheta)\in [0, t^{*}]\times[0, \ub^{*}]\times \mathbb{T}$, both of which resolve the initial singularity at $\mathbf{S}_{*}$.

Furthermore, in acoustical coordinates $(t, u, \vartheta)$, the solution $U_{R}$ exhibits a one-dimensional profile along the $u$-direction at $t=0$:
\begin{equation}
\begin{cases}
c_{R}(0, u, \vartheta)=c_{r}(0, \vartheta)-\frac{\gamma-1}{\gamma+1}u,\\
v_{R}^{1}(0, u, \vartheta)=v_{r}^{1}(0, \vartheta)-\frac{2}{\gamma+1}u,\\
v_{R}^{2}(0, u, \vartheta)=v_{r}^{2}(0, \vartheta).
\end{cases}
\end{equation}
For higher-order $C_{u}$-tangent jets at $t=0$, in view of Proposition \ref{prop:jet_dependence_singularity} and Proposition \ref{prop:top_order_dependence_Omega}, we have
\[
\begin{cases}
L^{n}U_{R}(0, u, \vartheta)=F_{n}\big(u, U_{r}(0, \vartheta)\big),\\
L^{n-1}(\Omega_{R})(u, \vartheta)=\frac{c_{R}^{\frac{\gamma+1}{\gamma-1}(n-1)}}{\rho_{r}^{n}}\partial_{1}^{n}v_{r}^{2}(0, \vartheta)+F_{n, n-1}(u, U_{r}(0, \vartheta)).
\end{cases}
\]

Similarly, in acoustical coordinates $(t, \ub, \vartheta)$, the solution $U_{L}$ exhibits a one-dimensional profile along the $\ub$-direction at $t=0$:
\begin{equation}\label{formula: 1-D pattern left}
\begin{cases}
c_{L}(0, \ub, \vartheta)=c_{l}(0, \vartheta)-\frac{\gamma-1}{\gamma+1}\ub,\\
v_{L}^{1}(0, \ub, \vartheta)=v_{l}^{1}(0, \vartheta)+\frac{2}{\gamma+1}\ub,\\
v_{L}^{2}(0, \ub, \vartheta)=v_{l}^{2}(0, \vartheta).
\end{cases}
\end{equation}
For higher-order $\Cb_{\ub}$-tangent jets at $t=0$, we have
\[
\begin{cases}
\Lb^{n}U_{L}(0, \ub, \vartheta)=F_{n}\big(\ub, U_{l}(0, \vartheta)\big),\\
\Lb^{n-1}(\Omega_{L})(\ub, \vartheta)=\frac{(-c_{L})^{\frac{\gamma+1}{\gamma-1}(n-1)}}{\rho_{l}^{n}}\partial_{1}^{n}v_{l}^{2}(0, \vartheta)+F_{n, n-1}(\ub, U_{l}(0, \vartheta)).
\end{cases}
\]

\begin{center}
\begin{tikzpicture}[scale=1.5]
\coordinate (O) at (0,0);
\fill (O) circle (0.8pt);

\draw[thick] (-3,0) -- (3,0);
\draw[thick] (-3,2) -- (3,2);

\node[right] at (3,0) {$\Sigma_{0}$};
\node[right] at (3,2) {$\Sigma_{t^{*}}$};
\node[below] at (O) {$x_{1}=0$};

\node[above] at (1.5,0) {$U_{r}$ on $\Sigma_{0}^{\geqslant 0}$};
\node[above] at (-1.5,0) {$U_{l}$ on $\Sigma_{0}^{\leqslant 0}$};

\foreach \x in {1.0,1.3,1.6,1.9,2.2} {
  \draw[dashed] (O) -- (\x,2);
}
\foreach \x in {-0.8,-1.1,-1.4,-1.7} {
  \draw[dashed] (O) -- (\x,2);
}

\draw[thick,dashed] (O) -- (2.5,2);
\node[above] at (2.5,2) {$C_{0}$};
\draw[thick,dashed] (O) -- (-2,2);
\node[above] at (-2,2) {$\Cb_{0}$};

\node at (2.2,1) {$U_{r}$};
\node at (0.9,1.2) {$U_{R}$};
\node at (-0.9,1.2) {$U_{L}$};
\node at (-2.2,1) {$U_{l}$};
\end{tikzpicture}
\end{center}

\item To determine the inner boundaries $H$ and $\Hb$ of the rarefaction wave regions, for each $x_{2}\in \mathbb{T}$, we let $u=g_{r}(x_{2})$ and $\ub=g_{l}(x_{2})$ denote the unique solutions to the system:
\[
\begin{cases}
c_{r}(0, x_{2})-\frac{\gamma-1}{\gamma+1}u=c_{l}(0, x_{2})-\frac{\gamma-1}{\gamma+1}\ub,\\
v^{1}_{r}(0, x_{2})-\frac{2}{\gamma+1}u=v_{l}^{1}(0, x_{2})+\frac{2}{\gamma+1}\ub.
\end{cases}
\]
Then the inner boundary $H$ is defined by the level set $u(t, x_{1}, x_{2})=g_{r}(x_{2})$, and the inner boundary $\Hb$ by $\ub(t, x_{1}, x_{2})=g_{l}(x_{2})$. By the implicit function theorem, $H$ and $\Hb$ can be parametrized by $x_{1}=\phi_{H}(t, x_{2})$ and $x_{1}=\phi_{\Hb}(t, x_{2})$, respectively.

\begin{center}
\begin{tikzpicture}[scale=1.5]
\coordinate (O) at (0,0);
\fill (O) circle (0.8pt);

\draw[thick] (-3,0) -- (3,0);
\draw[thick] (-3,2) -- (3,2);

\node[right] at (3,0) {$\Sigma_{0}$};
\node[right] at (3,2) {$\Sigma_{t^{*}}$};
\node[below] at (O) {$x_{1}=0$};

\node[above] at (1.5,0) {$U_{r}$ on $\Sigma_{0}^{\geqslant 0}$};
\node[above] at (-1.5,0) {$U_{l}$ on $\Sigma_{0}^{\leqslant 0}$};

\foreach \x in {1.3,1.6,1.9,2.2} {
  \draw[dashed] (O) -- (\x,2);
}
\foreach \x in {-1.1,-1.4,-1.7} {
  \draw[dashed] (O) -- (\x,2);
}

\draw[thick,dashed] (O) -- (2.5,2);
\node[above] at (2.5,2) {$C_{0}$};
\draw[thick,dashed] (O) -- (1.0,2);
\node[above] at (1.0,2) {$H$};
\draw[thick,dashed] (O) -- (-0.8,2);
\node[above] at (-0.8,2) {$\Hb$};
\draw[thick,dashed] (O) -- (-2,2);
\node[above] at (-2,2) {$\Cb_{0}$};

\node at (2.2,1) {$U_{r}$};
\node at (0.9,1.2) {$U_{R}$};
\node at (-0.9,1.2) {$U_{L}$};
\node at (-2.2,1) {$U_{l}$};
\end{tikzpicture}
\end{center}
 
\item Let $U_{H}$ and $U_{\Hb}$ be the restrictions of $U_{R}$ and $U_{L}$ to $H$ and $\Hb$, respectively. We need to prescribe\footnote{In view of the characteristic nature of $H$ and $\Hb$, the normal jets of $\left(\frac{c_{H}}{\gamma-1}+\frac{1}{2}v_{H}^{1}\right)$ on $\mathbf{S}_{*}$ and $\left(\frac{c_{\Hb}}{\gamma-1}-\frac{1}{2}v_{\Hb}^{1}\right)$ on $\mathbf{S}_{*}$ can be freely prescribed.} higher normal jets of $U_{H}$ and $U_{\Hb}$ on $\mathbf{S}_{*}$ inductively to ensure that $U_{H}$ and $U_{\Hb}$ are compatible on $\mathbf{S}_{*}$ up to order $N_{4}$ in the sense of admitting a single vortex sheet.

Assume that 
\[
\partial_{x_{1}}^{j}\left(\frac{c_{H}}{\gamma-1}+\frac{1}{2}v_{H}^{1}\right)(0, x_{2}) \quad \text{for } 0\leq j\leq k-1,
\]
together with 
\[
\partial_{x_{1}}^{j}\left(\frac{c_{\Hb}}{\gamma-1}-\frac{1}{2}v_{\Hb}^{1}\right)(0, x_{2}) \quad \text{for } 0\leq j\leq k-1,
\]
have been prescribed to ensure compatibility up to order $k-1$. Then 
\[
\partial_{x_{1}}^{k}\left(\frac{c_{H}}{\gamma-1}-\frac{1}{2}v_{H}^{1}\right)(0, x_{2}) \quad \text{and} \quad \partial_{x_{1}}^{k}\left(\frac{c_{\Hb}}{\gamma-1}+\frac{1}{2}v_{\Hb}^{1}\right)(0, x_{2}),
\]
as well as
\[
\partial_{x_{1}}^{k}v_{H}^{2}(0, x_{2}) \quad \text{and} \quad \partial_{x_{1}}^{k}v_{\Hb}^{2}(0, x_{2}),
\]
are determined by $\partial_{x_1}^{a}\partial_{x_2}^{b}U_{l, r}(0, x_{2})$ for $a+b\leq k$.

In view of \eqref{eq: characteristic system along H} and \eqref{eq: characteristic system along Hb}, we can freely prescribe 
\[
\partial_{x_{1}}^{k}\left(\frac{c_{H}}{\gamma-1}+\frac{1}{2}v_{H}^{1}\right)(0, x_{2}) \quad \text{and} \quad \partial_{x_{1}}^{k}\left(\frac{c_{\Hb}}{\gamma-1}-\frac{1}{2}v_{\Hb}^{1}\right)(0, x_{2})
\]
by solving the linear system:
\begin{equation}\label{local: RVR1}
\begin{cases}
\frac{\gamma-3}{2}\partial_{x_1}^{k}\wb_{H}(0, x_{2})+\frac{\gamma+1}{2}\partial_{x_1}^{k}w_{\Hb}(0, x_{2})=F_{k, k-1}\big((U_{\Hb}, U_{H})\big)(0, x_{2}),\\
\frac{\gamma+1}{2}\partial_{x_1}^{k}\wb_{H}(0, x_{2})+\frac{\gamma-3}{2}\partial_{x_1}^{k}w_{\Hb}(0, x_{2})=F_{k, k-1}\big((U_{\Hb}, U_{H})\big)(0, x_{2}).
\end{cases}
\end{equation}

This linear system is always uniquely solvable for any right-hand side $F_{k, k-1}$. Indeed, the determinant of its coefficient matrix is
\[
\det\begin{pmatrix}
\frac{\gamma-3}{2} & \frac{\gamma+1}{2} \\
\frac{\gamma+1}{2} & \frac{\gamma-3}{2}
\end{pmatrix}
= \left(\frac{\gamma-3}{2}\right)^2 - \left(\frac{\gamma+1}{2}\right)^2 = -2(\gamma-1).
\]

Since $\gamma>1$ for a polytropic gas, we have $-2(\gamma-1) \neq 0$, which implies that the coefficient matrix is non-singular. Therefore, \eqref{local: RVR1} implies that $U_{\Hb}$ and $U_{H}$ are compatible on $\mathbf{S}_{*}$ up to order $N_{4}$ in the sense of \eqref{eq: compatibility condition for a single vortex sheet}.

\item Since the normal jets of $U_{H}$ and $U_{\Hb}$ are determined on $\mathbf{S}_{*}$, they are uniquely determined by solving ODE systems along $H$ and $\Hb$, respectively; we refer to \cite{wang2025constructingcharacteristicinitialdata} for details. Moreover, the size of the normal jets of $U_{H}$ and $U_{\Hb}$ is of $O(\varepsilon)$. By Taylor expansion, we can construct auxiliary data $U_{au, r}(t^{*}, x_{1}, x_{2})$ and $U_{au, l}(t^{*}, x_{1}, x_{2})$ on $\Sigma_{t^{*}}^{x_{1} \geqslant \phi_{H}(t^{*}, x_{2})}$ and $\Sigma_{t^{*}}^{x_{1} \leqslant \phi_{\Hb}(t^{*}, x_{2})}$, respectively, to ensure that:
\begin{itemize}
\item The auxiliary data $U_{au, r}$ and $U_{au, l}$ are smoothly compatible with $U_{H}$ and $U_{\Hb}$ on $H\cap \Sigma_{t^{*}}$ and $\Hb\cap\Sigma_{t^{*}}$, respectively, up to order $N_{5}$\footnote{To obtain all the normal jets of $U_{H}$ along $H$, we need to use the wave equation $L(TU_{H})=\hat{X}^{2}(U_{H})+\text{lower order terms}$, which causes loss of derivatives since we need to use two tangential derivatives to estimate one normal derivative.},
\item The auxiliary data $U_{au, r}$ and $U_{au, l}$ have compact support\footnote{Note that we can use cut-off functions which vanish away from $\mathbf{S}_{*}$.},
\item The auxiliary data $U_{au, r}$ and $U_{au, l}$ are of size $O_{C^{N_{5}}}(\varepsilon)$.
\end{itemize}

\begin{center}
\begin{tikzpicture}[scale=1.5]
\coordinate (O) at (0,0);
\fill (O) circle (0.8pt);

\draw[thick] (-4,0) -- (4,0);
\draw[thick] (-4,2) -- (4,2);

\node[right] at (4,0) {$\Sigma_{0}$};
\node[right] at (4,2) {$\Sigma_{t^{*}}$};
\node[below] at (O) {$x_{1}=0$};

\draw[thick,dashed] (O) -- (1.0,2);
\node[above] at (1.0,2) {$H$};
\draw[thick,dashed] (O) -- (-0.8,2);
\node[above] at (-0.8,2) {$\Hb$};

\node at (0.6,1) {$U_{H}$ on $H$};
\node at (-0.6,1) {$U_{\Hb}$ on $\Hb$};
\node[above] at (2.5,2) {$U_{au, r}$ on $\Sigma_{t^{*}}^{\geqslant \phi_{H}(t^{*}, x_{2})}$};
\node[above] at (-2.5,2) {$U_{au, l}$ on $\Sigma_{t^{*}}^{\leqslant \phi_{\Hb}(t^{*}, x_{2})}$};
\end{tikzpicture}
\end{center}

\item \textbf{Given initial data $U_{H}$ on $H$, $U_{\Hb}$ on $\Hb$, $U_{au, r}$ on $\Sigma_{t^{*}}^{x_{1} \geqslant \phi_{H}(t^{*}, x_{2})}$ and $U_{au, l}$ on $\Sigma_{t^{*}}^{x_{1} \leqslant \phi_{\Hb}(t^{*}, x_{2})}$}, we can construct $C^{N_{6}}$ solutions $U_{au, r}$ and $U_{au, l}$ on the regions $\{(t, x_{1}, x_{2}): 0\leq t\leq t^{*}, x_{2}\in \mathbb{T}, x_{1}\geq \phi_{H}(t, x_{2})\}$ and $\{(t, x_{1}, x_{2}): 0\leq t\leq t^{*}, x_{2}\in \mathbb{T}, x_{1}\leq \phi_{\Hb}(t, x_{2})\}$, respectively; we refer to Speck-Yu \cite{SpeckYu} for details on the local well-posedness of the Goursat problem for general compressible Euler equations. In particular, we obtain $C^{N_{6}}$ auxiliary data $U_{au, r}(0, x_{1}, x_{2})$ on $\Sigma_{0}^{\geqslant 0}$ and $U_{au, l}(0, x_{1}, x_{2})$ on $\Sigma_{0}^{\leqslant 0}$. Moreover, $U_{au, l}$ and $U_{au, r}$ are compatible up to order $N_{6}$ in the sense of admitting a single vortex sheet \eqref{eq: compatibility condition for a single vortex sheet}.

\begin{center}
\begin{tikzpicture}[scale=1.5]
\coordinate (O) at (0,0);
\fill (O) circle (0.8pt);

\draw[thick] (-4,0) -- (4,0);
\draw[thick] (-4,2) -- (4,2);

\node[right] at (4,0) {$\Sigma_{0}$};
\node[right] at (4,2) {$\Sigma_{t^{*}}$};
\node[below] at (O) {$x_{1}=0$};

\node[above] at (2.5,0) {$U_{au, r}(0, x_{1}, x_2)$ on $\Sigma_{0}^{\geqslant 0}$};
\node[above] at (-2.5,0) {$U_{au, l}(0, x_1, x_2)$ on $\Sigma_{0}^{\leqslant 0}$};

\draw[thick,dashed] (O) -- (1.0,2);
\node[above] at (1.0,2) {$H$};
\draw[thick,dashed] (O) -- (-0.8,2);
\node[above] at (-0.8,2) {$\Hb$};
\end{tikzpicture}
\end{center}

\item We can apply the results of Coulombel and Secchi \cite{Coulombel-Secchi1, Coulombel-Secchi2} directly to the Cauchy data $(U_{au, l}, U_{au, r})$ to construct a single vortex sheet $V: x_{1}=\phi_{V}(t, x_{2})$ together with a piecewise smooth solution
\[
\begin{cases}
U_{-}(t, x_{1}, x_{2}), & x_{1}\leq \phi_{V}(t, x_{2}),\\
U_{+}(t, x_{1}, x_{2}), & x_{1}\geq \phi_{V}(t, x_{2}),
\end{cases}
\quad (t, x_{1}, x_{2})\in [0, t^{*}]\times \mathbb{R}\times \mathbb{R}/2\pi\mathbb{Z}.
\]

\begin{center}
\begin{tikzpicture}[scale=1.5]
\coordinate (O) at (0,0);
\fill (O) circle (0.8pt);

\draw[thick] (-4,0) -- (4,0);
\draw[thick] (-4,2) -- (4,2);

\node[right] at (4,0) {$\Sigma_{0}$};
\node[right] at (4,2) {$\Sigma_{t^{*}}$};
\node[below] at (O) {$x_{1}=0$};

\node[above] at (2.5,0) {$U_{au, r}(0, x_{1}, x_2)$ on $\Sigma_{0}^{\geqslant 0}$};
\node[above] at (-2.5,0) {$U_{au, l}(0, x_1, x_2)$ on $\Sigma_{0}^{\leqslant 0}$};

\draw[thick,dotted] (O) -- (0.1,2);
\node[above] at (0.1,2) {$V$};

\node at (2.5,1) {$U_{+}$};
\node at (-2.5,1) {$U_{-}$};
\end{tikzpicture}
\end{center}

\item Finally, we verify that
\[
 U(t, x_{1}, x_{2}):=\begin{cases}
 U_{l}, & x_{1}\leq \phi_{\Cb_{0}}(t, x_{2}),\\
 U_{L}, & \phi_{\Cb_{0}}(t, x_{2})\leq x_{1}\leq \phi_{\Hb}(t, x_{2}),\\
 U_{-}, & \phi_{\Hb}(t, x_{2})\leq x_{1}\leq \phi_{V}(t, x_{2}),\\
U_{+}, & \phi_{V}(t, x_{2})\leq x_{1}\leq \phi_{H}(t, x_{2}),\\
U_{R}, & \phi_{H}(t, x_{2})\leq x_{1}\leq \phi_{C_{0}}(t, x_{2}),\\
U_{r}, & \phi_{C_0}(t, x_{2})\leq x_{1},
 \end{cases}
 \quad (t, x_{1}, x_{2})\in [0, t^{*}]\times \mathbb{R}\times \mathbb{T},
 \]
is a piecewise smooth entropy solution with initial data $(U_{l}, U_{r})$. In fact, it suffices to note that $U_{R}=U_{+}$ on $H$ and $U_{L}=U_{-}$ on $\Hb$ by construction.

\begin{center}
\begin{tikzpicture}[scale=1.5]
\coordinate (O) at (0,0);
\fill (O) circle (0.8pt);

\draw[thick] (-3,0) -- (3,0);
\draw[thick] (-3,2) -- (3,2);

\node[right] at (3,0) {$\Sigma_{0}$};
\node[right] at (3,2) {$\Sigma_{t^{*}}$};
\node[below] at (O) {$x_{1}=0$};

\node[above] at (1.5,0) {$U_{r}$ on $\Sigma_{0}^{\geqslant 0}$};
\node[above] at (-1.5,0) {$U_{l}$ on $\Sigma_{0}^{\leqslant 0}$};

\foreach \x in {1.3,1.6,1.9,2.2} {
  \draw[dashed] (O) -- (\x,2);
}
\foreach \x in {-1.1,-1.4,-1.7} {
  \draw[dashed] (O) -- (\x,2);
}

\draw[thick,dotted] (O) -- (0.1,2);
\node[above] at (0.1,2) {$V$};
\draw[thick,dashed] (O) -- (1.0,2);
\node[above] at (1.0,2) {$H$};
\draw[thick,dashed] (O) -- (-0.8,2);
\node[above] at (-0.8,2) {$\Hb$};
\draw[thick,dashed] (O) -- (2.5,2);
\node[above] at (2.5,2) {$C_{0}$};
\draw[thick,dashed] (O) -- (-2,2);
\node[above] at (-2,2) {$\Cb_{0}$};

\node at (-2.2,1) {$U_{l}$};
\node at (-0.9,1.2) {$U_{L}$};
\node at (-0.2,1.2) {$U_{-}$};
\node at (0.3,1.2) {$U_{+}$};
\node at (0.9,1.2) {$U_{R}$};
\node at (2.2,1) {$U_{r}$};
\end{tikzpicture}
\end{center}
\end{itemize}

\section*{Acknowledgment}

\thanks{JJ is supported by Starting Research Fund from Hunan University under Grant No. 531118010918. TL is supported by  a grant from the Research Grants Council of the Hong Kong Special
Administrative Region, China (Project No. 11313025).}


\begin{thebibliography}{99}



\bibitem{AlinhacWaveRare0} S. Alinhac, \textit{Existence d'ondes de rar\'efaction pour des \'ecoulements isentropiques}, S\'eminaire sur les \'equations aux d\'eriv\'ees partielles 1986–1987, Exp. No. XVI, 16 pp., \'Ecole Polytech., Palaiseau, 1987. 

\bibitem{AlinhacWaveRare1} S. Alinhac, \textit{Existence d'ondes de rar\'efaction pour des syst\`emes quasi-lin\'eaires hyperboliques multidimensionnels}, Comm. Partial Differential Equations 14 (1989), no. 2, 173–230. 

\bibitem{AlinhacWaveRare2}S. Alinhac, \textit{Unicit\'e d'ondes de rar\'efaction pour des syst\`emes quasi-lin\'eaires hyperboliques multidimensionnels}, Indiana Univ. Math. J. 38 (1989), no. 2, 345–363. 

\bibitem{Benzoni-Gavage-Serre2007book}
S. Benzoni-Gavage, and D. Serre, \textit{Multi-dimensional Hyperbolic Partial Differential Equations}, Oxford University Press 2007.

\bibitem{18} B. Stevens, Short-time structural stability of compressible vortex sheets with surface tension, \textit{Arch. Ration. Mech. Anal.} 222 (2016) 603--730.



\bibitem{Chen-Chen} G.-Q. Chen and J. Chen, \textit{Stability of rarefaction waves and vacuum states for the multidimensional Euler equations}. J. Hyperbolic Differ. Equ. 4 (2007), no. 1, 105-122.

\bibitem{ChenXinYin} Chen, S., Xin, Z. and Yin, H, \textit{Global Shock Waves¶for the Supersonic Flow Past a Perturbed Cone} Commun. Math. Phys. 228, 47–84 (2002). https://doi.org/10.1007/s002200200652.


\bibitem{ChenLi}S. Chen and D. Li, \textit{Cauchy problem with general discontinuous initial data along a smooth curve for 2-d Euler system}, J. Differential Equations 257 (2014), no. 6, 1939–1988.

\bibitem{ChristodoulouShockDevelopment}  D. Christodoulou, \textit{The shock development problem}, EMS Monographs in Mathematics, European Mathematical Society (EMS), Zürich, 2019. 

\bibitem{CK} D. Christodoulou and S. Klainerman, \textit{The global nonlinear stability of the Minkowski space}, Princeton Mathematical Series, vol. 41, Princeton University Press, Princeton, NJ, 1993. 

\bibitem{ChiodaroliDeLellisKreml2015} E. Chiodaroli, C. De Lellis, O. Kreml, \textit{Global Ill-Posedness of the Isentropic System of Gas Dynamics}, Comm. Pure Appl. Math. 68 (2015), no. 7, 1157–1190.


\bibitem{ChristodoulouMiao} D. Christodoulou and S. Miao, \textit{Compressible flow and Euler's equations}, Surveys of Modern Mathematics, vol. 9, International Press, Somerville, MA; Higher Education Press, Beijing, 2014. 


\bibitem{CourantFriedrichs}R. Courant and K. O. Friedrichs, \textit{Supersonic flow and shock waves}, Reprinting of the 1948 original. Applied Mathematical Sciences, Vol. 21. Springer-Verlag, New York-Heidelberg, 1976. xvi+464 pp.


\bibitem{Coulombel-Secchi1} J.-F. Coulombel and P. Secchi, \textit{The stability of compressible vortex sheets in two space dimensions}, Indiana Univ. Math. J. {\bf 53} (2004), no.~4, 941--1012; MR2095445.

\bibitem{Coulombel-Secchi2} J.-F. Coulombel and P. Secchi, \textit{Nonlinear compressible vortex sheets in two space dimensions}, Ann. Sci. \'Ec. Norm. Sup\'er. (4) {\bf 41} (2008), no.~1, 85--139; MR2423311.

\bibitem{Dafermos} C. Dafermos, \textit{Hyperbolic conservation laws in continuum physics}, Third edition, Grundlehren der Mathematischen Wissenschaften, Vol. 325, Springer-Verlag, Berlin, 2010.

\bibitem{Dafermos1} C. Dafermos, \textit{The second law of thermodynamics and stability}, Arch. Rational Mech. Anal. 70, 167–179 (1979). https://doi.org/10.1007/BF00250353


\bibitem{Diperna4}
R.~{Diperna }.
\newblock Decay of solutions of hyperbolic systems of conservation laws with a
  convex extension.
\newblock {\em Arch. Rational Mech. Anal.}, 64(1):1--46, 1977.

\bibitem{diperna2}
R.~{Diperna }.
\newblock Convergence of the viscosity method for isentropic gas dynamics.
\newblock {\em Comm. Math. Phys.}, 91(1):1--30, 1983.

\bibitem{Diperna3}
R.~{Diperna}.
\newblock Decay and asymptotic behavior of solutions to nonlinear hyperbolic
  systems of conservation laws.
\newblock {\em Indiana Univ. Math. J.}, 24(11):1047--1071, 1974.



\bibitem{DiPerna79} R. J. DiPerna, \textit{Uniqueness of solutions to hyperbolic conservation laws}, Indiana Univ. Math. J. 28 (1979), no. 1, 137–188.



\bibitem{GangXuYinHuicheng} Gang, X., Huicheng, Y, \textit{On Global Multidimensional Supersonic Flows with Vacuum States at Infinity}. Arch Rational Mech Anal 218, 1189–1238 (2015). https://doi.org/10.1007/s00205-015-0878-6.

\bibitem{Glimm1965} J. Glimm, \textit{Solutions in the large for nonlinear hyperbolic systems of equations}. Comm. Pure Appl. Math. 18 (1965), 697–715.

\bibitem{glimmlax}
J.~{Glimm} and P.~{Lax}.
\newblock Decay of solutions of systems of nonlinear hyperbolic conservation
  laws.
\newblock {\em Memoirs of the American Mathematical Society}, 101:xvii+112 pp,
  1970.
  


\bibitem{ginsberg2024stabilityirrotationalshockslandau} Daniel Ginsberg and Igor Rodnianski, \textit{The stability of irrotational shocks and the Landau law of decay}. https://arxiv.org/abs/2403.13568

\bibitem{Kruzkov1970} S. N. Kru\v{z}kov, \textit{First order quasilinear equations with several independent variables}
Mat. Sb. (N.S.)81(123)(1970), 228–255.

\bibitem{16} J. A. Fejer, J. W. Miles, On the stability of a plane vortex sheet with respect to three-dimensional disturbances, \textit{J. Fluid Mech.} 15 (1963) 335--336.

\bibitem{17} J. W. Miles, On the disturbed motion of a plane vortex sheet, \textit{J. Fluid Mech.} 4 (1958) 538--552.

\bibitem{KrupaSzékelyhidi} Krupa, S.G., Székelyhidi, L, \textit{Contact Discontinuities for 2-D Isentropic Euler are Unique in 1-D but Wildly Non-unique Otherwise}. Commun. Math. Phys. 406, 109 (2025). https://doi.org/10.1007/s00220-025-05278-6

\bibitem{Lax1957} P. D. Lax, \textit{Hyperbolic systems of conservation laws. II}. Comm. Pure Appl. Math. 10 (1957), 537–566.


\bibitem{PL}
P.~{Lax}.
\newblock Hyperbolic systems of conservation laws and the mathematical theory
  of shock waves.
\newblock In {\em Conference Board of the Mathematical Sciences Regional
  Conference Series in Applied Mathematics}, pages v+48 pp, 1973.


\bibitem{Li1991} D. Li, \textit{Rarefaction and shock waves for multidimensional hyperbolic conservation laws.} Comm. Partial Differential Equations 16 (1991), no. 2-3, 425–450.

\bibitem{Liu2021book} T.-P. Liu, \textit{Shock waves}. Graduate Studies in Mathematics, 215. American Mathematical Society, Providence, RI, 202.

\bibitem{liutp2}
T.~{Liu}.
\newblock Decay to $n$-waves of solutions of general systems of nonlinear
  hyperbolic conservation laws.
\newblock {\em Comm. Pure Appl. Math.}, 30(5):586--611, 1977.

\bibitem{L2}
T.~{Liu}.
\newblock Linear and nonlinear large-time behavior of solutions of general
  systems of hyperbolic conservation laws.
\newblock {\em Comm. Pure Appl. Math.}, 30(6):767--796, 1977.


\bibitem{Luo-YuRare1} T.-W. Luo and P. Yu, \textit{On the stability of multi-dimensional rarefaction waves I: the energy estimates}, Ann. of Math. (2) {\bf 202} (2025), no.~2, 631--752; MR4964221.

\bibitem{Luo-YuRare2} T.-W. Luo and P. Yu, \textit{On the stability of multi-dimensional rarefaction waves II: existence of solutions and applications to the Riemann problem}, Ann. of Math. (2) {\bf 202} (2025), no.~2, 753--855; MR4964222.


\bibitem{LukSpeck2D} J. Luk and J. Speck, \textit{Shock formation in solutions to the 2D compressible Euler equations in the presence of non-zero vorticity}, Invent. Math. 214 (2018), no. 1, 1–169.

\bibitem{MajdaShock2} A. Majda, \textit{The existence of multidimensional shock fronts}, Mem. Amer. Math. Soc. 43 (1983), no. 281, v+93. 

\bibitem{MajdaShock3} A. Majda, \textit{The stability of multidimensional shock fronts}, Mem. Amer. Math. Soc. 41 (1983), no. 275, iv+95 pp.



\bibitem{Metivier2001book}
G. M\'{e}tivier, \textit{Stability of multidimensional shocks, in Advances in the theory of shock waves}, Progr. Nonlinear Differential Equations Appl. 47, Birkhäuser, 2001, 25–103.

\bibitem{Rauch} J. Rauch, \textit{BV estimates fail for most quasilinear hyperbolic systems in dimensions greater than one}, Comm. Math. Phys. 106 (1986), no. 3, 481–484.

\bibitem{Riemann} B. Riemann, \textit{\"{U}ber die Fortpflanzung ebener Luftwellen von endlicher Schwingungsweite}, Abh. Ges. Wiss. Göttingen 8 (1860), 43–65.

\bibitem{Speck} J. Speck, \textit{A New Formulation of the 3D Compressible Euler Equations with Dynamic Entropy: Remarkable Null Structures and Regularity Properties},  Arch. Ration. Mech. Anal. 234 (2019), no. 3, 1223–1279. 

\bibitem{Smoller} J. Smoller, \textit{Shock Waves and Reaction-Diffusion Equations}, Second edition, Grundlehren der Mathematischen Wissenschaften, Vol. 258, Springer-Verlag, New York, 1994.

\bibitem{SpeckYu} Jared Speck and Sifan Yu. \textit{Characteristic initial value problem for the 3{D} compressible{E}uler equations.}, in preparation.

\bibitem{LiTT} T. T. Li and W. C. Yu, "Boundary Value Problems for Quasilinear Hyperbolic Systems," Duke University, Durham. 

\bibitem{WangYin} Z. Wang and H. Yin, \textit{Local structural stability of a multidimensional centered rarefaction wave for the three-dimensional steady supersonic Euler flow around a sharp corner}, SIAM J. Math. Anal. 42 (2010), no. 4, 1639–1687.

\bibitem{Wang} Q. Wang, \textit{On global dynamics of $3$-D irrotational compressible fluids}, https://arxiv.org/abs/2407.13649.



\bibitem{QuXiang} Qu, A, Xiang, W, \textit{Three-Dimensional Steady Supersonic Euler Flow Past a Concave Cornered Wedge with Lower Pressure at the Downstream}, Arch Rational Mech Anal 228, 431–476 (2018). https://doi.org/10.1007/s00205-017-1197-x

\bibitem{wang2025constructingcharacteristicinitialdata} Y.-X. Wang and S. Yu and P. Yu, \textit{Constructing characteristic initial data for three dimensional compressible Euler equations}, https://arxiv.org/abs/2508.15199.

\bibitem{ZhangRuotong} R.-T. Zhang, \textit{Local and Global Results on Three Dimensional Rarefaction Waves in Spherical Symmetry}, https://arxiv.org/abs/2512.00353.



\end{thebibliography}
\end{document}